\documentclass[11pt]{amsart}
\newcommand{\preprint}[1]{}

\newcommand{\hide}[1]{}

\usepackage{amssymb}
\usepackage{amsbsy}
\usepackage{amscd}
\usepackage{amsmath}
\usepackage{amsthm}
\usepackage{enumerate}
\input xy
\xyoption{all}

\numberwithin{equation}{section}

\theoremstyle{plain}
\newtheorem{thm}{Theorem}[section]
\newtheorem{prop}[thm]{Proposition}
\newtheorem{claim}[thm]{Claim}

\newtheorem{conj}[thm]{Conjecture}
\newtheorem{cor}[thm]{Corollary}
\newtheorem{lem}[thm]{Lemma}
\newtheorem{assumption}[thm]{Assumption}

\theoremstyle{definition}
\newtheorem{defi}[thm]{Definition}

\theoremstyle{remark}
\newtheorem{example}[thm]{Example}
\newtheorem{question}[thm]{Question}
\newtheorem{rem}[thm]{Remark}

\newcommand{\A}{{\mathcal A}}

\newcommand{\C}{{\mathcal C}}
\newcommand{\CC}{{\mathbb C}}

\newcommand{\F}{{\mathcal F}}
\newcommand{\G}{{\mathcal G}}

\newcommand{\HH}{{\mathbb H}}

\newcommand{\K}{{\mathcal K}}

\newcommand{\LB}{{\mathcal L}}

\newcommand{\fM}{{\mathfrak M}}

\renewcommand{\P}{{\mathcal P}}
\newcommand{\PP}{{\mathbb P}}

\newcommand{\QQ}{{\mathbb Q}}

\newcommand{\RR}{{\mathbb R}}
\renewcommand{\S}{{\mathcal S}}
\newcommand{\fS}{{\mathfrak S}}

\newcommand{\U}{{\mathcal U}}

\newcommand{\X}{{\mathcal X}}
\newcommand{\Y}{{\mathcal Y}}

\newcommand{\ZZ}{{\mathbb Z}}
\newcommand{\RealNumbers}{{\mathbb R}}
\newcommand{\Integers}{{\mathbb Z}}
\newcommand{\ComplexNumbers}{{\mathbb C}}
\newcommand{\RationalNumbers}{{\mathbb Q}}
\newcommand{\LieAlg}[1]{{\mathfrak #1}}
\newcommand{\linsys}[1]{{\mid}#1{\mid}}
\newcommand{\LLV}{{LLV}}
\renewcommand{\mod}{{\rm mod}}
\newcommand{\Mon}{{\rm Mon}}
\newcommand{\monrep}{{mon}}

\newcommand{\RightArrowOf}[1]{\stackrel{#1}{\rightarrow}}
\newcommand{\LeftArrowOf}[1]{\stackrel{#1}{\leftarrow}}
\newcommand{\LongLeftArrowOf}[1]{\stackrel{#1}{\longleftarrow}}

\newcommand{\LongRightArrowOf}[1]{\stackrel{#1}{\longrightarrow}}

\newcommand{\StructureSheaf}[1]{{\mathcal O}_{#1}}

\newcommand{\restricted}[2]{#1_{\mid_{#2}}}

\newcommand{\rank}{{\rm rank}}

\newcommand{\Pic}{{\rm Pic}}
\newcommand{\Sym}{{\rm Sym}}
\newcommand{\Ext}{{\rm Ext}}
\newcommand{\Tor}{{\mathcal T}or}
\newcommand{\Hom}{{\rm Hom}}
\newcommand{\Aut}{{\rm Aut}}
\newcommand{\End}{{\rm End}}

\newcommand{\SheafHom}{{\mathcal H}om}
\newcommand{\SheafEnd}{{\mathcal E}nd}

\newcommand{\SheafExt}{{\mathcal E}xt}

\newcommand{\Spin}{{\rm Spin}}

\newcommand{\Ideal}[1]{{\mathcal I}_{#1}}

\newcommand{\Choose}[2]{\left(\!\!\begin{array}{c}#1\\#2\end{array}\!\!\right)}

\renewcommand{\span}{{\rm span}}
\newcommand{\fine}{{\rm fine}}
\renewcommand{\div}{{\rm div}}
\newcommand{\obs}{{\rm obs}}

\begin{document}
\title[Vector bundles  on a hyper-K\"{a}hler manifold with a rank $1$ obstruction map]{Stable vector bundles on a hyper-K\"{a}hler manifold with a rank $1$ obstruction map are modular
}
\author{Eyal Markman}
\address{Department of Mathematics and Statistics, 
University of Massachusetts, Amherst, MA 01003, USA}
\email{markman@math.umass.edu}

\date{\today}


\begin{abstract}
Let $X$ be an irreducible $2n$-dimensional holomorphic symplectic manifold. 
A reflexive sheaf $F$ is {\em very  modular}, if its Azumaya algebra $\SheafEnd(F)$ deforms with $X$ to every K\"{a}hler deformation of $X$. We show that if $F$ is a slope-stable reflexive sheaf of positive rank and the obstruction map $HH^2(X)\rightarrow \Ext^2(F,F)$ has rank $1$, then $F$ is very modular. We associate to such a sheaf a vector in the Looijenga-Lunts-Verbitsky lattice of rank $b_2(X)+2$. Three sources of examples of such modular sheaves emerge. 
The first source consists of slope-stable reflexive sheaves $F$ of positive rank which are isomorphic to the image
$\Phi(\StructureSheaf{X})$ of the structure sheaf via an equivalence 
$\Phi:D^b(X)\rightarrow D^b(Y)$ of the derived categories of two irreducible holomorphic symplectic manifolds.
The second source consists of such $F$, which are isomorphic to the image of a sky-scraper sheaf via a derived equivalence. 
The third source consists of images $\Phi(L)$ of torsion sheaves $L$ supported as line bundles on holomorphic lagrangian submanifolds $Z$, 
such that $Z$ deforms with $X$ in co-dimension one in moduli and $L$ is a rational power of the canonical line bundle of $Z$.

An example of the first source is constructed using a stable and rigid vector bundle $G$ on a $K3$ surface $X$ to get the very modular vector bundle $F$ on the Hilbert scheme $X^{[n]}$, associated to
the equivariant vector bundle $G\boxtimes \cdots \boxtimes G$ on $X^n$ via the BKR-correspondence. 
This builds upon and partially generalizes results of K. O'Grady for $n=2$ \cite{ogrady-modular}. A construction of the second source associates to a set $\{G_i\}_{i=1}^n$ of $n$ distinct stable vector bundles  in the same two dimensional moduli space of vector bundles on a $K3$ surface $X$ the very modular vector bundle $F$ on $X^{[n]}$ corresponding to the 
equivariant bundle $\oplus_{\sigma\in \fS_n}[G_{\sigma(1)}\boxtimes \cdots \boxtimes G_{\sigma(n)}]$ on $X^n$. 
\end{abstract}

\maketitle
\tableofcontents

%
\section{Introduction}
\label{sec-introduction}

%
\subsection{The obstruction map}
%
\subsubsection{The obstruction map of objects in the derived category}
An {\em irreducible holomorphic symplectic manifold} is a simply connected compact K\"{a}hler manifold $X$ such 
that $H^0(X,\Omega^2_X)$ is one-dimensional spanned by an everywhere non-degenerate holomorphic $2$-form.
If $X$ is $2$-dimensional, then it is a $K3$ surface. It is well known that if $E$ is a slope-stable vector bundle 
over a $K3$ surface $X$, then the projective bundle $\PP(E)$ deforms with $X$ to every complex deformation of $X$
\cite{verbitsky-hyperholomorphic}. The analogous statement fails if $\dim(X)>2$ even to first order, since $c_2(\SheafEnd(E))$
need not remain of Hodge type $(2,2)$. The starting point of this paper is the observation that an analogous statement holds for torsion free reflexive sheaves satisfying the  infinitesimal condition that the obstruction homomorphism
\begin{equation}
\label{eq-obs-introduction}
\obs_E:HH^2(X)\rightarrow \Ext^2(E,E)
\end{equation}
has rank $1$ (Theorem \ref{thm-introduction-image-via-FM-is-modular}). Above, $HH^2(X)$ is the second Hochschild cohomology of $X$, considered as the space of natural transformations from the identity functor $id:D^b(X)\rightarrow D^b(X)$ to its shift $id[2]$, and $\obs_E$ sends 
a natural transformation $\alpha$ to its value $\alpha_E$ on $E$. 
The space $HH^2(X)$ parametrizes first order deformations of the category $\mbox{Coh}(X)$ of coherent sheaves on $X$ 
and those yield also first order deformations of the derived category $D^b(X)$ \cite{toda}. The class
$\obs_E(\xi)$
is the obstruction to lifting $\xi\in HH^2(X)$ to a first order deformation of the pair $(D^b(X),E)$  (see \cite[Prop. 6.1]{toda}). 

Given an object $F\in D^b(X)$ of non-zero rank $r$, set $\kappa(F):=ch(F)\exp\left(\frac{-c_1(F)}{r}\right)$. 
If $F$ is a locally free coherent sheaf, then  $\kappa(F)$ is a characteristic class of the projective bundle $\PP(F)$. 
We first observe that if $X$ is a projective irreducible holomorphic symplectic manifold and $F$ is an object in $D^b(X)$ with $\obs_F$ of rank $1$, then $\kappa(F)$ remains a Hodge class for all K\"{a}hler deformations of $X$ (see Lemma \ref{lemma-alpha-remains-of-Hodge-type}).

Given an equivalence of derived categories $\Phi:D^b(X)\rightarrow D^b(Y)$, we have a commutative diagram
\[
\xymatrix{
HH^2(X) \ar[r]^{\Phi_*} \ar[d]_{\obs_F} & HH^2(Y) \ar[d]_{\obs_{\Phi(F)}}
\\
\Hom(F,F[2]) \ar[r]_-{\Phi} & \Hom(\Phi(F),\Phi(F)[2]),
}
\]
where the horizontal arrows are isomorphisms, and so the rank of $\obs_{\Phi(F)}$ and $\obs_F$ are equal. If $\rank(F)\neq 0$,
then the restriction of $\obs_F$ to the subspace $H^0(X,\StructureSheaf{X})$ of $HH^2(X)$ is injective, by \cite[Lemma 10.1.3]{huybrechts-lehn}. Thus, $\rank(\obs_F)\geq 1$ for every object $F$ such that  $\rank(\Phi(F))\neq 0$, for some equivalence of derived categories $\Phi:D^b(X)\rightarrow D^b(Y)$. 

%
\subsubsection{Objects with a rank $1$ obstruction map}
Three sources\footnote{The rank of $\obs_F$ is $1$ for every $\PP^n$-object $F$, since then $\dim\Hom(F,F[2])=\dim H^2(\PP^n)=1$. However, all examples of $\PP^n$-objects we are aware of arise via one of sources (1) or (3) listed via equivalences of derived categories and deformations.} of objects $F$ with $\obs_F$ of rank $1$ emerge:
\begin{enumerate}
\item
\label{item-source-1}
$F=\Phi(\StructureSheaf{X})$, where $\Phi:D^b(X)\rightarrow D^b(Y)$ is an equivalence of derived categories of two projective irreducible holomorphic symplectic manifolds. Note that $\Ext^2(\StructureSheaf{X},\StructureSheaf{X})$ is one dimensional and so $\obs_F$ has rank $1$.
\item
\label{item-source-2}
$F$ is the image of a sky-scraper sheaf of a point via an equivalence $\Phi$ as above (See Section \ref{sec-proof-of-modularity}).
\item
\label{item-source-3}
$F$ is the image $\Phi(\iota_*L)$ via an equivalence $\Phi$ as above, where $\iota:Z\hookrightarrow X$ 
is the embedding of $Z$ as a holomorphic lagrangian submanifold and $L$ is a
line bundle on $Z$, 
such that 
\begin{enumerate}
\item
$\iota^*:H^{1,1}(X)\rightarrow H^{1,1}(Z)$ has rank $1$, and  
\item
$L$ is a rational power of the canonical line bundle of $Z$.
\end{enumerate}
\end{enumerate}

Sources (\ref{item-source-1}) and (\ref{item-source-2}) above seem to be special cases of deformations of source (\ref{item-source-3}). If, for exampe, $X$ is the Hilbert scheme $S^{[n]}$ of $n$ points on a $K3$ surface $S$ and $S$ contains a smooth rational curve $C$, 
then $\StructureSheaf{S^{[n]}}$ is the image via an auto-equivalence of $D^b(S^{[n]})$ of the structure sheaf $\StructureSheaf{C^{[n]}}$ of the lagrangian projective space $C^{[n]}$ (see Example \ref{example-non-modular-P-n-objects}). Similarly, assume that $S$ admits an elliptic fibration
with a section and $C_i\subset S$, $1\leq i \leq n$, are $n$ distinct smooth fibers. The sky-scraper sheaf $\CC_p$ supported on the origin $p$ of the lagrangian abelian variety
 $A:=\prod_{i=1}^nC_i\subset S^{[n]}$ is the image of $\StructureSheaf{A}$ via an auto-equivalence of $D^b(S^{[n]})$.
 We prove the following statement.
 
 \begin{thm} (Theorem \ref{thm-support-is-either-lagrangian-or-point})
 Let $\iota:Z\rightarrow X$ be an embedding of a variety $Z$ as a proper subvariety of an irreducible holomorphic symplectic manifold $X$.  Let $F$ be a torsion free coherent sheaf on $Z$. If $\obs_{\iota_*F}$ has rank $1$, then $\iota(Z)$ is either a lagrangian subvariety or a point.
 \end{thm}
 
%
\subsubsection{Deformability of sheaves with a rank $1$ obstruction map}
The following theorem is a consequence of Verbitsky's results on hyperholomorphic sheaves in \cite{kaledin-verbitsky-book}  and the fact that $\kappa(F)$ remains a Hodge class under all deformations for a sheaf $F$ with a rank $1$ obstruction homomorphism $\obs_F$.

\begin{thm} 
\label{thm-introduction-image-via-FM-is-modular}
(Theorem
\ref{thm-modularity-of-a-stable-sheaf-with-a-rank-1-obstruction-map}) 
Let $F$ be a reflexive sheaf  of rank $r>0$ over an irreducible holomorphic symplectic manifold $X$. Assume that  $\obs_F$ has rank $1$ and that $F$ is $H$-slope-stable with respect to some
ample line-bundle $H$.
Then for every irreducible holomorphic symplectic manifold $Y$ deformation equivalent to $X$ there exists a possibly twisted coherent sheaf $E$ over $Y$ and a flat 
deformation of the pair $(X,F)$ to the pair (Y,E).
\end{thm}

Following is a sketch of the proof of Theorem \ref{thm-introduction-image-via-FM-is-modular}. Results of Verbitsky 
in \cite{kaledin-verbitsky-book}
reduce the proof to showing that $c_2(\SheafEnd(F))$ remains of Hodge type $(2,2)$ under all deformations of $X$ as a K\"{a}hler manifold. 
The latter reduces to showing that $c_2(\SheafEnd(F))$ remains of Hodge type $(2,2)$ under all first order infinitesimal complex deformations of $X$, by Lemma \ref{lemma-alpha-remains-of-Hodge-type}. 
Set $HT^2(X):=H^0(X,\wedge^2TX)\oplus H^1(X,TX)\oplus H^2(X,\StructureSheaf{X})$. 
The Hochschild-Kostant-Rosenberg isomorphism
$I^{HKR}:HT^2(X)\rightarrow HH^2(X)$ composes with $\obs_F$ to a rank $1$ homomorphism 
\[
\obs^{HT}_F:=\obs_F\circ I^{HKR} : HT^2(X)\rightarrow \Ext^2(F,F).
\]
The class $c_2(\SheafEnd(F))$ remains of Hodge type in the direction of $\xi\in H^1(X,TX)$, if and only if there exists a class 
$\theta\in H^2(X,\StructureSheaf{X})$, such that $F$ deforms to first order in the direction of $(0,\xi,\theta)\in HT^2(X)$. Toda proved that this is the case, if and only if $(0,\xi,\theta)\in HT^2(X)$ belongs to the kernel of 
$\obs^{HT}_F$ (here we use also a result of \cite{huang} that $\obs^{HT}_F$ is the obstruction map used by Toda,
see Lemma \ref{lem-Sigma-F-is-contained-in-stabilizer-of-Mukai-vector}). Now, if the rank of $F$ does not vanish, then 
 $\obs^{HT}_F$ maps $H^2(\StructureSheaf{X})$  injectively into $\Ext^2(F,F)$.
Thus, the intersection of $\ker(\obs^{HT}_F)$ and $H^1(X,TX)\oplus H^2(\StructureSheaf{X})$ 
must project onto $H^1(X,TX)$. Hence, every $\xi\in H^1(X,TX)$ lifts to a class $(0,\xi,\theta)$ in $\ker(\obs^{HT}_F)$.

\begin{rem}
Theorem \ref{thm-introduction-image-via-FM-is-modular} holds also for abelian surfaces (Remark \ref{rem-semi-homogeneous}). If $F$ is an object in the derived category $D^b(X)$ of an abelian variety $X$ of dimension $\geq 3$, then the minimal rank of $\obs_F$ is $h^{2,0}(X):=\dim(H^2(\StructureSheaf{X}))$. Slope-stable  vector bundles $F$ with $\obs_F$ of rank $h^{2,0}(X)$ are semi-homogeneous (see Lemma \ref{lemma-semi-homogeneous}). These were extensively studied by Mukai \cite{mukai-semihomogeneous}.
\end{rem}

%
\subsection{Examples of very modular vector bundles}
Theorems \ref{thm-BKR-of-tensor-product-of-spherical-object} and \ref{thm-Fourier-Mukai-images-of-sky-scraper-sheaves} below exhibit two examples of sheaves $F$ satisfying the conditions of the above Theorem.

%
\subsubsection{Very modular $\PP^n$-objects  over IHSMs of $K3^{[n]}$-type}
Let $(S,H)$ be a generic  polarized $K3$ surface, let $G$ be an $H$-slope-stable locally free sheaf on $S$ with vanishing
$\Ext^1(G,G)$. Let the symmetric group $\fS_n$ act on $S^n$ permuting the factors. 
The Hilbert scheme $S^{[n]}$ of length $n$ zero dimensional subschemes of $S$ is isomorphic to an irreducible component of the $\fS_n$-Hilbert scheme of length $n!$ zero-dimensional $\fS_n$-invariant subschemes $Z$ of $S^n$, such that $H^0(\StructureSheaf{Z})$ is the regular representation of $\fS_n$ \cite{haiman}. There exists a universal 
$\fS_n$-invariant subscheme $\Gamma\subset S^{[n]}\times S^n$ with a $\fS_n$-linearization $\rho$, such that the  object $(\StructureSheaf{\Gamma},\rho)$ is the Fourier-Mukai kernel of an equivalence
$BKR:D^b_{\fS_n}(S^n)\rightarrow D^b(S^{[n]})$ between the
$\fS_n$-equivariant derived category of $S^n$ and the derived category of $S^{[n]}$ \cite{BKR}. 
The vector bundle $G^{\boxtimes n}:=G\boxtimes G \boxtimes \cdots \boxtimes G$ over $S^{[n]}$ comes with a canonical $\fS_n$-linearizarion $\rho_\boxtimes$. 
The object $G^{[n]}:=BKR((G^{\boxtimes n},\rho_\boxtimes))$ is isomorphic to a rank $r^n$ locally free sheaf over $S^{[n]}$. Now, 
$\Ext^{2i}_{\fS_n}((G^{\boxtimes n},\rho_\boxtimes),(G^{\boxtimes n},\rho_\boxtimes))$
is $1$-dimensional, for $0\leq i\leq n$, hence $\Ext^{2i}(G^{[n]},G^{[n]})$ is one-dimensional as well, for $0\leq i\leq n$,  and so 
$\obs_{G^{[n]}}$ has rank $1$. 

\begin{thm} 
\label{thm-BKR-of-tensor-product-of-spherical-object}
For every irreducible holomorphic symplectic manifold $X$ of $K3^{[n]}$ deformation type there exists a possibly twisted locally free sheaf $E$ and a deformation of the pair $(S^{[n]},G^{[n]})$ to $(X,E)$. $E$ is a $\PP^n$-object, i.e., $\Ext^i(E,E)\cong H^i(\PP^n,\CC)$, and in particular $E$ is
infinitesimally rigid.  $E$ is slope-stable with respect to classes in an open subcone of the K\"{a}hler cone of $X$. 
\end{thm}

The theorem is proved in section \ref{sec-stability-of-images-of-structure-sheaf}.
The theorem was proved by O'Grady in case $n=2$, where it was further proven that the moduli space of vector bundles with the same Chern classes as $E$ is a single point \cite{ogrady-modular}.
O'Grady proved that the Chern class $c_2(\SheafEnd(G^{[2]}))$ remains of Hodge type under all K\"{a}hler deformations of $S^{[2]}$ by explicitly computing this class.
The current paper emerged from the author's attempt to provide a conceptual explanation to O'Grady's intriguing result.

\begin{rem}
\label{rem--is deformation-of-FM-image-of-structure-sheaf}
The pair  $(S,G)$ in Theorem \ref{thm-BKR-of-tensor-product-of-spherical-object} is related, via a finite sequence 
of deformations and derived equivalences, to the pair $(T,\StructureSheaf{T})$, for some $K3$ surface $T$. 
More precisely, there are two  sequences $(S_i,G_i)$, $0\leq i\leq N$, and $(S'_i,G'_i)$, $0\leq i\leq N-1$,
and derived equivalences $\Phi_i:D^b(S'_i)\rightarrow D^b(S_{i+1})$, $0\leq i\leq N-1$, such that
$(S_0,G_0)=(S,G)$, $(S_N,G_N)=(T,\StructureSheaf{T})$, 
$(S'_i,G'_i)$ is a deformation of $(S_i,G_i)$, and $G_{i+1}$ is a slope-stable locally free sheaf isomorphic to $\Phi_i(G'_i)$.
Now, $\StructureSheaf{T}^{[n]}=\StructureSheaf{T^{[n]}}$ and the BKR-conjugate $\Phi_i^{[n]}:D^b((S'_i)^{[n]})\rightarrow D^b(S_{i+1}^{[n]})$ of  the derived equivalence 
$\Phi_i^{\boxtimes n}:D^b_{\fS_n}((S'_i)^n)\cong D^b_{\fS_n}(S_{i+1}^n)$ is an equivalence mapping
$(G'_i)^{[n]}$ to $G_{i+1}^{[n]}$.
Hence, the pair $(S^{[n]},G^{[n]})$ is related to the pair $(T^{[n]},\StructureSheaf{T^{[n]}})$ via a sequence of deformations and derived equivalences.
\hide{
The pair  $(S,G)$ can be deformed to a pair $(S_0,G_0)$, such that $G_0$ 
is isomorphic to the image of $\StructureSheaf{T}$ via an equivalence of  derived categories $\Phi:D^b(T)\rightarrow G^b(S_0)$, for some $K3$ surface $T$.
Then $G_0^{[n]}$ is the image of $\StructureSheaf{T^{[n]}}$ via the BKR-conjugate $\Phi^{[n]}:D^b(T^{[n]})\rightarrow D^b(S_0^{[n]})$ of the derived equivalence 
$\Phi^{\boxtimes n}:D^b_{\fS_n}(T^n)\cong D^b_{\fS_n}(S_0^n)$. Hence, the pair $(S^{[n]},G^{[n]})$ is a deformation of $(S_0^{[n]},\Phi^{[n]}(\StructureSheaf{T^{[n]}}))$.
}
\end{rem}

%
\subsubsection{A very modular vector bundle on $S^{[n]}$ from $n$ vector bundles on $S$}
Let $(S,H)$ be a  polarized $K3$ surface and $M_H(v)$ a $2$-dimensional projective moduli space of
$H$-slope-stable locally free sheaves  over $S$ with Mukai vector $v$ of rank $r$. Then $M_H(v)$ is a $K3$ surface
\cite{mukai-hodge}.
Let $\U$ be a universal vector bundle over $S\times M_H(v)$, possibly twisted by a Brauer class $\theta$.
Then the Fourier-Mukai transformation $\Phi_\U:D^b(M_H(v),\theta^{-1})\rightarrow D^b(S)$ is an equivalence \cite[Theorem 1.3]{caldararu-non-fine}.
Choose $n$ distinct points $[G_i]$ in $M_H(v)$ corresponding to $n$ vector bundles $G_i$, $1\leq i\leq n$.
The vector bundle $G:=\oplus_{\sigma\in\fS_n}G_{\sigma(1)}\boxtimes G_{\sigma(2)}\boxtimes \cdots \boxtimes  G_{\sigma(n)}$
admits a $\fS_n$-linearization $\rho$ and any two linearizations yield isomorphic objects $(G,\rho)$ (Remark \ref{rem-chi-invariant-objects}).
The object $BKR(G,\rho)$ is isomorphic to a rank $(n!)r^n$ locally free sheaf $F$ over $S^{[n]}$. 
Let $z\in (M_H(v))^{[n]}$ be the point corresponding to the length $n$ subscheme of $M_H(v)$ supported on the set 
$\{[G_i] \ : \ 1\leq i\leq n\}.$ Let $\Phi_\U^{[n]}:D^b(M_H(v))^{[n]},(\theta^{-1})^{[n]})\rightarrow D^b(S^{[n]})$ be the BKR conjugate of 
$\Phi_\U^{\boxtimes n}:D^b_{\fS_n}((M_H(v))^n,(\theta^{-1})^{\boxtimes n})\rightarrow D^b_{\fS_n}(S^n)$.
$F$ is isomorphic to the image $\Phi_\U^{[n]}(\CC_z)$ of the sky-scraper sheaf $\CC_z$ (see Eq. (\ref{eq-F-z-is-image-of-sky-scpaper-via-Phi[n]}) for details). Hence,
$\obs_F$ has rank $1$. 

\begin{thm}
\label{thm-Fourier-Mukai-images-of-sky-scraper-sheaves}
For every irreducible holomorphic symplectic manifold $X$ of $K3^{[n]}$ deformation type there exists a possibly twisted locally free sheaf $E$ and a deformation of the pair $(S^{[n]},F)$ to $(X,E)$. $E$ is slope-stable with respect to classes in an open subcone of the K\"{a}hler cone of $X$. 
\end{thm}

The theorem is proved in Section \ref{sec-modular-vector-bundles-with-isotropic-LLV-line}. 
A criterion for the existence of an untwisted sheaf $E$ in Theorems \ref{thm-BKR-of-tensor-product-of-spherical-object} and \ref{thm-Fourier-Mukai-images-of-sky-scraper-sheaves} is provided in Section \ref{sec-lift-to-an-untwisted-sheaf}.

%
\subsection{The LLV line of an object with a rank $1$ obstruction map}
%
\subsubsection{The LLV-line of an object $F$ with $\rank(\obs_F)=1$}
We do not know if the property that $\obs_F$ has rank $1$ is an open condition. A priori, the rank of $\obs_F$ may increase after a deformation of $(X,F)$. However, the following theorem lists topological consequences of the fact that the rank of $\obs_F$ is one, and those are clearly deformation invariant.

Let $X$ be a $2n$-dimensional irreducible holomorphic symplectic manifold. 
Consider the vector space
\[
\widetilde{H}(X,\QQ) := \QQ\alpha\oplus H^2(X,\QQ)\oplus\QQ\beta,
\]
endowed with the symmetric bilinear pairing with respect to which $\alpha$ and $\beta$ are isotropic, $(\alpha,\beta)=-1$, $U:=\span\{\alpha,\beta\}$
is orthogonal to $H^2(X,\QQ)$, and the pairing restricts to the Beauville-Bogomolov-Fujiki pairing on $H^2(X,\QQ)$.
Verbitsky and Looijenga-Lunts constructed a canonical (topological) representation of the 
Lie algebra $\LieAlg{so}(\widetilde{H}(X,\QQ))$ on the cohomology $H^*(X,\QQ)$ (see \cite{looijenga-lunts,verbitsky-cohomology}). Denote by $\LieAlg{g}_X$ the image of $\LieAlg{so}(\widetilde{H}(X,\QQ))$ in 
$\LieAlg{gl}(H^*(X,\QQ))$. The subring $SH^*(X,\QQ)$ of $H^*(X,\QQ)$ generated by $H^2(X,\QQ)$ is isomorphic 
as a $\LieAlg{g}_X$-subrepresentation to the irreducible subrepresentation $V(n)$ 
of $\Sym^n(\widetilde{H}(X,\QQ))$ spanned by $n$-th powers of isotropic classes and $V(n)$  appears with multiplicity $1$ in $H^*(X,\QQ)$. Hence, we get a canonical projection from $H^*(X,\QQ)$ to $SH^*(X,\QQ)$.
Given an equivalence $\Phi:D^b(X)\rightarrow D^b(Y)$ between two projective irreducible holomorphic symplectic manifolds
we get an induced isomorphism $\phi:H^*(X,\QQ)\rightarrow H^*(Y,\QQ)$. 
Taelman proved that $\phi$ conjugates $\LieAlg{g}_X$ to $\LieAlg{g}_Y$ and the restriction of $\phi$ 
to $SH^*(X)$ is induced by an isometry 
$\tilde{\phi}:\widetilde{H}(X,\QQ)\rightarrow \widetilde{H}(Y,\QQ)$.
Furthermore, there exists a canonical
$\LieAlg{g}_X$-equivariant isomorphism
\[
\Psi : SH^*(X,\QQ)\rightarrow V(n),
\]
which is $(\phi,\tilde{\phi})$-equivariant, up to sign \cite[Theorems 4.7 and 4.9]{taelman}.
Given an object $F$ of $D^b(X)$, let $v(F):=ch(F)\sqrt{td_X}$ be its Mukai vector and let $\hat{v}(F)$ be the projection of $v(F)$ to $SH^*(X,\QQ)$. Note that $\hat{v}(F)$ does not vanish, if $\rank(F)\neq 0$ or if $F$ is a sky scraper sheaf, or if $F$ is the image of such an object via an equivalence of derived categories.

\begin{thm}
\label{thm-introduction-LLV-line}
(Theorem \ref{thm-Mukai-vector})
Let $F$ be an object of $D^b(X)$ with $\obs_F$ of rank $1$. Assume that $\hat{v}(F)$ does not vanish.
\begin{enumerate}
\item
\label{thm-item-stabilizers-of-Mukai-vector-equal-to-that-of-LLV-line}
The stabilizer $\LieAlg{g}_{v(F)}$ of $v(F)$ in $\LieAlg{g}_X$ is equal to the subalgebra of $\LieAlg{g}_X$
annihilating a line $\ell(F)$ in $\widetilde{H}(X,\QQ)$.
\item
The isomorphism $\Psi$ maps $\hat{v}(F)$ into the image of $\ell(F)^n$ via the projection from $\Sym^n(\widetilde{H}(X,\QQ))$ to $V(n)$.
\item
\label{thm-item-LLV-vector-for-non-zero-rank}
If the rank $r$ of $F$ does not vanish, then 
$\ell(F)=\span\{r\alpha+c_1(F)+s\beta\}$, for some rational number $s$.
\item
Given an equivalence $\Phi:D^b(X)\rightarrow D^b(Y)$ between two irreducible holomorphic symplectic manifolds,
we have $\ell(\Phi(F))=\tilde{\phi}(\ell(F))$.
\end{enumerate}
\end{thm}

Theorem \ref{thm-introduction-LLV-line} shows that 
$\widetilde{H}(X,\QQ)$ plays a role analogous to that of the Mukai lattice of a $K3$ surface in the classification of objects $F$ with $\obs_F$ of rank $1$. Note that the rank of $\obs_F$ is $1$ for every simple sheaf $F$ over a $K3$ surface, since $Ext^2(F,F)$ is $1$-dimensional. We refer to $\widetilde{H}(X,\QQ)$ as the {\em rational LLV lattice}.  When $X$ is of $K3^{[2]}$-type\footnote{
The author was notified by Thorsten Beckmann that he extended \cite[Theorem 9.8]{taelman} for every manifold of $K3^{[n]}$-type, for all $n\geq 1$. Partial results were obtained independently by the author and are included in this paper (Remark \ref{rem-integral-LLV-lattice}).
} 
there exists an integral lattice in $\widetilde{H}(X,\QQ)$, which is preserved by equivalences of derived categories and by monodromy operators \cite[Theorem 9.8]{taelman}. 

The proof of Theorem \ref{thm-introduction-LLV-line} proceeds as follows. The Hodge decomposition exhibits $H^*(X,\CC)$
as a module over the ring $HT^*(X):=\oplus_{a,b}H^a(X,\wedge^bTX)$. The action homomorphism
$m:HT^*(X)\rightarrow \LieAlg{gl}(H^*(X,\CC))$ maps $HT^2(X)$ into the LLV Lie algebra $\LieAlg{g}_\CC:=\LieAlg{g}_X\otimes_\QQ\CC,$ by \cite{verbitsky-mirror-symmetry} (see also \cite[Prop. 2.9]{taelman}).
Let  $\tilde{\Sigma}(F)$ be the image of $\ker\left(\obs^{HT}_F\right)$ via the graded Duflo automorphism of $HT^*(X)$ given by contraction with $\sqrt{td_X}$.
Then $m:HT^2(X)\rightarrow \LieAlg{g}_\CC$ maps  $\tilde{\Sigma}(F)$ into the stabilizer $\LieAlg{g}_{\CC,v(F)}$ in $\LieAlg{g}_\CC$ of the Mukai vector $v(F)$ of $F$ (see \cite{huang}, or Lemma
\ref{lem-Sigma-F-is-contained-in-stabilizer-of-Mukai-vector}). We prove that $m(\tilde{\Sigma}(F))$ and its complex conjugate
generate $\LieAlg{g}_{\CC,v(F)}$,  if the rank of $\obs_F$ is $1$  (Lemma \ref{lemma-m-Sigma-and-its-conjugate-generate-the-annihilator-of-ell}).
Composing $m$ with the action of $\LieAlg{g}_\CC$ on $\widetilde{H}(X,\CC)$ yields 
a homomorphism from $HT^2(X)$ into $\LieAlg{gl}(\widetilde{H}(X,\CC))$. Evaluation at 
a holomorphic $2$-form on $X$, considered as an element of $\widetilde{H}(X,\CC)$, gives rise to a homomorphism 
$\mu:HT^2(X)\rightarrow \widetilde{H}(X,\CC)$ which pulls back the pairing on $ \widetilde{H}(X,\CC)$ to one on $HT^2(X)$, canonical up to a scalar (Lemma \ref{lemma-formula-for-m-lambda}). We get the line $\tilde{\Sigma}(F)^\perp$ in $HT^2(X)$ orthogonal to $\tilde{\Sigma}(F)$
and we set $\ell(F):=\mu\left(\tilde{\Sigma}(F)^\perp\right)$. 
We show that 
$\ell(F)$ is defined over $\QQ$ and 
$\LieAlg{g}_{\CC,v(F)}$ is the complexification of the subalgebra of $\LieAlg{g}_X$
annihilating   $\ell(F)$  (Lemma \ref{lemma-pde}).

%
\subsubsection{Computation of LLV-lines}
The integral $\int_X\sqrt{td_X}$ for $X$ of $K3^{[n]}$ and generalized Kummer deformation types was computed in \cite[Prop. 19 and 21]{sawon-thesis}. We use Sawon's computations to show that
$\ell(\StructureSheaf{X})=\span\{4\alpha+(n+3)\beta\}$, if $X$ is of $K3^{[n]}$ deformation type, 
and $\ell(\StructureSheaf{X})=\span\{4\alpha+(n+1)\beta\}$, if $X$ is of generalized Kummer deformation type (Lemma
\ref{lemma-LLV-line-of-structure-sheaf-k3-type}). 

It is sometimes challenging to prove that $\obs_F$ has rank $1$, yet possible to prove 
the property in Theorem \ref{thm-introduction-LLV-line}(\ref{thm-item-stabilizers-of-Mukai-vector-equal-to-that-of-LLV-line}). We prove it in Lemma \ref{lemma-chern-character-of-Lagrangian-structure-sheaf-deforms-in-co-dimension-1} for 
the structure sheaf $\StructureSheaf{Z}$ of a smooth lagrangian surface $Z$ in $X$ of $K3^{[2]}$-type, when the restriction homomorphism $H^2(X,\QQ)\rightarrow H^2(Z,\QQ)$ has rank $1$ and $c_1(\omega_Z)$ belongs to its image.
If $F$ satisfies the property in Theorem \ref{thm-introduction-LLV-line}(\ref{thm-item-stabilizers-of-Mukai-vector-equal-to-that-of-LLV-line})
we say that 
$F$ {\em has a rank $1$ cohomological obstruction map}
(see Definition \ref{def-deforms-in-co-dimension-one} (\ref{def-item-remains-of-Hodge-type-in-co-dimension-1})).
The images of such objects $F$ via derived equivalences 
each has a rank $1$ cohomological obstruction map
as well. 
The modularity Theorem \ref{thm-introduction-image-via-FM-is-modular} above holds if we replace the property that $\obs_F$ has rank $1$ by the property that 
$F$ has a rank $1$ cohomological obstruction map.
The line $\ell(\StructureSheaf{Z})$ of the structure sheaf of a lagrangian submanifold $Z$ of $X$,  with 
$\StructureSheaf{Z}$  having a rank $1$ cohomological obstruction map,
is computed in Lemma \ref{lemma-Mukai-line-of-structure-sheaf-of-subcanonical-lagrangian} when $c_1(\omega_Z)$ belongs to the image of 
the pullback homomorphism $\iota^*:H^2(X,\QQ)\rightarrow H^2(Z,\QQ)$ via the embedding $\iota:Z\rightarrow X$.
The homomorphism $\iota^*$ necessarily has rank $1$ in this case and we have 
\[
\ell(\iota_*\StructureSheaf{Z})=\span\{\lambda+\frac{t(\lambda,\lambda)}{2}\beta\},
\]
where $c_1(\omega_Z)=t\iota^*(\lambda)$ and $\lambda^\perp=\ker(\iota^*)$  (Lemma
\ref{lemma-Mukai-line-of-structure-sheaf-of-subcanonical-lagrangian}).  

In Section \ref{sec-Effective-LLV-lines} we include a conjectural discussion of the following.

\begin{question}
Let $X$ be a projective irreducible holomorphic symplectic manifold. 
Which lines in $\widetilde{H}(X,\QQ)$ are of the form $\ell(F)$, for some object $F$
which has a rank $1$ cohomological obstruction map?
\end{question}

Very little is known even for $X$ of $K3^{[2]}$-deformation type. The question is related to the existence of lagrangian surfaces
in the generic polarized such $X$. When  $\iota:Z\rightarrow X$ is an embedding of a smooth lagrangian surface,
$\iota^*:H^2(X,\QQ)\rightarrow H^2(Z,\QQ)$ has rank $1$, and $c_1(\omega_Z)$ belongs to its image, then 
$\iota_*\iota^*L$ has a rank $1$ cohomological obstruction map
for every line bundle $L$ on $X$,
by Lemma \ref{lemma-chern-character-of-Lagrangian-structure-sheaf-deforms-in-co-dimension-1}.
We show that for a primitive polarization $L$, the existence of such a lagrangian surface, which locally deform with all deformations of $(X,L)$, implies that the Beauville-Bogomolov-Fujiki self-intersection $2d:=(c_1(L),c_1(L))$ of $L$ satisfies the following arithmetic constraint: $d/\gcd(3,d)$ {\em is a perfect square.}
See Lemma \ref{lemma-arithmetic-constraints} for a more detailed statement involving stronger arithmetic constraints. 
%
\subsection{The action of an auto-equivalence of a $K3$ surface $S$ on the LLV-lattice of $S^{[n]}$}
Equivalences of derived categories of irreducible holomorphic symplectic manifolds of $K3^{[n]}$ deformation type are morphisms of a groupoid $K3^{[n]}$, i.e., a category all of whose morphisms are isomorphisms. We have the functor ${}^{[n]}:K3^{[1]}\rightarrow K3^{[n]}$, sending a $K3$ surface $S$ to its Hilbert scheme $S^{[n]}$ and sending equivalences of derived categories 
$\Phi:S_1\rightarrow S_2$ to the equivalences $\Phi^{[n]}:D^b(S_1^{[n]})\rightarrow D^b(S_2^{[n]})$ as in Theorems \ref{thm-BKR-of-tensor-product-of-spherical-object} and \ref{thm-Fourier-Mukai-images-of-sky-scraper-sheaves} above  (see Section \ref{sec-BKR-equivalence} for the definition). 
Let  $\LLV$ be the groupoid, whose objects are irreducible holomorphic symplectic manifolds and whose morphisms are isometries of their rational LLV lattices. Taelman's work gives rise to a functor $\widetilde{H}:K3^{[n]}\rightarrow \LLV$.
The values of the composite functor $\widetilde{H}\circ ^{[n]}$ on morphisms are not equal to those one may naively expect. 
The composition $\widetilde{H}(\Phi^{[n]}):\widetilde{H}(S_1^{[n]},\QQ)\rightarrow \widetilde{H}(S_2^{[n]},\QQ)$ was computed  for every equivalence $\Phi:D^b(S_1)\rightarrow D^b(S_2)$ of derived categories of $K3$ surfaces and 
for $n=2$ in \cite[Theorem 9.4]{taelman}. We extend the computation for all $n$ in the following theorem as an application of Theorem \ref{thm-introduction-LLV-line} above. 

The space $H^2(S,\QQ)$ is naturally the subspace
of $H^2(S^{[n]},\QQ)$ orthogonal to the class $\delta$, which is half the class of the divisor of non-reduced subschemes.
The inclusion $H^2(S,\QQ)\subset H^2(S^{[n]},\QQ)$
extends to a natural isometric embedding of $\widetilde{H}(S,\QQ)$  onto 
the hyperplane $\delta^\perp$ in $\widetilde{H}(S^{[n]},\QQ)$ orthogonal to the class $\delta$.
Let $\eta_\Phi\in O(\widetilde{H}(S^{[n]},\QQ))$ be the extension of $\widetilde{H}(\Phi)$ leaving $\delta$ invariant. 
It turns out that $\widetilde{H}(\Phi^{[n]})$ is not
equal to $\eta_\Phi$ in general.
Associated to a class $\lambda\in H^2(S^{[n]},\QQ)$ is an isometry $B_\lambda$ of $\widetilde{H}(S^{[n]},\QQ)$,
which corresponds to the action of tensorization by a line bundle $L$ whenever $\lambda=c_1(L)$.

\begin{thm}
\label{thm-introduction-action-of-DMon-S-on-LLV-lattice}
(Theorem \ref{thm-action-of-DMon-S-on-LLV-lattice})
$\widetilde{H}(\Phi^{[n]})=\det\left(\widetilde{H}(\Phi)\right)^{n+1}(B_{-\delta/2} \circ \eta_\Phi\circ B_{\delta/2})$.
\end{thm}

Let $\chi$ be the sign character of $\fS_n$ and 
let $\Phi_\chi^{[n]}:D^b(S^{[n]})\rightarrow D^b(S^{[n]})$ be the $BKR$-conjugate of the auto-equivalence of
$D^b_{\fS_n}(S^n)$ of tensorization by the sign character. The proof of Theorem \ref{thm-introduction-action-of-DMon-S-on-LLV-lattice} is based on the observation that the isometries $\widetilde{H}(\Phi^{[n]})$ all commute with
$\widetilde{H}(\Phi_\chi^{[n]})$, and $\widetilde{H}(\Phi_\chi^{[n]})$ is, up to sign, the reflection of $\widetilde{H}(S^{[n]},\QQ)$
with respect to the hyperplane orthogonal to the class $B_{-\delta/2}(\delta)=\delta+(n-1)\beta$ 
(Lemma \ref{lemma-isometry-of-BKR-conjugate-of-tensorization-by-sign-character}).
%
\subsection{A reflexive sheaf with a positive LLV-line}
The vector bundle $E$ in Theorem \ref{thm-BKR-of-tensor-product-of-spherical-object} has a negative definite LLV line $\ell(E)$ and that in Theorem \ref{thm-Fourier-Mukai-images-of-sky-scraper-sheaves} has an isotropic LLV line. 
In Section \ref{sec-locally-free-FM-image-of-lagrangian-lb} we consider a particular example of a reflexive torsion free sheaf $E$, which 
 has a rank $1$ cohomological obstruction map
and with a positive LLV line $\ell(E)$.
Let $X$ be a cubic fourfold. The Fano variety of lines $F(X)$ on $X$ is holomorphic symplectic of $K3^{[2]}$-type, by \cite{beauville-donagi}. The Fano variety $F(C)$ of lines on a hyperplane section $C$ of $X$ is a lagrangian surface in $F(X)$
and $\StructureSheaf{F(C)}$ has a rank $1$ cohomological obstruction map,
by Lemma \ref{lemma-chern-character-of-Lagrangian-structure-sheaf-deforms-in-co-dimension-1}.
We exhibit  reflexive torsion free sheaves $E$, which are Fourier-Mukai images of 
$\StructureSheaf{F(C)}$. Each 
$E$ thus  has a rank $1$ cohomological obstruction map.
We expect $E$ to be slope-stable with respect to classes in an open subcone of the ample cone of $F(X)$. Once proven, $E$ would be very modular, by Theorem \ref{thm-introduction-image-via-FM-is-modular}. Results of  
\cite{LSV,voisin-tenfold} suggest 
that the irreducible component $M$ containing $E$ of the moduli space of slope-stable sheaves on $F(X)$  is birational to an irreducible holomorphic symplectic manifold 
of O'Grady10 deformation type. A possibly naive hope is that this would lead to a description of  irreducible holomorphic symplectic manifolds $M$ of O'Grady10 deformation type as moduli spaces of hyperholomorphic sheaves on holomorphic symplectic manifolds $Y$ of $K3^{[2]}$-deformation type. A program of constructing irreducible holomorphic symplectic manifolds as such moduli spaces is described in \cite{kaledin-verbitsky-book}. Such a modular description would enable the study of the cohomology ring of $M$ via a universal sheaf on $Y\times M$. A modular description is currently provided only for the complement of an exceptional divisor of a resolution $M$ of singularities of a moduli space of sheaves on a $K3$ surface \cite{ogrady10}. 

We refer the reader to the table of content for the organization of the paper. A summary of each section is included in the opening paragraph of the section. 

{\bf Acknowledgements:} 
This work is partially supported by  a grant  from the Simons Foundation (\#427110). I am grateful to Andrei C\u{a}ld\u{a}raru and
Shengyuan Huang for help with the proof of Lemma \ref{lem-Sigma-F-is-contained-in-stabilizer-of-Mukai-vector}. I thank Giulia Sacc\`{a} for pointing out references \cite{ferretti,iliev-manivel-fano-of-index10,logachev}. I am grateful to Justin Sawon for pointing out his results in \cite[Prop. 19 and 21]{sawon-thesis}. I thank Thorsten Beckmann for comments on an early draft of this paper. I thank Kieran O'Grady for pointing out that an assumption in the statement of Theorem \ref{thm-introduction-image-via-FM-is-modular} in an earlier version is automatically satisfied by a result in \cite{ogrady-modular}. I thank the referee for his insightful comments.

\medskip
{\bf Note:} Thorsten Beckmann independently 
assigned vectors in the LLV lattice to various objects in the derived categories of irreducible holomorphic symplectic manifolds, and he also proved Theorem \ref{thm-action-of-DMon-S-on-LLV-lattice} \cite{beckmann}. The work on objects with a rank $1$ cohomological obstruction map (Def. \ref{def-deforms-in-co-dimension-one}) was continued in \cite{beckmann-atomic}, where they 
are called {\em Atomic}.
%
\section{Real $(p,p)$-classes which remain of Hodge type to first order in all directions}
\label{sec-real-p-p-classes}
In Section \ref{sec-Mukai-lattice}
we review the action of the Looijenga-Lunts-Verbitsky Lie algebra on the cohomology of an irreducible holomorphic symplectic manifold $X$. In Section \ref{subsec-real-p-p-classes} we prove that if a real cohomology class of Hodge type $(p,p)$ on such $X$ remains a $(p,p)$ class under all first order classical deformations of $X$, then it remains a $(p,p)$ class under all K\"{a}hler deformations of $X$.
%
\subsection{The rational LLV lattice}
\label{sec-Mukai-lattice}

Let $U$ be the even unimodular rank $2$ lattice with basis $\{\alpha,\beta\}$ consisting of isotropic elements satisfying $(\alpha,\beta)=-1$. Set $U_\QQ:=U\otimes_\ZZ\QQ$.
Following Taelman \cite[Sec. 3.1]{taelman} we define 
the {\em rational LLV lattice}\footnote{Taelman calls this vector space the {\em Rational Mukai lattice}. We prefer to distinguish it from C\u{a}ld\u{a}raru's Mukai lattice defined in \cite{caldararu-I,caldararu-II}. LLV is short for Looijenga-Lunts-Verbitsky, as this vector space first appears in their work \cite{looijenga-lunts,verbitsky-cohomology}} 
of a $2n$-dimensional irreducible holomorphic symplectic manifold $X$ to be the orthogonal direct sum 
\begin{equation}
\label{eq-rational-LLV-lattice}
\tilde{H}(X,\RationalNumbers):=H^2(X,\RationalNumbers)\oplus U_\RationalNumbers,
\end{equation}
where the bilinear pairing on the first direct summand is the Beauville-Bogomolov-Fujiki (BBF)  pairing normalized so that
\[
\int_X\lambda^{2n}=c_X\frac{(2n)!}{2^nn!}(\lambda,\lambda)^n, \ \ \ \forall\lambda\in H^2(X,\RationalNumbers),
\]
where the Fujiki constant $c_X$ is a positive rational number such that the pairing $(\bullet,\bullet)$ is integral, primitive, and positive on K\"{a}hler classes \cite[Sec. 2.1]{ogrady-survey}. The Fujiki constant is calculated for all known irreducible holomorphic symplectic manifolds in \cite{rapagnetta}.
If $X$ is of $K3^{[n]}$-type, then $c_X=1$, and
if $X$ is a deformation equivalent to a generalized Kummer of dimension $2n$, then $c_X=n+1$. 
We endow $\tilde{H}(X,\RationalNumbers)$ with a weight $0$ Hodge structure extending that of $H^2(X,\QQ(1))$ so that classes in $U_\QQ$ are of type $(0,0)$.

Set $\LieAlg{g}:=\LieAlg{so}(\tilde{H}(X,\RationalNumbers))$. Let $\bar{\LieAlg{g}}$ be the subalgebra $\LieAlg{so}(H^2(X,\RationalNumbers))$ of $\LieAlg{g}$ annihilating $U_\RationalNumbers$.
Given a class $\lambda\in H^2(X,\RationalNumbers)$ denote by $e_\lambda$ the element of $\LieAlg{g}$ given by 
\begin{equation}
\label{eq-e-lambda}
e_\lambda(\alpha)=\lambda, \ \ e_\lambda(\lambda')=(\lambda,\lambda')\beta, \ \forall \lambda'\in H^2(X,\RationalNumbers), \ 
\mbox{and} \ \ e_\lambda(\beta)=0.
\end{equation}

\begin{example}
\label{example-Mukai-lattice-of-Hilbert-schemes}
Let $X$ be the Hilbert scheme $S^{[n]}$ of length $n$ subschemes of a $K3$ surface $S$. Let $\pi:S^{[n]}\rightarrow S^{(n)}$ be the Hilbert Chow morphism. The composition of the isomorphism $H^2(S,\Integers)\cong H^2(S^n,\Integers)^{\fS_n}\cong H^2(S^{(n)},\Integers)$ with $\pi^*:H^2(S^{(n)},\Integers)\rightarrow H^2(S^{[n]},\Integers)$ induces an isometric embedding
$\iota$ of $H^2(S,\Integers)$ into $H^2(S^{[n]},\Integers)$ \cite{beauville-varieties-with-zero-c-1}. 
Hence, there exists a unique extension of $\iota$ to an  isometric embedding, denoted by $\iota$ as well,  of $\tilde{H}(S,\RationalNumbers)$ into $\tilde{H}(S^{[n]},\RationalNumbers)$ satisfying 
$\iota\circ e_\lambda=e_{\iota(\lambda)}\circ\iota$, for all $\lambda\in H^2(S,\Integers)$.
It satisfies
$\iota(\alpha)=\alpha$ and $\iota(\beta)=\beta$.
\end{example}

Let $h:H^*(X,\QQ)\rightarrow H^*(X,\QQ)$ be the degree operator acting by $k-\dim_\CC(X)$ on $H^k(X,\QQ)$. Given a class $\omega\in H^2(X,\QQ)$, denote by $L_\omega:H^*(X,\QQ)\rightarrow H^*(X,\QQ)$ the operator of cup product with $\omega$. 
If $\omega$ is a K\"{a}hler class then there exists an operator $\Lambda_\omega:H^*(X,\QQ)\rightarrow H^*(X,\QQ)$,
such that the triple $(L_\omega,h,\Lambda_\omega)$ spans a Lie subalgebra of $\LieAlg{gl}(H^*(X,\QQ))$ isomorphic to $\LieAlg{sl}_2$, by the Hard-Lefschetz Theorem. Explicitly, we have
\[
[L_w,\Lambda_w]=h, \ [L_w,h]=2L_w, \ \mbox{and} \ [\Lambda_w,h]=-2\Lambda_w.
\]
We say that $(L_\omega,h,\Lambda_\omega)$ is an {\em $\LieAlg{sl}_2$-triple} with respect to $h$.
The set of classes $\omega\in H^2(X,\QQ)$, for which there exists a necessarily unique operator $\Lambda_w$ satisfying the above relations, is a Zariski dense subset of $H^2(X,\QQ)$, by the Jacobson-Morozov Theorem. 

\begin{thm}
\label{thm-LLV-Lie-algebra}
(\cite[Prop. 4.5]{looijenga-lunts} and \cite{verbitsky-cohomology})
There exists a unique isomorphism between the 
Lie algebra $\LieAlg{g}:=\LieAlg{so}(\tilde{H}(X,\QQ))$ and the subalgebra of $\LieAlg{gl}(H^*(X,\QQ))$ generated by all $\LieAlg{sl}_2$-triples $(L_\omega,h,\Lambda_\omega)$,  $\omega\in H^2(X,\QQ)$,
which sends $e_\lambda$ to $L_\lambda$. The subalgebra of $\LieAlg{g}$ preserving the grading of $H^*(X,\QQ)$ is $\bar{\LieAlg{g}}\oplus\QQ h$.
\end{thm}

The isomorphism in the above theorem is proved by Verbitsky and Looijenga-Luntz over $\RR$, but their proof goes through over $\QQ$ \cite[Theorem 2.7]{GKLR}.

%
\subsection{Real $(p,p)$-classes}
\label{subsec-real-p-p-classes}
Let $X$ be an irreducible holomorphic symplectic manifold. 

\begin{lem}
\label{lemma-alpha-remains-of-Hodge-type}
Let $\alpha\in H^{p,p}(X,\RealNumbers)$ be a real class of Hodge type $(p,p)$. Assume that $\alpha$ remains of Hodge type under all first order deformations of $X$ as a complex manifold. Then $\alpha$ remains of Hodge type under every K\"{a}hler deformation of $X$.
\end{lem}

\begin{proof}
Let $I$ be the complex structure of $X$ and let 
\begin{equation}
\label{f-I}
f_I:H^*(X,\RealNumbers)\rightarrow H^*(X,\RealNumbers)
\end{equation} 
be the Hodge operator acting on $H^{p,q}(X,I)$ by
$(q-p)\sqrt{-1}$. Then $f_I$ belongs to $\bar{\LieAlg{g}}_\RR:=\bar{\LieAlg{g}}\otimes_\QQ\RR$, by \cite[Theorem 8.1]{verbitsky-mirror-symmetry}. 
Every K\"{a}hler deformation of an irreducible holomorphic symplectic manifold is an irreducible holomorphic symplectic manifold
\cite{beauville-varieties-with-zero-c-1}.
Hence, if the class $\alpha$ is $\bar{\LieAlg{g}}$-invariant, then it remains of Hodge type under every K\"{a}hler deformation. 
We prove the $\bar{\LieAlg{g}}$-invariance of $\alpha$ in two steps.

\underline{Step 1:}
The class $\alpha$ belongs to the kernel of all homomorphisms in the image of the natural homomorphism
\[
H^1(X,TX)\rightarrow \Hom(H^{p,p}(X),H^{p-1,p+1}(X)),
\]
since $\alpha$ remains of Hodge type under all first order deformations. 
The above homomorphism is the restriction of the contraction homomorphism 
\[
\iota \ : \ H^1(X,TX)\rightarrow    \Hom(H^*(X,\ComplexNumbers),H^*(X,\ComplexNumbers)),
\]
where $\iota_\eta$, $\eta\in H^1(X,TX)$, maps $H^{p,q}(X)$ to $H^{p-1,q+1}(X)$.
The image of the homomorphism $\iota$ is contained in the subalgebra $\bar{\LieAlg{g}}_\CC$, by \cite[Lemma 4.5]{GKLR}, so that $\iota$ factors through a homomorphism 
\[
\tilde{\iota}\ : \ H^1(X,TX)\rightarrow  \bar{\LieAlg{g}}_\CC,
\]
which we recall next\footnote{
The homomorphism $\tilde{\iota}$ is extended in Corollary \ref{cor-image-of-m-is-in-LLV-algebra} below to a homomorphism $HT^2(X)\rightarrow \LieAlg{g}_\CC$.
}.
Fix a K\"{a}hler metric on $X$, let $I$ be the complex structure on $X$ and let $\omega_I:=g(I\bullet,\bullet)$ be the fundamental form. 
We get the twistor data consisting of complex structures $J$ and $K$, compatible with $g$ and satisfying the quaternionic relations
$IJ=K$, as well as the K\"{a}hler forms $\omega_J:=g(J(\bullet),\bullet)$ and $\omega_K:=g(K(\bullet),\bullet)$, such that
$\sigma:= \omega_J+\sqrt{-1}\omega_K$ is a global holomorphic $(2,0)$-form on $X$ with respect to the complex structure $I$ \cite{beauville-varieties-with-zero-c-1}.
Let $L_I\in \LieAlg{gl}(H^*(X,\ComplexNumbers))$ be cup product with $\omega_I$, define $L_J$ and $L_K$ similarly, and let
$\Lambda_I$, $\Lambda_J$, and $\Lambda_K$ be the adjoint Lefschetz operators. 
Set $\Lambda_\sigma:=\frac{1}{2}(\Lambda_J-\sqrt{-1}\Lambda_K)$. Given $\eta\in H^1(X,TX)$, let $L_{\iota_\eta(\sigma)}$ be cup product with the 
$(1,1)$-class $\iota_\eta(\sigma)$. 
Then $\iota_\eta=[L_{\iota_\eta(\sigma)},\Lambda_\sigma]$ and the latter belongs to $\bar{\LieAlg{g}}_\CC$, so that $\iota$ factors through the homomorphism $\tilde{\iota}$
sending $\eta\in H^1(X,TX)$ to $[L_{\iota_\eta(\sigma)},\Lambda_\sigma]$, by \cite[Lemma 4.5]{GKLR}.

The Lie algebra $\LieAlg{so}(H^{1,1}(X))$ is the subspace of 
$\Hom[H^{1,1}(X),H^{1,1}(X)]$  of anti-self-dual homomorphisms with respect to the BBF-pairing. 
The one-dimensional Lie algebra $\LieAlg{so}(H^{2,0}(X)\oplus H^{0,2}(X))$ is characterized similarly.
Elements of $\LieAlg{so}(H^{2,0}(X)\oplus H^{0,2}(X))$ have traceless diagonal matrices in the basis $\{\sigma,\bar{\sigma}\}$.
The complex structure $I$ induces a weight zero Hodge structure on $\bar{\LieAlg{g}}_\CC\cong \LieAlg{so}(H^2(X,\ComplexNumbers))$,  the Hodge decomposition is $\LieAlg{so}(H^{1,1}(X))$-invariant, and the Hodge direct summands are the following $\LieAlg{so}(H^{1,1}(X))$-modules:
\begin{eqnarray*}
\bar{\LieAlg{g}}^{-1,1}&\cong&\Hom(H^{2,0}(X),H^{1,1}(X)),
\\
\bar{\LieAlg{g}}^{0,0}&\cong& \LieAlg{so}(H^{1,1}(X)) \ \ \ \ \oplus \ \ \  \LieAlg{so}(H^{2,0}(X)\oplus H^{0,2}(X))
\\
\bar{\LieAlg{g}}^{1,-1}&\cong&\Hom(H^{0,2}(X),H^{1,1}(X)).
\end{eqnarray*}

Note that $\Hom(H^{2,0}(X),H^{1,1}(X))$ is isomorphic to $\Hom(H^{1,1}(X),H^{0,2}(X))$  via the BBF pairing and $\bar{\LieAlg{g}}^{-1,1}$ is embedded in the direct sum of the two. A similar remark holds for $\bar{\LieAlg{g}}^{1,-1}$.
The image of $\iota$ is equal to $\bar{\LieAlg{g}}^{-1,1}$ and so $\alpha$ is infinitesimally invariant with respect to $\bar{\LieAlg{g}}^{-1,1}$, i.e.,  annihilated by it.
The class $\alpha$ is assumed to be real. Hence, it is invariant with respect to the complex conjugate subalgebra 
$\bar{\LieAlg{g}}^{1,-1}$. 
It remains to prove that the two Lie subalgebras $\bar{\LieAlg{g}}^{-1,1}$ and $\bar{\LieAlg{g}}^{1,-1}$ generate $\bar{\LieAlg{g}}_\CC$.

\underline{Step 2:}
The Lie bracket induces the following homomorphism of $\LieAlg{so}(H^{1,1}(X))$-modules:
\begin{equation}
\label{eq-so-equivariant-Lie-braket}
\bar{\LieAlg{g}}^{-1,1}\otimes \bar{\LieAlg{g}}^{1,-1} \rightarrow \bar{\LieAlg{g}}^{0,0}.
\end{equation}
Normalize $\sigma$ so that $(\sigma,\bar{\sigma})=1$. 
We have 
$
(\iota_{\eta_1}(\sigma),\iota_{\eta_2}(\sigma))=(\sigma,-\iota_{\eta_1}(\iota_{\eta_2}(\sigma))),
$
since $\iota_{\eta_1}$ belongs to $\bar{\LieAlg{g}}_\CC$, and $\iota_{\eta_1}(\iota_{\eta_2}(\sigma))$ is a scalar multiple of $\bar{\sigma}$. We conclude that
$
\iota_{\eta_1}(\iota_{\eta_2}(\sigma))=-(\iota_{\eta_1}(\sigma),\iota_{\eta_2}(\sigma))\bar{\sigma}.
$
Given $\eta\in H^1(TX)$, let $\bar{\iota_\eta}(\bullet):=\overline{\iota_\eta(\bar{\bullet})}$ be the complex conjugate element. 
Note that $\iota_\eta(\bar{\sigma})=0$ and $\bar{\iota_\eta}(\sigma)=0.$
The following equalities thus hold for all $\eta_1, \eta_2, \eta_3 \in H^1(TX)$.
\begin{eqnarray*}
\hspace{0ex}
[\iota_{\eta_1},\bar{\iota_{\eta_2}}](\sigma) & = & (\overline{\iota_{\eta_2}(\sigma)},\iota_{\eta_1}(\sigma))\sigma,
\\
\hspace{0ex}
[\iota_{\eta_1},\bar{\iota_{\eta_2}}](\bar{\sigma}) & = & 
-(\iota_{\eta_1}(\sigma),\overline{\iota_{\eta_2}(\sigma)})\bar{\sigma},
\\
\hspace{0ex}
[\iota_{\eta_1},\bar{\iota_{\eta_2}}](\iota_{\eta_3}(\sigma)) & = & 
-(\overline{\iota_{\eta_2}(\sigma)},\iota_{\eta_3}(\sigma))\iota_{\eta_1}(\sigma)+
(\iota_{\eta_1}(\sigma),\iota_{\eta_3}(\sigma))\overline{\iota_{\eta_2}(\sigma)}.
\end{eqnarray*}
Choose $\eta\in H^1(TX)$ so that $\iota_\eta(\sigma)$ is a non-isotropic real $(1,1)$ class of self intersection $t$.
We see that $[\iota_\eta,\bar{\iota_\eta}](\sigma)=t\sigma$,
$[\iota_\eta,\bar{\iota_\eta}](\bar{\sigma})=-t\bar{\sigma}$, and $[\iota_\eta,\bar{\iota_\eta}]$ annihilates $H^{1,1}(X)$.
Hence, $[\iota_\eta,\bar{\iota_\eta}]$ spans the trivial $\LieAlg{so}(H^{1,1}(X))$ submodule $\LieAlg{so}(\Hom[H^{2,0}(X)\oplus H^{0,2}(X))$, and so the latter is contained in the image of (\ref{eq-so-equivariant-Lie-braket}).

Choose non-zero elements $\eta_1, \eta_2\in H^1(TX)$, so that each of $\omega_i:=\iota_{\eta_i}(\sigma)$, $i=1,2$, is real and 
$(\omega_1,\omega_2)=0.$ Then 
$[\iota_{\eta_1},\bar{\iota_{\eta_2}}]$ annihilates both $\sigma$ and $\bar{\sigma}$ and
$[\iota_{\eta_1},\bar{\iota_{\eta_2}}](\omega)=-(\omega_2,\omega)\omega_1+(\omega_1,\omega)\omega_2$, for all $\omega\in H^{1,1}(X)$. Thus,  $[\iota_{\eta_1},\bar{\iota_{\eta_2}}]$ is a non-zero element of $\LieAlg{so}(H^{1,1}(X))$.
Thus,  the image of (\ref{eq-so-equivariant-Lie-braket}) must contain the  irreducible $\LieAlg{so}(H^{1,1}(X))$-submodule 
$\LieAlg{so}(H^{1,1}(X))$ as well.
\hide{
Choose $\eta\in H^1(TX)$ so that $\iota_\eta(\sigma)=\omega_I$.
Let $\bar{\iota_\eta}(\bullet):=\overline{\iota_\eta(\bar{\bullet})}$ be the complex conjugate element. 
The operators $L_I$, $\Lambda_J$, and $\Lambda_K$ are real, so 
the equality $\iota_\eta=[L_I,\Lambda_\sigma]=\frac{1}{2}\left([L_I,\Lambda_J]-\sqrt{-1}[L_I,\Lambda_K]\right)$ yields
$\bar{\iota_\eta}=\frac{1}{2}\left([L_I,\Lambda_J]+\sqrt{-1}[L_I,\Lambda_K]\right)$.
We have the equality
\begin{equation}
\label{eq-so-W}
\span\{\iota_\eta,\bar{\iota_\eta}, [\iota_\eta,\bar{\iota_\eta}]\}=\LieAlg{so}(W),
\end{equation}
where $W:=\span\{\sigma,\omega_I,\bar{\sigma}\}$, and  $\LieAlg{so}(W)$ is embedded in $\LieAlg{so}[H^2(X,\ComplexNumbers)]$ acting trivially on the orthogonal subspace $W^\perp$. Indeed, let $f_J$ and $f_K$ be the analogues of (\ref{f-I}). Then $f_I=-[L_J,\Lambda_K]=-[L_K,\Lambda_J]$ and the equalities obtained from the latter two by cyclicly permuting $I$, $J$, and $K$ hold as well, by \cite[Theorem 8.1]{verbitsky-mirror-symmetry}  (see also \cite[Prop. 2.7]{soldatenkov}). We have $\LieAlg{so}(W)=\span\{f_I,f_J,f_K\}$. Now, 
$\iota_\eta=-\frac{1}{2}(f_K+\sqrt{-1}f_J)$, $\bar{\iota_\eta}=-\frac{1}{2}(f_K-\sqrt{-1}f_J)$, and $[f_J,f_K]=-2f_I$ (see \cite[Prop. 2.24]{GKLR} for the latter equality). Hence, 
$[\iota_\eta,\bar{\iota_\eta}]=-\sqrt{-1}f_I$ and Equation (\ref{eq-so-W}) holds.

The image of (\ref{eq-so-equivariant-Lie-braket}) contains $\LieAlg{so}(H^{2,0}(X)\oplus H^{0,2}(X))$, since
$[\iota_\eta,\bar{\iota_\eta}]=-\sqrt{-1}f_I$ and $f_I$ spans the trivial $\LieAlg{so}(H^{1,1}(X))$ submodule $\LieAlg{so}(\Hom[H^{2,0}(X)\oplus H^{0,2}(X))$. 
If $\omega'=\iota_{\eta'}(\sigma)$ is a non-zero real $(1,1)$-form
orthogonal to $\omega_I$, then $[\iota_\eta,\bar{\iota_{\eta'}}](\omega')=\iota_\eta(\bar{\iota_{\eta'}}(\omega'))-
\bar{\iota_{\eta'}}(\iota_\eta(\omega'))$, the second term vanishes as $\omega'$ is annihilated by $\LieAlg{so}(W)$, and the first term is a non-zero multiple of $\omega_I$. We conclude that the image of (\ref{eq-so-equivariant-Lie-braket}) is a $\LieAlg{so}[H^{1,1}(X)]$-submodule strictly larger than $\LieAlg{so}(H^{2,0}(X)\oplus H^{0,2}(X))$. Thus,  the image of (\ref{eq-so-equivariant-Lie-braket}) must contain the  irreducible $\LieAlg{so}(H^{1,1}(X))$-submodule 
$\LieAlg{so}(H^{1,1}(X))$ as well.
}
\end{proof}

\begin{rem}
\label{rem-p-p-classes-that-remain-of-Hodge-type-to-1st-order-on-torus}
The analogue of Lemma \ref{lemma-alpha-remains-of-Hodge-type} when $X$ is a $n$-dimensional complex torus holds as well, except that the conclusion is that the real $(p,p)$ class $\alpha$ must vanish if $0<p<n$. The proof is completely analogous.
The LLV Lie algebra of a complex torus $X$ is 
\[
\LieAlg{g}_\RR=\LieAlg{so}\left(H^1(X,\RR)\oplus H^1(X,\RR)^*\right),
\]
where the pairing on $H^1(X,\RR)\oplus H^1(X,\RR)^*$ is $((\alpha_1,\theta_1),(\alpha_2,\theta_2))=\theta_1(\alpha_2)+\theta_2(\alpha_1)$. The semi-simple direct summand of the reductive subalgebra of $\LieAlg{g}_\RR$ preserving the grading of $H^*(X,\RR)$ is $\bar{\LieAlg{g}}_\RR=\LieAlg{sl}(H^1(X,\RR))$,
by \cite[Prop. 3.2]{looijenga-lunts}. The Hodge decomposition of $\bar{\LieAlg{g}}_\CC$ is
\begin{eqnarray*}
\bar{\LieAlg{g}}^{-1,1}&=&\Hom(H^{1,0}(X),H^{0,1}(X)), 
\\
\bar{\LieAlg{g}}^{1,-1}&=&\Hom(H^{0,1}(X),H^{1,0}(X)), 
\\
\bar{\LieAlg{g}}^{0,0}&=&\ker\left[\LieAlg{gl}(H^{1,0}(X))\oplus\LieAlg{gl}(H^{0,1}(X))\right]\LongRightArrowOf{tr+tr}\CC.
\end{eqnarray*}
Again, $\bar{\LieAlg{g}}^{-1,1}$ and $\bar{\LieAlg{g}}^{1,-1}$ generate $\bar{\LieAlg{g}}_\CC$. However, in this case
$H^i(X,\CC)=\wedge^iH^1(X,\CC)$ is an irreducible representation of $\bar{\LieAlg{g}}_\CC$ and so every $\bar{\LieAlg{g}}_\CC$-invariant $(p,p)$ class vanishes, for $0<p<n$.
\end{rem}
%
\section{Stable reflexive sheaves with a rank $1$ obstruction map are very modular}
In Section \ref{sec-proof-of-modularity} we prove that stable reflexive sheaves with a rank $1$ obstruction map are very modular (Theorem \ref{thm-modularity-of-a-stable-sheaf-with-a-rank-1-obstruction-map}).
In Section \ref{sec-examples-of-sheaves-with-a-rank-1-obstruction-map} we exhibit examples of coherent sheaves with a rank $1$ obstruction map. Let $X$ and $Y$ be deformation equivalent irreducible holomorphic symplectic manifolds and $F$ a stable reflexive coherent (untwisted) sheaf on $X$ with $\obs_F$ of rank $1$.
In Section \ref{sec-lift-to-an-untwisted-sheaf} we provide a necessary and sufficient conditions for there to be a coherent (untwisted) sheaf $E$ on $Y$, such that $(Y,E)$ is related to $(X,F)$ by a sequence of deformations and tensorization by line bundles (Lemma \ref{lemma-a-necessariy-and-sufficient-condition-for-the-existence-of-a-lift-to-an-untwisted-sheaf}). We then provide an easy to check sufficient condition in the $K3^{[n]}$ and generalized kummer deformation types (Lemma \ref{lemma-easy-to-check-sufficient-condition}).

%
\subsection{Proof of modularity}
\label{sec-proof-of-modularity}
Let $F$ be a coherent, possibly twisted, sheaf of positive rank $r$.
We get the untwisted object $G:=F^{\otimes r}\otimes \det(F)^*$  of rank $r^r$ in the derived category of $X$ and we let  $\kappa(F)\in H^*(X,\RationalNumbers)$ be the $r$-th root of $ch(G)$ with value $r$ in $H^0(X,\RationalNumbers)$, see
\cite[Sec. 2.2]{markman-BBF-class-as-characteristic-class} for a detailed definition.

An equivalence 
$\Phi: D^b(X)\rightarrow D^b(Y)$ of derived categories
induces a graded isomorphism $\phi:HH^*(X)\rightarrow HH^*(Y)$ of Hochschild cohomologies, by
\cite{caldararu-I}. 
Set $HT^k(X):=\oplus_{p+q=k}H^q(X,\wedge^pTX)$. 
We have the graded Hochschild-Kostant-Rosenberg isomorphism 
$I^{HKR}:HT^*(X)\rightarrow HH^*(X)$. The space $HT^2(X)=H^2(\StructureSheaf{X})\oplus H^1(TX)\oplus H^0(\wedge^2TX)$ parametrizes first order generalized deformations of the category of coherent sheaves on $X$, and also of $D^b(X)$ \cite{toda}.

The Atiyah class of an object $G$ in $D^b(X)$ is a class $a_G$ in $\Hom(G,G\otimes \Omega^1_X[1])$ (see \cite{kapranov}).
Let $a_{G,i}\in \Hom(G,G\otimes \Omega^i_X[i])$ be the $i$ power of $a_G$. 
The exponential Atiyah class is 
\[
\exp a_G:=\oplus_{i\geq 0}a_{G,i}:G\rightarrow \oplus_{i\geq 0}(G\otimes \Omega_X^i[i]).
\]
The contraction homomorphism
\begin{equation}
\label{eq-obstruction-map}
\obs^{HT}_G
\ : \  HT^2(X) \ \rightarrow \ \Hom(G,G[2])
\end{equation}
sends a class $\xi\in HT^2(X)$ to 
the obstruction class $\xi \cdot \exp a_G$
to first order deformation of $G$ in the direction $\xi$ \cite[Prop. 6.1]{toda}.
Hence, if $\Hom(G,G[2])$ is one dimensional, 
then the kernel $\Sigma(G)\subset HT^2(X)$ has co-dimension at most one.

Given an object $G$ in the bounded derived category $D^b(X)$ of a smooth projective variety $X$, denote by $\Sigma(G)\subset
HT^2(X)$ the kernel of (\ref{eq-obstruction-map}). Note that 
$\Sigma(\StructureSheaf{X})$ contains $H^0(\wedge^2TX)\oplus H^1(TX)$, by the proof of \cite[Lemma 4.3]{toda} in which Toda constructs a first order deformation of $\StructureSheaf{X}$ for each class in $H^0(\wedge^2TX)\oplus H^1(TX)$. 
Given a closed point  $z\in X$ let $\StructureSheaf{z}$ be its sky-scraper sheaf and $I_z$ its ideal sheaf. 
Now, $\Ext^1(\StructureSheaf{z},\StructureSheaf{z}\otimes\Omega^1_X)\cong H^0(\SheafExt^1_X(\StructureSheaf{z},\StructureSheaf{z}))\otimes\Omega^1_{X,z}\cong \End(T_zX),$ the Atiyah class 
$a_{\StructureSheaf{z}}$ is the identity, and $a_{\StructureSheaf{z},2}$ is the identity endomorphism of $\wedge^2(T_zX)$.
Thus, 
\[
\Sigma(\StructureSheaf{z})=H^2(X,\StructureSheaf{X})\oplus H^1(X,TX)\oplus H^0(X,(\wedge^2TX)\otimes I_z).
\]
Hence, if $X$ is an irreducible holomorphic symplectic manifold, then both
$\Sigma(\StructureSheaf{X})$ 
and $\Sigma(\StructureSheaf{z})$ have co-dimension $1$ in $HH^2(X)$.

The co-dimension of $\Sigma(G)$ is one also 
for every object $G$ over an irreducible holomorphic symplectic manifold $Y$, which is isomorphic to the image 
$\Phi(F)$ of an object $F$ with $\Sigma(F)$ of codimension one, 
via an equivalence of derived categories $\Phi:D^b(X)\rightarrow D^b(Y)$,
by \cite[Theorem 4.7]{toda}.

\begin{lem}
\label{lemma--kappa-F-remains-of-Hodge-type}
Let $X$ be a projective irreducible holomorphic symplectic manifolds and
$F$  an object in $D^b(X)$ of positive rank. Assume that 
the projection from $HT^2(X)$ to its direct summand $H^1(TX)$ maps the intersection  
\begin{equation}
\label{transversality-condition}
\Sigma(F)
 \ \ \ \cap \ \ \ \left[H^2(\StructureSheaf{X})\oplus H^1(TX)\right]
\end{equation}
onto $H^1(TX)$. 
Then the class $\kappa(F)$ remains of Hodge type under every K\"{a}hler deformation of $X$. 
\end{lem}

\begin{proof}
Let $\beta$ be an element of $H^1(TX)$.   There exists a class $\alpha \in H^2(\StructureSheaf{X})$, such that the class $(\alpha,\beta,0)\in HT^2(X)$ belongs to $\Sigma(F)$, by assumption. 
The object $F$ deforms in the direction $(\alpha,\beta,0)$, by \cite[Prop. 6.1]{toda}. More precisely, there exists a perfect complex $\F$ (i.e., locally quasi-isomorphic to a bounded complex of free modules) of $\tilde{\alpha}$-twisted  sheaves 
in the category $Mod(\StructureSheaf{X}^\beta,\tilde{\alpha})$ constructed in
\cite[Sec. 4]{toda}, where $\StructureSheaf{X}^\beta$ is the structure sheaf of the first order deformation $\X$ of $X$ in the commutative direction $\beta$ and $\tilde{\alpha}$ is the natural lift of $\alpha$ to a Brauer class in $H^2(\X,\StructureSheaf{\X}^*)$. We get the rank $1$ locally free $\tilde{\alpha}^r$-twisted $\StructureSheaf{\X}$-module $\det(\F)$ over $\X$, by
\cite[Theorem 2]{knudsen-mumford}, and 
the untwisted deformation $\det(\F)^*\otimes (\F^\vee)^{\otimes r}$  over $\X$ of the object 
$E:=\det(F)^*\otimes (F^\vee)^{\otimes r}$. 
Hence, the class $ch(E)$ remains of Hodge type in every direction $\beta\in H^1(TX)$.
Thus, so does its $r$-th root $\kappa(F)$. The statement now follows from Lemma \ref{lemma-alpha-remains-of-Hodge-type}.
\end{proof}

\begin{rem}
\label{rem-chern-character-remains-of-Hodge-type}
Let $X$ be a projective irreducible holomorphic symplectic manifold and $F$ an object of $D^b(X)$.
We include the case where the rank of $F$ is zero. 
If $H^1(TX)$ is contained in the kernel $\Sigma(F)$ of the obstruction map (\ref{eq-obstruction-map}), then $ch(F)$ remains of
Hodge type under every K\"{a}hler deformation of $X$, by the proof of Lemma \ref{lemma--kappa-F-remains-of-Hodge-type}.
\end{rem}

\begin{prop}
\label{prop-kappa-class-remains-of-Hodge-type}
Let $\Phi: D^b(X)\rightarrow D^b(Y)$ be an equivalence of derived categories of coherent sheaves of two projective irreducible holomorphic symplectic manifolds. Assume that an object $G$ in $D^b(X)$ 
satisfies the following two conditions.
\begin{enumerate}
\item
$\Sigma(G)$ has co-dimension $1$ in $HT^2(X)$.
\item
The rank of $F:=\Phi(G)$ is positive. 
\end{enumerate}
Then the class $\kappa(F)$ remains of Hodge type under every K\"{a}hler deformation of $Y.$ 
\end{prop}

\begin{proof}
\hide{
Let $\beta$ be an element of $H^1(TY)$. By assumption (\ref{transversality-condition}) there exists a class $\alpha \in H^2(\StructureSheaf{Y})$, such that the class $(\alpha,\beta,0)\in HT^2(Y)$ belongs to $\phi(\Sigma(G))$. Now, $\Sigma(F)=\phi(\Sigma(G))$, by \cite[Theorem 4.7]{toda}.
The object $F$ deforms in the direction $(\alpha,\beta,0)$, by \cite[Prop. 6.1]{toda}. More precisely, there exists a perfect complex $\F$ of $\tilde{\alpha}$-twisted locally free sheaves 
in the category $Mod(\StructureSheaf{Y}^\beta,\tilde{\alpha})$ constructed in
\cite[Sec. 4]{toda}, where $\StructureSheaf{Y}^\beta$ is the structure sheaf of the first order deformation $\Y$ of $Y$ in the commutative direction $\beta$ and $\tilde{\alpha}$ is the natural lift of $\alpha$ to a Brauer class in $H^2(\Y,\StructureSheaf{\Y}^*)$. We get the untwisted deformation $(\wedge^r \F)^*\otimes (\F^\vee)^{\otimes r}$  over $\Y$ of the object 
$E:=(\wedge^r F)^*\otimes (F^\vee)^{\otimes r}$. 
Hence, the class $ch(E)$ remains of Hodge type in every direction $\beta\in H^1(TY)$.
Thus, so does its $r$-th root $\kappa(F)$. The statement now follows from Lemma \ref{lemma-alpha-remains-of-Hodge-type}.
}
The equality $\Sigma(F)=\phi(\Sigma(G))$ follows from \cite[Theorem 4.7]{toda}.
Hence, $\Sigma(F)$ has co-dimension $1$ in $HT^2(Y)$, since
$\Sigma(G)$ is assumed to have co-dimension $1$ in $HT^2(X)$.
We claim that the projection from $HT^2(Y)$ to its direct summand $H^1(TY)$ maps the intersection  
\[
\Sigma(F)
 \ \ \ \cap \ \ \ \left[H^2(\StructureSheaf{Y})\oplus H^1(TY)\right]
\]
onto $H^1(TY)$. 
This surjectivity statement is equivalent to the condition that either the two hyperplanes $\Sigma(F)$ and $H^2(\StructureSheaf{Y})\oplus H^1(Y,TY)$ are equal, or 
$H^2(\StructureSheaf{Y})$ is not contained in $\Sigma(F)$. 
The obstruction map (\ref{eq-obstruction-map}) is defined as 
contraction with the exponential Atiyah class 
\[
exp \ a_F\in \oplus_{i\geq1}\Hom(F,F\otimes\Omega_Y^i[i]),
\]
\cite[Def. 5.5]{toda}.
The direct summand of the latter in $\Hom(F,F\otimes\Omega_Y^0[0])=\Hom(F,F)$ is the identity.
Hence, (\ref{eq-obstruction-map}) restricts to $H^2(\StructureSheaf{Y})$ 
as the homomorphism $u:\Hom(\StructureSheaf{Y},\StructureSheaf{Y}[2])\rightarrow\Hom(F,F[2])$ induced by 
the unit natural transformation from the identity functor to the composition of the two adjoint functors of tensorization with $F$ and $F^\vee$. The composition of $u$ with the trace natural transformation $tr:\Ext^2(F,F)\rightarrow H^2(\StructureSheaf{Y})$
is multiplication of $H^2(\StructureSheaf{Y})$ by the rank $r$ of $F$ (see \cite[Lemma 10.1.3]{huybrechts-lehn}).
The obstruction map (\ref{eq-obstruction-map}) restricts to an injective homomorphism from 
$H^2(\StructureSheaf{Y})$ to $\Ext^2(F,F)$, since $r>0$, by assumption. Consequently, $H^2(\StructureSheaf{Y})$ is not contained in the kernel $\Sigma(F)$ of (\ref{eq-obstruction-map}). The statement now follows from Lemma \ref{lemma--kappa-F-remains-of-Hodge-type}.
\end{proof}

\begin{thm}
\label{thm-modularity-of-a-stable-sheaf-with-a-rank-1-obstruction-map}
Let $\Phi:D^b(X)\rightarrow D^b(Y)$ be an equivalence of the derived categories of two projective irreducible holomorphic symplectic manifolds $X$ and $Y$. 
Let $F$ be a reflexive sheaf on $Y$ of rank $r>0$, which is $H$-slope-stable with respect to some ample line-bundle $H$.
Assume that $F$ is isomorphic to $\Phi(G)$, where $\Sigma(G)$ has co-dimension one\footnote{$G$ could be, for example, 
$\StructureSheaf{X}$, or  the sky-scraper sheaf $\StructureSheaf{p}$ for some point $p$ in $X$, or $\iota_*\StructureSheaf{Z}$ for a lagrangian embedding  $\iota:Z\rightarrow X$ as in Example \ref{example-lagrangian-Z-with-modular-structure-sheaf}.} 
 in $HT^2(X)$. 
Then for every irreducible holomorphic symplectic manifold $Z$ deformation equivalent to $Y$ there exists a possibly twisted coherent sheaf $E$ over $Z$ and a flat 
deformation of the pair $(Y,F)$ to the pair (Z,E).
\end{thm}

\begin{proof}
The class $\kappa_2(F)$ remains of Hodge type $(2,2)$ under every K\"{a}hler deformation of $Y$, by Proposition \ref{prop-kappa-class-remains-of-Hodge-type}. It  follows that $F$ is {\em  modular}\footnote{Let $SH^*(Y)$ be the subalgebra of $H^*(Y,\QQ)$ generated by $H^2(Y,\QQ)$. We have the decomposition of $H^*(Y,\QQ)$ as an orthogonal direct sum  $SH^*(Y)\oplus  SH^*(Y)^\perp$ with respect to the Poincar\'{e} paring. A torsion sheaf $F$ is {\em modular} if the orthogonal projection of $\kappa_2(F)$ to $SH^*(Y)$ is a  scalar  multiple of the class in the image of  $\Sym^2H^2(Y,\QQ)$ corresponding to the inverse of the Beauville-Bogomolov-Fujiki  pairing \cite[Def. 1.1]{ogrady-modular}. The latter class  is the  unique  class in the image of $\Sym^2H^2(Y,\QQ)$ in $H^4(Y,\QQ)$, which remains of Hodge-type $(2,2)$ under all  K\"{a}hler deformations of $Y$. Hence, if $\obs_F$ has rank $1$, then $F$ is modular.
} in the sense of \cite[Def. 1.1]{ogrady-modular}. Hence, there exists an open subcone of the ample cone, containing  $H$, such that  $F$ is $H'$-slope-stable with  respect to every class $H'$ in this subcone, by \cite[Perop. 3.4]{ogrady-modular}.

A {\em $\Lambda$-marking} for an irreducible holomorphic symplectic manifold $Y$ is an isometry $\eta:H^2(Y,\Integers)\rightarrow \Lambda$ with a lattice $\Lambda$. There exists a moduli space $\fM_\Lambda$ of $\Lambda$-marked irreducible holomorphic symplectic manifolds, which is a non-Hausdorff manifold of dimension $b_2(Y)-2$ (see \cite{huybrects-basic-results}). 
Choose an isometry $\eta:H^2(Y,\Integers)\rightarrow \Lambda$ with a fixed lattice $\Lambda$. Denote by $\fM^0_\Lambda$ the connected component of the moduli space $\fM_\Lambda$ containing $(Y,\eta)$. By assumption, there exists an isometry $\eta':H^2(Z,\Integers)\rightarrow \Lambda$, such that $(Z,\eta')$ belongs to $\fM^0_\Lambda$. 
Associated to every K\"{a}hler class $\omega$ on $Y$ is a twistor line $\PP^1_\omega$ in $\fM^0_\Lambda$ through
$(Y,\eta)$ and a twistor family $\Y_\omega\rightarrow \PP^1_\omega$, such that its fiber $\Y_t$ over $t\in\PP^1_\omega$ 
represents the isomorphism class $t$ as a point of $\fM^0_\Lambda$ \cite{huybrects-basic-results}. 
A twistor path from $(Y,\eta)$ to $(Z,\eta')$ is a connected sequence of twistor lines in $\fM^0_\Lambda$, such that every two consecutive lines intersect. A choice of an intersection point of every two consecutive lines is included in the data of the twistor path. If the chosen intersection points of every two consecutive lines are pairs $(Y_i,\eta_i)$ such that $\Pic(Y_i)$ is trivial, then the twistor path is said to be generic. 
There exists a generic twistor path $\gamma$ from $(Y,\eta)$ to $(Z,\eta')$ in $\fM_\Lambda^0$, such that its first twistor line corresponds to a K\"{a}hler class $\omega$ on $Y$ with respect to which $F$ is slope-stable, by
\cite[Theorems 3.2 and 5.2e]{verbitsky-cohomology} and the fact that $F$ is slope-stable with respect to all ample classes in an open subcone of the ample cone. 
Let $\pi:\Y\rightarrow \gamma$ be the twistor deformation of $Y$ over the twistor path $\gamma$. 
The 
first Chern class of every direct summand of $\SheafEnd(F)$ vanishes, since  $F$ is $H$-stable, and so $\SheafEnd(F)$ is $H$-semistable, with respect to every $H$ in an open subcone of the ample cone of $Y$, see \cite[Lemma 7.2]{markman-BBF-class-as-characteristic-class}. 
The class $\kappa_2(F)$ remains of Hodge type along $\gamma$, by Proposition \ref{prop-kappa-class-remains-of-Hodge-type}.
Hence, the sheaf $F$ extends to a reflexive twisted sheaf $\F$ over the twistor deformation $\Y$ over the twistor line $\gamma$, by the generalization \cite[Prop. 6.17]{markman-BBF-class-as-characteristic-class} to the case of twisted reflexive sheaves of results of Verbitsky \cite{kaledin-verbitsky-book} for reflexive sheaves. We let $E$ be the restriction of $\F$ to the fiber of $\pi$ over 
the point $(Z,\eta')$ of the twistor path $\gamma$.
\end{proof}

\begin{rem}
\label{rem-semi-homogeneous}
\begin{enumerate}
\item
The analogues of Lemma \ref{lemma--kappa-F-remains-of-Hodge-type} and Proposition \ref{prop-kappa-class-remains-of-Hodge-type}  hold when $X$ is a complex torus
if we replace the assumption that $\Sigma(G)$ has codimension $1$ by the assumption that $\Sigma(G)$ has codimension equal to
$h^{0,2}(X):=\dim(H^{2,0}(X))$. 
The analogue of Theorem \ref{thm-modularity-of-a-stable-sheaf-with-a-rank-1-obstruction-map} holds as well
if $X$ is even dimensional (and so admits a hyperk\"{a}hler structure).
The proof of Lemma \ref{lemma--kappa-F-remains-of-Hodge-type} is identical using 
Remark \ref{rem-p-p-classes-that-remain-of-Hodge-type-to-1st-order-on-torus} instead of Lemma \ref{lemma-alpha-remains-of-Hodge-type}. The proof of Proposition \ref{prop-kappa-class-remains-of-Hodge-type} shows that the restriction of $\exp a_F$  to $H^2(\StructureSheaf{Y})$ is injective, and so $H^2(\StructureSheaf{Y})$ intersect $\Sigma(F)$ trivially. Hence, the intersection 
$\Sigma(F)\cap [H^2(\StructureSheaf{Y})\oplus H^1(TY)]$ again projects onto $H^1(TY)$ and
the hypothesis of Lemma \ref{lemma--kappa-F-remains-of-Hodge-type} is satisfied. 
The rest of the proof of Proposition \ref{prop-kappa-class-remains-of-Hodge-type} is identical.
The proof of Theorem \ref{thm-modularity-of-a-stable-sheaf-with-a-rank-1-obstruction-map} is analogous, using the same result of Verbitsky and the twistor path connectivity of the moduli space of even dimensional complex tori \cite{buskin-izadi}. 
\item
\label{rem-item-codimension-h-2-0}
Following are examples of sheaves $G$ over abelian varieties $X$ for which the codimension of  $\Sigma(G)$ is $h^{0,2}(X)$.
Over an abelian surface every simple sheaf is an example. 
Examples of such sheaves over compact complex tori of dimension $>2$ include the sky scraper sheaf, a line bundle on $X$, the structure sheaf of a connected abelian subvariety, as well as their images under Fourier-Mukai transformations. When locally free sheaves, these images are semi-homogeneous vector bundles. Recall that a vector bundle $F$ on $X$ is {\em semi-homogeneous}, if every translate of $F$ by a point of $X$ is isomorphic to $F\otimes L$, for some line bundle $L\in\Pic^0(X).$
If $F$ is a simple semi-homogeneous vector bundle over an abelian variety $X$, then it is stable and $\dim\Ext^2(F,F)=h^{0,2}(X),$ by \cite[Theorem 5.8]{mukai-semihomogeneous}. 
In particular, one gets via the analogue of Proposition \ref{prop-kappa-class-remains-of-Hodge-type} (with $\Phi$ the identity endo-functor) an alternative proof of the fact that the graded summands  $\kappa_p(F)\in H^{p,p}(X)$ of the class $\kappa(F)$ of a simple semi-homogeneous vector bundle $F$ vanishes for $0<p<\dim(X)$ (this also follows from \cite[Prop. 7.3]{mukai-semihomogeneous}, which implies the vanishing of $\kappa_{\dim(X)}(F)$ as well).
\end{enumerate}
\end{rem}

\begin{lem}
\label{lemma-semi-homogeneous}
Let $X$ be an abelian variety of dimension $\geq 3$, $H$ an ample line bundle on $X$,  and $F$ an $H$-slope-stable vector bundle. Then $F$ is semi-homogeneous if and only if 
the rank of the obstruction map $HT^2(X)\RightArrowOf{(\ref{eq-obstruction-map})} Ext^2(F,F)$
is $h^{0,2}(X)$. 
\end{lem}

\begin{proof}
If $F$ is semi-homogeneous, then $\dim \Ext^2(F,F)=h^{0,2}(X),$ by
by \cite[Theorem 5.8]{mukai-semihomogeneous}. 
The composition 
\[
H^2(\StructureSheaf{X})\rightarrow HT^2(X)\RightArrowOf{(\ref{eq-obstruction-map})} Ext^2(F,F) \RightArrowOf{tr}H^2(\StructureSheaf{X})
\]
is multiplication by the rank of $F$, so an isomorphism, so the map (\ref{eq-obstruction-map}) is surjective. 
Conversely, if the co-dimension of $\Sigma(F)$ in $HH^2(X)$ is $h^{0,2}(X)$, then  $\kappa_2(F)=0$, by Remark \ref{rem-semi-homogeneous}(\ref{rem-item-codimension-h-2-0}), and so the statement follows from \cite[Prop. A.2]{ogrady-modular}.
\end{proof}
%
\subsection{Examples of sheaves $F$ with a rank $1$ obstruction map $HT^2(X)\rightarrow \Ext^2(F)$}
\label{sec-examples-of-sheaves-with-a-rank-1-obstruction-map}
\begin{example}
Let $E$ be a simple sheaf of positive rank $r$ over a $K3$ surface $X$, so that $\Hom(E,E)$ is one dimensional and by Serre's Duality theorem so is $\Ext^2(E,E)$. Then $\Sigma(E)\subset HT^2(X)$ satisfies the condition in Part \ref{transversality-condition}
of Proposition \ref{prop-kappa-class-remains-of-Hodge-type}. The conclusion of the Proposition, that  $\kappa(E)$ remains of Hodge type, is a tautology in the $2$-dimensional case, since $\kappa(E)=r+\kappa_2(E)$, $\kappa_2(E)\in H^{4}(X,\RationalNumbers)$, and every class in the top cohomology $H^{4}(X,\RationalNumbers)$ remains of Hodge type under every K\"{a}hler deformation.
\end{example}

\begin{example}
\label{example-lagrangian-Z-with-modular-structure-sheaf}
Let $X$ be a $2n$-dimensional irreducible holomorphic symplectic manifold. Let 
$Z$ be an $n$-dimensional compact complex manifold with Hodge numbers $h^{2,0}(Z)=0$ and $h^{1,1}(Z)=1$. 
Assume given a holomorphic embedding $\iota:Z\hookrightarrow X$. 
Then $\iota(Z)$ is a lagrangian submanifold. 
The obstruction map $HT^2(X)\rightarrow \Ext^2(\iota_*\StructureSheaf{Z},\iota_*\StructureSheaf{Z})$
has rank $1$, by Lemma \ref{lemma-lagrangian-submanifolds-with-one-dimensional-second-cohomology} below.
Every smooth curve $Z$ in a $K3$ surface is an example. 
Examples where $Z=\PP^n$ are studied in \cite{bakker,hassett-tschinkel}. 
Examples with $Z$ a Grassmannian $G(k,2k+1)$ are given in \cite[Example 3.3]{markman-brill-noether}.
Every smooth cubic fourfold $V$ that does not contain a plane
embeds as a lagrangian subvariety of the projective irreducible holomorphic symplectic manifold of $K3^{[4]}$ deformation type parametrizing equivalence classes of generalized twisted cubics in $V$ \cite[Theorem B]{LLSv}.
\end{example}

\begin{lem}
\label{lemma-lagrangian-submanifolds-with-one-dimensional-second-cohomology}
Let $\iota:Z\rightarrow X$ be as in Example \ref{example-lagrangian-Z-with-modular-structure-sheaf}. 
Then $\Ext^2(\iota_*\StructureSheaf{Z},\iota_*\StructureSheaf{Z})$ is one-dimensional and the obstruction map $HT^2(X)\rightarrow \Ext^2(\iota_*\StructureSheaf{Z},\iota_*\StructureSheaf{Z})$
has rank $1$.
\end{lem}

\begin{proof}
The normal bundle $N_{Z/X}$ is isomorphic to $\Omega^1_{Z}$, as $Z$ is a lagrangian subvariety,  and so 
the extension sheaf $\SheafExt^i(\iota_*\StructureSheaf{Z},\iota_*\StructureSheaf{Z})$
is isomorphic to $\Omega^i_{Z}$. 
We have the seven-terms exact sequence of the local to global spectral sequence, where $N=N_{Z/X}$:
\begin{eqnarray*}
0&\rightarrow& H^1(\wedge^0N) \rightarrow \Ext^1_X(\iota_*\StructureSheaf{Z},\iota_*\StructureSheaf{Z})
\rightarrow H^0(N)
\\
&\rightarrow & H^2(\wedge^0N)\rightarrow 
\ker\left[\Ext^2_X(\iota_*\StructureSheaf{Z},\iota_*\StructureSheaf{Z})\rightarrow H^0(\wedge^2N)
\right]\rightarrow H^1(N)\rightarrow H^3(\wedge^0N)
\end{eqnarray*}
\cite[Appendix B, page 309]{milne}.
Our assumptions imply that $H^0(\wedge^2N)$ and $H^2(\wedge^0N)$ vanish as they are isomorphic to
$H^0(\Omega^2_Z)$ and $H^2(\StructureSheaf{Z})$.
It follows that $\Ext^2(\iota_*\StructureSheaf{Z},\iota_*\StructureSheaf{Z})\rightarrow H^1(Z,N_{Z/X})\cong H^1(Z,\Omega^1_{Z})$ is injective and so $\Ext^2(\iota_*\StructureSheaf{Z},\iota_*\StructureSheaf{Z})$ is at most one-dimensional. 

We claim that $\Ext^2(\iota_*\StructureSheaf{Z},\iota_*\StructureSheaf{Z})$ does not vanish and the obstruction map (\ref{eq-obstruction-map}) restricts to a non-trivial, hence surjective, homomorphism from $H^1(TX)$ onto $\Ext^2(\iota_*\StructureSheaf{Z},\iota_*\StructureSheaf{Z})$.
Assume, otherwise, that $H^1(TX)$ is contained in the kernel of the obstruction map. 
Then $ch(\iota_*\StructureSheaf{Z})$ remains of Hodge type in the direction of every $\xi\in H^1(TX)$, by
Remark \ref{rem-chern-character-remains-of-Hodge-type}. Hence, so does   the cohomology class $[Z]\in H^{n,n}(X,\ZZ)$ of $\iota(Z)$,
since $[Z]$ is a graded direct summand of $ch(\iota_*\StructureSheaf{Z})$.
Let $T_ZX$ be the subsheaf of $TX$ of  vector fields which are tangent to $\iota(Z)$. So $T_ZX$ is the kernel in the short exact sequence 
\[
0\rightarrow T_ZX\rightarrow TX\rightarrow \iota_*(N_{Z/X})\rightarrow 0.
\]
$H^1(X,T_ZX)$ is the space of first order complex deformations of the pair $(X,\iota(Z))$.
It is known that the pair $(X,\iota(Z))$ deforms with $X$ over a co-dimension $1$ smooth locus in  the local Kuranishi deformation space of $X$  and the sequence
\[
H^1(X,T_ZX)\rightarrow H^1(X,TX)\RightArrowOf{c_1(L)}H^2(X,\StructureSheaf{X})\rightarrow 0
\]
is exact, where $L$ is the line bundle on $X$ such that the kernel of $H^2(X,\CC)\rightarrow H^2(Z,\CC)$ is orthogonal to $c_1(L)$ via the BBF-pairing  \cite[Cor. 1.2]{voisin-lagrangian}. In particular, the left homomorphism above has a one-dimensional co-kernel.
Furthermore, the cohomology class $[Z]$ remains of Hodge type in a direction $\xi\in H^1(TX)$,
if and only if $c_1(L)$ does \cite[Cor. 1.4]{voisin-lagrangian}. A contradiction.
We conclude that $\Ext^2(\iota_*\StructureSheaf{Z},\iota_*\StructureSheaf{Z})$ does not vanish and  
$\Sigma(\iota_*\StructureSheaf{Z})$ has co-dimension $1$ in $HT^2(X)$. 
\end{proof}

\begin{rem}
\label{rem-expected-rank-one}
\begin{enumerate}
\item
Let $X$ be a $2n$-dimensional irreducible holomorphic symplectic manifold,
$Z$ an $n$-dimensional compact complex manifold, and $\iota:Z\hookrightarrow X$ an embedding as a lagrangian submanifold.
Let $\K$ be a line bundle on $Z$ satisfying $\K^k\cong\omega_Z^j$ for some integers $j, k$ with $k\neq 0$. Set $F:=\iota_*\K$.
We expect the obstruction map  $\obs^{HT}_F$, given in (\ref{eq-obstruction-map}), to have rank $1$ whenever the restriction homomorphism
$H^2(X,\CC)\rightarrow H^2(Z,\CC)$ has rank $1$.
The latter condition is equivalent to the homomorphism $H^{1,1}(X)\rightarrow H^{1,1}(Z)$ having rank $1$. 
Equivalently, the homomorphism $H^1(X,TX)\rightarrow H^1(Z,N_{Z/X})$, induced by the
sheaf homomorphism $TX\rightarrow \iota_*N_{Z/X}$, has rank $1$, as the symplectic structure of $X$ induces a commutative diagram
\[
\xymatrix{
TX\ar[r] \ar[d]_{\cong} & \iota_*N_{Z/X}\ar[d]^{\cong}
\\
\Omega^1_X \ar[r] & \iota_*\Omega^1_Z.
}
\]
The domain $HT^2(X)$ of $\obs^{HT}_F$ admits a direct sum decomposition by definition and its co-domain $\Ext^2(F,F)$ decomposes  as a direct sum 
$\bigoplus_{p+q=2}H^p(X,\wedge^q N_{Z/X})$, by the degeneration of the local-to-global spectral sequence  at the $^2E$-page 
\cite[Theorem 0.1.3]{mladenov}.
The 
homomorphism 
$H^2(X,\StructureSheaf{X})\rightarrow H^2(X,\SheafExt^0_X(\iota_*\K,\iota_*\K))
\cong H^2(Z,\StructureSheaf{Z})$ vanishes, due to $Z$ being lagrangian.
The sheaf homomorphism $TX\rightarrow \iota_*N_{Z/X}$ induces the zero homomorphism
\[
H^0(X,\wedge^2TX)\rightarrow H^0(Z,\wedge^2N_{Z/X})\cong H^0(Z,\SheafExt^2_X(\iota_*\K,\iota_*\K)), 
\]
again due to $Z$ being lagrangian. However, 
the obstruction map (\ref{eq-obstruction-map}) need not factor through
the direct sum of the
above homomorphisms.
\item
In order to prove that the obstruction map has rank $1$ it suffices to prove that the image of $H^0(X,\wedge^2TX)$ via the 
obstruction map is contained in the image of $H^1(X,TX)$. 
When $\K^2\cong \omega_Z$ the deformation quantization result of \cite{DAngolo-Schapira} seems to suggest that $H^0(X,\wedge^2TX)$ is contained in the kernel of the obstruction map and so the latter has rank $1$. It follows that whenever a square root of $\omega_Z$ belongs to the image of $\iota^*\Pic(X)\rightarrow \Pic(Z)$, then the obstruction map 
$\obs^{HT}_F$  has rank $1$ for $F=\iota_*\iota^*L$, for every $L\in\Pic(X)$, since if $\iota^*(L_0)^2\cong\omega_Z$, then
tensorization by $L\otimes L_0^{-1}$ is an auto-equivalnce of $D^b(X)$ mapping $\iota_*\iota^*L_0$ to $\iota_*\iota^*L$ and so the rank of the obstruction map for $F=\iota_*\iota^*L$ is equal to that for $F=\iota_*\iota^*L_0$.
\item
\label{rem-item-lagrangian-fixed-locus}
If $Z$ is a lagrangian connected component of the fixed locus of a finite group $G$ acting on $X$, then the image of $H^0(X,\wedge^2TX)$ via the obstruction map  $\obs^{HT}_{\iota_*\StructureSheaf{Z}}$ vanishes  and the obstruction map has rank $1$. This is seen as follows. 
The {\em universal Atiyah class} $a_X$  is the extension class of the short exact sequence 
\begin{equation}
\label{eq-short-exact-sequence-of-diagonal}
0\rightarrow I_{\Delta_X}/I^2_{\Delta_X}\rightarrow \StructureSheaf{X\times X}/I^2_{\Delta_X}\rightarrow 
\StructureSheaf{X\times X}/I_{\Delta_X}\rightarrow 0
\end{equation}
over $X\times X$. It induces a natural transformation from the identity endofunctor of $D^b(X)$ to the endofunctor of tensorization by $\Omega^1_X[1]$. We denote by $a_F\in \Hom(F,F\otimes \Omega^1_X[1])$ the value of $a_X$ at an object $F$ of $D^b(X)$. 
The Atiyah class $a_{\iota_*\StructureSheaf{Z}}\in \Ext^1_X(\iota_*\StructureSheaf{Z}, \iota_*\iota^*(\Omega^1_X))$ is thus
the extension class of the short exact sequence of $\StructureSheaf{X}$-modules 
\begin{equation}
\label{eq-exact-sequence-of-O-X-times-X-modules}
0\rightarrow \iota_*\iota^*(\Omega^1_X)\rightarrow 
E
\rightarrow \iota_*\StructureSheaf{Z}\rightarrow 0,
\end{equation}
involving the Fourier-Mukai transforms of $\iota_*\StructureSheaf{Z}$ with respect to the objects in 
(\ref{eq-short-exact-sequence-of-diagonal}). Note that the scheme theoretic support of $E$ is the (non-reducded) first order neighborhood of $Z$ in $X$.
The exact sequence (\ref{eq-exact-sequence-of-O-X-times-X-modules}) 
decomposes as a direct sum of (a) the short exact sequence of the $G$-invariant direct summands, 
\[
0\rightarrow \iota_*\Omega^1_Z\rightarrow E^G
\rightarrow \iota_*\StructureSheaf{Z}\rightarrow 0,
\]
and (b)  the short exact sequence
$
0\rightarrow N^*_{Z/X} \rightarrow N^*_{Z/X} \rightarrow 0\rightarrow 0.
$
Hence, $a_{\iota_*\StructureSheaf{Z}}:\iota_*\StructureSheaf{Z}\rightarrow \iota_*\iota^*\Omega^1_X[1]$
factors through a morphism\footnote{This morphism corresponds to the extension class as $\StructureSheaf{X}$-modules and is different in general from the extension class as  $\StructureSheaf{Z}$-modules.} 
$\iota_*\StructureSheaf{Z}\rightarrow \iota_*\Omega^1_Z[1]$
via the inclusion of $\iota_*\Omega^1_Z[1]$ as the $G$-invariant  direct summand of $\iota_*\iota^*\Omega^1_X[1]$.
It follows that $a_{\iota_*\StructureSheaf{Z},2}:\iota_*\StructureSheaf{Z}\rightarrow \iota_*\iota^*\Omega^2_X[2]$ factors through 
the direct summand $\iota_*\Omega^2_Z[2]$ and so the image of $H^0(X,\wedge^2TX)$ via the obstruction map  (\ref{eq-obstruction-map}) vanishes, since the homomorphism
$\iota^*\Omega^2_X\rightarrow \StructureSheaf{Z}$ induced by an element of the non-trivial $1$-dimensional $G$-representation $H^0(X,\wedge^2TX)$ vanishes on the $G$-invariant direct summand $\Omega^2_Z$.
\item
We will verify in Lemma \ref{lemma-chern-character-of-Lagrangian-structure-sheaf-deforms-in-co-dimension-1} that the rank of the cohomological obstruction map is $1$ (a condition slightly weaker than the rank of the obstruction map being $1$) in some examples of lagrangian surfaces.
\end{enumerate}
\end{rem}

\begin{example}
\label{examples-of-maximally-deformable-Lagrangian-surfaces}
Following are examples of embeddings $\iota:Z\rightarrow X$ of holomorphic lagrangian submanifolds $Z$ with a rank $1$ restriction homomorphism $\iota^*:H^2(X,\ZZ)\rightarrow H^2(Z,\ZZ)$ even though the second Betti number of $Z$ is larger than $1$. 
\begin{enumerate}
\item
\label{example-item-lagrangian-torus}
$Z$ is a lagrangian complex torus. In this case the kernel of $H^2(X,\ZZ)\rightarrow H^2(Z,\ZZ)$ is $c_1(L)^\perp$, for a line bundle $L$ with isotropic $c_1(L)$ \cite{wieneck}. For a generic pair $(X,Z)$ some power $d$ of the line bundle $L$ is base point free and 
the morphism $\varphi_{L^d}:X\rightarrow \linsys{L^d}^*$ is a  lagrangian fibration onto its image with fiber $Z$
\cite{GLR}. Hence, $[Z]$ is a scalar multiple of $c_1(L)^n$, where $n=\dim(X)/2$. If $X$ is of $K3^{[n]}$-type, then $d=1$ and $\varphi_L$ is surjective, so $[Z]=c_1(L)^n$, by \cite{markman-isotropic}.
\item
\label{example-item-Fano-variety-of-lines-on-a-cubic-threefold}
Let $W$ be a $6$ dimensional complex vector space and 
let $V\subset \PP(W)$ be a smooth cubic fourfold. Let $X$ be the Fano variety of lines on $V$. Then $X$ is of $K3^{[2]}$-type, by \cite{beauville-donagi}. Let $Z$ be the Fano variety of lines on a smooth hyperplane section of $V$. Then $Z$ is a lagrangian surface in $X$, by \cite{voisin-lagrangian}. 
The Hodge numbers of $Z$ are $h^{1,0}=5$, $h^{2,0}=10$, and $h^{1,1}=25$, and $\chi(Z)=27$, by
\cite[Eq. (9.5), (9.12), and (10.12)]{clemens-griffiths}.
Note that $Z$ is the zero subscheme of a regular section of $\tau^*$, where $\tau$ is the restriction to $X$ of the tautological rank $2$ subbundle over $Gr(2,6)$. Hence, $[Z]=c_2(\tau^*)$.
Let $L$ be the restriction to $X$ of the ample generator of $\Pic(Gr(2,6))$. The restriction of $L$ to $Z$ is the canonical line bundle of $Z$, by \cite[Prop. 10.3]{clemens-griffiths}.
We have $(c_1(L),c_1(L))=6$. 
One can show that 
$[Z]=\frac{1}{8}(5c_1(L)^2-c_2(TX))$. 
If $\tilde{\tau}:W\rightarrow W$ is an involution with a $5$-dimensional $1$-eigenspace $W^+$ and $V$ is $\tilde{\tau}$-invariant
and we choose $Z$ to be the Fano variety of lines on $V\cap\PP(W^+)$, then $Z$ is a connected component\footnote{The fixed locus consists of two connected components, the other being a cubic surface, see \cite[Sec. 3.6]{beauville-antisymplectic-involutions}.} 
of the fixed locus of the restriction $\tau:X\rightarrow X$ of $\tilde{\tau}$ and so the obstruction map   $HT^2(X)\rightarrow \Ext^2(\iota_*\StructureSheaf{Z},\iota_*\StructureSheaf{Z})$ given in (\ref{eq-obstruction-map}) has rank $1$, by Remark \ref{rem-expected-rank-one}(\ref{rem-item-lagrangian-fixed-locus}).

An example of an immersion of a smooth surface as a lagrangian subvariety in $X$  of class $63c_2(\tau^*)$ with a finite number of self intersection points is studied in \cite{iliev-manivel}.
\item
\label{example-item-fixed-locus-of-anti-symplectic-involution}
Let $X$ be of $K3^{[2]}$-type admitting an ample line-bundle $L$ satisfying $(c_1(L),c_1(L))=2$. The generic such $X$ admits an anti-symplectic involution. O'Grady proved that its fixed locus, which is a lagrangian surface $Z$, embeds in $\linsys{L}^*$ as a smooth surface of degree $40$ and Euler characteristic $192$ \cite{ogrady-sextics-duke}. One can show that 
$[Z]=5c_1(L)^2-\frac{1}{3}c_2(TX)$. The obstruction map   $HT^2(X)\rightarrow \Ext^2(\iota_*\StructureSheaf{Z},\iota_*\StructureSheaf{Z})$ given in (\ref{eq-obstruction-map}) has rank $1$, by Remark \ref{rem-expected-rank-one}(\ref{rem-item-lagrangian-fixed-locus}).
\end{enumerate}
\end{example}

\begin{example}
\label{example-non-modular-P-n-objects}
Let $X$ be a $K3$ surface containing a smooth rational curve $C$, let $\delta:X\rightarrow X\times X$ be the diagonal embedding,
and let $\pi_i$, $i=1,2$, be the projection from $X\times X$ onto the $i$-th factor.
The shift $\Ideal{\Delta}[1]$ of the ideal sheaf  of the diagonal $\Delta:=\delta(X)$ is quasi-isomorphic to  the complex 
$\P:=[\StructureSheaf{X\times X} \rightarrow \delta_*\StructureSheaf{X}]$, with $\delta_*\StructureSheaf{X}$ in degree $0$, and the latter is the Fourier-Mukai kernel of the auto-equivalence $\Phi^\P:=R\pi_{1,*}\circ (\P\stackrel{L}{\otimes} L\pi_1^*(\bullet))$ of $D^b(X)$ known as the spherical twist by the spherical object $\StructureSheaf{X}$ \cite{seidel-thomas}. 
$\Phi^\P(\StructureSheaf{X}(C))$ is isomorphic to the co-kernel of the evaluation homomorphism
$H^0(X,\StructureSheaf{X}(C))\otimes_\ComplexNumbers \StructureSheaf{X}\rightarrow \StructureSheaf{X}(C)$ and is thus the push-forward $\iota_*N_{C/X}$ of the normal bundle of $C$ by the inclusion $\iota:C\rightarrow X$. 
Thus, the composition of $\Phi^\P$ with tensorization by $\StructureSheaf{X}(C)$ is 
an auto-equivalence $\Psi$ of the derived category of $X$ which maps $\StructureSheaf{X}$ to $\iota_*N_{C/X}$. Note the isomorphism $\iota_*N_{C/X}\cong\iota_*\omega_C$. 
Clearly, the latter sheaf deforms with $X$ only over the locus in the moduli space where the class of $C$ in $H^2(X,\CC)$ remains of Hodge type. In this case we have
\[
\Sigma(\Psi(\StructureSheaf{X}))\cap \left[H^2(\StructureSheaf{X})\oplus H^1(TX)\right]
=\Sigma(\iota_*\omega_C)\cap \left[H^2(\StructureSheaf{X})\oplus H^1(TX)\right]= H^2(\StructureSheaf{X})\oplus [C]^\perp,
\]
where $[C]^\perp$ is the hyperplane of $H^1(TX)$ along which $C$ remains of type $(1,1)$. The second equality above si due to the fact that the right hand side is the unique hyperplane of $H^2(\StructureSheaf{X})\oplus H^1(TX)$ projecting onto $[C]^\perp$.
\end{example}

\hide{
The characteristic class $\kappa(F)$ is not defined for objects of rank $0$. Examples \ref{example-lagrangian-Z-with-modular-structure-sheaf} and  \ref{example-non-modular-P-n-objects}  of  objects supported on lagrangian submanifolds in higher dimensional irreducible holomorphic symplectic manifolds suggest the following conjectural version of 
Proposition \ref{prop-kappa-class-remains-of-Hodge-type}
omitting the rank constraint but restricting to a characteristic class in the smaller subring 
$\oplus_{k\geq 0} H^{4k}(X,\RationalNumbers)$.

\begin{conj}
\label{conj-constraint-on-cohomological-action-of-derived-equivalences}
Let $X$ be a projective irreducible holomorphic symplectic manifold. Let $G$ be an object in $D^b(X)$ 
with $\Sigma(G)$ of co-dimension $1$ in $HT^2(X)$. 
Then $ch(G)ch(G)^\vee$ belongs to the subring of $H^*(X,\RationalNumbers)$ of classes that remain of 
Hodge type under every K\"{a}hler deformation of $X$.
\end{conj}

Note that the conclusion holds if $G$ is an object supported on a lagrangian subvariety, since in this case  the product $ch(G)ch(G)^\vee$ is a class in $H^{4n}(X,\QQ)$. Note also that the statement of the conjecture follows from Proposition \ref{prop-kappa-class-remains-of-Hodge-type}, if the
orbit of $v(G)$ under the derived monodromy group of $X$ (Definition \ref{def-DMon}) (??? make sure it is already defined ???) has a Zariski dense open subset consisting of classes of non-zero rank, since being annihilated by $\bar{\LieAlg{g}}$ is a closed condition for classes in $H^*(X,\QQ)$ and $ch(G)ch(G)^\vee$ depends continuously on $v(G)$. Density of the orbit holds for $X$ of $K3^{[n]}$-type, by Remark
\ref{rem-DMon-is-large-for-K3-[n]-type}.
}
%
\subsection{When can a very modular sheaf on $X$ be deformed to an untwisted sheaf on $Y$?}
\label{sec-lift-to-an-untwisted-sheaf}
Let $X$ and $Y$ be deformation equivalent irreducible holomorphic symplectic manifolds. 
Let $E$ be be a reflexive sheaf of positive rank $r$ on $X$ with a rank $1$ obstruction map $\obs^{HT}_E$, given in
(\ref{eq-obstruction-map}). Assume that $E$ is slope-stable with respect to K\"{a}hler classes in an open subcone of the K\"{a}hler cone of $X$. 
Then $E$ can be deformed to a twisted sheaf $F$ over $Y$, by Theorem \ref{thm-modularity-of-a-stable-sheaf-with-a-rank-1-obstruction-map}. When can the Azumaya algebra $\SheafEnd(E)$ be deformed to the Azumaya algebra $\SheafEnd(F)$ of 
an untwisted sheaf $F$ over $Y$? The answer depends on $c_1(E)$ and 
$\Pic(Y)$.
Lemma \ref{lemma-a-necessariy-and-sufficient-condition-for-the-existence-of-a-lift-to-an-untwisted-sheaf} provides a general necessary and sufficient condition and Lemma \ref{lemma-easy-to-check-sufficient-condition} provides a condition in a form easier to check for $Y$ of $K3^{[n]}$ or generalized Kummer deformation type.

\begin{lem}
\label{lemma-a-necessariy-and-sufficient-condition-for-the-existence-of-a-lift-to-an-untwisted-sheaf}
There exists an untwisted reflexive sheaf $F$ such that $(Y,\SheafEnd(F))$ is deformation equivalent to $(X,\SheafEnd(E))$, if and only if there exists a parallel transport operator $g:H^2(X,\Integers)\rightarrow H^2(Y,\Integers)$, such that $g(c_1(E))+rH^2(Y,\Integers)$
belongs to the image of $H^{1,1}(Y,\Integers)$ in $H^2(Y,\Integers)/rH^2(Y,\Integers)$.
\end{lem}

\begin{proof}
The existence of such a parallel-transport operator $g$ is clearly necessary. We prove that it is sufficient.
Assume first that $F$ is locally free. Let $\mu_r$ be the group of $r$-th roots of unity. Denote by $\mu_r$ also the associated trivial local system over $X$. 
Let
\[
\tilde{\theta}:H^1_{an}(X,PGL_r(\StructureSheaf{}))\rightarrow H^2(X,\mu_r)
\]
be the connecting homomorphism of the short exact sequence
\[
0\rightarrow \mu_r\rightarrow SL_r(\StructureSheaf{})\rightarrow PGL_r(\StructureSheaf{})\rightarrow 0.
\]
The projective bundle $\PP(E)$ determines a class $[\PP(E)]$ in $H^1(X,PGL_r(\StructureSheaf{}))$.
The exponential function $\exp:\CC\rightarrow \CC^*$ maps the subgroup $\frac{2\pi i}{r}\Integers$ of $\CC$ onto $\mu_r$ and induces the homomorphism
\[
\exp:H^2\left(X,\frac{2\pi i}{r}\Integers\right)\rightarrow H^2(X,\mu_r).
\]
The following equality is proved in \cite[Lemma 2.5]{huybrechts-schroer}
\begin{equation}
\label{eq-HS-lemma2.5}
\tilde{\theta}([\PP(E)])=
\exp\left(
\frac{2\pi i}{r}c_1(E)
\right).
\end{equation}

Choose an isometry $\eta:H^2(X,\Integers)\rightarrow \Lambda$ with a fixed lattice $\Lambda$.
Let $\fM_\Lambda$ be the moduli space  of marked irreducible holomorphic symplectic manifolds recalled in the proof of Theorem \ref{thm-modularity-of-a-stable-sheaf-with-a-rank-1-obstruction-map}.
Let $\fM_\Lambda^0$ be the connected component of $\fM_\Lambda$ containing $(X,\eta)$. Then
$(Y,g\circ \eta)$ belongs to $\fM_\Lambda^0$, since $g$ is a parallel-transport operator. Choose a generic twistor path $\gamma$ from $(X,\eta)$ to $(Y,g\circ \eta)$ with first twistor line associated to a K\"{a}hler class with respect to which $F$ is slope-stable. Then $\PP(E)$ deforms along $\gamma$ to a projective bundle $B$ with 
\begin{equation}
\label{eq-characteristic-class-of-B}
\tilde{\theta}([B])=\exp\left(\frac{2\pi i}{r}g(c_1(E))\right).
\end{equation} 
Let $\iota:\mu_r\rightarrow \StructureSheaf{X}^*$ be the inclusion. 
There exists a locally free sheaf $F$ over $Y$ with $\PP(F)=B$, if and only if the Brauer class 
$\iota(\tilde{\theta}(B))\in
H^2_{an}(Y,\StructureSheaf{}^*)$ of $B$ vanishes. 
This is the case, if and only if $\tilde{\theta}([B])$ belongs to the image of the connecting homomorphism 
$\delta: H^1_{an}(Y,\StructureSheaf{}^*)\rightarrow H^2_{an}(Y,\mu_r)$ of the short exact sequence
\[
0\rightarrow \mu_r \RightArrowOf{\iota} \StructureSheaf{Y}^*\RightArrowOf{(\bullet)^r} \StructureSheaf{Y}^*\rightarrow 0.
\]
This is the case, if and only if there exists a line bundle $L\in \Pic(Y)$, such that the class $c_1(L)-g(c_1(E))$
belong to $rH^2(Y,\Integers)$, as seen from  the commutative diagram
\[
\xymatrix{
& & [L] & \theta([B])
\\
0 \ar[r] &H^1(Y,\StructureSheaf{}^*) \ar[r]^{(\bullet)^r} \ar[d]_{c_1}^{\cong} &
H^1(Y,\StructureSheaf{}^*) \ar[r]^{\delta} \ar[d]_{c_1}^{\cong} & H^1(Y,\mu_r) \ar[r]^{\iota} \ar[d]_{\cong}& H^2(Y,\StructureSheaf{}^*)
\\
0 \ar[r] & H^2(Y,\ZZ) \ar[r]_{\cdot r} & H^2(Y,\ZZ) \ar[r] & H^2(Y,\ZZ/r\ZZ)
\\
& & c_1(L) & c_1(L)+rH^2(Y,\ZZ)
}
\] 
and the fact that the right vertical isomorphism above maps the right hand side of (\ref{eq-characteristic-class-of-B})
to $g(c_1(E))+rH^2(Y,\ZZ)$.

The above argument goes through even if $E$ is reflexive, but not locally free, since then $E$ is locally free away from a closed analytic subset $Z$ of co-dimension $\geq 3$. In that case Poincare and
Lefschetz dualities yield the commutative diagram whose vertical arrows are isomorphisms.
\[
\xymatrix{
H_{4n-2}(X,\Integers) \ar[r]^\cong \ar[d]_\cong^{P.D.} & H_{4n-2}(X,Z,\Integers)\ar[d]^\cong_{L.D.}
\\
H^2(X,\Integers) \ar[r] & H^2(X\setminus Z,\Integers),
}
\]
where $2n$ is the complex dimension of $X$. The top horizontal arrow is an isomorphism as well, since $H_{4n-k}(Z,\Integers)$ vanishes for $k<6$. Hence, the bottom vertical arrow is an isomorphism as well. Consequently, the restriction homomorphism $H^2(X,\mu_r) \rightarrow H^2(X\setminus Z,\mu_r)$ is an isomorphism as well and the class $\tilde{\sigma}(\PP(\restricted{E}{X\setminus Z}))\in H^2(X\setminus Z,\mu_r)$ is the restriction of the class 
$\exp\left(
\frac{2\pi i}{r}c_1(E)
\right)$
in $H^2(X,\mu_r)$. Similarly,
$\tilde{\sigma}(\PP(\restricted{F'}{Y\setminus Z'}))$  is the restriction of the parallel transport 
$\exp\left(
\frac{2\pi i}{r}g(c_1(E))
\right)$
in $H^2(Y,\mu_r)$. 
where $Z'$ is the singular locus of the deformed twisted reflexive sheaf $F'$ along the twistor path $\gamma$, by
\cite[Construction 6.9]{markman-BBF-class-as-characteristic-class}.
\end{proof}

\begin{defi}
\label{def-divisibility}
Let $\Lambda$ be a lattice. A non-zero class $\lambda\in \Lambda$ is {\em primitive}, if the rank $1$ lattice $\span_\ZZ\{\lambda\}$ is saturated in $\Lambda$. 
The {\em divisibility} $\div(\lambda)$ of a primitive class $\lambda$ of $\Lambda$ is the greatest common divisor of $\{(\lambda,\lambda') \ : \lambda'\in\Lambda\}$.
\end{defi}

Let $L$ be a lattice. Denote by $\widetilde{O}(L)$ the subgroup of isometries acting trivially on the discriminant group $L^*/L$. Let $O_+(L)$ be the subgroup of isometries preserving the orientation of the positive cone \cite[Sec. 4]{markman-survey}.
Let $SO(L)$ be the subgroup of isometries with determinant $1$. Denote by
$S\widetilde{O}_+(L)$ the intersection of these three subgroups.

\begin{lem}
\label{lemma-faithful-invariants-of-orbits}
Let $X$ be an irreducible holomorphic symplectic manifold of $K3^{[n]}$-type or generalized Kummer deformation type of dimension $2n$, $n\geq 2$, and let $\alpha$ be a primitive class in $H^2(X,\Integers)$ with $\div(\alpha)=1$ or $2$.
Then the $\Mon^2(X)$-orbit of $\alpha$ is determined by $(\alpha,\alpha)$ and $\div(\alpha)$.
\end{lem}

\begin{proof}
The assumption on $\div(\alpha)$ determines the image of $(\alpha,\bullet)/\div(\alpha)$ in the discriminant group 
$H^2(X,\Integers)^*/H^2(X,\Integers)$, since the latter is cyclic of even order, so contains a unique element of order $2$.
Eichler's criterion then implies that the $\widetilde{O}(H^2(X,\Integers))$-orbit of $\alpha$ is determined by $\div(\alpha)$ and
$(\alpha,\alpha)$ \cite[Lemma 3.5]{GHS}. 
It remains to prove that the $\Mon^2(X)$-orbit of $\alpha$ is equal to the
$\widetilde{O}(H^2(X,\Integers)$-orbit of $\alpha$. 
The discriminant group $H^2(X,\ZZ)^*/H^2(X,\ZZ)$ is cyclic of even order and 
The $\Mon^2(X)$-orbit of $\alpha$ is contained in the
$\widetilde{O}(H^2(X,\Integers)$-orbit of $\alpha$, since both $(\alpha,\alpha)$ and $\div(\alpha)$ are $\Mon^2(X)$-invariant. 
Now $\Mon^2(X)$ contains $S\widetilde{O}_+(H^2(X,\Integers)),$
by \cite[Cor. 1.8]{markman-monodromy} and \cite[Theorem 1.4]{markman-generalized-kummers}.
Hence, it suffices to prove that the $\widetilde{O}(H^2(X,\Integers))$-orbit 
and $S\widetilde{O}_+(H^2(X,\Integers))$-orbit
of $\alpha$ are equal. The orthogonal complement $\alpha^\perp$ of every class $\alpha$ contains a copy of a unimodular rank $2$ even lattice $U$ of signature $(1,1)$ and $O(U)$ and $S\widetilde{O}_+(H^2(X,\Integers))$ generate
$\widetilde{O}(H^2(X,\Integers)$, where we embed $O(U)$ as a subgroup of $\widetilde{O}(H^2(X,\Integers)$ stabilizing $\alpha$ by extending each isometry via the identity on $U^\perp$. Hence, indeed, the $\widetilde{O}(H^2(X,\Integers))$-orbit 
and $S\widetilde{O}_+(H^2(X,\Integers))$-orbit
of $\alpha$ are equal. 
\end{proof}

We provide next a sufficient condition for the existence of a parallel transport operator as in Lemma 
\ref{lemma-a-necessariy-and-sufficient-condition-for-the-existence-of-a-lift-to-an-untwisted-sheaf} in a special case.
Let $X$ and $Y$ be deformation equivalent $2n$-dimensional, $n\geq 2$, irreducible holomorphic symplectic manifolds  of $K3^{[n]}$-type or of generalized Kummer type. Let $E$ be the sheaf in 
Lemma \ref{lemma-a-necessariy-and-sufficient-condition-for-the-existence-of-a-lift-to-an-untwisted-sheaf} and $r$ its rank.
Write $c_1(E)=a\alpha$, such that $\alpha$ is a primitive class in $H^2(X,\Integers)$. 
Set $\rho:=r/\gcd(a,r)$. 

\begin{lem}
\label{lemma-easy-to-check-sufficient-condition}
Assume that $\gcd(\rho,\div(\alpha))$ is $1$ or $2$.
There exists a reflexive sheaf $F$ on $Y$, such that $(Y,\SheafEnd(F))$ is deformation equivalent to $(X,\SheafEnd(E))$  
if and only if 
there exist a primitive class $\beta\in H^{1,1}(Y,\Integers)$ and an integer $k$ satisfying $\gcd(k,\rho)=1$, 
\begin{eqnarray}
\label{eq-divisibilities-are-related}
\gcd(\rho,\div(\beta))&=&\gcd(\rho,\div(\alpha)),
\\
\label{eq-congruence-relation-involving-degrees-of-alpha-beta}
k^2\frac{(\beta,\beta)}{2}&\equiv& \frac{(\alpha,\alpha)}{2} \ \ \left\{
\begin{array}{ccl}
\left(\mod \ \rho\right) & \mbox{if} & \gcd(\rho,\div(\alpha))=1,
\\
\left(\mod \ 2\rho\right) & \mbox{if} & \gcd(\rho,\div(\alpha))=2.
\end{array}
\right.
\end{eqnarray}
If such $\beta$ and $k$ exist, then $F$ may be chosen with $c_1(F)=ka\beta$.
\end{lem}

Note that $\div(\alpha)$ and $\div(\beta)$ divide $2n-2$ in the $K3^{[n]}$-type case and they divide $2n+2$ in the generalized Kummer case. In the $K3^{[2]}$-type case  the condition that $\gcd(\rho,\div(\alpha))$ is $1$ or $2$ is automatically satisfied.

\begin{proof}
Such $F$ exists, if and only if 
there exists a parallel-transport operator $g:H^2(X,\Integers)\rightarrow H^2(Y,\Integers)$, such that 
$g(a\alpha)+rH^2(Y,\Integers)=b\beta+rH^2(Y,\Integers)$, for some integer $b$ and a primitive class $\beta\in H^{1,1}(Y,\Integers)$, 
by Lemma 
\ref{lemma-a-necessariy-and-sufficient-condition-for-the-existence-of-a-lift-to-an-untwisted-sheaf}. 
The order of the left hand side coset is $r/\gcd(a,r)$ and the order of the right hand side coset is $r/\gcd(b,r)$. Hence, 
two integers $a$ and $b$ satisfying the latter equality must satisfy $\gcd(a,r)=\gcd(b,r)$ and so $\frac{a}{\gcd(a,r)}$
and $\frac{b}{\gcd(a,r)}$ are both invertible modulo $\rho$. Furthermore, the equation of cosets above  is equivalent to
\[
\frac{a}{\gcd(a,r)}g(\alpha)+\rho H^2(Y,\Integers)=\frac{b}{\gcd(a,r)}\beta+\rho H^2(Y,\Integers)
\]
The latter equality is equivalent to
\begin{equation}
\label{eq-the-needed-parallel-transport-operator}
g(\alpha)+\rho H^2(Y,\Integers)=k\beta+\rho H^2(Y,\Integers),
\end{equation}
for some integer $k$ satisfying $\gcd(k,\rho)=1$, by the invertibility of $\frac{a}{\gcd(a,r)}$ modulo $\rho$.
Equations (\ref{eq-divisibilities-are-related}) and  (\ref{eq-congruence-relation-involving-degrees-of-alpha-beta}) follow from Equation (\ref{eq-the-needed-parallel-transport-operator}). The latter is clear if modulo $\rho$ and if $\gcd(\rho,\div(\alpha))=2$
and $g(\alpha)=k\beta+\rho\lambda$, then
\[
(\alpha,\alpha)=(g(\alpha),g(\alpha))=(k\beta+\rho\lambda,k\beta+\rho\lambda)=
k^2(\beta,\beta)+2k\rho (\beta,\lambda)+\rho^2(\lambda,\lambda),
\]
and so the congruence (\ref{eq-the-needed-parallel-transport-operator}) holds modulo $2\rho$, since both $\rho$ and $(\beta,\lambda)$ are even.
Hence, if a sheaf $F$ satisfying the condition of the Lemma exists, then an integer $k$ and a class $\beta\in H^{1,1}(Y,\Integers)$ satisfying the conditions of the Lemma exist as well.

Conversely, assume that an integer $k$ and a class $\beta\in H^{1,1}(Y,\Integers)$ satisfying the conditions of the Lemma exist.
We prove the existence of a sheaf $F$ by 
proving that there exists a parallel-transport operator $g:H^2(X,\Integers)\rightarrow H^2(Y,\Integers)$
satisfying Equation (\ref{eq-the-needed-parallel-transport-operator}).

Step 1: We prove first that there exists a class $\alpha_1\in H^2(X,\Integers),$ such that $\div(\alpha+\rho\alpha_1)=\gcd(\rho,\div(\alpha))$.
There exists a class $\lambda\in H^2(X,\Integers)$, such that $\gcd((\alpha,\lambda),\rho)=\gcd(\rho,\div(\alpha))$.
Choose $x,y\in \Integers$, such that $x(\alpha,\lambda)+y\rho=\gcd(\rho,\div(\alpha))$.
There exists a class $\delta\in H^2(X,\Integers)$, such that $\delta^\perp$ is unimodular isomorphic to the orthogonal direct sum $U^{\oplus 3}$ (generalized Kummer case) or
$U^{\oplus 3}\oplus E_8(-1)^{\oplus 2}$ ($K3^{[n]}$-type case), where $U$ is the even unimodular rank $2$ lattice of signature $(1,1)$ and $E_8(-1)$ is the negative definite $E_8$ lattice. Write $\lambda=\lambda'+\lambda''$ with $\lambda'\in\delta^\perp$ and $\lambda''\in\Integers\delta$. Decompose $\alpha=\alpha'+\alpha''$ similarly. 
We may and do choose $\lambda$ so that the saturation of the lattice spanned by $\lambda'$ and $\alpha'$ is isometric to $U$. 
Indeed, choose a sublattice $U_0$ of $\delta^\perp$ isometric to $U$ containing $\alpha'$. Then there exists an element $u$ of
$U_0$, of the same divisibility as $\lambda'$, such that $(u,\alpha')=(\lambda',\alpha')$ and so we can replace $\lambda$ by $u+\lambda''$.
Now $\{\alpha',\lambda',\delta\}^\perp$ contains a copy of $U$  and so it contains classes $\alpha_1$ and $\alpha_2$ with $(\alpha_1,\alpha_2)=y$. Then $\div(\alpha+\rho\alpha_1)=\gcd(\rho,\div(\alpha))$, since
\[
(\alpha+\rho\alpha_1,x\lambda+\alpha_2)=x(\alpha,\lambda)+\rho(\alpha_1,\alpha_2)=\gcd(\rho,\div(\alpha)).
\]

Step 2: There exists a class $\beta_1\in H^2(Y,\Integers)$, such that $\div(\beta+\rho\beta_1)=\gcd(\rho,\div(\beta))$, by Step 1.
Set $\tilde{\alpha}=\alpha+\rho\alpha_1$ and $\tilde{\beta}:=\beta+\rho\beta_1$. The $\Mon^2(X)$-orbit of $\tilde{\alpha}$ consists of primitive classes $\lambda$ with $(\lambda,\lambda)=(\tilde{\alpha},\tilde{\alpha})$ and $\div(\lambda)=\div(\alpha)$,
by 
Lemma \ref{lemma-faithful-invariants-of-orbits}.
Similarly, the
$\Mon^2(Y)$-orbit of $\tilde{\beta}$ consists of primitive classes $\lambda$ with $(\lambda,\lambda)=(\tilde{\beta},\tilde{\beta})$ and $\div(\lambda)=\div(\beta)$. 
Let $f:H^2(X,\Integers)\rightarrow H^2(Y,\Integers)$ be a parallel-transport operator. 
There exists a class $\delta\in H^2(Y,\Integers)$, such that $\delta^\perp$ is unimodular. 
Set $\Lambda_2:=\delta^\perp$.
Fix $k$ solving the congruence (\ref{eq-congruence-relation-involving-degrees-of-alpha-beta}).
Then there exist integers $x$, $y$ satisfying
\[
\frac{(\tilde{\alpha},\tilde{\alpha})-k^2(\tilde{\beta},\tilde{\beta})}{2\rho}=xk+y\rho,
\]
since $\gcd(k,\rho)=1$. We claim that $x$ can be chosen to be even, if $\div(\alpha)=2$. Indeed, if 
$\gcd(\rho,\div(\alpha))=2$, then $\rho$ is even, $k$ is odd, and 
the left hand side of the above displayed equation is even, by assumption (\ref{eq-congruence-relation-involving-degrees-of-alpha-beta}), and so $x$ must be even. If $\div(\alpha)=2$ and $\rho$ is odd, then a solution $(x,y)$ can be replaced by $(x-\rho,y+k)$, and so $x$ may be chosen even.

\underline{
Case  $\div(\tilde{\beta})=1$:}
In this case $\delta^\perp$ contains classes in the $\Mon^2(Y)$-orbit of $\tilde{\beta}$. So $\tilde{\beta}^\perp$ contains classes in the $\Mon^2(Y)$-orbit of $\delta$ and we may choose $\delta$ so that 
 $(\tilde{\beta},\delta)=0$. 
Let $T$ be the rank $2$ lattice with basis $\{t_1,t_2\}$ satisfying
\[
\left(
\begin{array}{cc}
(t_1,t_1) & (t_1,t_2)
\\
(t_2,t_1) & (t_2,t_2)
\end{array}
\right)=
\left(
\begin{array}{cc}
(\tilde{\beta},\tilde{\beta}) & x
\\
x & 2y
\end{array}
\right)
\]
There exists a primitive embedding $\tau:T\rightarrow \Lambda_2$ such that $\tau(t_1)=\tilde{\beta}$, by results of Nikulin (see \cite[Lemma 8.1]{markman-monodromy}).
Then $\gamma:=k\tilde{\beta}+\rho\tau(t_2)$ is a primitive class of $\Lambda_2$ satisfying $(\gamma,\gamma)=(\tilde{\alpha},\tilde{\alpha})$.
Furthermore, $\div(\gamma)=1$, since $\Lambda_2$ is unimodular.
Hence, there exists an element $h$ of $\Mon^2(Y)$, such that $h(f(\tilde{\alpha})=\gamma$ and $h(f(\delta))=f(\delta)$, since $Mon^2(Y)$ contains $SO^+(\Lambda_2)$, by \cite[Theorem 1.2]{markman-monodromy} and \cite[Theorem 1.4]{markman-generalized-kummers}. 
So $g:=h\circ f$ is a parallel-transport operator satisfying Equation
(\ref{eq-the-needed-parallel-transport-operator}).

\underline{
Case  $\div(\tilde{\beta})=2$:}
The $\Mon^2(Y)$-orbit of $\tilde{\beta}$ contains a class of the form
$2\tilde{\beta}'+\delta$, where $\tilde{\beta}'$ is a primitive class in $\Lambda_2$, 
by Lemma \ref{lemma-faithful-invariants-of-orbits}.
Hence, we may assume that $\tilde{\beta}=2\tilde{\beta}'+\delta$ is of this form, possibly after modifying the choice of $\delta$.
Recall that $x$ can be chosen even in the current case.
Let $T$ be the rank $2$ lattice with basis $\{t_1,t_2\}$ satisfying
\[
\left(
\begin{array}{cc}
(t_1,t_1) & (t_1,t_2)
\\
(t_2,t_1) & (t_2,t_2)
\end{array}
\right)=
\left(
\begin{array}{cc}
(\tilde{\beta}',\tilde{\beta}') & x/2
\\
x/2 & 2y
\end{array}
\right).
\]
There exists a primitive embedding $\tau:T\rightarrow \Lambda_2$ such that $\tau(t_1)=\tilde{\beta}'$, by results of Nikulin (see \cite[Lemma 8.1]{markman-monodromy}.
Then $\gamma:=k\tilde{\beta}+\rho\tau(t_2)=(2k\tilde{\beta}'+\rho\tau(t_2))+k\delta$ is a primitive class satisfying $(\gamma,\gamma)=(\tilde{\alpha},\tilde{\alpha})$.
We claim that $\div(\gamma)=2$. Indeed, $2$ divides $\div(\gamma)$, since $\rho$ is even, as $\gcd(\rho,\div(\beta))=2$.
In addition, $\div(\alpha)$ divides $2$, since $k\tilde{\beta}'+(\rho/2)\tau(t_2)$ is a primitive class of $\Lambda_2$ and $\Lambda_2$ is unimodular.
Hence, there exists an element $h$ of $\Mon^2(Y)$, such that $h(f(\tilde{\alpha}))=\gamma$, 
by Lemma \ref{lemma-faithful-invariants-of-orbits}. Setting
$g:=h\circ f$ we get a parallel-transport operator satisfying Equation
(\ref{eq-the-needed-parallel-transport-operator}).
\end{proof}

Given a $K3$ surface $S$, let $S^{(n)}$ be its $n$-th symmetric power and let 
\begin{equation}
\label{eq-theta}
\theta:H^2(S,\Integers)\rightarrow H^2(S^{[n]},\Integers)
\end{equation}
be the composition of the isomorphism $H^2(S,\Integers)\cong H^2(S^{(n)},\Integers)$ with the pull-back homomorphism $H^2(S^{(n)},\Integers)\rightarrow H^2(S^{[n]},\Integers)$ via the Hilbert-Chow morphism.

\begin{example}
Let $X=S^{[2]}$, $S$ a $K3$ surface, and let $E$ be the vector bundle $G^{[2]}$ over $S^{[2]}$ associated to a stable and rigid vector bundle $G$ of rank $r_0$ on $S$ as in Theorem \ref{thm-BKR-of-tensor-product-of-spherical-object}. Let $\delta\in H^2(S^{[2]},\Integers)$ be half the class of the divisor of non-reduced subschemes.
Then $c_1(E)=r_0\theta(c_1(G))-\frac{r_0(r_0-1)}{2}\delta$, by \cite[Prop. 5.8]{ogrady-modular}. 
If $v(G)=(r_0,c_1(G),s),$ then $(v(G),v(G))=-2$, and so 
$
\frac{(c_1(G),c_1(G))}{2}=sr_0-1.
$
Assume that $c_1(G)$ is a primitive class.
If $r_0$ is odd, then $c_1(E)=r_0\alpha$, where $\alpha=\theta(c_1(G))-\frac{r_0-1}{2}\delta$ is a primitive class of divisibility $1$ satisfying 
\begin{equation}
\label{eq-constraint-on-c-1-G}
\frac{(\alpha,\alpha)}{2}=\frac{(c_1(G),c_1(G))}{2}-\left(\frac{r_0-1}{2}\right)^2=sr_0-1-\left(\frac{r_0-1}{2}\right)^2.
\end{equation}
If $r_0$ is even, then $c_1(E)=(r_0/2)\alpha$, where $\alpha=2\theta(c_1(G))-(r_0-1)\delta$ is a primitive class of divisibility $2$ satisfying 
\[
\frac{(\alpha,\alpha)}{2}=2(c_1(G),c_1(G))-\left(r_0-1\right)^2=4sr_0-4-\left(r_0-1\right)^2.
\]

If $r_0=2$ and $c_1(G)$ is primitive, then
$c_1(E)=2\theta(c_1(G))-\delta$ is a primitive class of divisibility $2$. So $\rho=\rank(E)=4$. An untwisted vector bundle $F$ exists on an irreducible holomorphic symplectic manifold of $K3^{[2]}$-type $Y$, such that $(S^{[2]},\PP(E))$ is deformation equivalent to $(Y,\PP(F))$, if and only if there exists a primitive class $\beta\in H^{1,1}(Y,\Integers)$ of divisibility $2$, such that
$(\beta,\beta)/2\equiv (\alpha,\alpha)/2$ (mod $8$),
by Lemma \ref{lemma-easy-to-check-sufficient-condition}. 
The latter congruence is equivalent to (\ref{eq-congruence-relation-involving-degrees-of-alpha-beta}), since if $k$ is odd, then $k^2\equiv 1$ (mod $8$).

If $r_0=3$ and $c_1(G)$ is primitive, then $c_1(E)=3\alpha$, where $\alpha=\theta(c_1(G))-\delta$ is a primitive class of divisibility $1$ satisfying $(\alpha,\alpha)/2\equiv 1$ (mod $3$), by Equation (\ref{eq-constraint-on-c-1-G}). So $\rank(E)=9$ and $\rho=3$. An untwisted vector bundle $F$ exists on an irreducible holomorphic symplectic manifold of $K3^{[2]}$-type $Y$, such that $(S^{[2]},\PP(E))$ is deformation equivalent to $(Y,\PP(F))$, if and only if there exists a primitive class $\beta\in H^{1,1}(Y,\Integers)$ (with no constraint on its divisibility), such that $(\beta,\beta)/2\equiv 1$ (mod $3$),
by Lemma \ref{lemma-easy-to-check-sufficient-condition}. Note: 
The congruence condition $(\beta,\beta)/2\equiv 1$ (mod $3$) follows also from Remark \ref{rem--is deformation-of-FM-image-of-structure-sheaf} and Equation (\ref{r-0-odd}), which imply that  $5+2(\beta,\beta)\equiv 0$ (mod $3$). 
\end{example}

%
\section{The BKR equivalence}
\label{sec-BKR-equivalence}

Let $S$ be a smooth projective surface.
We review in this section the equivalence between the equivariant derived category of coherent sheaves on the $n$-th cartesian product $S^n$ and the derived category of coherent sheaves on the Hilbert scheme $S^{[n]}$. This equivalence relies on results of Bridgeland-King-Ried \cite{BKR} and Haiman \cite{haiman}.

Let $X$ be a smooth projective variety on which a finite group $G$ acts. A {\em $G$-linearization} of a coherent sheaf $E$ on $X$
is given by isomorphisms $\lambda_g:E\rightarrow g^*E$, for $g\in G$, satisfying $\lambda_1=id_E$ and 
$\lambda_{gh}=h^*(\lambda_g)\circ \lambda_h$. A morphism of linearized sheaves $f:(E_1,\lambda_1)\rightarrow (E_2,\lambda_2)$ is an $\StructureSheaf{X}$-module homomorphism $f:E_1\rightarrow E_2$ satisfying 
$f=\lambda_{2,g}^{-1}\circ g^*(f)\circ \lambda_{1,g}$, for all $g\in G$. We get the abelian category of $G$-linearized coherent sheaves, whose bounded derived category is denoted by $D^b_G(X)$ \cite{BKR,ploog}.

Let $(X,G)$ and $(X',G')$ be smooth projective varieties with group actions. 
An equivariant morphism from $(X,G)$ to $(X',G')$ consists of a 
morphism $f:X\rightarrow X'$ and a group homomorphism $\varphi:G\rightarrow G'$ satisfying
$f\circ g=\varphi(g)\circ f$, for all $g\in G$. We get the derived pullback $L(f,\varphi)^*:D^b_{G'}(X')\rightarrow D^b_G(X)$.
When $\varphi$ is surjective we get also the derived pushforward $R(f,\varphi)^G_*:D^b_G(X)\rightarrow D^b_{G'}(X')$ \cite[Sec. 1.3]{ploog}. 
When $G=G'$ and $\varphi$ is the identity, or when $G'=1$, we will denote $L(f,\varphi)^*$ by $Lf^*$ and
$R(f,\varphi)^G_*$ by $Rf^G_*$. When $G'=1$, then $Rf^G_*(E,\lambda)$ is the $G$-invariant sub-object of
$Rf_*E$ with respect to the $G$-action on $Rf_*E$ induced by $\lambda$.

Let $G$ act on two smooth projective varieties $X$ and $Y$. Denote by $G_\Delta$ the diagonal subgroup of $G\times G$.
Let $(P,\rho)$ be an object of $D^b_{G_\Delta}(X\times Y)$. 
Let $(\pi_X,\pi_1):(X\times Y,G_\Delta)\rightarrow (X,G)$ and $(\pi_Y,\pi_2):(X\times Y,G_\Delta)\rightarrow (Y,G)$ be the projections from $X\times Y$ to $X$ and $Y$ equivariant with respect to the projections from the subgroup $G_\Delta$ of $G\times G$ to the two factors. We get the Fourier-Mukai transformation
$\Phi_{(P,\rho)}:D^b_G(X)\rightarrow D^b_G(Y)$, given by
\begin{equation}
\label{eq-equivariant-FM-transformation}
\Phi_{(P,\rho)}(\bullet)=R(\pi_Y,\pi_2)_*^{G}\left(
L(\pi_X,\pi_1)^*(\bullet)\stackrel{L}{\otimes} (P,\rho)
\right).
\end{equation}
Tensorization of a linearization $\lambda$ by a character $\chi$ of $G$ yields the linearization $(\chi\otimes\lambda)_g=\chi(g)\lambda_g$. We get an auto-equivalence $\Phi_\chi$ of each of $D^b_G(X)$ and $D^b_G(Y)$
sending an object $(E,\lambda)$ to $(E,\chi\otimes\lambda)$ and sending a morphism $f:(E_1,\lambda_1)\rightarrow (E_2,\lambda_2)$ to itself under the equality of $\Hom((E_1,\lambda_1),(E_2,\lambda_2))$ and $\Hom((E_1,\chi\otimes\lambda_1),(E_2,\chi\otimes\lambda_2))$ as subspaces of $\Hom(E_1,E_2)$. 
Note that $\Phi_\chi$ commutes with the Fourier-Mukai transformation $\Phi_{(P,\rho)}$ above with kernel $(P,\rho)\in D^b_{G_\Delta}(X\times Y)$,
\begin{equation}
\label{eq-Phi-chi-commutes-with-FM-with-diagonally-linearized-kernel}
\Phi_\chi\circ\Phi_{(P,\rho)}\cong \Phi_{(P,\rho)}\circ\Phi_\chi.
\end{equation}

Let $S_1$ and $S_2$ be smooth projective varieties and assume that the Fourier-Mukai transformation $\Phi_\P:D^b(S_1)\rightarrow D^b(S_2)$ with kernel an object $P\in D^b(S_1\times S_2)$ is an equivalence. 
Let $X=S_1^n$, $Y=S_2^n$, and let $G=\fS_n$ act on both by permuting the factors. We get the object
$P^{\boxtimes n}:=P\boxtimes \cdots \boxtimes P$ in $D^b_{\fS_{n,\Delta}}(S_1^n\times S_2^n)$, endowed with the linearization $\rho_\boxtimes$ permuting the factors, and 
\begin{equation}
\label{eq-outer-product-of-FM}
\Phi_{(P^{\boxtimes n},\rho_\boxtimes)}:D^b_{\fS_n}(S_1^n)\rightarrow D^b_{\fS_n}(S_2^n)
\end{equation} 
is an equivalence as well. 

The definition of the functor $\Phi_{(P,\rho)}$ in (\ref{eq-equivariant-FM-transformation}) admits an obvious generalization when $\varphi:G\rightarrow G'$ is a surjective homomorphism, $G$ acts on $X$, $G'$ acts on $Y$, and the subgroup $G_\Delta$ is replaced by the subgroup $G_\varphi$ of $G\times G'$, which is the graph of $\varphi$. We will use this generalization in the next paragraph with $G'=1$. 

Let a finite group $G$ act on a smooth projective variety $X$. Let $G$-$\mbox{Hilb}(X)$ be moduli space of $G$-clusters in $X$, i.e., of $G$-invariant $0$-dimensional subschemes $Z$, such that $H^0(Z,\StructureSheaf{Z})$ is the regular representation of $G$. 
When $S$ is a surface, $X$ is the $n$-th cartesian power $S^n$ of $S$, and $G=\fS_n$ acts by permuting the factors, then the Hilbert scheme 
$S^{[n]}$ is isomorphic to the connected component of $\fS_n$-$\mbox{Hilb}(S^n)$ which contains the free orbits, by \cite{haiman}. We get a universal $\fS_n\times 1$-invariant subscheme $\Gamma\subset S^n\times S^{[n]}$, so  that 
$\StructureSheaf{\Gamma}$, with the natural linearization $\rho$, is an object of $D^b_{\fS_n\times 1}(S^n\times S^{[n]})$. 
When $S$ is a smooth projective surface, the Fourier-Mukai transformation
\begin{equation}
\label{eq-BKR}
BKR:=\Phi_{(\StructureSheaf{\Gamma},\rho)}:D^b_{\fS_n}(S^n)\rightarrow D^b(S^{[n]})
\end{equation} 
is an equivalence, by \cite[Cor. 1.3]{BKR}. 
We are using here the convention\footnote{
The Fourier-Mukai transformation from
$D^b(S^{[n]})$ to $D^b_{\fS_n}(S^n)$ 
with kernel the linearized object $(\StructureSheaf{\Gamma},\rho)$ is an equivalence as well, which is often referred to as the derived McKay correspondence.
} 
of \cite{krug} for the derived McKay correspondence.
Denote by $b:\Gamma\rightarrow S^n$ the restriction of the first projection and by
$q:\Gamma\rightarrow S^{[n]}$ that of the second projection. 
\begin{equation}
\label{eq-b-q}
S^n \LongLeftArrowOf{b} \Gamma \LongRightArrowOf{q} S^{[n]}.
\end{equation}
The morphism $q$ is $\fS_n$-invariant and flat of degree $n!$. 
Note that $q$ factors through the quotient of $\Gamma$ by the index $2$ alternating subgroup ${\mathfrak A}_n$ of $\fS_n$, and
$\bar{q}:\Gamma/{\mathfrak A}_n\rightarrow S^{[n]}$ is a double cover of $S^{[n]}$ branched over the divisor $D$ of non-reduced subschemes. 
Hence, 
\[
\bar{q}_*\StructureSheaf{\Gamma/{\mathfrak A}_n}\cong (q_*\StructureSheaf{\Gamma})^{{\mathfrak A}_n}\cong \StructureSheaf{S^{[n]}}\oplus\LB,
\]
where $\LB^2\cong \StructureSheaf{S^{[n]}}(-D)$. Let $\chi$ be the sign character of $\fS_n$. Denote by $\chi$ also the linearization of $\StructureSheaf{S^n}$ acting via the character $\chi$. We get the isomorphisms
\begin{eqnarray}
\label{eq-BKR-of-structure-sheaf}
BKR(\StructureSheaf{S^n},1)&\cong& Rq_*^{\fS_n}(\StructureSheaf{\Gamma})\cong\StructureSheaf{S^{[n]}},
\\
\nonumber
BKR(\StructureSheaf{S^n},\chi)&\cong& Rq_*^{\fS_n}(\StructureSheaf{\Gamma},\chi)\cong \LB.
\end{eqnarray}

Given two smooth projective surfaces $S_1$ and $S_2$ and an object $P\in D^b(S_1\times S_2)$, denote by
\begin{equation}
\label{eq-BKR-conjugate-of-outer-product-of-FM}
\Phi_P^{[n]}:D^b(S_1^{[n]})\rightarrow D^b(S_2^{[n]})
\end{equation}
the $BKR$-conjugate $\Phi_P^{[n]}:=BKR\circ \Phi_{(P^{\boxtimes n},\rho_\boxtimes)}\circ BKR^{-1}$ of
(\ref{eq-outer-product-of-FM}). Denote by 
\begin{equation}
\label{eq-Phi-chi-[n]}
\Phi_\chi^{[n]}:=BKR\circ\Phi_\chi\circ BKR^{-1}:D^b(S_i^{[n]})\rightarrow D^b(S_i^{[n]}),
\end{equation}
$i=1,2$,
the auto-equivalence obtained by conjugating the functor
$\Phi_\chi$ of tensorization by the sign character of $\fS_n$. Equation  (\ref{eq-Phi-chi-commutes-with-FM-with-diagonally-linearized-kernel}) yields 
\begin{equation}
\label{eq-Phi-chi-[n]-commutes-with-surface-FM}
\Phi_\chi^{[n]}\circ \Phi_P^{[n]}\cong \Phi_P^{[n]}\circ \Phi_\chi^{[n]}.
\end{equation}
Equations (\ref{eq-BKR-of-structure-sheaf}) yield
\begin{equation}
\label{eq-Phi-chi-sends-structure-sheaf-to-LB}
\Phi_\chi^{[n]}(\StructureSheaf{S^{[n]}})\cong \LB.
\end{equation}

\begin{rem}
\label{rem-chi-invariant-objects}
\begin{enumerate}
\item
Let $G$ be a finite group acting on a variety $X$, let $Y$ be a subscheme of $X$, let $\iota:Y\hookrightarrow X$ be the inclusion, and let $F$ be a sheaf on $Y$. Set $E:=\oplus_{g\in G}[g^*(\iota_*F)]$. Endow $E$ with the linearization $\lambda$, such that
$\lambda_f:E\rightarrow f^*E$ maps the direct summand $g^*(\iota_*F)$ to itself via the identity map. Let $\chi:G\rightarrow \CC^*$ be a character. Let $\eta:E\rightarrow E$ be the automorphism acting on the direct summand $g^*(\iota_*F)$ by $\chi(g^{-1})$. 
Then $\eta$ lifts to an isomorphism $\eta:(E,\lambda)\rightarrow (E,\chi\otimes\lambda)$. 
\item
Assume, furthermore, that the subschemes $Y$ and $g(Y)$ are disjoint, for every $g\in G$. Set $Z:=\cup_{g\in G} \ g(Y)$ and let $e:Z\hookrightarrow X$ be the inclusion. Let $(A,\rho)$ be an object of $D^b_G(X)$. Set $F:=\iota^*A$. 
The linearization $\rho$ induces an isomorphism between $e_*e^*A$ and the object $E=\oplus_{g\in G}[g^*(\iota_*F)]$ above. Indeed,
$e$ restricts to $g^{-1}(Y)$ as $g^{-1}\circ \iota\circ g$ and so the direct summand of $e_*e^*A$ supported on $g^{-1}(Y)$ is
\[
(g^{-1}\circ \iota\circ g)_*(g^{-1}\circ \iota\circ g)^*A=
g^*\left[\iota_*(\iota^*(g^{-1})^*A)\right]\RightArrowOf{\rho_{g^{-1}}}g^*(\iota_*\iota^*(A)),
\]
and the right hand side is the direct summand of $E$ supported on $g^{-1}(Y)$.
We claim that $e_*e^*(A,\rho)$ is isomorphic to the object $(E,\lambda)$ above. 
In particular, $e_*e^*(A,\rho)$ is isomorphic to $e_*e^*(A,\chi\otimes\rho)$.
It suffices to prove, more generally, that any two linearizations $(E,\rho)$ and $(E,\lambda)$ of $E$ are isomorphic.
We need to construct an automorphism $\phi$ of $E$ satisfying
\begin{equation}
\label{eq-isomorphism-phi-of-two-linearizations}
\xymatrix{
E\ar[r]^{\lambda_g} \ar[d]_{\phi}
& g^*E \ar[d]^{g^*(\phi)}
\\
E\ar[r]_{\rho_g} & g^*E,
}
\end{equation}
for all $g\in G$.
Let $E_{g(Y)}:=g_*\iota_*F$ be the direct summand of $E$ supported on $g(Y)$.
Let $\rho_g(f(Y))$ be the component of $\rho_g:E\rightarrow g^*E$ supported on $f(Y)$. We identify the factor of $\Aut(g^*E)$ supported on $Y$ with $\Aut(E_{g(Y)})$ and 
define $\phi(g(Y))\in\Aut(E_{g(Y)})$  by
\[
\phi(g(Y))= \rho_g(Y)(\lambda_g(Y))^{-1}.
\]
Note that $\phi(Y)=id$, since $\rho_1$ and $\lambda_1$ are both the identity.
The diagram above commutes over the component $Y$ of $Z$, by definition of $\phi$. Consider the diagram
\[
\xymatrix{
E\ar[r]^{\lambda_f} \ar[d]_{\phi} \ar@/^2pc/[rr]^{\lambda_{gf}}
& f^*E \ar[d]^{f^*(\phi)} \ar[r]^{f^*(\lambda_g)}
&  f^*g^*E\ar[d]^{(gf)^*(\phi)}
\\
E\ar[r]_{\rho_f} \ar@/_2pc/[rr]_{\rho_{gf}}
& 
f^*E \ar[r]_{f^*(\rho_g)} 
& f^*g^*E.
}
\]
The upper and lower domes commute, by definition of a linearization. 
The left and outer squares commute over the component $Y$, by definition of $\phi$. Hence, the right square commute as well over $Y$. Hence, the square (\ref{eq-isomorphism-phi-of-two-linearizations}) commutes also over $f(Y)$, for all $f\in G$.
\end{enumerate}
\end{rem}

%
\section{The derived monodromy group}
Taelman introduced a natural action of the group $\Aut(D^b(X))$ of auto-equivalences of $D^b(X)$ on the rational LLV lattice $\tilde{H}(X,\RationalNumbers)$  \cite[Theorem C]{taelman}.
We review Taelman's results, as they play a central role in Section \ref{sec-the-Mukai-line-of-a-maximally-deformable-F}.

Let $IHSM$ be the following groupoid. Objects of $IHSM$ are pairs $(X,\epsilon)$, where $X$ is a projective irreducible holomorphic symplectic manifold satisfying the condition that if $\dim_\CC(X)$ is divisible by $4$ then the dimension of $H^2(X,\QQ)$ is odd. All known examples satisfy this latter condition.
If the complex dimension of $X$ is divisible by $4$, then $X$ is enriched with the data $\epsilon$ of an orientation of the vector space $H^2(X,\QQ)$ (which determines an orientation of $\widetilde{H}(X,\QQ)$).  If the complex dimension of $X$ is not divisible by $4$, then the data $\epsilon$ is empty. 
The morphisms in 
$\Hom_{IHSM}((X,\epsilon),(Y,\epsilon'))$
are vector spaces isomorphisms in $\Hom(H^*(X,\RationalNumbers),H^*(Y,\RationalNumbers))$, which are 
compositions of parallel transport operators and isomorphisms induced by 
derived equivalences (independent of the data $\epsilon$).
Explicitly, a morphism $f\in \Hom_{IHSM}(X,Y)$ is the composition $f_n\circ \cdots \circ f_2\circ f_1$, where $X_i$ is an irreducible holomorphic symplectic manifold deformation equivalent to $X$, $X_0=X$, $X_n=Y$, and 
$f_i:H^*(X_{i-1},\RationalNumbers)\rightarrow H^*(X_i,\RationalNumbers)$ is either a parallel transport operator\footnote{A parallel transport operator $f:H^*(X,\QQ)\rightarrow H^*(Y,\QQ)$ 
is the parallel-transport in the local system $R\pi_*\QQ$ associated to a path from a point $b_0$ to a point $b_1$ in the analytic base $B$ of a smooth and proper family $\pi:\X\rightarrow B$ of irreducible holomorphic symplectic manifolds, not necessarily projective, with  fiber $X$ over $b_0$ and $Y$ over $b_1$.} 
or the isomorphism induced by the correspondence $\pi_{i-1}^*\sqrt{td_{X_{i-1}}}ch(\P)\pi_i^*\sqrt{td_{X_i}}$, where $X_{i-1}$ and $X_i$ are projective, $\P\in D^b(X_{i-1}\times X_i)$ is the kernel of the Fourier-Mukai functor $\Phi_\P:D^b(X_{i-1})\rightarrow D^b(X_i)$, and $\Phi_\P$ is an exact equivalence of triangulated categories. 

\begin{defi}
\label{def-DMon}
The {\em derived monodromy group} $DMon(X)$ is the subgroup $\Hom_{IHSM}((X,\epsilon),(X,\epsilon))$ of $GL(H^*(X,\QQ))$
(which is independent of $\epsilon$).
A {\em derived parallel transport operator} is a linear transformation $\phi:H^*(X,\QQ)\rightarrow H^*(Y,\QQ)$, which belongs to  $\Hom_{IHSM}((X,\epsilon),(Y,\epsilon'))$.
\end{defi}

Let $Lat$ be the category  of rational vector spaces with a non-degenerate inner product. Morphisms in $Lat$ are isometries.
Taelman defines the functor 
\begin{equation}
\label{eq-functor-tilde-H}
\tilde{H}:IHSM\rightarrow Lat 
\end{equation}
sending an object $X$ to its rational LLV lattice $\tilde{H}(X,\RationalNumbers)$ given in (\ref{eq-rational-LLV-lattice}). 
The functor $\tilde{H}$ sends a morphism in 
$f\in \Hom_{IHSM}((X,\epsilon),(Y,\epsilon'))$ to an isometry $\tilde{H}(f):\tilde{H}(X,\RationalNumbers)\rightarrow\tilde{H}(Y,\RationalNumbers)$.
If $f$ is induced by a parallel transport operator and $\dim_\CC(X)=2n$ for odd $n$, then $\tilde{H}(f)$ is induced by the restriction of $f$ to the second cohomology extended via the identity to the direct summand $U_\RationalNumbers$. 
If $f$ is induced by a parallel transport operator, the dimension of $X$ is divisible by $4$, 
and $H^2(X,\QQ)$ is odd dimensional, then $\tilde{H}(f)$ is defined by the above description only up to sign, and the sign is chosen so that $\tilde{H}(f)$ is orientation preserving. 
If $f$ is induced by a derived equivalence $\Phi_\P$ as above, then
$\tilde{H}(f)$ is the isometry defined
in \cite[Theorems 4.7 and 4.9]{taelman}, which we will recall in Theorem \ref{thm-def-of-H-tilde} below.  We get the homomorphism
\begin{equation}
\label{eq-action-of-DMon-on-Mukai-lattice}
\tilde{H}:DMon(X)\rightarrow O(\tilde{H}(X,\RationalNumbers)).
\end{equation}
If $\lambda\in H^2(X,\RationalNumbers)$ is $c_1(L)$, for some line bundle $L\in \Pic(X)$, then 
tensorization by $L$ corresponds to the element of $DMon(X)$ acting on $H^*(X,\QQ)$ by multiplication by $ch(L)=\exp(\lambda)$  and the latter is mapped  via $\tilde{H}$ to
$exp(e_\lambda):=id+e_\lambda+e^2_\lambda/2$ in $SO(\tilde{H}(X,\RationalNumbers))$, where $e_\lambda$ is given in Equation (\ref{eq-e-lambda}), by \cite[Theorem 3.1 and Prop. 3.2]{taelman}.

Let $K3^{[n]}$ be the full subgroupid of $IHSM$, whose objects are projective irreducible holomorphic symplectic manifolds deformation equivalent to the Hilbert scheme $S^{[n]}$ of length $n$ subschemes of a $K3$ surface $S$. We get the functor
\begin{equation}
\label{eq-functor-[n]}
{}^{[n]}:K3^{[1]}\rightarrow K3^{[n]}
\end{equation}
sending a $K3$ surface $S$ to its Hilbert scheme $S^{[n]}$, sending a parallel transport operator associated to the path in the base of a family of $K3$ surfaces to the parallel transport operator associated to the same path for the relative family of Hilbert schemes over the same base, and sending the morphism associated to an equivalence $\Phi:D^b(S_1)\rightarrow D^n(S_2)$ to the morphism associated to $\Phi^{[n]}:D^b(S_1^{[n]})\rightarrow D^b(S_2^{[n]})$ obtained by conjugating
$\Phi\boxtimes \cdots \boxtimes \Phi:D^b_{\fS_n}(S_1^n)\rightarrow D^b_{\fS_n}(S_2^n)$ with the BKR equivalence 
$D^b(S_i^{[n]})\rightarrow D^b_{\fS_n}(S_i^n)$, $i=1,2$, given in (\ref{eq-BKR}).

\hide{
Let $LLV\!\!-\!\!Mod$ be the category of pairs $(\Lambda,A)$ of a vector space $\Lambda$ over $\RationalNumbers$ endowed with a non-degenerate bilinear pairing and an algebra $A$ over $\RationalNumbers$, whose underlying vector space is endowed with an $\LieAlg{so}(\Lambda)$-module structure. Morphisms (???) 

Let $SH^*(X,\RationalNumbers)$ be the subalgebra of $H^*(X,\RationalNumbers)$ generated by $H^2(X,\RationalNumbers)$. 
We have two functors $SH^*(\bullet,\RationalNumbers)$ and $H^*(\bullet,\RationalNumbers)$ from $IHSM$ to $LLV\!\!-\!\!Mod$,
by \cite[Theorems A, B, and C]{taelman}.
Both algebras $SH^*(X,\RationalNumbers)$ and $H^*(\bullet,\RationalNumbers)$ are endowed with a 
$\LieAlg{so}\tilde{H}(X,\RationalNumbers)$-module structure.
The inclusion homomorphism 
$\eta_X:SH^*(X,\RationalNumbers)\rightarrow H^*(X,\RationalNumbers)$ is a natural transformation 
\[
\eta:SH^*(\bullet,\RationalNumbers)\rightarrow H^*(\bullet,\RationalNumbers)
\]
between the two functors. 
In particular, the inclusion $\eta_X$ is equivariant with respect to the $\Aut(D^b(X))$-action and is a homomorphism of 
$\LieAlg{so}\tilde{H}(X,\RationalNumbers)$-modules (as well as of algebras). 

Now $ch(\StructureSheaf{X})$ belongs to $SH^*(X,\RationalNumbers)$.
Furthermore, $SH^*(X,\RationalNumbers)$ is isomorphic to $\Sym^n\tilde{H}(X,\RationalNumbers)$ as a 
$\LieAlg{so}\tilde{H}(X,\RationalNumbers)$-module, where $n$ is half the complex dimension of $X$, by 
\cite[Prop. 3.5]{taelman}. 
(??? 
We need to consider $ch(\StructureSheaf{X})\sqrt{td_X}$, but the latter does not belong to $SH^*(X,\RationalNumbers)$ but rather to the translate of the latter by $\sqrt{td_X}$. Could it be that the latter is the invariant subrepresentation, not the former?
???)
}

\hide{
\begin{rem}
\label{rem-DMon-is-large-for-K3-[n]-type}
The monodromy group $\Mon(X)$ is a finite index subgroup of $O(H^2(X,\Integers))$, whenever $\dim(H^2(X,\QQ))\geq 5$, by \cite[Theorem 8.2]{bakker-lehn}, 
and it embeds as a subgroup of $DMon(X)$, by definition of the latter.
When $X$ is the Hilbert scheme $S^{[n]}$ of a $K3$ surface $S$, then $DMon(X)$ contains also the image of 
$DMon(S)$ via the functor (\ref{eq-functor-[n]}).
The group $DMon(S)$, of the $K3$ surface, is $O^+(\tilde{H}(S,\Integers))$, by \cite[Prop. 9.2]{taelman}.
The Lie algebras $\LieAlg{so}(H^2(S^{[n]},\RationalNumbers))$ and 
$\LieAlg{so}(H^2(S,\RationalNumbers)\oplus U_\RationalNumbers))$ generate\footnote{Indeed, let $\LieAlg{a}$ be the subalgebra of $\LieAlg{so}(\tilde{H}(S^{[n]},\RationalNumbers))$ generated by these two and let $\LieAlg{b}$ be the subspace of $\LieAlg{so}(\tilde{H}(S^{[n]},\RationalNumbers))$ spanned by these two. Then $\LieAlg{so}(\tilde{H}(S^{[n]},\RationalNumbers))/\LieAlg{b}$
is isomorphic to $U_\QQ$ as a $\LieAlg{so}(H^2(S,\RationalNumbers))\oplus \LieAlg{so}(U_\RationalNumbers)$-module. Hence,  $\LieAlg{a}$ has co-dimension $\leq 2$ and
 $Q:=\LieAlg{so}(\tilde{H}(S^{[n]},\RationalNumbers))/\LieAlg{a}$ is a representation of $\LieAlg{a}$ and hence of both $\LieAlg{so}(H^2(S^{[n]},\RationalNumbers))$ and 
$\LieAlg{so}(H^2(S,\RationalNumbers)\oplus U_\RationalNumbers)$. The representation $Q$ must be a trivial representation
of both $\LieAlg{so}(H^2(S^{[n]},\RationalNumbers))$ and 
$\LieAlg{so}(H^2(S,\RationalNumbers)\oplus U_\RationalNumbers)$, by dimension reasons, yet 
if $Q$ is non-zero then it is a quotient of $\LieAlg{so}(\tilde{H}(S^{[n]},\RationalNumbers))/\LieAlg{b}$ and so 
it restricts to a non-trivial representation of the subalgebra $\LieAlg{so}(U_\RationalNumbers)$, a contradiction.
} 
$\LieAlg{so}(\tilde{H}(S^{[n]},\RationalNumbers))$.
Hence, the Zariski closure of the image of $DMon(X)$ via (\ref{eq-action-of-DMon-on-Mukai-lattice}) contains $SO(\tilde{H}(X,\RationalNumbers))$, for every $X$ of $K3^{[n]}$-type.
\end{rem}
}

Let $SH^*(X,\RationalNumbers)$ be the subalgebra of $H^*(X,\RationalNumbers)$ generated by $H^2(X,\RationalNumbers)$. 
Then $SH^*(X,\RationalNumbers)$ is an irreducible $\LieAlg{g}$-sub-module of $H^*(X,\QQ)$, which is isomorphic to the irreducible submodule of the symmetric power
$\Sym^n(\tilde{H}(X,\RationalNumbers))$ containing $\alpha^n$, by \cite[Prop. 3.5 and 3.7]{taelman}. The embedding is given by
\begin{eqnarray}
\label{eq-Psi}
\Psi:SH^*(X,\RationalNumbers)&\rightarrow &\Sym^n(\tilde{H}(X,\RationalNumbers)),
\\
\nonumber
\lambda_1\cdots\lambda_n & \mapsto & e_{\lambda_1}\cdots e_{\lambda_n}(\alpha^n/n!),
\end{eqnarray}
where $\lambda_i\in H^2(X,\QQ)$, $1\leq i\leq n$, and 
$e_\lambda(x_1\cdots x_n):=\sum_ix_1\cdots e_\lambda(x_i)\cdots x_n$. Note that $\Psi(1)=\alpha^n/n!$ and
\begin{equation}
\label{eq-Psi-pt}
\Psi([pt])=\beta^n/c_X, 
\end{equation}
where $c_X$ is the Fujiki constant, 
by  \cite[Lemma 3.6]{taelman}.
The image of $\Psi$ is equal to the kernel of the $\LieAlg{g}$-module homomorphism
\begin{equation}
\label{eq-Delta}
\Delta:\Sym^n(\tilde{H}(X,\RationalNumbers))\rightarrow \Sym^{n-2}(\tilde{H}(X,\RationalNumbers)),
\end{equation}
$\Delta(x_1\cdots x_n)=\sum_{i<j}(x_i,x_j)x_1\cdots \hat{x}_i\cdots \hat{x}_j\cdots x_n$, by \cite[Lemma 3.7]{taelman}.
Hence, the image of $\Psi$ is equal to the subspace spanned by $n$-th powers of isotropic elements of $\tilde{H}(X,\RationalNumbers)$. The latter subspace is an irreducible $\LieAlg{g}$-submodule that appears with multiplicity one in $\Sym^n(\tilde{H}(X,\RationalNumbers))$.
We get a natural $\LieAlg{g}$-equivariant projection
\begin{equation}
\label{eq-projection-to-Im-Psi}
\Sym^n(\tilde{H}(X,\RationalNumbers)) \rightarrow Im(\Psi).
\end{equation}

The submodule $SH^*(X,\QQ)$ appears with multiplicity one in $H^*(X,\QQ)$, by \cite[Prop. 2.32]{GKLR}, so we get a $\LieAlg{g}$-equivariant projection 
\begin{equation}
\label{eq-projection-to-verbitsky-component}
H^*(X,\QQ)\rightarrow SH^*(X,\QQ).
\end{equation}

\begin{example}
\label{example-projection-to-Im-Psi}
Suppose that $\dim(X)=4$, so that $n=2$. 
Let $\tilde{q}\in\Sym^2(\tilde{H}(X,\QQ))$ be the unique $\LieAlg{g}$-invariant class, such that 
$\Delta(\tilde{q})=1$. If $\{x_i\}$ is an orthogonal basis of $\tilde{H}(X,\QQ)$, 
then $\tilde{q}=\frac{1}{\dim\tilde{H}(X,\RationalNumbers)}\sum_i\frac{1}{(x_i,x_i)}x_i^2$.
The projection from $\Sym^2(\tilde{H}(X,\RationalNumbers))$ to 
the image of $\Psi:SH^*(X,\RationalNumbers)\rightarrow \Sym^2(\tilde{H}(X,\RationalNumbers))$ maps
the square $\gamma^2$ of a class
$\gamma\in \tilde{H}(X,\QQ)$ to $\gamma^2-(\gamma,\gamma)\tilde{q}$.
\end{example}

Theorem A in \cite{taelman} states that the morphism $\phi$ in $\Hom_{IHSM}((X,\epsilon_1),(Y,\epsilon_2))$ induced by an equivalence of derived categories $\Phi:D^b(X)\rightarrow D^b(Y)$ conjugates the LLV algebra $\LieAlg{g}_X$ to
$\LieAlg{g}_Y$. The latter statement is clear for parallel-transport operators, and so holds for every morphism $\phi$
in $\Hom_{IHSM}((X,\epsilon_1),(Y,\epsilon_2))$.
It follows that $\phi$ maps $SH^*(X,\QQ)$ to $SH^*(Y,\QQ)$, as each is an irreducible subrepresentation of the LLV algebra, which appears with multiplicity $1$. Denote by
\[
SH(\phi): SH^*(X,\QQ)\rightarrow SH^*(Y,\QQ)
\]
the restriction of $\phi$ to $SH^*(X,\QQ)$. 

\begin{thm} 
\label{thm-def-of-H-tilde}
(\cite[Theorems 4.7 and 4.9]{taelman})
Assume that $X$ and $Y$ are of dimension $2n$ and $\phi$ is a morphism in $\Hom_{IHSM}((X,\epsilon_1),(Y,\epsilon_2))$.
There exists a unique isometry $\widetilde{H}(\phi):\widetilde{H}(X,\QQ)\rightarrow \widetilde{H}(Y,\QQ)$ making the following diagram commutative
\[\xymatrixcolsep{5pc}
\xymatrix{
SH^*(X,\QQ) \ar[r]^{\epsilon(\widetilde{H}(\phi))SH(\phi)}
\ar[d]_\Psi &
SH^*(Y,\QQ) \ar[d]^\Psi
\\
\Sym^n\widetilde{H}(X,\QQ) \ar[r]_{\Sym^n(\widetilde{H}(\phi))} &
\Sym^n\widetilde{H}(Y,\QQ),
}
\]
where 
$\epsilon(\widetilde{H}(\phi))=\left\{\begin{array}{ccl}
1 & \mbox{if} & n \ \mbox{is odd},
\\
1 & \mbox{if} & n \ \mbox{is even and} \ \widetilde{H}(\phi) \ \mbox{is orientation preserving,}
\\
-1& \mbox{if} & n \ \mbox{is even and} \ \widetilde{H}(\phi) \ \mbox{is orientation reversing.} 
\end{array}\right.$
\end{thm}

\begin{rem}
\label{rem-functor-chi}
Assume $n$ is even. In that case $\Sym^n(\widetilde{H}(\phi))=\Sym^n(-\widetilde{H}(\phi))$. 
Let $IHSM_n$ be the full subcategory of $IHSM$, whose objects $(X,\epsilon_X)$ consist of $X$ of dimension $2n$.
There is a functor $\chi$ from the groupoid $IHSM_n$ to the group $\mu_2:=\{\pm 1\}$, which is defined as follows. The image of $\Psi_X:SH^*(X,\QQ)\rightarrow \Sym^n\widetilde{H}(X,\QQ)$ is an irreducible $\LieAlg{g}_X:=\LieAlg{so}(\widetilde{H}(X,\QQ))$-representation spanned by the subvariety $Isot_X$ of $n$-th powers of non-zero isotropic classes in 
$\widetilde{H}(X,\QQ)$. The analogous subvariety $Isot_{X,\RR}$  of $\Sym^n\widetilde{H}(X,\RR)$ is connected (in the classical topology) being the image of the complement of the origin in the quadric cone in $\widetilde{H}(X,\RR)$.
Indeed, the signature of  $\widetilde{H}(X,\RR)$ is $(4,b_2(X)-2)$ and so the complement of the origin in the quadric cone in $\widetilde{H}(X,\RR)$ is homotopic to the product of a $3$-sphere and a $b_2(X)-3$ sphere. The image 
$-Isot_{X,\RR}$ of $Isot_{X,\RR}$ under multiplication by $-1$ is disjoint from $Isot_{X,\RR}$, since $n$ is even. 
If $\phi$ belongs to
$\Hom_{IHSM}((X,\epsilon_X),(Y,\epsilon_Y))$, then $\Psi_Y\circ SH(\phi)\circ \Psi_X^{-1}$ maps 
$Isot_{X,\RR}$ isomorphically to one of  $Isot_{Y,\RR}$ or $-Isot_{Y,\RR}$, and we set $\chi(\phi)=1$, if
the image is $Isot_{Y,\RR}$ and $-1$ if it is $-Isot_{Y,\RR}$. Note that the functor $\chi$ is independent of the data of orientations.
If $\phi$ arrises from a parallel transport operator, then $\chi(\phi)=1$. If $\phi$ corresponds to an equivalence of derived categories $\Phi:D^b(X)\rightarrow D^b(Y)$ mapping the sky-scraper sheaf of some point to an object $F$ of positive rank, then 
$\chi(\phi)=1$, since then $(\Psi_Y\circ SH(\phi)\circ \Psi_X^{-1})(\beta^n/c_X)=\Psi(v(F))$ and both $c_X$ and the coefficient $\rank(F)/n!$ of $\alpha^n$
in $\Psi(v(F))$ are positive.
Taelman first proves that there exists an isometry 
$\widetilde{H}(\phi)$, unique up to sign, such that the diagram
\[\xymatrixcolsep{5pc}
\xymatrix{
SH^*(X,\QQ) \ar[r]^{\chi(\phi)SH(\phi)}
\ar[d]_{\Psi_X} &
SH^*(Y,\QQ) \ar[d]^{\Psi_Y}
\\
\Sym^n\widetilde{H}(X,\QQ) \ar[r]_{\Sym^n(\widetilde{H}(\phi))} &
\Sym^n\widetilde{H}(Y,\QQ)
}
\]
commutes. The orientations of the objects in $IHSM_n$ are introduced in order to determine the sign of $\widetilde{H}(\phi)$ so that 
$\chi(\phi)=\epsilon(\widetilde{H}(\phi))$.
\end{rem}

\begin{rem}
\label{rem-Psi-natural-transformation}
We sketch the more functorial description of $\Psi$ as a natural transformation. We have the functor $SH:IHSM\rightarrow Lat$ sending an irreducible holomorphic symplectic manifold $X$ to the subring $SH^*(X,\QQ)$ of $H^*(X,\QQ)$, forgetting the ring structure but endowing it with the restriction of the Mukai pairing of $H^*(X,\QQ)$ \cite[Sec. 3.2]{taelman}. We have the endofunctor 
$Isot_n:Lat\rightarrow Lat$ sending a rational  vector space $V$ with a bilinear pairing to the subspace $V(n)$  of $\Sym^n(V)$ spanned by $n$-th powers of isotropic vectors. 
We have the full functor $F_n:IHSM_n\rightarrow IHSM$ from the category $IHSM_n$ of irreducible holomorphic symplectic manifolds of dimension $2n$ into that without dimension restriction.
Then the above construction produces the natural transformation 
$\Psi:SH\circ F_n \rightarrow Isot_n\circ \tilde{H}\circ F_n$.
\end{rem}

%
\section{The LLV line of an object of $D^b(X)$ which deforms in co-dimension $1$ in $HH^2(X)$}
\label{sec-the-Mukai-line-of-a-maximally-deformable-F}
Let $X$ be a projective irreducible holomorphic symplectic manifold. 
In Section \ref{sec-the-Hochschild-lie-algebra} we review results of Taelman and Verbitsky 
stating that two Lie subalgebras of $\End(H^*(X,\CC))$ coincide, one is 
 the LLV Lie algebra $\LieAlg{g}_{X,\CC}$ of $X$, defined in terms of the cohomology ring of $X$, and the other is an analogous Lie algebra defined in terms of the Hochschield cohomology of $X$. 
 The vector space $HT^2(X)$ naturally embeds in $\LieAlg{g}_{X,\CC}$. 
 
 In Section \ref{sec-subalgebras-of-LLV-generated-by-a-hyperplane} 
 we show that if a hyperplane $\Sigma$ in $HT^2(X)$ annihilates a non-zero rational class in the subspace of
 $\Sym^n\widetilde{H}(X,\QQ)$, $n=\frac{1}{2}\dim_\CC(X)$, spanned by powers of isotropic classes, then $\Sigma$ and its complex conjugate generate 
 the Lie subalgebra of $\LieAlg{g}_{X,\CC}$ annihilating some non-zero rational class in $\widetilde{H}(X,\QQ)$
 (Lemma \ref{lemma-pde}).

Let $F$ be an object  in $D^b(X)$.
In Section \ref{sec-LLV-subspace} we show that the kernel in $HT^2(X)$ of the obstruction map $\obs^{HT}_F$, given in (\ref{eq-obstruction-map}),  
is mapped via the Duflo automorphism of $HT^2(X)$ into the stabilizer of the Mukai vector $v(F)$ of $F$. We use a result of Huang relating $\obs^{HT}_F$ and $\obs_F$, given in (\ref{eq-obs-introduction})
\cite{huang}. 

In section \ref{sec-LLV-line} we associate   to every object $F$ in $D^b(X)$ with a rank $1$ obstruction map 
$HT^2(X)\rightarrow \Hom(F,F[2])$
a line $\ell(F)$ in 
$\tilde{H}(X,\RationalNumbers)$ (Theorem \ref{thm-Mukai-vector}). 
The construction is equivariant with respect to the $\Aut(D^b(X))$-action. 
We calculate the line $\ell(F)$ when $F$ is a sky-scraper sheaf (Example \ref{example-Mukai-line-of-sky-scraper-sheaf}) and when $F$ is the structure sheaf $\StructureSheaf{X}$ of $X$ of $K3^{[n]}$ and generalized Kummer deformation types (Lemma \ref{lemma-LLV-line-of-structure-sheaf-k3-type}).

In Section \ref{sec-torsion-sheaves-with-rank-1-obstruction-map}
we show that if a torsion sheaf $F$ has a rank $1$ obstruction map and is 
supported on a proper irreducible subvariety $Z$ of $X$, then $Z$ is either a lagrangian subvariety or a point (Theorem \ref{thm-support-is-either-lagrangian-or-point}). We calculate the line $\ell(F)$ for the structure sheaf of a lagrangian subvariety (Lemma \ref{lemma-Mukai-line-of-structure-sheaf-of-subcanonical-lagrangian}). 

%
\subsection{The Hochschild Lie algebra of a holomorphic symplectic manifold}
\label{sec-the-Hochschild-lie-algebra}

Let $X$ be a $2n$-dimensional irreducible holomorphic symplectic manifold. Let $h$ be the endomorphism of $H^*(X,\CC)$ acting on $H^k(X,\CC)$ by multiplication by $k-2n$.
Let $h'$ be the endomorphism of $H^*(X,\CC)$ acting on $H^{p,q}(X,\CC)$ by multiplication by $q-p$. 
Choose a holomorphic symplectic form $\sigma$ in $H^0(X,\Omega^2_X)$. 
The volume form $\sigma^n$ induces the isomorphism of vector bundles $\gamma:\wedge^iTX\rightarrow \Omega^{2n-i}_X$.
Denote by $\gamma:HT^*(X)\rightarrow H^*(X,\CC)$ the induced isomorphism. Set 
$H\Omega_*(X):=\oplus_{p,q}H^{p,q}(X)$.
We identify $H\Omega_*(X)$ with $H^*(X,\CC)$ via the Hodge decomposition, but the grading of $H\Omega_*(X)$ is determined by $h'$, i.e., 
\[
H\Omega_d(X)=\bigoplus_{q-p=d}H^{p,q}(X).
\]
Let
\begin{equation}
\label{eq-m}
m:HT^*(X)\rightarrow \End(H\Omega_*(X))
\end{equation}
be the associative algebras homomorphism corresponding to the $HT^*(X)$-module structure.
The isomorphism $\gamma$ relates the ring structure of $HT^*(X)$ with the  $HT^*(X)$-module structure of  $H\Omega_*(X)$ as follows. Given $a\in HT^{2i}(X)$ and $b\in HT^*(X)$ we have 
\begin{equation}
\label{eq-ring-structure-versus-module-structure}
m_a(\gamma(b))=\gamma(a\wedge b).
\end{equation}
For $a\in HT^k(X)$, the endomorphism $m_a$ has degree $k$ with respect to the grading $h'$, i.e. 
$m_a:H\Omega_i(X)\rightarrow H\Omega_{i+k}(X)$.

We say that $m_e$, $e\in HT^2(X)$, has the {\em hard Lefschetz property} with respect to $h'$, if there exists an element $f\in \End(H\Omega_*(X))$, such that the triple $(m_e,h',f)$ is an $\LieAlg{sl}_2$ triple, i,e.,
\[
[h',m_e]=2m_e, \ [h',f]=-2f, \ [m_e,f]=h'.
\]
The symplectic form induces the isomorphism $\sigma:\wedge^iTX\rightarrow \Omega^i_X$. On the level of cohomology it induces a graded ring isomorphism $\sigma:HT^*(X)\rightarrow H^*(X,\CC)$. 
A Zariski dense open subset of $H^2(X,\CC)$ consists of elements $e$ which have the Hard Lefschetz property with respect to $h$,
by the Hard Lefschetz Theorem.
It follows that a Zariski dense open subset of $HT^2(X)$ consists of elements $e$, such that the endomorphism of $HT^*(X)$ of multiplication by $e$ has the Hard Lefschetz property with respect to the grading of $HT^*(X)$. Consequently, a Zariski dense open subset of $HT^2(X)$ consists of elements $e$, such that the endomorphism $m_e$ of $H\Omega_*(X)$ has the Hard Lefschetz property with respect to 
the grading $h'$, since $\gamma$ is an isomorphism of graded $HT^*(X)$-modules, by Equation (\ref{eq-ring-structure-versus-module-structure}).
Let $\LieAlg{g}'$ be the subalgebra of $\End(H\Omega_*(X))$ generated by all $\LieAlg{sl}_2$ triples $(m_e,h',f)$, where $e$ is in $HT^2(X)$.

Composing the isomorphism $\gamma$ with the inverse of $\sigma$ we get the isomorphism $\eta=\gamma\circ \sigma^{-1}:\Omega^j_X\rightarrow \Omega^{2n-j}_X$ inducing the isomorphism
\[
\eta^{i,j}:H^i(\Omega^j_X)\rightarrow H^i(\Omega^{2n-j}_X).
\]
Set $\eta:=\oplus_{0\leq i,j\leq 2n}\eta^{i,j}:H^*(X,\CC)\rightarrow H^*(X,\CC)$. The following equality is evident
\[
\eta\circ h \circ \eta^{-1} = h'.
\]
Equation (\ref{eq-ring-structure-versus-module-structure}) translates to 
\begin{equation}
\label{eq-cohomology-ring-structure-versus-module-structure}
m_a(\gamma(b))=\eta(\sigma(a)\cup\sigma(b)),
\end{equation}
for $a\in HT^{2i}(X)$, 
where on the right the cup product is that of the cohomology ring structure. 
Denote by $e_{\sigma(a)}$ the endomorphism of $H^*(X,\CC)$ of multiplication by $\sigma(a)$, for $a\in HT^2(X)$. 
Equation (\ref{eq-cohomology-ring-structure-versus-module-structure}) yields 
\begin{equation}
\label{eq-eta-conjugates-m-sigma-a-to-m-a}
\eta\circ e_{\sigma(a)}\circ \eta^{-1}=m_{a} 
\end{equation}
and $(e_{\sigma(a)},h,f)$ is a $\LieAlg{sl}_2$ triple, if and only if $(m_a,h',\eta\circ f\circ \eta^{-1})$
is. Thus, $\eta$ conjugates  $\LieAlg{g}_\CC$ to $\LieAlg{g}'$.
\begin{equation}
\label{eq-eta-conjugates-Lie-algebra-g-to-g-prime}
\eta\LieAlg{g}_\CC\eta^{-1}=\LieAlg{g}'.
\end{equation}

The group $\Spin(\tilde{H}(X,\CC))$ acts on $H^*(X,\CC)$ integrating the infinitesimal action of the LLV Lie algebra $\LieAlg{g}_\CC:=\LieAlg{so}(\tilde{H}(X,\CC)).$
Denote the resulting representation by  
\begin{equation}
\label{eq-rho}
\rho: \Spin(\tilde{H}(X,\CC)) \rightarrow GL(H^*(X,\CC)).
\end{equation}
It is known that  $\rho$ is injective, if the odd cohomology of $X$ is non trivial. Furthermore, the restriction of $\rho$ to the even cohomology factors through an injective homomorphism  
\begin{equation}
\label{eq-bar-rho}
\bar{\rho}:SO(\tilde{H}(X,\CC))\rightarrow GL(H^{ev}(X,\CC)).
\end{equation}
The following is a fundamental result of Verbitsky.

\begin{thm}
\label{thm-Verbitsky}
\cite[Theorem 9.1 and 9.7(i)]{verbitsky-mirror-symmetry}
The operator $\eta$ is an involution and it belongs to the image of $\rho$. 
\end{thm}

Consequently, we get the equality 
\begin{equation}
\label{eq-of-two-Lie-algebras}
\LieAlg{g}_\CC=\LieAlg{g}', 
\end{equation}
since the left hand side of (\ref{eq-eta-conjugates-Lie-algebra-g-to-g-prime}) is equal to $\LieAlg{g}_\CC$. An alternative proof of Equality (\ref{eq-of-two-Lie-algebras}) is provided in \cite[Prop. 2.9]{taelman}. 

\begin{cor}
\label{cor-image-of-m-is-in-LLV-algebra}
The image $m(HT^2(X)))$ of $HT^2(X)$ in $\End(H\Omega_*(X))$, via the homomorphism $m$ given in (\ref{eq-m}), is contained in $\LieAlg{g}_\CC$, yielding
\begin{equation}
\label{eq-m-from-HT-2-into-LLV-algebra}
m: HT^2(X)\rightarrow \LieAlg{g}_\CC.
\end{equation}
The operator $h'$ belongs to $\LieAlg{g}_\CC$.
\end{cor}

\begin{proof}
The operator $h'$ belongs to $\LieAlg{g}'$, by definition, hence also to $\LieAlg{g}_\CC$.
Similarly, $m_a$ belongs to $\LieAlg{g}'$, for a Zariski dense open subset of $HT^2(X)$, by definition, hence also to $\LieAlg{g}_\CC$.
\end{proof}

The $\LieAlg{g}$-equivariance of $\Psi$, given in (\ref{eq-Psi}), and the fact that $\Psi$ maps the class $1\in SH^0(X,\CC)$
to $\alpha^n/n!\in \Sym^n(\tilde{H}(X,\CC))$ and the class $[pt]\in SH^{4n}(X,\CC)$ to $\beta^n/c_X$, induces a Hodge structure on the image of $\Psi$ such that $\alpha^n$ is of type $(0,0)$, $\beta^n$ is of type $(2n,2n)$, $\Psi$ is a morphism of Hodge structures, and the Hodge direct summand of weight $(p,q)$ in $Im(\Psi)$
is the common eigenspace of $h$ and  $h'$ of weights $p+q-2n$ and $q-p$ respectively. The Lie algebra $\LieAlg{g}$ is a sub Hodge structure of $\End(H^*(X,\QQ))$ and, given a class $\lambda\in H^{p,q}(X,\CC)$ with $p+q=2$, the operator 
$e_\lambda:SH^*(X,\CC)\rightarrow SH^*(X,\CC)$ has Hodge type $(p,q)$. Hence, the same holds for the operator $e_\lambda$ acting on $Im(\Psi)$ with respect to the induced Hodge structure. Any $\LieAlg{g}$-module gets endowed with a Hodge structure and in particular so does the rational Hodge lattice $\tilde{H}(X,\QQ)$ and the latter induces one on $\Sym^n(\tilde{H}(X,\CC))$ of which the image of $\Psi$ is a sub Hodge structure with respect to the same Hodge structure induced above. 
We will denote by $\tilde{H}^{p,q}_{\fine}(X)$ the Hodge direct summand of weight $(p,q)$, in order to distinguish it from the Hodge decomposition of the pure Hodge structure of weight $0$ on $\tilde{H}(X,\QQ)$ introduced in Section \ref{sec-Mukai-lattice}.
The class $\alpha$ has weight $(0,0)$ and the class $\beta$ has weight $(2,2)$ and the direct summand $H^2(X,\QQ)$ in $\tilde{H}(X,\QQ)$ is a sub Hodge structure with its original weight $2$ Hodge structure. 

Theorem \ref{thm-Verbitsky} determines an action of $\eta$ on $\tilde{H}(X,\CC)$, 
by the inverse of $\bar{\rho}$, given in (\ref{eq-bar-rho}), of the restriction of $\eta$ to the even cohomology of $X$.
Conjugation by the operator $\eta$ interchanges $h$ and $h'$. Hence, $\eta(\tilde{H}^{p,q}_\fine(X))=\tilde{H}^{2-p,q}_\fine(X)$.
Denote by $\sigma$ the class in $\tilde{H}^{2,0}_\fine(X,\CC)$ corresponding to the chosen symplectic form under the identification 
$\tilde{H}^{2,0}_\fine(X,\CC)=H^{2,0}(X)$. 
Given $\lambda\in H^i(\wedge^jTX)$, $i+j=2$, the operator $m_\lambda$, given in (\ref{eq-m-from-HT-2-into-LLV-algebra}), has Hodge type $(-j,i).$
Hence, $m_\lambda(\sigma)$ has Hodge type $(i,i)$. We get the homomorphism
\begin{eqnarray}
\label{eq-embedding-of-HT-2-in-Mukai-lattice}
\mu:HT^2(X)&\rightarrow & \tilde{H}(X,\CC)
\\
\nonumber
\lambda & \mapsto & m_\lambda(\sigma).
\end{eqnarray}
We have $(\eta\circ e_{\sigma(\lambda)}\circ\eta^{-1})(\sigma)=m_\lambda(\sigma)$, by 
(\ref{eq-eta-conjugates-m-sigma-a-to-m-a}). Now $\eta^{-1}(\sigma)=c\alpha$, for a non-zero scalar $c$, and the homomorphism
$\lambda\mapsto e_{\sigma(\lambda)}(\alpha)$ is injective. Hence, the homomorphism 
$\mu$ is injective as well. We have the commutative diagram
\[
\xymatrix{
H^2(X,\CC) \ar[r]^e \ar@/^2pc/[rr]^{\iota} & \LieAlg{g}_\CC \ar[r]^{ev_\alpha} & \tilde{H}(X,\CC)
\\
HT^2(X) \ar[u]^\sigma \ar[r]_m \ar@/_2pc/[rr]_{\mu}&  \LieAlg{g}_\CC \ar[u]^{Ad_\eta} \ar[r]_{ev_\sigma} & 
\tilde{H}(X,\CC), \ar[u]_{\frac{1}{c}\eta}
}
\]
where $\iota$ is the natural inclusion and $ev_\alpha$ and $ev_\sigma$ are the evaluations of an endomorphism of $\tilde{H}(X,\CC)$ on the elements $\alpha$ and $\sigma$. The top dome commutes by Equation (\ref{eq-e-lambda}) and the bottom by definition of $\mu$.
The left square commutes by Equation  
(\ref{eq-eta-conjugates-m-sigma-a-to-m-a}) and the right by the equality $\eta^{-1}(\sigma)=c\alpha$ and the fact that $\eta=\eta^{-1}$
(Theorem \ref{thm-Verbitsky}). Note the equalities
\begin{equation}
\label{eq-mu-takes-Poisson-tensor-to-alpha}
\mu(H^0(\wedge^2TX))=\CC\alpha, \ \ \ \mu(H^1(TX))=H^{1,1}(X), \ \ \ \mbox{and} \ \ \ \mu(H^2(\StructureSheaf{X}))=\CC\beta.
\end{equation}

The Hodge structure of Section \ref{sec-Mukai-lattice} is related to the above by setting 
$\tilde{H}^{1,-1}(X)=\tilde{H}^{2,0}_\fine(X)$, $\tilde{H}^{-1,1}(X)=\tilde{H}^{0,2}_\fine(X)$, and
\[
\tilde{H}^{0,0}(X)=\tilde{H}^{0,0}_\fine(X)\oplus \tilde{H}^{1,1}_\fine(X)\oplus \tilde{H}^{2,2}_\fine(X).
\]
In particular, the homomorphism $\mu$, given in (\ref{eq-embedding-of-HT-2-in-Mukai-lattice}), is an isomorphism
between $HT^2(X)$ and $\tilde{H}^{0,0}(X)$. The homomorphism $\mu$ thus pulls back the Mukai pairing to a non-degenerate pairing on $HT^2(X)$ satisfying 
\begin{equation}
\label{eq-Mukai-pairing-on-HT-2}
(\lambda_1,\lambda_2):=(m_{\lambda_1}(\sigma),m_{\lambda_2}(\sigma)),
\end{equation}
for all $\lambda_1,\lambda_2\in HT^2(X)$.
The homomorphism $\mu$
depends only on the choice of $\sigma$, and is thus canonical up to a scalar multiple. Hence, so is the pairing (\ref{eq-Mukai-pairing-on-HT-2}). In particular, a hyperplane $\Sigma$ in
$HT^2(X)$ determines the orthogonal line $\Sigma^\perp$ in  $HT^2(X)$, and hence the line $\ell(\Sigma)$ in 
$\tilde{H}^{0,0}(X)$, which is the image of $\Sigma^\perp$ via (\ref{eq-embedding-of-HT-2-in-Mukai-lattice}).
\begin{equation}
\label{eq-ell-Sigma}
\ell(\Sigma) := \mu(\Sigma^\perp):=\{m_\lambda(\sigma) \ : \ \lambda\in\Sigma^\perp\}.
\end{equation}

Assume $\sigma$ is normalized so that $(\sigma,\bar{\sigma})=1$.

\begin{lem}
\label{lemma-formula-for-m-lambda}
The following equality holds for all $\lambda_1,\lambda_2\in HT^2(X)$.
\begin{equation}
\label{eq-formula-for-m-lambda}
m_{\lambda_1}(m_{\lambda_2}(\sigma))=-(\mu(\lambda_1),\mu(\lambda_2))\bar{\sigma}.
\end{equation}
\end{lem}

\begin{proof}
We have
\[
(m_{\lambda_1}(m_{\lambda_2}(\sigma)),\sigma)=-(m_{\lambda_2}(\sigma),m_{\lambda_1}(\sigma))=-(\mu(\lambda_1),\mu(\lambda_2)),
\]
where the first equality is due to the fact that $m_{\lambda_1}$ belongs to $\LieAlg{g}_\CC$, by Corollary \ref{cor-image-of-m-is-in-LLV-algebra}, and the second equality follows by the definition of $\mu$.
Now $m_\lambda$ has degree $2$ with respect to $h'$, for all $\lambda\in HT^2(X)$, 
and so $m_{\lambda_1}(m_{\lambda_2}(\sigma))$ belongs to $\tilde{H}^{0,2}_\fine(X)$ and is thus of the form $t\bar{\sigma}$, for a scalar $t$. We get that $t=(\sigma,t\bar{\sigma})=(\sigma,m_{\lambda_1}(m_{\lambda_2}(\sigma)))=-(\mu(\lambda_1),\mu(\lambda_2))$.
\end{proof}

%
\subsection{Subalgebras of the  Hochschild Lie algebra generated by a hyperplane in $HT^2(X)$}
\label{sec-subalgebras-of-LLV-generated-by-a-hyperplane}
\begin{lem}
\label{lemma-line-in-Mukai-lattice-associated-to-hyperplane-in-HT-2}
Let $\Sigma$ be a hyperplane in $HT^2(X)$. Let $\ell_0$ be a line in $\tilde{H}^{0,0}(X)=\CC\alpha\oplus H^{1,1}(X)\oplus \CC\beta$. The image  $m(\Sigma)$ of $\Sigma$ in $\LieAlg{g}_\CC$ via (\ref{eq-m-from-HT-2-into-LLV-algebra}) is contained in the subalgebra $\LieAlg{g}_{\CC,\ell_0}$ of $\LieAlg{g}_\CC$ annihilating $\ell_0$, if and only if  $\ell_0=\ell(\Sigma)$.
\end{lem}

\begin{proof}
Let $\lambda_0$ be a non-zero class in $\Sigma^\perp$. Then $m_{\lambda_0}(\sigma)$ spans $\ell(\Sigma)$. Given a class $\lambda\in HT^2(X)$ we have
$
m_\lambda(m_{\lambda_0}(\sigma))=-(\lambda_0,\lambda)\bar{\sigma},
$
by Equation (\ref{eq-formula-for-m-lambda}) and the definition of the pairing on $HH^2(T)$.
We conclude that $m_\lambda$ annihilates $\ell(\Sigma)$, if and only if $\lambda$ belongs to $\Sigma$.
\end{proof}

\begin{lem}
\label{lemma-conjugation-of-m-by-g}
Let $g\in O(\tilde{H}(X,\CC))$ be an isometry which leaves invariant every element of $\span\{\sigma,\bar{\sigma}\}$.
Then
\begin{equation}
\label{eq-conjugating-m-lambda-by-g}
gm_\lambda g^{-1}=m_{\mu^{-1}(g(\mu(\lambda)))},
\end{equation}
where $\mu^{-1}:\tilde{H}^{0,0}(X)\rightarrow HT^2(X)$ is the inverse of $\mu$.
\end{lem}

\begin{proof} We have
$(gm_\lambda g^{-1})(\sigma)=g(m_\lambda(\sigma))=g(\mu(\lambda))=m_{\mu^{-1}(g(\mu(\lambda)))}(\sigma).$
Every element of $\tilde{H}^{0,0}(X)$ is of the form $\mu(\lambda')$, for some $\lambda'\in HT^2(X)$, and
\[
m_{\mu^{-1}(g(\mu(\lambda)))}(\mu(\lambda'))=-(g(\mu(\lambda)),\mu(\lambda'))\bar{\sigma}=-(\mu(\lambda),g^{-1}(\mu(\lambda')))\bar{\sigma},
\]
\[
(gm_\lambda g^{-1})(\mu(\lambda'))=-g[(\mu(\lambda),g^{-1}(\mu(\lambda'))) \bar{\sigma}]=
-(\mu(\lambda),g^{-1}(\mu(\lambda')))\bar{\sigma}.
\]
Finally, both sides of (\ref{eq-conjugating-m-lambda-by-g}) annihilate $\bar{\sigma}$.
\end{proof}

\begin{lem}
\label{lemma-on-planes-of-signature-1-1}
Let $U_0\subset HT^2(X)$ be a $2$-dimensional subspace such that the subspace $\mu(U_0)$ of $\tilde{H}(X,\CC)$ is defined over $\RR$ and the restriction of the pairing to $\mu(U_0)$ is non-degenerate.
Let $U_0^\perp$ be the subspace of $HT^2(X)$ which is the orthogonal complement of $U_0$ with respect to the pairing (\ref{eq-Mukai-pairing-on-HT-2}). Then the subalgebra of
$\LieAlg{g}_\CC:=\LieAlg{so}(\tilde{H}(X,\CC))$ generated by $m(U_0^\perp)$ and its complex conjugate
$\overline{m(U_0^\perp)}$ is equal to the subalgebra
$\LieAlg{so}[\tilde{H}^{1,-1}(X)\oplus \mu(U_0^\perp)\oplus \tilde{H}^{-1,1}(X)]$ annihilating $\mu(U_0)$.
\end{lem}

\begin{proof}
The proof is identical to Step 2 in the proof of Lemma \ref{lemma-alpha-remains-of-Hodge-type} replacing $\iota:H^1(TX)\rightarrow \LieAlg{g}_\CC$ by $m:U_0^\perp\rightarrow \LieAlg{g}_\CC$.
\hide{
Let $g\in O(\tilde{H}(X,\RR))$ be an isometry which leaves invariant every element of 
$\tilde{H}(X,\RR)\cap \span\{\sigma,\bar{\sigma}\}$ (so that $\tilde{H}^{0,0}(X)$ is $g$ invariant) and maps $\mu(U_0)$ to $U_\CC=\span_\CC\{\alpha,\beta\}$. Then $g(\mu(U_0^\perp))=\mu(H^1(TX))$. 
Given $\lambda\in U_0^\perp$, we have
$gm_\lambda g^{-1}=m_{\mu^{-1}(g(\mu(\lambda))}$, by Lemma \ref{lemma-conjugation-of-m-by-g}, 
and $\mu^{-1}(g(\mu(\lambda)))$ belongs to $H^1(TX)$. 
The Lie algebra generated by $m(H^1(TX))$ and its complex conjugate 
$\overline{m(H^1(TX))}$ is $\bar{\LieAlg{g}}=\LieAlg{so}[\tilde{H}^{1,-1}(X)\oplus \mu(H^1(TX))\oplus \tilde{H}^{-1,1}(X)]$, by the proof of Lemma \ref{lemma-alpha-remains-of-Hodge-type}.
Hence, $m(U_0^\perp)$ and its complex conjugate
$\overline{m(U_0^\perp)}$ generate 
\[
g^{-1}\bar{\LieAlg{g}}g=\LieAlg{so}[\tilde{H}^{1,-1}(X)\oplus \mu(U_0^\perp)\oplus \tilde{H}^{-1,1}(X)] 
\]
as claimed.
}
\end{proof}

\begin{lem}
\label{lemma-m-Sigma-and-its-conjugate-generate-the-annihilator-of-ell}
Let $\ell\subset HT^2(X)$ be a $1$-dimensional subspace such that the line $\mu(\ell)$ in $\tilde{H}(X,\CC)$ is defined over $\RR$. Let $\Sigma$ be the subspace of $HT^2(X)$ which is the orthogonal complement of $\ell$ with respect to the pairing (\ref{eq-Mukai-pairing-on-HT-2}). Then the subalgebra of
$\LieAlg{g}_\CC$ generated by $m(\Sigma)$ and its complex conjugate
$\overline{m(\Sigma)}$ is equal to the subalgebra
$\LieAlg{g}_{\CC,\ell}$ annihilating $\mu(\ell)$.
\end{lem}

\begin{proof}
The subspace $m(\Sigma)$ of $\LieAlg{g}_\CC$ is contained in $\LieAlg{g}_{\CC,\ell}$, 
by Lemma \ref{lemma-line-in-Mukai-lattice-associated-to-hyperplane-in-HT-2}, and  its complex conjugate $\overline{m(\Sigma)}$
is contained in  $\LieAlg{g}_{\CC,\ell}$, since $\mu(\ell)$ is defined over $\RR$. Hence, the subalgebra $\LieAlg{a}$  of $\LieAlg{g}_\CC$ generated by $m(\Sigma)$ and
$\overline{m(\Sigma)}$ is contained in $\LieAlg{g}_{\CC,\ell}$.
Furthermore, $\LieAlg{a}$  contains the subalgebra $\LieAlg{g}_{\CC,U_0}$ annihilating $U_0$, for every two dimensional subspace $U_0$ of $\tilde{H}^{0,0}(X)$, which is defined over $\RR$, contains $\mu(\ell)$,
and such that the induced pairing on $U_0$ is non-degenerate,
by Lemma \ref{lemma-on-planes-of-signature-1-1}.

Choose such a $U_0$. We get the flag $\LieAlg{g}_{\CC,U_0}\subset \LieAlg{a}\subseteq \LieAlg{g}_{\CC,\ell}$, where the first inclusion is strict.
The quotient
$\LieAlg{g}_\CC/\LieAlg{g}_{\CC,U_0}$ is isomorphic to the direct sum of $\LieAlg{so}(U_0)$ and $U_0\otimes U_0^\perp$.
If $\LieAlg{a}$ contains  $U_0\otimes U_0^\perp$ then it is the whole of $\LieAlg{g}_\CC$, since 
$\LieAlg{so}(U_0)$ is contained\footnote{
Choose and orthogonal basis $\{e_i\}$ of $\tilde{H}(X,\CC)$ with $e_1, e_2\in U_0$. Let
$\{e_i^*\}$ be the dual basis. Set $A=(e_1,e_1)e_1^*\otimes e_3-(e_3,e_3)e_3^*\otimes e_1$, 
$B=(e_2,e_2)e_2^*\otimes e_3-(e_3,e_3)e_3^*\otimes e_2$, and
$C=(e_1,e_1)e_1^*\otimes e_2-(e_2,e_2)e_2^*\otimes e_1$. Then $A$ and $B$ belong to the submodule of $\LieAlg{g}_\CC$ isomorphic to 
$U_0\otimes U_0^\perp$, $C$ spans $\LieAlg{so}(U_0)$, and 
$[A,B]=(e_3,e_3)C$.
} 
in the Lie bracket of $U_0\otimes U_0^\perp$ with itself.
Hence, $\LieAlg{a}$ intersects  $U_0\otimes U_0^\perp$ either trivially, or along 
$\ell_0\otimes U_0^\perp$, for some line $\ell_0\subset U_0$. 
If the intersection is trivial, then $\LieAlg{a}=\LieAlg{so}(U_0)\times \LieAlg{g}_{\CC,U_0}$, which is impossible, since 
$\LieAlg{so}(U_0)$ is not contained in $\LieAlg{g}_{\CC,\ell}$.  Hence, the intersection is 
$\ell_0\otimes U_0^\perp$, for some line $\ell_0\subset U_0$. 
Finally, $\ell_0$ must be orthogonal to $\mu(\ell)$, since $\LieAlg{a}$ is contained in $\LieAlg{g}_{\CC,\ell}$. 
\end{proof}

\hide{
\begin{rem}
If we replace in Lemma \ref{lemma-m-Sigma-and-its-conjugate-generate-the-annihilator-of-ell} the assumption that $\mu(\ell)$ is defined over $\RR$ by the assumption that $U_0:=\mu(\ell)+\overline{\mu(\ell)}$ is two dimensional and the induced pairing on $U_0$ is non-degenerate, then the subalgebra $\LieAlg{a}$ of
$\LieAlg{g}_\CC$ generated by $m(\Sigma)$ and its complex conjugate
$\overline{m(\Sigma)}$ is the whole of $\LieAlg{g}_\CC$. Indeed, $\LieAlg{a}$ contains $\LieAlg{g}_{\CC,U_0}$,
by Lemma \ref{lemma-on-planes-of-signature-1-1},
the proof of Lemma \ref{lemma-m-Sigma-and-its-conjugate-generate-the-annihilator-of-ell} 
shows that the subalgebra $\LieAlg{b}$ generated by $m(\Sigma)$ and $\LieAlg{g}_{\CC,U_0}$ is
$\LieAlg{g}_{\CC,\mu(\ell)}$, 
the subalgebra $\LieAlg{c}$ generated by $\overline{m(\Sigma)}$ and $\LieAlg{g}_{\CC,U_0}$ is
$\LieAlg{g}_{\CC,\overline{\mu(\ell)}}$. Hence, $\LieAlg{a}$ contains both  $\LieAlg{g}_{\CC,\mu(\ell)}$ and $\LieAlg{g}_{\CC,\overline{\mu(\ell)}}$ and so $\LieAlg{a}/\LieAlg{g}_{\CC,U_0}$ contains $U_0\otimes U_0^\perp$. It follows that $\LieAlg{a}$
is the whole of $\LieAlg{g}_\CC$, as we saw in the proof of Lemma \ref{lemma-m-Sigma-and-its-conjugate-generate-the-annihilator-of-ell}.
\end{rem}
}

\begin{lem}
\label{lemma-Veronese-is-defined-over-Q}
Let $K$ be a subfield of $\CC$,
let $\ell$ be a line in $\widetilde{H}(X,\CC)$, and let $j$ a positive integer. If the line $\ell^j$ in $\Sym^j\widetilde{H}(X,\CC)$
is defined over $K$, then so does $\ell$. 
\end{lem}

\begin{proof}
The Veronese embedding $\varphi:\PP(\widetilde{H}(X,\CC))\rightarrow \PP(\Sym^{j}\widetilde{H}(X,\CC))$, $\ell\mapsto \ell^j$, is defined over $\QQ$ and is thus equivariant with respect to every automorphism $g$ of $\CC$. 
Assume that $\ell^j$ is defined over $K$.
The injectivity of  $\varphi$ implies that $\ell$ is invariant with respect to every automorphism of $\CC$ leaving $K$ invariant.
Hence, $\ell$ is defined over $K$.
\end{proof}

\begin{lem}
\label{lemma-pde}
Let $f$ be a non-zero element of the kernel of $\Delta:\Sym^n(\tilde{H}(X,\QQ))\rightarrow \Sym^{n-2}(\tilde{H}(X,\QQ))$ of Hodge type $(0,0)$.
Let $\Sigma$ be a subspace of $HT^2(X)$ of co-dimension $1$. 
Assume that $m(\Sigma)$ is contained in the subalgebra $\LieAlg{g}_{\CC,f}$ annihilating $f$. Then the line $\ell(\Sigma)\subset \tilde{H}(X,\CC)$, given in (\ref{eq-ell-Sigma}), is defined over $\QQ$, is contained in $\tilde{H}^{0,0}(X)$,  and
$f$ spans the image of the projection of $\ell(\Sigma)^n$ to $\ker(\Delta)$. Furthermore, $\LieAlg{g}_{\CC,f}=\LieAlg{g}_{\CC,\ell(\Sigma)}$.
\end{lem}

\begin{proof}
The element $f$ has the form
${\displaystyle 
\sum_{i=0}^{\lfloor n/2\rfloor}(\sigma\bar{\sigma})^i f_i,
}$
where $f_i$ belongs to $\Sym^{n-2i}(\tilde{H}^{0,0}(X))$, since $f$ has Hodge type $(0,0)$.

\underline{Step 1:} We derive a system (\ref{eq-first-recursive-equation}) and (\ref{eq-second-recursive-equation}) 
of partial differential equations that the $f_i$ satisfy.
Let $\{x_1, \dots, x_N\}$ be an orthonormal basis of $\tilde{H}^{0,0}(X)$. 
We have
\[
\Delta\left((\sigma\bar{\sigma})^i f_i\right)=
i^2(\sigma\bar{\sigma})^{i-1}f_i+(\sigma\bar{\sigma})^i\Delta(f_i).
\]
Hence, 
\begin{equation}
\label{eq-first-recursive-equation}
\Delta(f_{i-1})+i^2f_{i}=0,
\end{equation}
for $1\leq i\leq\lfloor n/2\rfloor$, 
since $f$ belongs to the kernel of $\Delta$. 
Given $\lambda\in HT^2(X)$, we have $m_\lambda(x_i)=-(\mu(\lambda),x_i)\bar{\sigma}$, by Equation (\ref{eq-formula-for-m-lambda}). Thus,
\[
m_\lambda\left(
(\sigma\bar{\sigma})^i x_{i_1}\cdots x_{i_{n-2i}}
\right)=
i\mu(\lambda)\sigma^{i-1}\bar{\sigma}^i x_{i_1}\cdots x_{i_{n-2i}}-
\sigma^i \bar{\sigma}^{i+1} \sum_{j=1}^{n-2i}(\mu(\lambda),x_{i_j})x_{i_1}\cdots \hat{x}_{i_j} \cdots x_{i_{n-2i}}.
\]
Let $\mu(\lambda)^*$ be the element of $\tilde{H}^{0,0}(X)^*$ satisfying $\mu(\lambda)^*(\lambda')=(\mu(\lambda),\lambda')$, for all $\lambda'\in \tilde{H}^{0,0}(X)$. Then
\[
m_\lambda((\sigma\bar{\sigma})^i f_i)=i\mu(\lambda)\sigma^{i-1}\bar{\sigma}^i f_i
- \sigma^i \bar{\sigma}^{i+1} [f_i\lrcorner \mu(\lambda)^*].
\]

\[
m_\lambda(f)=
\sum_{i=1}^{\lfloor n/2\rfloor+1} \left(
i\mu(\lambda)f_i-[f_{i-1}\lrcorner \mu(\lambda)^*]
\right)\sigma^{i-1}\bar{\sigma}^i=0,
\]
for all $\lambda\in \Sigma$, where we set $f_{\lfloor n/2\rfloor+1}=0$. 
Hence, for all $\lambda\in \Sigma$ and $0<i\leq \lfloor n/2\rfloor+1$ we have
\begin{equation}
\label{eq-second-recursive-equation}
i\mu(\lambda)f_i-[f_{i-1}\lrcorner \mu(\lambda)^*]=0.
\end{equation}
So
$
ix_jf_i=\frac{\partial f_{i-1}}{\partial x_j},
$ 
if $x_j$ belongs to $\mu(\Sigma)$.

\underline{Step 2:} We prove that $\ell(\Sigma)$ is defined over $\RR$.
If $n$ is odd, then $f_{\lfloor n/2\rfloor}$ is linear and Equation (\ref{eq-second-recursive-equation}) for $i=1+\lfloor n/2\rfloor$ implies that $f_{\lfloor n/2\rfloor}$ is orthogonal to $\mu(\Sigma)$ and so belongs to $\ell(\Sigma)$. Note that $f_i$ is real, for all $i$, since $f$ is rational. 
If $f_{\lfloor n/2\rfloor}$ does not vanish, then  $\ell(\Sigma)$ is defined over $\RR$.
If $f_{\lfloor n/2\rfloor}$ does vanish, let $j$ be the maximal index, such that $f_j\neq 0$, but $f_{j+1}=0$.
Equation  (\ref{eq-second-recursive-equation}) for $i=j+1$ implies that $f_j$ belongs to $\ell(\Sigma)^{n-2j}$, which is thus defined over $\RR$, and again we get that
$\ell(\Sigma)$ is defined over $\RR$, by Lemma \ref{lemma-Veronese-is-defined-over-Q}.

If $n$ is even, then $f_{n/2}$ is constant. If it vanishes, then the above argument proves that $\ell(\Sigma)$ is defined over $\RR$. Assume that $f_{n/2}=t$, where $t\neq 0$. Set $g:=f_{n/2-1}$. Assume first that $\ell(\Sigma)$ is not isotropic. We can then assume that $x_1$ belongs to $\ell(\Sigma)$. 
Equation (\ref{eq-second-recursive-equation}) for $i={n/2}$ implies that
$\frac{\partial g}{\partial x_j}=\frac{nt}{2}x_j$, for $j\geq 2$. So $g=\frac{nt}{4}(x_2^2+ \ \cdots \ +x_N^2)+cx_1^2$.
Equation (\ref{eq-first-recursive-equation}) yields $-\frac{tn^2}{4}=\Delta(g)=\frac{nt(N-1)}{4}+c$, so
$c=\frac{-tn(n+N-1)}{4}$.
We get
\[
f_{n/2-1}=\frac{nt}{4}(x_1^2+ \ \cdots \ +x_N^2)-\left(\frac{tn(N+n)}{4}\right)x_1^2.
\]
Now, $t$ is real, as is $f_{n/2-1}$. The element $x_1^2+ \ \cdots \ +x_N^2$ is independent of the choice of an orthonormal basis and is thus real. The coefficient $\frac{tn(N+n)}{4}$ is a non-zero real number. Hence, $x_1^2$ is real and the line
$\ell(\Sigma)$ is defined over $\RR$.

Assume that $n$ is even and $\ell(\Sigma)$ is isotropic. We claim that $f_{n/2}=0.$ 
Set $t:=f_{n/2}$ and $g:=f_{n/2-1}$. We may assume that $\ell(\Sigma)$ is spanned by $x_1+\sqrt{-1}x_2$. Then
$x_1+\sqrt{-1}x_2$ belongs to $\mu(\Sigma)$. 
Equation (\ref{eq-second-recursive-equation}) for $i={n/2}$ implies that
$\frac{\partial g}{\partial x_j}=\frac{nt}{2}x_j$, for $j\geq 3$ and
\[
(n/2)(x_1+\sqrt{-1}x_2)t=g\lrcorner (x_1+\sqrt{-1}x_2)^*
\] 
So
\[
f_{n/2-1}=\frac{nt}{4}(x_3^2+ \ \cdots \ +x_N^2)+\frac{nt}{4}(x_1+\sqrt{-1}x_2)(x_1-\sqrt{-1}x_2)
+a(x_1+\sqrt{-1}x_2)^2,
\]
 for some scalar $a$. Equation (\ref{eq-first-recursive-equation}) yields 
 \[
-\frac{tn^2}{4}= \Delta(f_{n/2-1})=\frac{nt}{4}(N-2)+\frac{nt}{2}=\frac{nNt}{4}.
 \]
 Hence, $t=0$ as claimed. We conclude once again that $\ell(\Sigma)$ is defined over $\RR$.

\underline{Step 3:} The subalgebra $\LieAlg{a}$ of $\LieAlg{g}_\CC$ generated by $m(\Sigma)$ and its complex conjugate
$\overline{m(\Sigma)}$ is equal to the subalgebra $\LieAlg{g}_{\CC,\ell(\Sigma)}$  of $\LieAlg{g}_\CC$ annihilating $\ell(\Sigma)$,
by Lemma \ref{lemma-m-Sigma-and-its-conjugate-generate-the-annihilator-of-ell}. The subalgebra $\overline{m(\Sigma)}$ 
is contained in $\LieAlg{g}_{\CC,f}$, since $f$ is rational and $m(\Sigma)$ is contained in $\LieAlg{g}_{\CC,f}$. Hence, 
$\LieAlg{g}_{\CC,\ell(\Sigma)}$ is contained in $\LieAlg{g}_{\CC,f}$.

Let $\lambda$ be a non-zero element of $\ell(\Sigma)$. 
The trivial $\LieAlg{g}_{\CC,\ell(\Sigma)}$-submodule in $\Sym^n(\tilde{H}(X,\CC))$ has the basis
$\{\lambda^{n-2i}\tilde{q}^i\}_{i=0}^{\lfloor n/2\rfloor}$, where $\tilde{q}$ is the $\LieAlg{g}_\CC$-invariant degree $2$ 
element given in Example \ref{example-projection-to-Im-Psi}. 
$\Delta:\Sym^n(\tilde{H}(X,\CC))\rightarrow \Sym^{n-2}(\tilde{H}(X,\CC))$ is a surjective $\LieAlg{g}_\CC$-modules homomorphism, hence it induces a surjective homomorphism of the respective trivial $\LieAlg{g}_{\CC,\ell(\Sigma)}$-submodules.
We conclude that the trivial $\LieAlg{g}_{\CC,\ell(\Sigma)}$-submodule $\ker(\Delta)^{\LieAlg{g}_{\CC,\ell(\Sigma)}}$
of $\ker(\Delta)$ is one-dimensional. 
If $\ell(\Sigma)$ is isotropic, then $\lambda^n$ spans $\ker(\Delta)^{\LieAlg{g}_{\CC,\ell(\Sigma)}}$, 
and so $f$ is a scalar multiple of $\lambda^n$. We conclude that $\ell(\Sigma)^n$ is defined over $\QQ$, hence so is $\ell(\Sigma)$, by Lemma \ref{lemma-Veronese-is-defined-over-Q}. 

If $\ell(\Sigma)$ is non-isotropic, then the projection of $\lambda^n$ to $\ker(\Delta)$ spans 
$\ker(\Delta)^{\LieAlg{g}_{\CC,\ell(\Sigma)}}$, and so $f$ belongs to the projection of $\ell(\Sigma)^n.$
Hence, $\LieAlg{g}_{\CC,f}$ is contained in $\LieAlg{g}_{\CC,\ell(\Sigma)}$.
The equality $\LieAlg{g}_{\CC,\ell(\Sigma)}=\LieAlg{g}_{\CC,f}$ follows, regardless of whether $\ell(\Sigma)$ is isotropic.

It remains to prove that $\ell(\Sigma)$ is defined over $\QQ$ when $\ell(\Sigma)$ is non-isotropic.
The quadric $Q$ in $\PP(\widetilde{H}(X,\CC))$ of isotropic lines is defined over $\QQ$ and the morphism
\[
\PP(\widetilde{H}(X,\CC))\setminus Q\rightarrow Gr\left(\Choose{N+1}{2},\LieAlg{g}_\CC\right),
\]
sending $\ell$ to $\LieAlg{g}_{\CC,\ell}$, is defined over $\QQ$ (and is thus equivariant with respect to the group of field automorphisms of $\CC$, and is injective. Hence, if $\LieAlg{g}_{\CC,\ell}$ is defined over $\QQ$, then so is $\ell$, as both are invariant under all field automorphisms of $\CC$. We already know that 
$\LieAlg{g}_{\CC,\ell(\Sigma)}=\LieAlg{g}_{\CC,f}$ is defined over $\QQ$. Hence, so it $\ell(\Sigma)$.
\end{proof}

%
\subsection{The LLV subspace of an object in $D^b(X)$}
\label{sec-LLV-subspace}
The {\em Mukai vector}  of an object $F\in D^b(X)$ is the class $v(F):=ch(F)\sqrt{td_{X}}$ in $H^*(X,\RationalNumbers)$.
Let $\LieAlg{g}_{v(F)}$ be the subalgebra of $\LieAlg{g}$ annihilating $v(F)$. 
Let $\obs^{HT}_F$ be the obstruction map (\ref{eq-obstruction-map}) and set $\Sigma(F):=\ker\left(\obs^{HT}_F\right)$.
Let $\tilde{\Sigma}(F)$ be the image of the subspace $\Sigma(F)$  of $HT^2(X)$ via the graded Duflo isomorphism
$D:HT^*(X)\rightarrow HT^*(X)$, given by contraction with $\sqrt{td_X}$. 
\[
\tilde{\Sigma}(F):=D(\Sigma(F))=D\left(\ker\left(\obs^{HT}_F\right)\right).
\]

\begin{lem}
\label{lem-Sigma-F-is-contained-in-stabilizer-of-Mukai-vector}
Let $X$ be a projective irreducible holomorphic symplectic manifold and $F\in D^b(X)$ an object. 
The image $m(\tilde{\Sigma}(F))$ of the subspace $\tilde{\Sigma}(F)$  of $HT^2(X)$ via (\ref{eq-m-from-HT-2-into-LLV-algebra}) is contained in the complexification of $\LieAlg{g}_{v(F)}$. 
\end{lem}

\begin{proof}
The image $m(\tilde{\Sigma}(F))$ is contained in $\LieAlg{g}_\CC$, by Corollary \ref{cor-image-of-m-is-in-LLV-algebra}. We prove that it annihilates $v(F)$ in two steps.

\underline{Step 1:} Let $\Delta:X\rightarrow X\times X$ be the diagonal embedding. 
Then $HH^i(X):=\Hom(\Delta_*\StructureSheaf{X},\Delta_*\StructureSheaf{X}[i])$ and 
$HH_i(X):=\Hom(\Delta_!\StructureSheaf{X}[i],\Delta_*\StructureSheaf{X})$, where $\Delta_!$ is the left adjoint of $\Delta^*$.
We have $\Delta_!\StructureSheaf{X}\cong \Delta_*\omega_X^{-1}[-2n]$, where $\dim(X)=2n$. 
Hence, $HH_0(X)=\Hom(\Delta_*\omega_X^{-1}[-2n],\Delta_*\StructureSheaf{X})$.
Elements of $HH^i(X)$ yield natural transformations from the identity functor $id$ of $D^b(X)$ to $id[i]$. We get the homomorphism
$
\obs_F:HH^2(X)\rightarrow \Hom(F,F[2])
$
sending $\alpha$ to the value $\alpha_F$ of the natural transformation on the object $F$.
The Chern character $\widetilde{ch}(F)$ is defined as an element of $HH_0(X)$ \cite[Def. 6.2]{caldararu-I}.
We prove first that the kernel $\Sigma'(F)$ of $\obs_F$ is contained in the ideal of $HH^*(X)$ annihilating $\widetilde{ch}(F)$.

Given a smooth complex projective variety $Y$ and an object $E$ of $D^b(Y)$, let
Let $Tr_Y:\Hom(E,E\otimes\omega_Y[\dim Y])\rightarrow \CC$ be the Serre's Duality trace (see \cite[Sec. 2]{caldararu-I}).
The Chern character $\widetilde{ch}(F)$ satisfies 
\[
Tr_{X\times X}(\widetilde{ch}(F)\circ \nu)=Tr_X(\nu_F),
\]
for all $\nu\in \Hom(\Delta_*\StructureSheaf{X},\Delta_*\omega_X[2n])$, where $\nu_F\in \Hom(F,F\otimes\omega_X[2n])$
is the value on $F$ of the natural transformation $\nu$ from $id$ to the Serre functor of tensorization by $\omega_X[2n]$
\cite[Def. 6.2]{caldararu-I}. In particular, for every $\alpha\in HH^2(X)$ and $\beta\in \Hom(\Delta_*\StructureSheaf{X},\Delta_*\omega_X[2n-2])$, setting $\nu=\alpha\circ\beta$ above we get
\[
Tr_{X\times X}([\widetilde{ch}(F)\circ \alpha]\circ\beta)=Tr_X(\alpha_F\circ \beta_F).
\]
If $\alpha_F=0$, then the element $\widetilde{ch}(F)\circ \alpha$ of $HH_{-2}(X)=\Hom(\Delta_*\omega_X^{-1}[-2-2n],\Delta_*\StructureSheaf{X})$ pairs to zero with every class $\beta$ as above. Hence, $\widetilde{ch}(F)\circ \alpha=0,$
because the trace pairing is non-degenerate.

\underline{Step 2:} 
Consider the  Hochschild-Kostant-Rosenberg isomorphisms
\begin{eqnarray*}
I^{HKR} : HT^*(X) & \rightarrow & HH^*(X),
\\
I_{HKR} : HH_*(X) & \rightarrow & H\Omega_*(X).
\end{eqnarray*}
Set $I^K:=I^{HKR}\circ D^{-1}$ and let $I_K$ be given by $I_K(\alpha)=I_{HKR}(\alpha)\sqrt{td_X}$.
We have $I_K(\widetilde{ch}(F))=ch(F)\sqrt{td_X}=v(F)$, by \cite[Theorem 4.5]{caldararu-II},  
and 
$I_K(I^K(\alpha)\circ \widetilde{ch}(F))=\alpha\cdot I_K(\widetilde{ch}(F))$,  for all $\alpha\in HT^*(X)$, by \cite[Theorem 1.4]{calaque-et-al}.
Hence, the image via $(I^K)^{-1}$ of 
the kernel $\Sigma'(F)$ of $\obs_F$ annihilates $v(F)$, by Step 1 of the proof.
\begin{equation}
\label{I^K-inverse-maps-Sigma-prime-to-stabilizer}
m\left((I^K)^{-1}(\Sigma'(F))\right)\subset \LieAlg{g}_{\CC,v(F)}.
\end{equation}

We have $\obs_F(I^{HKR}(\alpha))=(\exp a_F)(\alpha)$, for all $\alpha\in HT^2(X)$, by \cite[Theorem A]{huang}.
Hence,  $I^{HKR}$ maps 
the kernel $\Sigma(F)$ of the obstruction map $\obs^{HT}_F$, given in (\ref{eq-obstruction-map}), to 
the kernel $\Sigma'(F)$ of $\obs_F$. So $(I^K)^{-1}(\Sigma'(F))=D(\Sigma(F))=\tilde{\Sigma}(F)$.
We conclude that  $m(\tilde{\Sigma}(F))$ annihilates $v(F)$, by Equation (\ref{I^K-inverse-maps-Sigma-prime-to-stabilizer}).
%
\end{proof}

Let $X$ be a $2n$-dimensional irreducible holomorphic symplectic manifold. Composing the homomorphism $m:HT^2(X)\rightarrow \LieAlg{g}_\CC$, given in Equation (\ref{eq-m-from-HT-2-into-LLV-algebra}), with the inclusion 
$\LieAlg{g}_\CC\subset \End(H^*(X,\CC))$ followed by evaluation at $v(F)$ we get the {\em cohomological obstruction map}
\[
\obs_F^{Coh}:HT^2(X)\rightarrow H^*(X,\CC).
\]

\begin{defi}
\label{def-deforms-in-co-dimension-one}
\begin{enumerate}
\item
We say that an object $F$ in $D^b(X)$ 
{\em deforms in co-dimension $c$}, if the kernel $\Sigma(F)$ of the obstruction map (\ref{eq-obstruction-map}) has co-dimension $c$ in $HT^2(X)$. 
\item
\label{def-item-remains-of-Hodge-type-in-co-dimension-1}
We say that $F$ {\em has a rank $c$ cohomological obstruction map}, if $\rank(\obs_F^{Coh})=c$, i.e., if
the annihilator $\LieAlg{g}_{v(F)}$ of the Mukai vector $v(F)$ intersects $m(HT^2(X))$ in a co-dimension $c$ subspace of the latter.
\item
If $F$ is $\theta$-twisted of non-zero rank, we say that 
$F$ {\em has a rank $c$ cohomological obstruction map}
if the assumption in Part (\ref{def-item-remains-of-Hodge-type-in-co-dimension-1}) holds with $v(F)$ replaced by  
$\bar{v}(F):=\kappa(F)\sqrt{td_X}$.
\end{enumerate}
\end{defi}

If $F$ deforms in co-dimension $c$, then $F$ has a rank $\leq c$ cohomological obstruction map,
since $m(\tilde{\Sigma}(F))$ is contained in $\LieAlg{g}_{v(F)}$
by Lemma \ref{lem-Sigma-F-is-contained-in-stabilizer-of-Mukai-vector}.
Important examples of very modular sheaves, which deform in co-dimension $2$, are constructed in \cite{markman-BBF-class-as-characteristic-class} in the $K3^{[n]}$ deformation type and in \cite{markman-generalized-kummers} in the generalized Kummer deformation type.

Let $k$ be a positive integer. If $F_i$, $1\leq i\leq k$, each deforms in co-dimension one and each has positive rank, then $\otimes_{i=1}^k F_i$ deforms in co-dimension $\leq 2$. Indeed, in this case $\Sigma(F_i)$ does not contain $H^2(\StructureSheaf{X})$, as was explained in the proof of 
Proposition \ref{prop-kappa-class-remains-of-Hodge-type}, 
and so the intersection $\Sigma(F_i)\cap [H^1(TX)\oplus H^2(\StructureSheaf{X})]$ 
is the graph of a linear homomorphism $\theta_i:H^1(TX)\rightarrow H^2(\StructureSheaf{X})$
and $F_i$ deforms in the direction $\xi+\theta_i(\xi)$ for all $\xi\in H^1(TX)$. Then $\otimes_{i=1}^k F_i$ deforms in the direction 
$\xi+(\sum_{i=1}^k\theta_i)(\xi)$, for each $\xi\in H^1(TX)$. So $\Sigma(\otimes_{i=1}^k F_i)\cap [H^1(TX)\oplus H^2(\StructureSheaf{X})]$ contains the graph of $\sum_{i=1}^k\theta_i$.
The linear homomorphism $\theta_i$ is computed in Lemma \ref{lemma-theta-i} below.

\begin{defi}
The {\em LLV subspace} of an object $F$ in $D^b(X)$, is the subspace of $\tilde{H}(X,\QQ)$ given by 
\[
W(F):=\mu(\tilde{\Sigma}(F)^\perp)
\]
generalizing Equation (\ref{eq-ell-Sigma}). 
\end{defi}

Note that $F$ deforms in co-dimension $\dim(W(F))$. 
If one-dimensional, the LLV subspace $W(\otimes_{i=1}^k F_i)$ of the tensor product discussed above is described by 
Theorem \ref{thm-Mukai-vector} (\ref{prop-item-spanning-Mukai-vector}), while if it is two dimensional it is computed in Lemma \ref{lemma-theta-i}.
Given a subspace $W\subset \tilde{H}(X,\RationalNumbers)$, let $\LieAlg{g}_W$ be the subalgebra of $\LieAlg{g}$ annihilating $W$. 
If $F$  has a rank $1$ cohomological obstruction map, 
and the projection $\hat{v}(F)$ of $v(F)$ to $SH^*(X)$ does not vanish, 
then $\LieAlg{g}_{v(F)}=\LieAlg{g}_{\ell(F)}$, for a rational line $\ell(F)$ in
$\tilde{H}(X,\QQ)$, by Lemma \ref{lemma-pde} applied with $f=\Psi(\hat{v}(F))$. The equality $\ell(F)=W(F)$ holds, if furthermore $F$ deforms in co-dimension $1$.

%
\subsection{The LLV line}
\label{sec-LLV-line}

\begin{thm}
\label{thm-Mukai-vector}
Let $X$ be a projective irreducible holomorphic symplectic manifold and $F\in D^b(X)$ an object which deforms in co-dimension $1$, so that $\tilde{\Sigma}(F)$ has co-dimension one in $HT^2(X)$.
Denote by $\ell(F)$ the line $\ell(\tilde{\Sigma}(F))$ in  $\tilde{H}(X,\CC)$ given in (\ref{eq-ell-Sigma}).
Assume\footnote{The assumption holds if $\rank(F)\neq 0$, or $c_1(F)\neq 0$. The assumption depends only on the orbit of $v(F)$ under the derived monodromy group.} 
that the projection $\hat{v}(F)$ of $v(F)$ to $SH^*(X)$  via (\ref{eq-projection-to-verbitsky-component}) does not vanish.
\begin{enumerate}
\item
\label{prop-item-line-in-rational-Mukai-lattice} 
The line $\ell(F)$ is defined over $\QQ$ and 
$\LieAlg{g}_{v(F)}$ is equal to the Lie subalgebra $\LieAlg{g}_{\ell(F)}$ annihilating $\ell(F)$.
\item
\label{prop-item-isotropic}
The embedding 
$\Psi$, given in (\ref{eq-Psi}), maps $\hat{v}(F)$ 
into the  image of $\ell(F)^n$ via the projection (\ref{eq-projection-to-Im-Psi}).
\item
\label{prop-item-spanning-Mukai-vector}
If the rank $r$ of $F$ does not vanish, then 
the line $\ell(F)$ is spanned by an element of the form $r\alpha+c_1(F)+s(F)\beta$, 
for some $s(F)\in\RationalNumbers$.
\item
\label{prop-item-functoriality}
Given an equivalence $\Phi:D^b(X)\rightarrow D^b(Y)$, let $\tilde{H}(\phi):\tilde{H}(X,\QQ)\rightarrow 
\tilde{H}(Y,\QQ)$ be its associated isometry via the functor (\ref{eq-functor-tilde-H}).
We have
$\ell(\Phi(F))=\tilde{H}(\phi)(\ell(F))$.
\end{enumerate}
\end{thm}

\begin{proof}
Part
(\ref{prop-item-line-in-rational-Mukai-lattice}) 
The inclusion 
$\LieAlg{g}_{\ell(F)} \subset \LieAlg{g}_{v(F)}$
follows from the assumption that $\Sigma(F)$ has co-dimension one in $HT^2(X)$, Lemma \ref{lemma-m-Sigma-and-its-conjugate-generate-the-annihilator-of-ell}, and Lemma \ref{lem-Sigma-F-is-contained-in-stabilizer-of-Mukai-vector}. 
The equality $\LieAlg{g}_{\ell(F)}=\LieAlg{g}_{\hat{v}(F)}$ and the fact that $\ell(F)$ is defined over $\QQ$ both follow from 
Lemma \ref{lemma-pde}, by letting $f$ be $\Psi(\hat{v}(F))$.
We conclude that $\LieAlg{g}_{\ell(F)}=\LieAlg{g}_{v(F)}$ from the evident inclusion $\LieAlg{g}_{v(F)}\subset \LieAlg{g}_{\hat{v}(F)}$.

Part (\ref{prop-item-isotropic}) 
follows from Lemma \ref{lemma-pde}, by letting $f$ be $\Psi(\hat{v}(F))$.

(\ref{prop-item-spanning-Mukai-vector})
The class $\kappa(F)$ is $\bar{\LieAlg{g}}$-invariant, by Proposition \ref{prop-kappa-class-remains-of-Hodge-type}.
Hence, so is $\kappa(F)\sqrt{td_X}$. 
Set $g:=\exp(e_{-c_1(F)/r})\in SO(\tilde{H}(X,\QQ))$. Then $g$ acts on $H^{ev}(X,\CC)$, via (\ref{eq-bar-rho}), and  $\kappa(F)\sqrt{td_X}=g(v(G)),$ by Theorem \ref{thm-LLV-Lie-algebra}.
The map $F\mapsto \ell(F)$ factors through $v$ and the map $v(F)\mapsto \ell(F)$ extends to an $SO(\tilde{H}(X,\CC))$-equivariant map from the $SO(\tilde{H}(X,\CC))$-orbit of $v(F)$ to the $SO(\tilde{H}(X,\CC))$-orbit of $\ell(F)$,
since the  embedding $\Psi$, the two projections 
(\ref{eq-projection-to-Im-Psi}) and 
(\ref{eq-projection-to-verbitsky-component}), and the $n$-power map $\ell(F)\mapsto \ell(F)^n$, are all
$SO(\tilde{H}(X,\QQ))$-equivariant. 
Hence, $g(\ell(F))$ is $\bar{\LieAlg{g}}$-invariant. So
$g(\ell(F))$ is contained in $U_\QQ$.

Let $w$ be a non-zero vector in $\ell(F)$. 
Write $w=r_0\alpha+\lambda+s\beta\in \tilde{H}(X,\RationalNumbers)$, with $r_0,s\in\RationalNumbers$,  and $\lambda\in H^2(X,\RationalNumbers)$. 
Note that $r_0\neq 0$. Indeed, $\exp(e_{\lambda'})(v(F))$
is not $\bar{\LieAlg{g}}$-invariant, for $\lambda'\neq -c_1(F)/r$, $\lambda'\in H^2(X,\QQ)$, and so 
neither is $\exp(e_{\lambda'})(\ell(F))$,
by the equivariance of the map $v(F)\mapsto \ell(F)$. But if $r_0=0$, then $\ell(F)$ is $\bar{\LieAlg{g}}$-invariant, if and only if
$\exp(e_{\lambda'})(\ell(F))$ is, for all $\lambda'\in H^2(X,\QQ)$.
Now, 
$\exp(e_{-\lambda'/r'})w=
r_0\alpha+\left(\lambda-\frac{r_0}{r'}\lambda'\right)+
\left(s-\frac{(\lambda,\lambda')}{r'}+\frac{r_0(\lambda,\lambda')}{2(r')^2}\right)\beta$  
belongs to $U_\RationalNumbers$ if and only if $\lambda'/r'=\lambda/r_0$. We know that 
$\exp(e_{-c_1(F)/r})w$ belongs to $U_\QQ$. Hence, $\lambda/r_0=c_1(F)/r$ and  $(r/r_0)w$ has the stated form.

(\ref{prop-item-functoriality}) Let $\phi:H^*(X,\QQ)\rightarrow H^*(Y,\QQ)$ be the isomorphism associated to $\Phi$.
Then $\phi(v(F))=v(\Phi(F))$ and so $\phi(\hat{v}(F))=\hat{v}(\Phi(F))$. 
Theorem \ref{thm-def-of-H-tilde} yields the commutative diagram
\[
\xymatrix{
SH^*(X,\QQ) \ar[d]_{\Psi_X} \ar[r]^{\epsilon(\tilde{H}(\phi))\phi} & SH^*(Y,\QQ) \ar[d]^{\Psi_Y}
\\
\ker(\Delta_X) \ar[r]_{\Sym^n(\tilde{H}(\phi))} & \ker(\Delta_Y),
}
\]
where $\Delta_X:\Sym^n(\tilde{H}(X,\QQ))\rightarrow \Sym^{n-2}(\tilde{H}(X,\QQ))$ is given in (\ref{eq-Delta}) and $\Delta_Y$ is its analogue.
Hence, $\Sym^n(\tilde{H}(\phi))$ maps $\Psi(\hat{v}(F))$ to $\Psi(\hat{v}(\Phi(F)))$.
This means that $\Sym^n(\tilde{H}(\phi))$ maps the projection of $\ell(F)^n$ in $\ker(\Delta_X)$ to the projection of $\ell(F)^n$
in $\ker(\Delta_Y)$. The map 
sending $\ell(F)$ to the projection of $\ell(F)^n$, is injective, by the equality $\LieAlg{g}_{\CC,\hat{v}(F)}=\LieAlg{g}_{\CC,\ell(F)}$ in Lemma 
\ref{lemma-pde}. Hence, $\tilde{H}(\phi)(\ell(F))=\ell(\Phi(F))$.
\hide{
The class $\kappa(F)$ is $\bar{\LieAlg{g}}$-invariant, by Proposition \ref{prop-kappa-class-remains-of-Hodge-type}.
Hence, so is $\kappa(F)\sqrt{td_X}$. Set $n=\dim_\CC(X)/2$ and let $h$ be the grading operator acting on 
$H^k(X,\RationalNumbers)$ by $k-2n$. The operator $h$ belongs to $\LieAlg{g}$  and commutes with $\bar{\LieAlg{g}}$ and hence belongs to $\LieAlg{so}(U_\RationalNumbers)$, yet it does not annihilate $\kappa(F)\sqrt{td_X}$, since the graded summand in $H^0(X,\RationalNumbers)$ of the latter is $r$ and so it does not vanish,  by assumption. Hence, 
the  class $\kappa(F)\sqrt{td_X}$ is not invariant under the Lie algebra $\LieAlg{so}(U_\RationalNumbers)$.

The Lie algebra $\LieAlg{g}_{\kappa(F)\sqrt{td_X}}$ is strictly bigger than $\bar{\LieAlg{g}}$, since $\Sigma(F)$ has co-dimension one in $HT^2(X)$ and its image 
n $\End(H^*(X,\RationalNumbers))$ is contained in the complexification of $\LieAlg{g}_{v(F)}$, by 
Lemma \ref{lem-Sigma-F-is-contained-in-stabilizer-of-Mukai-vector}.
(??? need the following to relate $\Sigma(F)$, which is contained in the stabilizer of $v(F)=ch(F)\sqrt{td_X}$, to its translate by $\exp(-c_1(F)/r)$ which is contained in the stabilizer of $\kappa(F)\sqrt{td_X}$: ???) Note that given an object $F\in D^b(X)$, we have $H^1(TX)\subset\Sigma(F)$, if and only if the subalgebra 
$\bar{\LieAlg{g}}$  
is contained in $\LieAlg{g}_{v(F)}$, by Lemma \ref{lemma-alpha-remains-of-Hodge-type}. 
$\Pic(X)$ embeds in $\Aut(D^b(X))$ and hence in
$DMon(X)$. $\Aut(D^b(X))$ acts on $HT^2(X)$, hence so does $\Pic(X)$, as well as $\Pic(X)_\RationalNumbers$. 
The map $Ob(D^b(X)) \ni F\mapsto \Sigma(F)\subset HT^2(X)$ is $\Aut(D^b(X))$ equivariant. Hence, if the rank $r$ of $F$ is non-zero, then 
$H^1(X,TX)\subset \det(F)^{-1/r}\otimes \Sigma(F)$ 
if and only if  $\bar{\LieAlg{g}}$ is contained in
$\LieAlg{g}_{\exp(-c_1(F)/r)v(F)}$.

There are two types of Lie subalgebras $\LieAlg{a}$ satisfying $\bar{\LieAlg{g}} \subset \LieAlg{a}\subset \LieAlg{g}$, with both inclusions strict, one is $\LieAlg{so}(U_\RationalNumbers)\times \bar{\LieAlg{g}}$, and the other is the stabilizer 
$\LieAlg{g}_\ell$ of a line $\ell$ in $U_\RationalNumbers$. Indeed, the quotient 
$\LieAlg{g}/\bar{\LieAlg{g}}$ is isomorphic to the direct sum of $\LieAlg{so}(U_\RationalNumbers)$ and $U_\RationalNumbers\otimes H^2(X,\RationalNumbers)$ as a $\bar{\LieAlg{g}}\times \LieAlg{so}(U_\RationalNumbers)$-module, and both direct summands are irreducible. 
If $\LieAlg{a}$ contains $\LieAlg{so}(U_\RationalNumbers)$, then $\LieAlg{a}$
must be $\bar{\LieAlg{g}}\times \LieAlg{so}(U_\RationalNumbers)$. 
If $\LieAlg{a}$ contains the $\bar{\LieAlg{g}}\times \LieAlg{so}(U_\RationalNumbers)$-submodule
$U_\RationalNumbers\otimes H^2(X,\RationalNumbers)$ of $\LieAlg{g}$, then it contains $\LieAlg{so}(U_\RationalNumbers)$, as the latter is contained\footnote{
Choose and orthogonal basis $\{e_i\}$ of $\tilde{H}(X,\RationalNumbers)$ with $e_1, e_2\in U_\QQ$. Let
$\{e_i^*\}$ be the dual basis. Set $A=(e_1,e_1)e_1^*\otimes e_3-(e_3,e_3)e_3^*\otimes e_1$, 
$B=(e_2,e_2)e_2^*\otimes e_3-(e_3,e_3)e_3^*\otimes e_2$, and
$C=(e_1,e_1)e_1^*\otimes e_2-(e_2,e_2)e_2^*\otimes e_1$. Then $A$ and $B$ belong to the submodule of $\LieAlg{g}$ isomorphic to 
$U_\RationalNumbers\otimes H^2(X,\RationalNumbers)$, $C$ spans $\LieAlg{so}(U_\RationalNumbers)$, and 
$[A,B]=(e_3,e_3)C$.
}
 in the Lie bracket of the former with itself. Hence, if  $\LieAlg{a}$ does not contain $\LieAlg{so}(U_\RationalNumbers)$, then  
the quotient $\LieAlg{a}/\LieAlg{g}$ is a $\bar{\LieAlg{g}}$-module, which is not a $\LieAlg{so}(U_\RationalNumbers)$-module, hence equal to $\ell'\otimes H^2(X,\RationalNumbers)$, for a line $\ell'$ in $U_\RationalNumbers$. We have seen that $\LieAlg{g}_{\kappa(F)\sqrt{td_X}}$ does not contain $\LieAlg{so}(U_\RationalNumbers)$. 
We conclude that $\LieAlg{g}_{\kappa(F)\sqrt{td_X}}$ is equal to $\LieAlg{g}_{\ell'}$, for some line $\ell'$ in $U_\RationalNumbers$. Hence, $\LieAlg{g}_{v(F)}=\LieAlg{g}_{\ell(F)}$, where $\ell(F)=\exp(c_1(F)/r)\ell'$.

The equality $\ell(F)\otimes_\QQ\CC=\ell(\Sigma(F))$ follows from 
Lemma \ref{lemma-line-in-Mukai-lattice-associated-to-hyperplane-in-HT-2} 
and the fact that $m(\Sigma(F))$ is contained in the complexification of $\LieAlg{g}_{v(F)}$ and hence of $\LieAlg{g}_{\ell(F)}$.

(\ref{prop-item-isotropic}) 
The orthogonal projection of $v(F)$ to
$SH^*(X,\RationalNumbers)$ is non-zero, since $r\neq 0$. 
If $v(F)$ is $\LieAlg{g}_\ell$-invariant, then so is 
its projection to $SH^*(X,\RationalNumbers)$. 
It remains to show that the projection of $\ell^n$ into the kernel of 
$\Delta:\Sym^n(\tilde{H}(X,\RationalNumbers))\rightarrow \Sym^{n-2}(\tilde{H}(X,\RationalNumbers))$
surjects onto 
the $\LieAlg{g}_\ell$-invariant subspace of $\ker(\Delta)$. 

Assume first that $\ell$ is non-isotropic. 
We have the direct sum decomposition
\[
\Sym^{n}(\tilde{H}(X,\RationalNumbers))\cong 
\oplus_{i=0}^n[\ell^{n-i}\otimes\Sym^{i}(\ell^\perp)].
\]
Choose an element $\lambda\in \ell$. 
Set $\tilde{q}_{\ell^\perp}:=\tilde{q}-\frac{1}{(\lambda,\lambda)\dim\tilde{H}(X,\RationalNumbers)}\lambda^2$.
Then the trivial $\LieAlg{g}_\ell$-submodule $\Sym^{i}(\ell^\perp)^{\LieAlg{g}_\ell}$ in $\Sym^{i}(\ell^\perp)$
vanishes, for odd $i$, and it is spanned by $\tilde{q}_{\ell^\perp}^{i/2}$ for even $i$. Hence, the trivial $\LieAlg{g}_\ell$-submodule $\Sym^{n}(\tilde{H}(X,\RationalNumbers))^{\LieAlg{g}_\ell}$ in 
$\Sym^{n}(\tilde{H}(X,\RationalNumbers))$ is spanned by $\{\lambda^{n-2i}\tilde{q}_{\ell^\perp}^i\}_{i=0}^{\lfloor\frac{n}{2}\rfloor}$.
Now $\Delta$ is a surjective homomorphism of $\LieAlg{g}$-modules, hence also of $\LieAlg{g}_\ell$-modules.
We see that 
$\Delta:\Sym^{n}(\tilde{H}(X,\RationalNumbers))^{\LieAlg{g}_\ell}\rightarrow \Sym^{n-2}(\tilde{H}(X,\RationalNumbers))^{\LieAlg{g}_\ell}$ is surjective and so its kernel is one-dimensional.
The  direct sum decomposition $\Sym^{n}(\tilde{H}(X,\RationalNumbers))=\ker(\Delta)\oplus \tilde{q}\Sym^{n-2}(\tilde{H}(X,\RationalNumbers))$ is $\LieAlg{g}$-equivariant and $\ell^n$ is not contained in the second direct summand, hence it projects to the first non-trivially and so surjectively onto $\ker(\Delta)^{\LieAlg{g}_\ell}.$

If $\ell$ is isotropic then $\ell^n$ is contained in $\ker(\Delta)$ is thus  equal to its projection into the latter. The two projections, of $v(F)$ and of $\ell^n$ do not vanish and the equality of the latter to the image of the former then follows by continuity from the non-isotropic case.
}
\end{proof}

\begin{rem}
Assume that the rank of $F$ is non-zero and $X$ is of $K3^{[2]}$ or $K3^{[3]}$-type.
Then $v(F)$ is determined by its rank and the line $\ell(F)$, by  parts 
(\ref{prop-item-isotropic}) and (\ref{prop-item-spanning-Mukai-vector}) of 
Theorem \ref{thm-Mukai-vector}, since $\hat{v}(F)=v(F)$ in this case. 
If $X$ is of $K3^{[2]}$-type the equality $\hat{v}(F)=v(F)$ follows from the equality $H^*(X,\QQ)=SH^*(X,\QQ)$.
When $X$ is of $K3^{[3]}$-type, $H^*(X,\QQ)$ decomposes as the direct sum of two irreducible $\LieAlg{g}$-module
$H^*(X,\QQ)\cong SH^*(X,\QQ)\oplus \wedge^2\tilde{H}(X,\QQ)$, by \cite[Cor. 3.2]{GKLR}.
The projection of $v(F)$ to $\wedge^2\tilde{H}(X,\QQ)$ must vanish, since the $\LieAlg{g}_\ell$-invariant subspace
of the latter vanishes. Hence, $v(F)$ is contained in $SH^*(X,\QQ)$ and again we get the equality $\hat{v}(F)=v(F)$.
\end{rem}

Let $\tilde{D}:\tilde{H}(X,\QQ)\rightarrow \tilde{H}(X,\QQ)$ be given by $\tilde{D}(r\alpha+\lambda+s\beta)=r\alpha-\lambda+s\beta$.
Let $D:H^{ev}(X,\QQ)\rightarrow H^{ev}(X,\QQ)$ be the automorphism of the even cohomology ring acting on
$H^{2i}(X,\QQ)$ by $(-1)^i$.
If $\Sigma(F)$ has codimension one in $HT^2(X)$, then so does the subspace $\Sigma(F^\vee)$ of the dual object $F^\vee$.

\begin{lem}
\label{lemma-Mukai-line-of-F-dual}
$\ell(F^\vee)=\tilde{D}(\ell(F))$.
\end{lem}

\begin{proof}
Let $\Pi:\Sym^n(\tilde{H}(X,\QQ))\rightarrow \ker(\Delta)$ be the projection (\ref{eq-projection-to-Im-Psi}).
Then $\Pi$ commutes with $\Sym^n(\tilde{D})$. The projection (\ref{eq-projection-to-verbitsky-component}) commutes with $D$.
The isometry $\Psi:SH^*(X,\QQ)\rightarrow \ker(\Delta)$, given in (\ref{eq-Psi}), 
satisfies $\Psi(D(\theta))=\Sym^n(\tilde{D})(\Psi(\theta))$, for all $\theta\in SH^*(X,\QQ)$.
The $n$-th power map satisfies $\Sym^n(\tilde{D})(\lambda^n)=(\tilde{D}(\lambda))^n$, for all
$\lambda\in \tilde{H}(X,\QQ)$. Let $\lambda$ be a non-zero vector in $\ell(F)$ satisfying
$\Psi(\hat{v}(F))=\Pi(\lambda^n)$. Then
\[
\Psi(\hat{v}(F^\vee))=\Psi(D(\hat{v}(F)))=\Sym^n(\tilde{D})(\Psi(\hat{v}(F)))=\Sym^n(\tilde{D})(\Pi(\lambda^n))=
\Pi(\tilde{D}(\lambda)^n).
\]
Hence, $\tilde{D}(\lambda)$ spans $\ell(F^\vee)$.
\end{proof}

\begin{example}
\label{example-Mukai-line-of-sky-scraper-sheaf}
Let $X$ be a $2n$-dimensional irreducible holomorphic symplectic manifold.
Given a closed point $z\in X$ we have $\ell(\StructureSheaf{z})=\span_\QQ\{\beta\}.$
Indeed, $v(\StructureSheaf{z})=[pt]$ belongs to $SH^*(X,\QQ)$ and $\Psi(1)=\beta^n/c_X$, by (\ref{eq-Psi-pt}).
\end{example}

We will need the following 
multiplication table in the cohomology of $X$ of $K3^{[2]}$type, where $\alpha$ is a class in $H^2(X)$ and $c_2$ is $c_2(TX)$. 
\begin{equation}
\label{eq-multiplication-table}
\left[
\begin{array}{cccccc}
1 & \alpha & \alpha^2 & c_2 & c_2\alpha & [\mbox{pt}]
\\
\alpha & \alpha^2 & \frac{(\alpha,\alpha)}{10}c_2\alpha & 
c_2\alpha & 30(\alpha,\alpha)[\mbox{pt}] & 0 
\\
\alpha^2 & \frac{(\alpha,\alpha)}{10}c_2\alpha & 3(\alpha,\alpha)^2[\mbox{pt}]& 
30(\alpha,\alpha)[\mbox{pt}] & 0 & 0 
\\
c_2 & c_2\alpha & 30(\alpha,\alpha)[\mbox{pt}] & 828[\mbox{pt}] & 0 & 0
\\
c_2\alpha & 30(\alpha,\alpha)[\mbox{pt}] & 0 & 0 & 0 & 0
\\
\mbox{[pt]} & 0 & 0 & 0 & 0 & 0
\end{array}
\right]
\end{equation}

Given classes $\alpha$, $\beta$, $\gamma$, $\delta$ in $H^2(X)$, we have
\[
\alpha\beta\gamma\delta=[(\alpha,\beta)(\gamma,\delta)+(\alpha,\gamma)(\beta,\delta)+(\alpha,\delta)(\beta,\gamma)][pt].
\]

\begin{example}
\label{example-Mukai-line-of-structure-sheaf}
Assume $X$ is of $K3^{[2]}$-type. We claim that 
$
\ell(\StructureSheaf{X})=\span_\QQ\{4\alpha+5\beta\}.
$
We have $v(\StructureSheaf{X})=\sqrt{td_X}=1+\frac{1}{24}c_2(X)+\frac{25}{32}[pt].$
Set 
\begin{equation}
\label{eq-b-X}
b_X:=\frac{1}{23}\sum_1^{23}y_i^2/(y_i,y_i), 
\end{equation}
where $\{y_i\}$ is an orthogonal basis of $H^2(X,\QQ)$. 
We have $\int_Xy_i^4=3(y_i,y_i)^2$ and for  $i\neq j$ we have $\int_X y_i^2y_j^2=(y_i,y_i)(y_j,y_j)$, due to the vanishing of $(y_i,y_j)$. Hence,
\begin{eqnarray*}
\int_Xb_X^2=\frac{1}{23^2}\left(\sum_{i<j}\frac{2}{(y_i,y_i)(y_j,y_j)}\int_Xy_i^2y_j^2+
\sum_{i=1}^{23} \frac{1}{(y_i,y_i)^2}\int_X y_i^4
\right)=\frac{25}{23},
\\
\int_Xb_Xy_i^2=\frac{1}{23(y_i,y_i)}\int_Xy_i^4+\frac{1}{23}\sum_{j=1,j\neq i}^{23}\frac{1}{(y_j,y_j)}\int_Xy_i^2y_j^2=
\frac{25}{23}(y_i,y_i)=\frac{5}{23\cdot 6}\int_Xc_2(X)y_i^2.
\end{eqnarray*}
It follows that 
\begin{equation}
\label{eq-b-X-in-terms-of-c-2}
c_2(X)=\frac{23\cdot 6}{5}b_X,
\end{equation} 
as both span the one-dimensional space of $O(H^2(X,\QQ))$-invariant classes.
We have
\begin{equation}
\label{eq-values-of-Psi}
\Psi(1)=\frac{\alpha^2}{2}, \
\Psi([pt])=\beta^2, \ \mbox{and} \ 
\Psi(b_X)=\tilde{b}_X+\alpha\beta,
\end{equation}
where $\tilde{b}_X$ denotes the element of $\Sym^2(\tilde{H}(X,\QQ))$ given by the right hand side of Equation (\ref{eq-b-X}). 
Hence,
\begin{eqnarray}
\label{eq-Psi-of-c-2}
\Psi(c_2(TX)) & = & 30(\tilde{q}+\alpha\beta),
\\
\label{eq-Psi-of-sqrt-td-class}
\Psi(\sqrt{td_X})&=&\frac{1}{2}\left[\alpha^2+\frac{5}{2}(\tilde{q}+\alpha\beta)+\frac{25}{16}\beta^2\right]=
\frac{1}{2}\left[(\alpha+\frac{5}{4}\beta)^2+\frac{5}{2}\tilde{q}\right],
\end{eqnarray}
where $\tilde{q}$ is defined in Example \ref{example-projection-to-Im-Psi}  and we used the equality 
\begin{equation}
\label{eq-tilde-q-in-terms-of-tilde-b-X}
\tilde{q}=\frac{1}{25}[23\tilde{b}_X-2\alpha\beta].
\end{equation} 
The right hand side of the latter displayed equation  is the projection of $\frac{1}{2}(\alpha+\frac{5}{4}\beta)^2$ from $\Sym^2(\tilde{H}(X,\QQ))$
to the image of $\Psi$, by Example \ref{example-projection-to-Im-Psi}. So
$
\ell(\StructureSheaf{X})=\span_\QQ\{4\alpha+5\beta\}
$ as claimed. 
\end{example}

More generally, we have the following. 

\begin{lem}
\label{lemma-LLV-line-of-structure-sheaf-k3-type}
Let $X$ be a $2n$-dimensional irreducible holomorphic symplectic manifold. There exists a rational number $t$, such that 
$t^n=\frac{n!}{c_X}\int_X\sqrt{td_X}$ and
$
\ell(\StructureSheaf{X})=\span_\QQ\{\alpha+t\beta\}.
$
If  $X$ is of $K3^{[n]}$-type, $n\geq 2$, then $\ell(\StructureSheaf{X})=\span_\QQ\{4\alpha+(n+3)\beta\}$.
If  $X$ is deformation equivalent to a $2n$-dimensional generalized Kummer, $n\geq 2$, then $\ell(\StructureSheaf{X})=\span_\QQ\{4\alpha+(n+1)\beta\}$.
\end{lem}

\begin{proof}
Let $\tau_{2i}$ be the graded summand of $\sqrt{td_X}$ in $H^{4i}(X,\QQ)$. So $\tau_0=1$ and $\tau_{2n}=s[pt]$ for $s:=\int_X\sqrt{td_X}$.
It is known that $s>0$, by \cite[Sec. 6]{hitchin-sawon}.
If $X$ is of $K3^{[n]}$-type, then $s=\frac{(n+3)^n}{4^nn!}$, by \cite[Prop. 19]{sawon-thesis}. 
If  $X$ is deformation equivalent to a $2n$-dimensional generalized Kummer, then
$s=\frac{(n+1)^{n+1}}{4^n(n!)}$,
 by \cite[Prop. 21]{sawon-thesis}.
Choose a basis of $\widetilde{H}(X,\QQ)$ consisting of $\alpha$, $\beta$, and a basis of $H^2(X,\QQ)$.
We get a well defined notion of the degree $\deg_\alpha(f)$ of an element of $\Sym^n\widetilde{H}(X,\QQ)$ in $\alpha$.  Define 
$\deg_\beta(f)$ similarly. 
Let $\bar{\tau}_{2i}$ be the projection of $\tau_{2i}$ to $SH^*(X)$. Note that $\bar{\tau}_{0}=\tau_0=1$ and
$\bar{\tau}_{2n}=\tau_{2n}=s[pt]$. 
Then $\Psi(\bar{\tau}_0)=\alpha^n/n!$,
$\Psi(\bar{\tau}_{2n})=s\beta^n/c_X$, and for $1\leq i\leq n-1$ we have $\deg_{\alpha}(\Psi(\bar{\tau}_{2i}))<n$
and $\deg_{\beta}(\Psi(\bar{\tau}_{2i}))<n$, by definition of $\Psi$. We know that 
$\ell(\StructureSheaf{X})=\span\{\alpha+t\beta\}$, for some $t\in \QQ$, by Theorem \ref{thm-Mukai-vector}.
Furthermore, the projection of $(\alpha+t\beta)^n$ to $\ker(\Delta)$ is a rational scalar multiple of $\Psi(\sum_{i=0}^n\bar{\tau}_{2i})$,
by Theorem \ref{thm-Mukai-vector}. Let $V(d)$ be the kernel of $\Delta:\Sym^d(\widetilde{H}(X,\QQ))\rightarrow \Sym^{d-2}(\widetilde{H}(X,\QQ))$. 
We have the decomposition into irreducible $\LieAlg{so}(\widetilde{H}(X,\QQ))$-sub-representations
\[
\Sym^n(\widetilde{H}(X,\QQ))=\oplus_{j=0}^{\lfloor n/2\rfloor}\tilde{q}^jV(n-2j),
\]
where $\tilde{q}$ is defined in Example \ref{example-projection-to-Im-Psi} (see the proof of \cite[Prop. 2.14]{looijenga-lunts}).
Now $\deg_\alpha(\tilde{q})=\deg_\beta(\tilde{q})=1$, and so 
only the direct summand $V(n)$ in the above decomposition of $\Sym^n(\widetilde{H}(X,\QQ))$ contains elements $f$ with
$\deg_\alpha(f)=n$ and similarly for $\deg_\beta$.
Thus the coefficient of $\alpha^n$ 
 in  the projection of $\frac{1}{n!}(\alpha+t\beta)^n$ to $\ker(\Delta)=V(n)$ is $\frac{1}{n!}$ and the coefficient of $\beta^n$ is
 $\frac{t^n}{n!}$.
Hence, $t^n=\frac{s(n!)}{c_X}$. We see that $t^n=\left(\frac{n+3}{4}\right)^n$ in the $K3^{[n]}$-type and $t^n=\left(\frac{n+1}{4}\right)^n$ in the generalized Kummer case. 
If $n$ is odd, we are done. If $n$ is even it remains to prove that $t>0$. 
The inequality $t>0$  is 
equivalent to $(\alpha+t\beta,\alpha+t\beta)<0$.
The inequality $(\alpha+t\beta,\alpha+t\beta)<0$ follows from Remark \ref{rem-the-LLV-line-of-the-structure-sheaf-for-n-geq-2} for $X$ of $K3^{[n]}$-type. Below we will prove the inequality $t>0$ directly in the $K3^{[n]}$ and generalized Kummer deformation types.

The coefficient of $\alpha^{n-1}\beta$ 
 in  the projection of $(\alpha+t\beta)^n$ to $\ker(\Delta)=V(n)$ has the same sign as $t$.
 Indeed, we have $\Delta(\tilde{b}_X)=\Delta(\tilde{q})=1$, 
 \begin{eqnarray*}
 \Delta\left(\alpha^{n-1}\beta+(n-1)\alpha^{n-2}\tilde{b}_X
 \right)&=& (n-1)(\alpha,\beta)\alpha^{n-2}+(n-1)\alpha^{n-2}\Delta(\tilde{b}_X)=0,
 \\
 \tilde{b}_X&=& \left(\frac{b_2(X)+2}{b_2(X)}\right)\tilde{q}+\left(\frac{2}{b_2(X)}\right)\alpha\beta,
 \\
 \alpha^{n-1}\beta+(n-1)\alpha^{n-2}\tilde{b}_X&=&
\left( \frac{(n-1)(b_2(X)+2)}{b_2(X)}\right)\alpha^{n-2}\tilde{q}+\left(\frac{b_2(X)+2n-2}{b_2(X)}\right)\alpha^{n-1}\beta.
 \end{eqnarray*}
 Multiplying both sides of the last equation by 
 $\frac{b_2(X)}{b_2(X)+2n-2}$ we see that the projection of $\alpha^{n-1}\beta$ to $\ker(\Delta)$ is
 \[
 \left(\frac{b_2(X)}{b_2(X)+2n-2} \right)
 \left[
 \alpha^{n-1}\beta+(n-1)\alpha^{n-2}\tilde{b}_X
 \right],
 \]
 and the coefficient of $\alpha^{n-1}\beta$ is indeed positive.
 
The only term in $\sum_{i=0}^n\Psi(\bar{\tau}_{2i})$ that contributes to the coefficient of $\alpha^{n-1}\beta$ is $\Psi(\bar{\tau}_2)$, by definition of $\Psi$. It suffices to show that the coefficient of $\alpha^{n-1}\beta$ in $\Psi(\bar{\tau}_2)$ is positive.
Now, $\tau_2=c_2(TX)/24$. The projection of $c_2(TX)$ to $SH^4(X)$ is monodromy invariant, hence a  multiple of $b_X$, and we claim that it is a positive multiple.
Indeed, given a line bundle $L$ on $X$, the coefficient of $(c_1(L),c_1(L))^{n-1}$ in the formula for
$\int_X ch(L)td_X$ depends only on $\int_X \frac{c_1(L)^{2n-2}}{(2n-2)!} \frac{c_2(TX)}{12}$ and it is 
positive, by \cite[Examples 7 and 8]{huybrechts-survey}, as is the coefficient of $(c_1(L),c_1(L))^{n-1}$ in the formula for
$\int_X b_Xc_1(L)^{2n-2}$, by \cite[Eq. (2.1.4)]{ogrady-survey}. 
We conclude that $\Psi(\bar{\tau}_2)$ is a positive scalar multiple of $\Psi(b_X)$.
Finally, the coefficient of $\alpha^{n-1}\beta$ in $\Psi(b_X)=\frac{1}{(n-2)!}\alpha^{n-2}[\tilde{b}_X+\alpha\beta]$ is positive. 
\end{proof}

\hide{
\begin{rem}
 The proof of Lemma \ref{lemma-LLV-line-of-structure-sheaf-k3-type} shows that the positive $n$-th root $\sqrt[n]{\frac{s(n!)}{c_X}}$ of $\frac{s(n!)}{c_X}$ is rational  
and $\ell(\StructureSheaf{X})=\span\left\{\alpha+\epsilon\sqrt[n]{\frac{s(n!)}{c_X}}\beta\right\}$, with $\epsilon=1$, if $n$ is odd, and $\epsilon=\pm 1$, if $n$ is even.
\hide{
In \cite{taelman} the BBF form $b(\bullet,\bullet)$ on a $2d$-dimensional irreducible holomorphic symplectic manifold $X$ is normalized as
 \[
 \int_X\lambda^{2d}=\frac{(2d)!}{2^dd!}b(\lambda,\lambda)^d,
 \]
for $\lambda\in H^2(X,\QQ)$. 
If $X$ is a $2d$-dimensional irreducible holomorphic symplectic manifold deformation equivalent to a generalized Kummer, then
the integral and primitive BBF form $b_0(\bullet,\bullet)$ satisfies
\[
\int_X\lambda^{2d}=(d+1)(2d-1)!!b_0(\lambda,\lambda)^d,
\]
for $\lambda\in H^2(X,\QQ)$, by O'Grady's arxiv:1805.12075.v1 formula (2.2.5), where the double factorial is defined in formula 
(1.4.1) of that paper.  When $d=2$, then 
$(d+1)(2d-1)!!=9$, while $\frac{(2d)!}{2^dd!}=3$. Hence
$b(\lambda,\lambda)^2=3b_0(\lambda,\lambda)^2$. Thus, pairing $(\bullet,\bullet)$ has values in $\sqrt{3}\QQ$.
The Chern numbers of a fourfold $X$ of generalized Kummer type 
are $\int_Xc_2(X)^2=756$, and $\int_Xc_4(X)=108$, by Appendix A in \cite{nieper}. Hence, $td_X=1+\frac{1}{12}c_2(X)+3[pt]$ and
\[
\sqrt{td_X}=1+\frac{1}{24}c_2+\frac{27}{32}[pt].
\]
}
If $X$ is $2n$-dimensional of  generalized Kummer deformation type, then  
$c_X=n+1$ and
\[
 \int_X\sqrt{td_X}=\frac{(n+1)^{n+1}}{4^n(n!)},
 \]
 by \cite[Prop. 21]{sawon-thesis}.
We see that $\frac{s(n!)}{c_X}=\left(\frac{n+1}{4}\right)^n$ and 
$\ell(\StructureSheaf{X})=\span\{\alpha\pm \left(\frac{n+1}{4}\right)\beta\}$.
\end{rem}
}

The computations of  Example \ref{example-Mukai-line-of-structure-sheaf} generalize to yield  the following statement.

\begin{lem}
Let $X$ be of $K3^{[2]}$-type and $F$ an object of $D^b(X)$ of non-zero rank.
\begin{enumerate}
\item
\label{item-LLV-line-of-an-objects-which-deforms-in-co-dimension-1}
If $F$ deforms in co-dimension $1$, then $\kappa_3(F)=0$ and
$
\kappa(F)=x+yc_2(TX)+z[pt],
$
where $x, y, z\in \RationalNumbers$ and 
\begin{equation}
\label{eq-necessary-condition--on-kappa-class-for-F-to-deform-in-co-dimension-1}
450y^2+3xy-xz=0. 
\end{equation}
In this case,
\begin{equation}
\label{eq-LLV-line-of-an-objects-which-deforms-in-co-dimension-1}
\ell(F)=\span\left\{x^2\alpha+xc_1(F)+\left(\frac{5x^2}{4}+30xy+\frac{(c_1(F),c_1(F))}{2}\right)\beta\right\}.
\end{equation}
\item
\label{item-kappa-class-of-an-object-with-Mikai-vector-deforming-in-co-dimension-1}
Conversely, if $\kappa_3(F)=0$, 
$
\kappa(F)=x+yc_2(TX)+z[pt],
$
and $450y^2+3xy-xz=0$, then 
$F$  has a rank $1$ cohomological obstruction map.
\item
\label{item-tensor-product-does-not-deform-in-co-dimension-1}
If $F_1, F_2\in D^b(X)$ are of non-zero rank and $F_1$, $F_2$, and $F_1\otimes F_2$, all deform in co-dimension $1$,
then $\kappa(F_i)$ is a scalar multiple of $\kappa(\StructureSheaf{X})$, for $i=1$ or $i=2$.
\end{enumerate}
\end{lem}

\begin{proof}
\ref{item-LLV-line-of-an-objects-which-deforms-in-co-dimension-1})
The class $\kappa(F)$ remains of Hodge type under every K\"{a}hler deformation of $X$, by Proposition \ref{prop-kappa-class-remains-of-Hodge-type}. Hence, $\kappa_3(F)=0$ and $\kappa_2(F)$ is a scalar multiple of $c_2(TX)$. 
We have
\begin{eqnarray*}
\kappa(F)\sqrt{td_X}&=&(x+yc_2(X)+z[pt])\left(1+\frac{1}{24}c_2(X)+\frac{25}{32}[pt]\right)
\\&=& x+(x/24+y)c_2(X)+(z+25x/32+69y/2)[pt].
\end{eqnarray*}
Set $t=x/24+y$ and $s=z+25x/32+69y/2$.
We have
\begin{eqnarray*}
\Psi(x+tc_2(X)+s[pt])&=&
x\alpha^2/2+30t(\tilde{q}+\alpha\beta)+s\beta^2
\\
&=& 30t\tilde{q}+\frac{x}{2}\left(\alpha^2+\frac{60t}{x}\alpha\beta+\frac{2s}{x}\beta^2\right).
\end{eqnarray*}
The latter belongs to the projection via (\ref{eq-projection-to-Im-Psi}) of the  square of a line, if and only if $450t^2=sx$, in which case it belongs to the projection of the square of
$\span\{\alpha+\frac{30t}{x}\beta\}=\span\{\alpha+\left(\frac{5}{4}+\frac{30y}{x}\right)\beta\}$, by Example \ref{example-projection-to-Im-Psi}. This is the case, if and only if 
$450y^2+3xy-xz=0$. The equality $v(F)=\exp(e_{c_1(F)/x})\kappa(F)\sqrt{td_X}$ implies that $\ell(F)$ is spanned by 
$\exp(e_{c_1(F)/x})\left(\alpha+\left(\frac{5}{4}+\frac{30y}{x}\right)\beta\right)$, by the $SO(\widetilde{H}(X,\QQ))$-equivariance of $\Psi$.
Finally observe that 
\[
\exp(e_{c_1(F)/x})\left(\alpha+\left(\frac{5}{4}+\frac{30y}{x}\right)\beta\right)=
\alpha+c_1(F)/x+\left(\frac{5}{4}+30y/x+\frac{(c_1(F),c_1(F))}{2x^2}\right)\beta.
\]

\ref{item-kappa-class-of-an-object-with-Mikai-vector-deforming-in-co-dimension-1})
The assumptions imply that $\Psi(\kappa(F)\sqrt{td_X})$ is invariant under the stabilizer of the line in Equation (\ref{eq-LLV-line-of-an-objects-which-deforms-in-co-dimension-1}) in the LLV algebra $\LieAlg{g}:=\LieAlg{so}(\tilde{H}(X,\QQ))$. Hence, 
$\kappa(F)\sqrt{td_X}$ is invariant under this subalgebra, by the $\LieAlg{g}$-equivariance of $\Psi$. 

\ref{item-tensor-product-does-not-deform-in-co-dimension-1})
Set $\kappa(F_i)=x_i+y_ic_2(X)+z_i[pt]$. Then
\[
\kappa(F_1\otimes F_2)=
x_1x_2+[x_1y_2+x_2y_1]c_2(X)+
[x_1z_2+x_2z_1+828y_1y_2][pt].
\]
Set $\bar{y}_i:=y_i/x_i$ and $\bar{z_i}=z_i/x_i$. 
Equation (\ref{eq-necessary-condition--on-kappa-class-for-F-to-deform-in-co-dimension-1}) for $F_1\otimes F_2$ translates to the vanishing of
\begin{eqnarray*}
450(\bar{y}_1+\bar{y_2})^2+3(\bar{y}_1+\bar{y}_2)-(\bar{z}_1+\bar{z}_2+828\bar{y}_1\bar{y}_2)&=&
\\
(450\bar{y}_1^2+3\bar{y}_1-\bar{z}_1)+(450\bar{y}_2^2+3\bar{y}_2-\bar{z}_2)+72\bar{y}_1\bar{y}_2&=&72\bar{y}_1\bar{y}_2.
\end{eqnarray*}
Hence, $y_i=0$, for $i=1$ or $i=2$. Equation (\ref{eq-necessary-condition--on-kappa-class-for-F-to-deform-in-co-dimension-1}) and the vanishing of $y_i$ imply that $z_i=0$ as well.
\end{proof}

\begin{rem}
Theorem \ref{thm-Mukai-vector} imposes the following cohomological constraint on the Todd class of a $2n$ dimensional irreducible holomorphic symplectic manifold $X$. 
Denote by $\hat{c}_{2k}$ 
the projection of $c_{2k}(X)$ to the primitive subspace of $H^{4k}(X,\QQ)$ defined in \cite[Cor. 1.13]{looijenga-lunts}.
Then $\hat{c}_{2k}$ generates an irreducible subrepresentation $\LieAlg{g}(\hat{c}_{2k})$ of the LLV algebra $\LieAlg{g}$  isomorphic to the subspace $V(n\!-\!2k)$ of 
$\Sym^{n-2k}\tilde{H}(X,\QQ)$ spanned by powers of isotropic classes, by \cite[Lemma 2.8]{looijenga-lunts}. 
The direct sum  $\oplus_{k=0}^{\lfloor n/2\rfloor}\LieAlg{g}(\hat{c}_{2k})$ is the minimal $\LieAlg{g}$ invariant subspace containing $\sqrt{td_X}$. Theorem \ref{thm-Mukai-vector} implies that 
the projection of $\sqrt{td_X}$ to $\LieAlg{g}(\hat{c}_{2k})$ spans the image of the power $\ell(\StructureSheaf{X})^{n-2k}$ of the same line $\ell(\StructureSheaf{X})$ under the unique, up to scalar, isomorphism $\LieAlg{g}(\hat{c}_{2k})\cong V(n\!-\!2k)$. Note that when $X$ is of $K3^{[n]}$-type and $n\geq 4$, then $\hat{c}_{2k}$ is non-trivial
for $1\leq k\leq n/2$, by  \cite[Theorem 1.10]{markman-constraints}.
\end{rem}

Theorem \ref{thm-Mukai-vector} lists topological consequences about the Mukai vector of an object $F$ 
deforming in co-dimension $1$ (Definition \ref{def-deforms-in-co-dimension-one}).
We emphasize next the evident fact that these topological consequences remain valid under deformations of the pair $(X,F)$. For simplicity we restrict to the case of locally free sheaves.
Let $\pi:\X\rightarrow B$ be a smooth and proper morphism over a connected analytic space $B$, all of which fibers are irreducible holomorphic symplectic manifolds. Let  $\theta$ be a class in $H^2(\X,\StructureSheaf{\X}^*)$, in the analytic topology, 
and let $\F$ be a locally free $\theta$-twisted sheaf of rank $r$. Denote by $X_b$ the fiber of $\pi$ over a point $b$ of $B$ and by $\theta_b\in H^2(X_b,\StructureSheaf{X_b}^*)$ and $F_b$ the restrictions of $\theta$ and $\F$ to $X_b$. The rational LLV lattices $\tilde{H}(X_b,\QQ)$ then form a local system $\underline{\QQ\alpha}\oplus R^2\pi_*\QQ\oplus \underline{\QQ\beta}$, where
$\underline{\QQ\alpha}\oplus\underline{\QQ\beta}$ is a trivial local subsystem. 
Denote by $\LieAlg{g}(X_b)$ the LLV Lie algebra of $X_b$.
Set $\bar{v}(F_b):=\kappa(F_b)\sqrt{td_{X_b}}$ and let $\hat{\bar{v}}(F_b)$ be its projection to $SH^*(X_b,\QQ)$.

\begin{cor}
\label{cor-constant-LLV-line-of-a-family-of-twisted-sheaves}
Assume that for some point $b_0$ of $B$ the class $\theta_{b_0}$ is trivial\footnote{Strictly speaking, $\F$ is a twisted sheaf with respect to a co-cycle $\tilde{\theta}$ representing the class $\theta$ and $F_{b_0}$ is a twisted sheaf with respect to the restriction $\tilde{\theta}_{b_0}$ of $\tilde{\theta}$. The co-cycle $\tilde{\theta}_{b_0}$ is a co-boundary. We then abuse notation and denote by $F_{b_0}$ also an untwisted coherent sheaf under an equivalence of the category of coherent sheaves on $X_{b_0}$ with the category of $\tilde{\theta}_{b_0}$-twisted sheaves. The isomorphism depends on a choice of a $1$-co-chain whose co-boundary is $\tilde{\theta}_{b_0}$, and so is unique up to tensorization by a line-bundle. In particular, the condition that 
$F_{b_0}$ deforms in co-dimension $1$
is independent of the choice of the equivalence.}
and that $F_{b_0}$ deforms in co-dimension $1$.
Then 
\[
\ell_0:=\exp\left(e_{-c_1(F_{b_0})/r}\right)(\ell(F_{b_0}))
\]
is a line in $\QQ\alpha\oplus\QQ\beta$ and the following statements hold for every $b\in B$.
\begin{enumerate}
\item The subalgebra $\LieAlg{g}(X_b)_{\ell_0}$ of $\LieAlg{g}(X_b)$ annihilating the line $\ell_0$ is equal to the subalgebra
$\LieAlg{g}(X_b)_{\bar{v}(F_b)}$ annihilating the class $\bar{v}(F_b)$. In particular, $\bar{v}(F_b)$ is annihilated by 
$\LieAlg{so}(H^2(X_b,\QQ))$.
\item
The  embedding 
$\Psi:SH^*(X_b,\QQ)\rightarrow \Sym^n\tilde{H}(X_b,\QQ)$, given in (\ref{eq-Psi}), maps $\hat{\bar{v}}(F_b)$ 
into the  image of the $n$-th power $\ell_0^n$ of $\ell_0$ via the projection (\ref{eq-projection-to-Im-Psi}).
\end{enumerate}
\end{cor}

\begin{proof}
The line $\ell_0$ is in $\QQ\alpha\oplus\QQ\beta$, by Theorem \ref{thm-Mukai-vector}(\ref{prop-item-spanning-Mukai-vector}).
The statements hold for $X_{b_0}$, by the equivariance of $\Psi$ with respect to the isometry 
$\exp\left(e_{-c_1(F_{b_0})/r}\right)$ and by Theorem \ref{thm-Mukai-vector} parts (\ref{prop-item-line-in-rational-Mukai-lattice}) and (\ref{prop-item-isotropic}). It follows for all $b\in B$, by the flatness of the section associated to $\kappa(\F)\sqrt{td_\pi}$ 
in $R\pi_*\QQ$ and the assumption that the base $B$ is connected. 
\end{proof}

\begin{lem}
\label{lemma-theta-i}
Assume that $F_i$, $1\leq i\leq k$, each deforms in co-dimension one and its rank $r_i$ is positive.
Then the following statements hold.
\begin{enumerate}
\item
\label{lemma-item-theta-i}
$\Sigma(F_i)\cap [H^1(TX)\oplus H^2(\StructureSheaf{X})]$ is the graph of the homomorphism
$\theta_i:H^1(TX)\rightarrow H^2(\StructureSheaf{X})$ given by 
$\lambda\mapsto -m_\lambda\left(c_1(F_i)/r_i\right)$.
\item 
\label{lemma-item-LLV-subspace}
The LLV subspace  $W(\otimes_{i=1}^k F_i)$ is contained in
${\displaystyle
\exp\left(e_{\sum_{i=1}^k c_1(F_i)/r_i}\right)(\span_\QQ\{\alpha,\beta\}).
}$
\end{enumerate}
\end{lem}

\begin{proof}
(\ref{lemma-item-theta-i})
The isomorphism $\mu: HT^2(X)\rightarrow \tilde{H}^{0,0}(X,\CC)$, given by $\mu(\lambda)=m_\lambda(\sigma)$
in (\ref{eq-embedding-of-HT-2-in-Mukai-lattice}),
maps $H^1(TX)\oplus H^2(\StructureSheaf{X})$ to $H^{1,1}(X)\oplus \ComplexNumbers\beta$ 
(see Equation (\ref{eq-mu-takes-Poisson-tensor-to-alpha})).
The homomorphism $\mu$ maps $\Sigma(F_i)\cap [H^1(TX)\oplus H^2(\StructureSheaf{X})]$ to 
$\ell(F_i)^\perp\cap[H^{1,1}(X)\oplus \CC\beta]$, hence to 
the graph of $\tilde{\theta}_i:H^{1,1}(X)\rightarrow \CC\beta$ given by 
$\tilde{\theta}_i(x)=(x,c_1(F_i)/r_i)\beta$, by Theorem \ref{thm-Mukai-vector}(\ref{prop-item-spanning-Mukai-vector}).
We have $m_{\theta_i(\lambda)}(\sigma)=\mu(\theta_i(\lambda))=\tilde{\theta}_i(\mu(\lambda))=(m_\lambda(\sigma),c_1(F_i)/r_i)\beta$, for all $\lambda\in H^1(TX)$. Pairing with $\alpha$, we get:
\begin{eqnarray*}
(\alpha,\mu(\theta_i(\lambda)))&=& -(m_{\theta_i(\lambda)}(\alpha),\sigma),
\\
(\alpha,\tilde{\theta}_i(\mu(\lambda)))&=& -(m_\lambda(\sigma),c_1(F_i)/r_i)=(\sigma,m_\lambda(c_1(F_i)/r_i)).
\end{eqnarray*}
It remains to prove that $m_{\theta_i(\lambda)}(\alpha)=\theta_i(\lambda)$.
Now, $m_{\theta_i(\lambda)}=\eta\circ e_{\theta_i(\lambda)}\circ \eta^{-1}$, 
by Equation (\ref{eq-eta-conjugates-m-sigma-a-to-m-a}), and $\eta\circ e_{\theta_i(\lambda)}\circ \eta^{-1}=e_{\theta_i(\lambda)},$
since $\theta_i(\lambda)$ belongs to $H^2(\wedge^{0}TX)$, and $e_{\theta_i(\lambda)}(\alpha)=\theta_i(\lambda)$, by definition
(Eq. (\ref{eq-e-lambda})).

(\ref{lemma-item-LLV-subspace})
The intersection $\Sigma(\otimes_{i=1}^k F_i)\cap [H^1(TX)\oplus H^2(\StructureSheaf{X})]$
is the graph of $\sum_{i=1}^k\theta_i:H^1(TX)\rightarrow H^2(\StructureSheaf{X})$, 
as observed in Section \ref{sec-LLV-subspace}. Hence,
$\mu$ maps the intersection to the graph of the homomorphism from $H^{1,1}(X)$ to $\CC\beta$ given by
\[
x\mapsto \left(x,\sum_{i=1}^k c_1(F_i)/r_i\right)\beta,
\]
by the proof of Part (\ref{lemma-item-theta-i}). 
The latter graph is precisely 
$\exp\left(e_{\sum_{i=1}^k c_1(F_i)/r_i}\right)(H^{1,1}(X))$. The subspace 
$W(\otimes_{i=1}^k F_i)$ is thus contained in the subspace ${\displaystyle
\exp\left(e_{\sum_{i=1}^k c_1(F_i)/r_i}\right)(\span_\QQ\{\alpha,\beta\})
}$
orthogonal to $\exp\left(e_{\sum_{i=1}^k c_1(F_i)/r_i}\right)(H^{1,1}(X))$. 
\end{proof}

%
\subsection{Torsion sheaves which deform in co-dimension $1$ are supported on lagrangian subvarieties or points}
\label{sec-torsion-sheaves-with-rank-1-obstruction-map}
Let $X$ be a $2n$-dimensional irreducible holomorphic symplectic manifold.
Let $\sigma$ be a non-zero class in $H^0(X,\Omega^2_X)$.

\begin{lem}
\label{lemma-vanishing-of-ch-p-for-p-less-than-n}
Let $F$ be an object of $D^b(X)$ such that 
$F$  has a rank $1$ cohomological obstruction map (Def. \ref{def-deforms-in-co-dimension-one}).
Assume that the line $\ell(F)$ in
$\tilde{H}(X,\QQ)$ is spanned by a class of the form $0\alpha+\lambda+s\beta$, where $\lambda$ is a class in $H^2(X,\QQ)$. 
Then $ch_i(F)=0$, for $i<n$, and $ch_n(F)\cup \sigma=0$.
\end{lem}

\begin{proof}
Let $\bar{\sigma}\in H^2(X,\StructureSheaf{X})$ be the class conjugate to that of $\sigma$ and let $m_{\bar{\sigma}}$ be the image of $\bar{\sigma}$ in $\LieAlg{g}_\CC$ via (\ref{eq-m-from-HT-2-into-LLV-algebra}).
The element $m_{\bar{\sigma}}$ has weight $(0,2)$ and thus it maps the direct summand $\tilde{H}_\fine^{p,q}(X,\CC)$ to
$\tilde{H}_\fine^{p,q+2}(X,\CC)$. It follows that $m_{\bar{\sigma}}$ annihilates $\tilde{H}_\fine^{p,q}(X,\CC)$ if $(p,q)$ is $(1,1)$ or $(2,2)$. 
Consequently, $m_{\bar{\sigma}}$
annihilates $\ell(F)$, as it is contained in $\tilde{H}_\fine^{1,1}(X,\CC)\oplus \tilde{H}_\fine^{2,2}(X,\CC)$.
The subalgebra $\LieAlg{g}_{v(F)}$ annihilating $v(F)$ is equal to $\LieAlg{g}_{\ell(F)}$
(use the proof of Theorem \ref{thm-Mukai-vector} part (\ref{prop-item-line-in-rational-Mukai-lattice}) 
replacing the use of Lemma \ref{lem-Sigma-F-is-contained-in-stabilizer-of-Mukai-vector} by the assumption that 
$F$  has a rank $1$ cohomological obstruction map).
Hence, $m_{\bar{\sigma}}$ belongs to $\LieAlg{g}_{v(F)}$.

Wedge product with $\sigma^k$ induces an isomorphism $\Omega_X^{n-k}\rightarrow\Omega_X^{n+k}$, for $0\leq k\leq n$, and thus an isomorphism $H^{n-k,q}(X)\cong H^{n+k,q}(X)$. It follows that
cup product with the class of $\sigma$ induces an injective homomorphism from $H^{p,p}(X,\QQ)$ into $H^{p+2,p}(X,\QQ)$, for $p<n$. Thus, cup product with $\bar{\sigma}$ 
 induces an injective homomorphism from $H^{p,p}(X,\QQ)$ into $H^{p,p+2}(X,\QQ)$, for $p<n$. We conclude that the direct summand $v_p(F)$ of $v(F)$ in $H^{p,p}(X,\QQ)$ vanishes, for $p<n$, as the operator $m_{\bar{\sigma}}$ acts on $H^*(X,\QQ)$ via cup product with $\bar{\sigma}$. The equality $v(F)=ch(F)\sqrt{td_X}$ implies that the first non zero graded summand of $ch(F)$ and $v(F)$ are equal. Thus,  $ch_p(F)=0$, for $p<n$, and 
$ch_n(F)=v_n(F)$. Thus $ch_n(F)\cup \sigma$ vanishes, being the complex conjugate of $v_n(F)\cup \bar{\sigma}$ which vanishes.
\end{proof}

\begin{thm}
\label{thm-support-is-either-lagrangian-or-point}
Let $\iota:Z\hookrightarrow X$ be an embedding of an irreducible variety $Z$ of dimension $< 2n$ and let $F$ be a torsion free sheaf over $Z$. Assume that $v(\iota_*F)$ deforms in co-dimension $1$ and the projection $\hat{v}(\iota_*F)$ does not vanish. Then $\iota(Z)$ is either a lagrangian subvariety or a point.
\end{thm}

\begin{proof}
Let $p$ be the complex co-dimension of $\iota(Z)$ in $X$. 
Let $[Z]\in H^{p,p}(X,\ZZ)$ be the class Poincar\'{e} dual to $\iota(Z)$. 
Then $ch_i(\iota_*F)=0$, for $i<p$, and $v_p(\iota_*F))=ch_p(\iota_*F)=\rank(F)[Z]$. The inequality $p\geq n$ and the vanishing of $ch_n(\iota_*F)\cup\sigma$ follow, by 
Lemma \ref{lemma-vanishing-of-ch-p-for-p-less-than-n}.

Let $r\alpha+\lambda+s\beta$ be a non-zero element of $\ell(\iota_*F)$. Then $r=0$, since $\iota(Z)$ is a proper subvariety of $X$. If $\lambda=0$, then $v(F)$ is annihilated by $\LieAlg{g}_\beta$.  In particular, $v(F)$ is annihilated by 
$L_\omega$, for every K\"{a}hler class $\omega$ on $X$. Hence, $v_p(\iota_*F)=\rank(F)[Z]$ 
is annihilated by 
$L_\omega$, for every K\"{a}hler class $\omega$ on $X$. 
Hence, $\iota^*\omega=0$, for every K\"{a}hler class $\omega$ on $X$, and so $Z$ is a point.

Assume that $\lambda\neq 0$. Let $\omega$ be a K\"{a}hler class on $X$ such that $(\lambda,\omega)\neq 0$. 
Then $L_\omega^n((\lambda+s\beta)^n)=n!(\omega,\lambda)^n\beta^n$ does not vanish. 
Now $L_\omega$ belongs to $\LieAlg{g}_\CC$, $\Psi$ is $\LieAlg{g}_\CC$-equivariant, and 
$\Psi(\hat{v}(\iota_*F))$ spans $\ell(\iota_*F)^n$.
Hence, $L_\omega^n$ does not annihilate $\hat{v}(\iota_*F)$ and thus neither $v(\iota_*F)$.
Now $L_\omega^n$ annihilates 
the graded summand $v_i(\iota_*F)$ in $H^{2i}(X,\QQ)$, for $i>n$. Thus $L_\omega^n(v_n(\iota_*F))\neq 0$.
We conclude that  $Z$ is $n$-dimensional.
The subvariety $Z$ is lagrangian by the vanishing of 
$ch_n(\iota_*F)\cup\sigma$ (see \cite[Prop.  5]{leung}).
\end{proof}

Let $\iota:Z\rightarrow X$ be an embedding of a compact K\"{a}hler manifold as a holomorphic lagrangian submanifold.
Assume that the obstruction map (\ref{eq-obstruction-map}) from $HT^2(X)$ to $\Ext^2(\iota_*\StructureSheaf{Z},\iota_*\StructureSheaf{Z})$ has rank $1$. Then the homomorphism $H^1(X,TX)\rightarrow H^1(Z,N_{Z/X})$,
induced by the sheaf homomorphism $TX\rightarrow \iota_*N_{Z/X}$, has rank $1$. Conjugating via the isomorphisms 
$TX\cong \Omega^1_X$ and $N_{Z/X}\cong \Omega^1_Z$, induced by the symplectic structure, we get that 
$H^1(X,\Omega^1_X)\rightarrow H^1(Z,\Omega^1_Z)$ has rank $1$. 
The restriction homomorphism $H^2(X,\CC)\rightarrow H^2(Z,\CC)$ factors through the latter, since $Z$ is lagrangian, and hence
has rank $1$.
Let $L\in\Pic(X)$ be a line bundle, such that the sublattice $c_1(L)^\perp$ orthogonal to $c_1(L)$ with respect to the BBF-pairing is equal to the kernel of the restriction homomorphism $\iota^*:H^2(X,\ZZ)\rightarrow H^2(Z,\ZZ)$. 
Assume that there exists 
a class $\tau\in H^2(X,\QQ)$, such that
$\iota^*(\tau)=c_1(\omega_Z).$ This is the case if $h^{1,1}(Z)=1$, as in 
Example \ref{example-lagrangian-Z-with-modular-structure-sheaf}, if $Z$ is a complex torus, or if $Z$ is the zero subscheme of a global section of a rank $n$ vector bundle on $X$.

\begin{lem}
\label{lemma-Mukai-line-of-structure-sheaf-of-subcanonical-lagrangian}
$\ell(\iota_*\StructureSheaf{Z})$ is spanned by ${\displaystyle 
c_1(L)-\frac{(\tau,c_1(L))}{2}\beta
}$.
\end{lem}

\begin{proof}
Let $\gamma:=r\alpha+\lambda+s\beta$ be a non-zero element of $\ell(\iota_*\StructureSheaf{Z})$.
The restriction homomorphism $\iota^*:H^2(X,\StructureSheaf{X})\rightarrow H^2(Z,\StructureSheaf{Z})$ vanishes, so $\iota_*\StructureSheaf{Z}$ deforms along gerby deformations of $X$. Hence,
$H^2(X,\StructureSheaf{X})$ is contained in $\Sigma(\iota_*\StructureSheaf{Z})$. The Duflo automorphism of $HT^*(X)$ restricts to $H^2(X,\StructureSheaf{X})$ as the identity, and so $H^2(X,\StructureSheaf{X})$ is contained in $\tilde{\Sigma}(\iota_*\StructureSheaf{Z})$.
Thus 
$\mu(H^2(X,\StructureSheaf{X}))$ is orthogonal to $\ell(\iota_*\StructureSheaf{Z})$, by Theorem \ref{thm-Mukai-vector}.
Now, $\mu(H^2(X,\StructureSheaf{X}))=\CC\beta$ (see Equation (\ref{eq-mu-takes-Poisson-tensor-to-alpha})). 
Hence, $\beta$ is orthogonal to $\gamma$ and thus the coefficient $r$ of $\alpha$  is $0$.

The sheaf $\iota_*\StructureSheaf{Z}$ deforms with $X$ in a direction $\xi\in H^1(TX)$, if and only if the line bundle $L$ does, by
\cite{voisin-lagrangian}. Hence, $\mu(\tilde{\Sigma}(\iota_*\StructureSheaf{Z})\cap H^1(TX))=c_1(L)^\perp\cap H^{1,1}(X)$ and
$\lambda$ is a non-zero scalar multiple of $c_1(L)$. We may and  thus do choose $\lambda$ to be $c_1(L)$, so that $\gamma=c_1(L)+s\beta$.

We have $\ell((\iota_*\StructureSheaf{Z})^\vee)=\tilde{D}(\gamma)=-c_1(L)+s\beta$, by Lemma \ref{lemma-Mukai-line-of-F-dual}.
Local duality yields $(\iota_*\StructureSheaf{Z})^\vee\cong\iota_*(\omega_Z)[-n]$, so that $v((\iota_*\StructureSheaf{Z})^\vee)=(-1)^n\exp(\tau)\cdot v(\iota_*\StructureSheaf{Z})$.   Hence,
$\ell((\iota_*\StructureSheaf{Z})^\vee)=-\exp(e_{\tau})(\ell(\iota_*\StructureSheaf{Z}))$, which is spanned by
$-\exp(e_{\tau})(\gamma)=-c_1(L)-[s+(\tau,c_1(L))]\beta$. Comparing coefficients of $\beta$ we get 
$s=-s-(\tau,c_1(L))$, so $s=-\frac{(\tau,c_1(L))}{2}$.
\end{proof}

\begin{rem}
\label{rem-the-LLV-line-of-the-structure-sheaf-for-n-geq-2}
Let $S$ be a $K3$ surface and let $\iota:C\rightarrow S$ be an embedding of a smooth rational curve $C$. 
Set $X:=S^{[n]}$, $n\geq 2$. We get the embedding  
$\iota^{[n]}:C^{(n)}\rightarrow X$ of the $n$-th symmetric product of $C$, hence a lagrangian $\PP^n$ in $X$.
Example \ref{example-non-modular-P-n-objects} exhibits an auto-equivalence $\Phi:D^b(S)\rightarrow D^b(S)$ which maps $\StructureSheaf{S}$ to $\iota_*\omega_C$. We get the auto-equivalence $\Phi^{[n]}:D^b(X)\rightarrow D^b(X)$ which maps $\StructureSheaf{X}$ to the object in $D^b(X)$ corresponding to the $\fS_n$-equivariant sheaf
$\boxtimes_{i=1}^n \iota_*\omega_C$ over $S^n$ via the BKR-correspondence.
The latter object in $D^b(X)$ is isomorphic to $\iota^{[n]}_*\omega_{C^{(n)}}$. Hence, the isometry $\phi^{[n]}$ of
$\tilde{H}(X,\QQ)$ associated to $\Phi^{[n]}$  maps $\ell(\StructureSheaf{X})$ to $\ell(\iota^{[n]}_*\omega_{C^{(n)}})$. 
The latter is spanned by a class $\gamma$ of the form $0\alpha+\lambda+s\beta$, where $s\in\QQ$ and $\lambda$ is a primitive class in $H^2(X,\ZZ)$, such that $\lambda^\perp$ is the kernel of the restriction homomorphism $H^2(X,\ZZ)\rightarrow H^2(C^{(n)},\ZZ)$. The class $\lambda$ has BBF self-intersection $-2(n+3)$ 
\cite[Theorem 3]{bakker}. 
This agrees with the fact that $\ell(\StructureSheaf{X})$
is spanned by the class $2\alpha+\frac{n+3}{2}\beta$ of self-intersection $-2(n+3)$ (see Lemma \ref{lemma-LLV-line-of-structure-sheaf-k3-type}).
\end{rem}

%
\section{Lagrangian surfaces in IHSMs of $K3^{[2]}$-type}
In Section \ref{sec-Chern-character-of-structure-sheaf-of-lagrangian-surface}
we use the determination of $\ell(\iota_*\StructureSheaf{Z})$ in Lemma \ref{lemma-Mukai-line-of-structure-sheaf-of-subcanonical-lagrangian}
to calculate the Chern character of $\iota_*\StructureSheaf{Z}$ for the embedding 
$\iota:Z\rightarrow X$ of a smooth lagrangian surface $Z$ in $X$ of $K3^{[2]}$-type (Lemma \ref{lemma-ch-of-Lagrangian-surface-in-K3-2}).
In Section \ref{sec-structure-sheaf-of-lagrangian-surface-has-rank-1-obstruction-map}
we show that $\iota_*\StructureSheaf{Z}$ has a rank $1$ cohomological obstruction map if $Z$ is a
smooth lagrangian surface in $X$ of $K3^{[2]}$-type with a rank $1$ restriction homomorphism $H^2(X)\rightarrow H^2(Z)$ with an image containing $c_1(\omega_Z)$  (Lemma \ref{lemma-chern-character-of-Lagrangian-structure-sheaf-deforms-in-co-dimension-1}).
Integral constraints on the topological invariants of such surfaces are presented in Section \ref{sec-integral-constraints} (Lemma \ref{lemma-arithmetic-constraints}).
In Section \ref{sec-new-examples?} we explain why examples of such smooth lagrangian subvarieties with a non-vanishing first Betti number may yield new examples of holomorphic symplectic varieties.
%
\subsection{The Chern character of $\StructureSheaf{Z}$ for a maximally deformable lagrangian surface $Z$}
\label{sec-Chern-character-of-structure-sheaf-of-lagrangian-surface}
Keep the notation and assumptions of Lemma \ref{lemma-Mukai-line-of-structure-sheaf-of-subcanonical-lagrangian}. 
Set $\lambda:=c_1(L)$ and assume that $(\lambda,\lambda)\neq 0$ and $c_1(\omega_Z)=\iota^*\tau$,
for some $\tau\in H^2(X,\QQ)$. Then 
\[
c_1(\omega_Z)=t\iota^*\lambda, 
\]
for some\footnote{In Section \ref{sec-integral-constraints} we will assume that $\lambda$ is a primitive class and derive arithmetic constraints on $\chi(Z)$ and $(\lambda,\lambda)$. This is the reason we introduce the variable $t$, rather than allow $\lambda$ to be a rational class.
}
 $t\in\QQ$, since $\lambda$ does not belong to $\lambda^\perp$. Set 
\begin{equation}
\label{eq-c}
c:=\frac{5}{4|(\lambda,\lambda)|}\sqrt{\frac{\chi(Z)}{3}},
\end{equation}
where $\chi(Z)$ is the topological Euler characteristic of $Z$. The following Lemma expresses $ch(\iota_*\StructureSheaf{Z})$
in terms of $\chi(Z)$ and $\lambda$. 

\begin{lem}
\label{lemma-ch-of-Lagrangian-surface-in-K3-2}
Assume, furthermore, that $X$ is of $K3^{[2]}$-type. 
Then $\chi(Z)>0$ and the following equalities hold.
\begin{eqnarray}
\label{eq-t}
t^2 & = & \frac{4\sqrt{3\chi(Z)}}{5|(\lambda,\lambda)|}-\frac{6}{5(\lambda,\lambda)}.
\\
\label{eq-Euler-characteristic-of-structure-sheaf-of-lagrangian-surface}
\chi(\StructureSheaf{Z})&=& c\left(\frac{(\lambda,\lambda)}{10}+\frac{t^2(\lambda,\lambda)^2}{4}
\right).
\\
\label{eq-ch-2-of-lagrangian-structure-sheaf}
ch_2(\iota_*\StructureSheaf{Z}) & = & [Z]=c\left(\lambda^2-\frac{(\lambda,\lambda)}{30}c_2(X)\right).
\\
\label{eq-ch-3-of-lagrangian-structure-sheaf}
ch_3(\iota_*\StructureSheaf{Z}) & = & -\frac{ct}{3}\lambda^3.
\\
\label{eq-ch-4-of-lagrangian-structure-sheaf}
ch_4(\iota_*\StructureSheaf{Z}) & = & c\left(
\frac{t^2(\lambda,\lambda)^2}{4}-\frac{(\lambda,\lambda)}{10}
\right)[pt].
\end{eqnarray}
\end{lem}

\begin{proof}
Lemma \ref{lemma-Mukai-line-of-structure-sheaf-of-subcanonical-lagrangian}
yields the equality
\begin{equation}
\label{eq-Psi-of-Mukai-vector-of-lagrangian-structure-sheaf}
\Psi(ch(\iota_*\StructureSheaf{Z})\sqrt{td_X})
=
a\Pi\left(\lambda^2-t(\lambda,\lambda)\lambda\beta+\frac{t^2(\lambda,\lambda)^2}{4}\beta^2
\right),
\end{equation}
for some $a\in\QQ$, where $\Pi$ is the projection from $\Sym^2(\tilde{H}(X,\QQ))$ to $\ker(\Delta)$.
We get
\begin{eqnarray*}
& &
\Psi\left([Z]+ch_3(\iota_*\StructureSheaf{Z})+ch_4(\iota_*\StructureSheaf{Z})+\frac{1}{24}[Z]c_2(X)
\right) \ \ \ =
\\
& &
a\left(
\lambda^2-(\lambda,\lambda)\tilde{q}-t(\lambda,\lambda)\lambda\beta+\frac{t^2(\lambda,\lambda)^2}{4}\beta^2
\right)
\end{eqnarray*}
Comparing the weight $(2,2)$ summands we get Equation (\ref{eq-ch-2-of-lagrangian-structure-sheaf}), with $a$ instead of $c$, using also Example \ref{example-projection-to-Im-Psi}, Equation (\ref{eq-Psi-of-c-2}), and the equality $\Psi(\lambda^2)=\lambda^2+(\lambda,\lambda)\alpha\beta$. We have
\[
\chi(Z)= \int_X[Z]^2=a^2\int_X\left(\lambda^4-\frac{(\lambda,\lambda)}{15}c_2(X)\lambda^2+\frac{(\lambda,\lambda)^2}{900}c_2(X)^2
\right)
= a^2\frac{48}{25}(\lambda,\lambda)^2.
\]
Hence, $\chi(Z)>0$ and $a=c$. 
Equation (\ref{eq-ch-3-of-lagrangian-structure-sheaf}) follows from the equality $\Psi(\lambda^3)=3(\lambda,\lambda)\lambda\beta$, by comparing the weight $(3,3)$ summands. Equation (\ref{eq-ch-4-of-lagrangian-structure-sheaf})
follows by equating the weight $(4,4)$ summands using Equation 
(\ref{eq-ch-2-of-lagrangian-structure-sheaf}). Equation (\ref{eq-Euler-characteristic-of-structure-sheaf-of-lagrangian-surface}) follows from the equality 
\[
\chi(\StructureSheaf{Z})=\chi(\iota_*\StructureSheaf{Z})=\int_Xch(\iota_*\StructureSheaf{Z})td_X=
\int_X\left(ch_4(\iota_*\StructureSheaf{Z})+\frac{1}{12}[Z]c_2(X)\right)
\] 
and  (\ref{eq-ch-4-of-lagrangian-structure-sheaf}).

It remains to prove Equation (\ref{eq-t}).
We have the short exact sequence
\[
0\rightarrow TZ\rightarrow \iota^*TX \rightarrow N_{Z/X}\rightarrow 0.
\]
The isomorphism $TZ\cong N^*_{Z/X}$ yields
\[
1+\iota^*c_2(TX)=1-c_1(TZ)^2+2c_2(TZ).
\]
Hence, $c_2(TZ)=\frac{1}{2}\left[\iota^*c_2(TX)+c_1(\omega_Z)^2\right].$
The equality $\chi(Z)=\int_Zc_2(TZ)$ yields
\[
\int_Zc_1(\omega_Z)^2=2\chi(Z)-\int_X[Z]c_2(TX).
\]
Hirzebruch-Riemann-Roch over $Z$ yields the first equality below.
\[
\chi(\StructureSheaf{Z})=\frac{1}{12} \int_Z\left(c_1(\omega_Z)^2+c_2(TZ)\right)=
\frac{\chi(Z)}{4}-\frac{1}{12}\int_X[Z]c_2(TX).
\]
Equation (\ref{eq-ch-2-of-lagrangian-structure-sheaf}) and the definition of $c$ yield
\begin{equation}
\label{eq-Euler-characteristic-of-structure-sheaf-depends-only-on-topological-Euler-characteristic}
\chi(\StructureSheaf{Z}) = \frac{\chi(Z)}{4}-\frac{c(\lambda,\lambda)}{5}=\frac{1}{4}\left(
\chi(Z)-\frac{(\lambda,\lambda)}{|(\lambda,\lambda)|}\sqrt{\frac{\chi(Z)}{3}}
\right).
\end{equation}
Comparing the right hand sides of the above equation with that of Equation (\ref{eq-Euler-characteristic-of-structure-sheaf-of-lagrangian-surface}) and solving for $t^2$ we get Equation (\ref{eq-t}).
\end{proof}

\begin{rem}
Note that Equation (\ref{eq-t}) and the definition of $c$ yield
\begin{equation}
\label{eq-t-squared}
t^2=\frac{48}{25}c-\frac{6}{5(\lambda,\lambda)}.
\end{equation}
The latter, combined with 
Equation (\ref{eq-Euler-characteristic-of-structure-sheaf-of-lagrangian-surface}), yields
\begin{equation}
\label{eq-two-Euler-characteristics}
\chi(\StructureSheaf{Z}) = \frac{1}{4}\left(
\chi(Z)-\frac{(\lambda,\lambda)}{|(\lambda,\lambda)|}\sqrt{\frac{\chi(Z)}{3}}
\right).
\end{equation}
The above and the definition of $\chi(Z)$ yield the two linear equations in three variables.
\begin{eqnarray*}
2\pm \sqrt{\frac{\chi(Z)}{3}}&=&h^{1,1}(Z)-2h^{2,0}(Z),
\\
\chi(Z)-2&=&h^{1,1}(Z)+2h^{2,0}(Z)-4h^{1,0}(Z),
\end{eqnarray*}
where the sign in the first equation is that of $(\lambda,\lambda)$.
If $(\lambda,\lambda)>0$, then 
\begin{eqnarray}
\label{eq-h-2-0-when-h-1-0-vanishes}
h^{2,0}(Z)&=&\frac{1}{4}\left[\chi(Z)-\sqrt{\frac{\chi(Z)}{3}}+4h^{1,0}(Z)-4\right],
\\
\nonumber
h^{1,1}(Z) & = & \frac{1}{2}\left[\chi(Z)+\sqrt{\frac{\chi(Z)}{3}}+4h^{1,0}(Z)\right].
\end{eqnarray}
Equation (\ref{eq-two-Euler-characteristics}) yields $\chi(Z)=3\left[h^{1,1}(Z)-2-2h^{2,0}(Z)\right]^2$ (the Euler characteristic is $3$ times the signature minus $2$ squared), or equivalently,
\[
h^{1,0}(Z)=\frac{1}{4}\left(
2+2h^{2,0}(Z)+h^{1,1}(Z)-3\left[h^{1,1}(Z)-2-2h^{2,0}(Z)\right]^2
\right).
\]
If $h^{1,1}(Z)=1$, then $h^{1,0}(Z)=-\frac{1}{2}(5h^{2,0}(Z)+6h^{2,0}(Z)^2)$, so $h^{1,0}(Z)=h^{2,0}(Z)=0$.
If, for example, $h^{2,0}(Z)=0$, then the pair $(h^{1,0}(Z),h^{1,1}(Z))$ is either $(0,1)$, or $(1,2)$. The former is realized by $\PP^2$, while the latter implies $\chi(Z)=0$ which is impossible if $(\lambda,\lambda)\neq 0$, by Lemma \ref{lemma-ch-of-Lagrangian-surface-in-K3-2}. If $h^{2,0}=1$, then the pair $(h^{1,0}(Z),h^{1,1}(Z))$ is either 
$(1,3)$, or $(2,4)$. The latter is realized by a lagrangian abelian surface, while the former can not occur if $(\lambda,\lambda)>0$,
by Equation (\ref{eq-h-2-0-when-h-1-0-vanishes}). More generally, the last displayed equality implies that $h^{1,1}(Z)$ and $2+2h^{2,0}(Z)$ are ``close'', in the sense that if $2+2h^{2,0}(Z)=ah^{1,1}(Z)$, then 
$h^{1,1}(Z)\leq \frac{a+1}{3(a-1)^2}$, since $h^{1,0}(Z)\geq 0$, so $a$ is close to $1$ for large values of $h^{1,1}(Z)$. 
\end{rem}

%
\subsection{Lagrangian surfaces $Z$ with $\StructureSheaf{Z}$ having a rank $1$ cohomological obstruction map}
\label{sec-structure-sheaf-of-lagrangian-surface-has-rank-1-obstruction-map}

\begin{lem}
\label{lemma-chern-character-of-Lagrangian-structure-sheaf-deforms-in-co-dimension-1}
Assume that $X$ is of $K3^{[2]}$-type and that $\iota^*:H^2(X,\ZZ)\rightarrow H^2(Z,\ZZ)$ has rank $1$,
Let $L$ be the line bundle over $X$ such that $\ker(\iota^*)=c_1(L)^\perp$. Assume that $(c_1(L),c_1(L))\neq 0$
and that $c_1(\omega_Z)=t\iota^*c_1(L)$, for some $t\in \QQ$.
Then $\LieAlg{g}_{v(\iota_*\StructureSheaf{Z})}=\LieAlg{g}_\ell$, where $\ell$ is spanned by 
$c_1(L)-t\frac{(c_1(L),c_1(L))}{2}\beta$.
\end{lem}

\begin{proof}
Set $\lambda:=c_1(L)$.
It suffices to verify Equation (\ref{eq-Psi-of-Mukai-vector-of-lagrangian-structure-sheaf}) with $a=c$. 
Equation (\ref{eq-Psi-of-Mukai-vector-of-lagrangian-structure-sheaf}) would follow once we verify Equations
(\ref{eq-ch-2-of-lagrangian-structure-sheaf}), (\ref{eq-ch-3-of-lagrangian-structure-sheaf}), and (\ref{eq-ch-4-of-lagrangian-structure-sheaf}).
The class $[Z]\in H^4(X,\QQ)$ is invariant under the stabilizer of $\lambda$ in the monodromy group of $X$.
Hence, $[Z]=a\lambda^2+bc_2(TX)$, for some $a, b\in\QQ$. Let $\sigma$ be a non-zero element of $H^0(X,\Omega^2_X)$. The pull back $\iota^*\sigma$ vanishes.
Hence, the cohomology class $[Z]\cup \sigma$ vanishes, since $\int_X[Z]\cup\sigma \cup (\bullet)=\int_Z\iota^*(\sigma) \cup (\bullet)$.
We drop the symbol $\cup$ below.
We have
\[
0=\int_X[Z]\sigma\bar{\sigma}=\int_X [a\lambda^2+bc_2(TX)]\sigma\bar{\sigma}=
a[2(\lambda,\sigma)(\lambda,\bar{\sigma})+(\lambda,\lambda)(\sigma,\bar{\sigma})]+
30b(\sigma,\bar{\sigma})
\]
and $(\lambda,\sigma)=0$. 
Hence, $b=-a(\lambda,\lambda)/30$ and Equation (\ref{eq-ch-2-of-lagrangian-structure-sheaf}) holds with $c$ replaced by $a$. The equality $a=c$ is proven using the equality $\chi(Z)=\int_X[Z]^2$ as in the proof of 
Lemma \ref{lemma-ch-of-Lagrangian-surface-in-K3-2}. 

Equation (\ref{eq-ch-4-of-lagrangian-structure-sheaf}) is equivalent to Equation (\ref{eq-Euler-characteristic-of-structure-sheaf-of-lagrangian-surface}), by Hirzebruch-Riemann-Roch (and the fact that Equation (\ref{eq-ch-2-of-lagrangian-structure-sheaf}) is verified and the vanishing of the direct summand of $td_X$ in $H^2(X,\QQ)$). Now Equation (\ref{eq-Euler-characteristic-of-structure-sheaf-of-lagrangian-surface}) follows from Equation (\ref{eq-ch-2-of-lagrangian-structure-sheaf}) and the fact that $Z$ is lagrangian, see the derivation of Equation (\ref{eq-Euler-characteristic-of-structure-sheaf-depends-only-on-topological-Euler-characteristic}) which exhibits $\chi(\StructureSheaf{Z})$ as a function of only  $\chi(Z)$ and the sign of $(\lambda,\lambda)$.

It remains to prove Equation (\ref{eq-ch-3-of-lagrangian-structure-sheaf}). The class $ch_3(\iota_*\StructureSheaf{Z})$
must be equal to $a\lambda^3$, for some $a\in\QQ$, as it is invariant under the stabilizer of $\lambda$ in the monodromy group of $X$. We need to prove the equality $a=-ct/3$. We calculate $\chi(\iota^*(L^{\otimes nt}))$ in two ways.
Using Grothendieck-Riemann-Roch on $X$ we get the polynomial in $n$
\[
\chi(\iota^*(L^{\otimes nt}))=
\int_X\left(1+nt\lambda+\dots +\frac{n^4t^4}{4!}\lambda^4\right)ch(\iota_*\StructureSheaf{Z})td_X.
\]
The coefficient of $n^1$ is $3at(\lambda,\lambda)^2$.
On the other hand, Hirzebruch-Riemann-Roch on $Z$ yields the second equality below:
\[
\chi(\iota^*(L^{\otimes nt}))=\chi(\omega_Z^{\otimes n})=
\chi(\StructureSheaf{Z})+\frac{n(n-1)}{2}\int_Z c_1(\omega_Z)^2.
\]
The coefficient of $n^1$ is
\[
-\frac{1}{2}
\int_Z c_1(\omega_Z)^2
= -\frac{1}{2}\int_X c(\lambda^2-\frac{(\lambda,\lambda)}{30}c_2(TX)) t^2\lambda^2 = -ct^2(\lambda,\lambda)^2.
\]
Comparing the two coefficients of $n^1$ we get the equality $a=-ct/3$.
\end{proof}

\begin{example}
\label{example-lagrangian-surfaces-which-deform-in-co-dimension-1}
If we let $Z=\PP^2$ in Lemma \ref{lemma-ch-of-Lagrangian-surface-in-K3-2}, then $(\lambda,\lambda)=-10$, the divisibility (Definition \ref{def-divisibility}) of $\lambda$ is $2$, $c=\frac{1}{8}$, and $t=\pm\frac{3}{5}$. If $Z$ is the Fano variety of lines on a hyperplane section of a cubic $4$-fold, as in Example \ref{examples-of-maximally-deformable-Lagrangian-surfaces}(\ref{example-item-Fano-variety-of-lines-on-a-cubic-threefold}), then $(\lambda,\lambda)=6$, the divisibility of $\lambda$ is $2$, $c=\frac{5}{8}$, and $t=1$. 
If $Z$ is the fixed locus of an anti-symplectic involution, as in Example \ref{examples-of-maximally-deformable-Lagrangian-surfaces}, then $(\lambda,\lambda)=2$, the divisibility of $\lambda$ is $1$, $c=5$ (since $\chi(Z)=192$), and $t=3$, by \cite[Prop. 1.10]{ferretti}. 
\end{example}
%
\subsection{Integral constraints for the existence of smooth lagrangian surfaces}
\label{sec-integral-constraints}

\begin{lem}
\label{lemma-integral-lagrangian-classes}
Let  $X$ be of $K3^{[2]}$-type and let $\lambda$ be a primitive class in $H^2(X,\ZZ)$.
\begin{enumerate}
\item
If $\div(\lambda)=1$, then $\frac{1}{\gcd(5,(\lambda,\lambda))}\left(5\lambda^2-\frac{(\lambda,\lambda)}{6}c_2(TX)\right)$
is a primitive class in $H^4(X,\ZZ)$.
\item
If $\div(\lambda)=2$, then $\frac{1}{\gcd\left(40,5+\frac{(\lambda,\lambda)}{2}\right)}\left(5\lambda^2-\frac{(\lambda,\lambda)}{6}c_2(TX)\right)$
is a primitive class in $H^4(X,\ZZ)$.
\end{enumerate}
\end{lem}

The proof of Lemma \ref{lemma-integral-lagrangian-classes} uses the integral generators for the cohomology ring $H^*(X,\ZZ)$ found in \cite{markman-integral-generators}. The proof is omitted.

Smooth lagrangian surfaces satisfying the hypotheses of Lemma \ref{lemma-ch-of-Lagrangian-surface-in-K3-2} 
must satisfy the following arithmetic constraints. Let $c$ be the rational number given in Equation (\ref{eq-c}).

\begin{lem}
\label{lemma-arithmetic-constraints}
Let $\iota:Z\rightarrow X$ be an embedding of a smooth projective lagrangian surface in $X$ of $K3^{[2]}$-type. Assume that
the kernel of the restriction homomorphism $\iota^*:H^2(X,\ZZ)\rightarrow H^2(Z,\ZZ)$ is equal to $c_1(L)^\perp$ for a line bundle $L$ with primitive $\lambda:=c_1(L)$, such that $(\lambda,\lambda)>0$. 
Assume furthermore that $c_1(\omega_Z)=t\iota^*\tau$, for some class $\tau\in H^2(X,\ZZ)$ and $t\in\QQ$. Then 
\begin{enumerate}
\item
\label{lemma-item-5-divides}
$5$ does not divide $(\lambda,\lambda)$.
\item
\label{lemma-item-3-divides-x-but-5-does-not}
If $3$ divides $x:=\frac{(\lambda,\lambda)}{2}$, 
then 
$\div(\lambda)=2$, $x\equiv 3$ modulo $8$,  $8c/5$ is an integer, and both 
$3x$ and $16cx-5$ are perfect squares. 
\item
\label{lemma-item-neither-3-nor-5}
If 
$3$ does not divide  
$x:=\frac{(\lambda,\lambda)}{2}$, then $x$ and $3(16cx-5)$ are perfect squares. 
If $\div(\lambda)=1$, then $c/5$ is an integer. If $\div(\lambda)=2$, then $\gcd\left(8,5+\frac{(\lambda,\lambda)}{2}\right)\frac{c}{5}$ is an integer. 
\end{enumerate}
\end{lem}

\begin{proof}
Assume first that $\gcd(5,(\lambda,\lambda))=1$.
If $\div(\lambda)=1$, set $n:=0$. If $\div(\lambda)=2$, let $n$ the integer, such that  $2^n:=\gcd\left(8,5+\frac{(\lambda,\lambda)}{2}\right)$.
Then $[Z]=\frac{k}{2^n}\left(5\lambda^2-\frac{(\lambda,\lambda)}{6}c_2(TX)\right)$, for an integer $k=\left\{
\begin{array}{ccl}
c/5 & \mbox{if} & \div(\lambda)=1,
\\
2^nc/5 & \mbox{if} & \div(\lambda)=2,
\end{array}
\right.$
 by Lemmas \ref{lemma-ch-of-Lagrangian-surface-in-K3-2} and \ref{lemma-integral-lagrangian-classes}.
 Hence,
\[
\frac{48k}{5\cdot 2^n}-\frac{6}{5(\lambda,\lambda)}
\ \ = \ \ 
\frac{6(2^{3-n}k(\lambda,\lambda)-1)}{5(\lambda,\lambda)}
\]
must be the square of the rational number $t$, by Equation (\ref{eq-t-squared}).
Equivalently,
\[
15\frac{(\lambda,\lambda)}{2}(2^{3-n}k(\lambda,\lambda)-1)
\] 
is a perfect square.

Set $x:=\frac{(\lambda,\lambda)}{2}$.
Then $15x(2^{4-n}kx-1)$ is a perfect square. 
Let $p$ be a prime integer. If $p$ divides $x$, then 
$p$ does not divide $(2^{4-n}kx-1)$. 
Thus, either both $3x$  and $5(2^{4-n}kx-1)$ are perfect squares, or both 
$x$  and $15(2^{4-n}kx-1)$ are perfect squares. Part (\ref{lemma-item-neither-3-nor-5}) follow from the second case. 
The first case implies that $n=3$, since otherwise $5(2^{4-n}kx-1)\equiv -1$ modulo $4$ and it can not be a square.
Hence, if $3$ divides $x$, then 
$\div(\lambda)=2$, $x=(\lambda,\lambda)/2\equiv 3$ modulo $8$, $k=\frac{8c}{5}$ is an integer, and both $3x$  and $5(2kx-1)$ are perfect squares. Part (\ref{lemma-item-3-divides-x-but-5-does-not}) follows.

Assume next that $5$ divides $(\lambda,\lambda)$.
Then the primitive lagrangian class is 
$\frac{1}{2^n}\left(\lambda^2-\frac{(\lambda,\lambda)}{30}c_2(TX)\right)$, where $n=0$, if
$\div(\lambda)=1$, or
$2^n=\gcd\left(8,\frac{(\lambda,\lambda)}{10}+1\right)$, if 
$\div(\lambda)=2$, by Lemma
\ref{lemma-integral-lagrangian-classes}.
Thus, $[Z]=\frac{k}{2^n}\left(\lambda^2-\frac{(\lambda,\lambda)}{30}c_2(TX)\right)$, for the integer $k=2^nc$.
Set $x:=\frac{(\lambda,\lambda)}{10}$. Then 
\[
\sqrt{\frac{48k}{25\cdot 2^n}-\frac{6}{5(\lambda,\lambda)}}
=
\frac{1}{5}\sqrt{\frac{3(2^{4-n}kx-1)}{x}}
\]
should be a rational number. Hence, $3x(2^{4-n}kx-1)$ is a perfect square.
If a prime $p$ divides $x$, then it does not divide $(2^{4-n}kx-1)$.
Hence, either both $x$ and  $3(2^{4-n}kx-1)$ are perfect squares, 
or both $3x$ and $2^{4-n}kx-1$ are perfect squares. 
If both $x$ and  $3(2^{4-n}kx-1)$ are perfect squares, then $n\geq 2$, since $-3$ is not a square modulo $8$. Hence $\div(\lambda)=2$. It not possible for 
both $3x$ and $2^{4-n}kx-1$ to be perfect squares, since in that case 
$3$ divides $x$ and so $2^{4-n}kx-1\equiv -1$ modulo $3$, which is not a square. 
We conclude that $x$ is a perfect square. Now $10x\equiv 2x$ modulo $8$ and $10x=(\lambda,\lambda)\equiv 6$ modulo $8$, since $\div(\lambda)=2$. Hence $x\equiv 3$ or $-1$ modulo $8$. The latter contradicts the fact that $x$ is a perfect square. Thus $5$ does not divide $(\lambda,\lambda)$ and 
Part (\ref{lemma-item-5-divides}) follows.
\end{proof}

\begin{rem}
\label{rem-two-lagrangian-classes-satisfying-arithmetic-constraints}
Following are two cases satisfying the arithmetic constraints above.
The case $(\lambda,\lambda)=8$, $\div(\lambda)$ either $1$ or $2$, $c=620$ (so that $\chi(Z)=3\cdot 2^{14}\cdot 31^2$), and $t=69/2$.
The case $(\lambda,\lambda)=54$, $\div(\lambda)=2$, $c=245/8$ (so $\chi(Z)=3^77^4$), and $t=23/3$.
A statement similar to Lemma \ref{lemma-arithmetic-constraints} can be formulated for $(\lambda,\lambda)<0$ and proven essentially with the same argument.
\end{rem}

\hide{
The moduli space of sheaves on $X$, which are deformations of $\iota_*\StructureSheaf{Z}$, 
with $Z$ as in Example \ref{examples-of-maximally-deformable-Lagrangian-surfaces}(\ref{example-item-Fano-variety-of-lines-on-a-cubic-threefold}), 
is related to O'Grady's $10$-dimensional example of an irreducible holomorphic symplectic manifold with second Betti number equal to $24$ \cite{LSV,voisin-tenfold}. Though rather constrained, additional examples 
of lagrangian surfaces satisfying the assumptions of Lemma \ref{lemma-ch-of-Lagrangian-surface-in-K3-2}
would be potentially interesting (see Remark \ref{rem-potentially-interesting}). 
}

\begin{rem}
\label{rem-singular-Lagrangian-surfaces}
Singular lagrangian ruled surfaces, which deform with a polarized holomorphic symplectic manifold $(X,L)$ of $K3^{[2]}$ deformation type, are more common and their existence is proven\footnote{The existence of a uniruled divisor in $X$ in the linear system of $L$ is proven  in \cite{CMP} and a choice of a curve in the base of the ruling determines a lagrangian surface.
}
 in \cite{CMP} for all but finitely many deformation types of the polarization $L$.
\end{rem}

%
\subsection{Lagrangian subvarieties with a positive LLV line may lead to new examples of IHSMs}
\label{sec-new-examples?}
Though rather constrained, additional examples 
of lagrangian surfaces satisfying the assumptions of Lemma \ref{lemma-ch-of-Lagrangian-surface-in-K3-2}
would be potentially interesting.
Smooth lagrangian subvarieties $Z$ of projective irreducible holomorphic symplectic manifolds 
$X$ satisfying the conditions
\begin{enumerate}
\item
$H^{1,0}(Z)$ does not vanish,
\item
the restriction homomorphism $\rho:H^2(X,\ZZ)\rightarrow H^2(Z,\ZZ)$ has rank $1$,
\item
the lattice orthogonal to $\ker(\rho)$ is spanned by an ample class $\lambda$,
\end{enumerate}
are interesting from the point of view of the classification of holomorphic symplectic varieties for the following reason. The ampleness of $\lambda$ implies that $(\lambda,\lambda)>0$
and so $\lambda^\perp$ is a non-degenerate sublattice of co-rank $1$ in $H^2(X,\ZZ)$. 
The pair $(X,Z)$ deforms over an open subset of the period domain $\Omega_{\lambda^\perp}$ of weight $2$ Hodge structures  associated to the lattice $\lambda^\perp$ \cite{voisin-lagrangian}. The class $\lambda$ remains ample over an open dense subset in moduli. 
Note that $\dim(\Omega_{\lambda^\perp})=h^{1,1}(X)-1$. Let $M$ be the connected component of the moduli space of $\lambda$-stable sheaves on $X$ containing 
$\iota_*L$, where $\iota:Z\rightarrow X$ is the inclusion and $L$ is a line bundle on $Z$ with 
$c_1(L)\in \span_\QQ\{\iota^*\lambda,c_1(\omega_Z)\}$. Let $B$ be the irreducible component of the 
Hilbert scheme of $X$ containing $Z$. 
Then $Z$ is a smooth point of $B$ and $\dim(B)=h^{1,0}(Z)$, since the normal bundle of $Z$ in $X$ is isomorphic to 
$T^*Z$ \cite{voisin-lagrangian}.
A Zariski open subset $M^0$ of $M$ parametrizes pairs $(Z',\iota_*L')$, where $Z'$ is a smooth lagrangian subvariety of $X$ and
$L'$ is a line bundle over $Z'$. $M^0$ is holomorphic symplectic and the support morphism $\pi:M^0\rightarrow B$ is projective and a lagrangian fibration, by
\cite[Theorem 8.1]{donagi-markman}. Being quasi-projective, $M^0$ admits an ample divisor class, which restricts non-trivially to the projective fibers of $\pi$, as well as the pull back of an ample class from $B$.
In particular, the rank of  the Neron-Severi group  of any compactification $\bar{M}$ of $M^0$
is at least $2$. 

Assume that $M^0$ admits a compactification $\bar{M}$, which is an irreducible symplectic variety with at worst $\QQ$-factorial terminal singularities. 
Then the moduli space of complex deformations of $\bar{M}$ admits a surjective and generically injective period map onto the period domain $\Omega_\Lambda$ of weight $2$ pure Hodge structures on the lattice $\Lambda:=H^2(\bar{M},\ZZ)$, by \cite{bakker-lehn}. In particular, the moduli space of $\bar{M}$ has dimension $h^{1,1}(\bar{M})$ and 
the locus $\LB$ in moduli of $\bar{M}$, arising from varying $(X,Z)$ via the above construction, has co-dimension $\geq 2$,
since such symplectic varieties with Picard rank $r$ vary in co-dimension $r$ in moduli. Hence, the dimension of $\LB$ is at most 
$h^{1,1}(\bar{M})-2.$
The locus $\LB$ is expected
to have the same dimension $h^{1,1}(X)-1$ as $\Omega_{\lambda^\perp}$, which would imply the inequality $h^{1,1}(\bar{M})\geq h^{1,1}(X)+1$. Only one deformation type of irreducible symplectic varieties with second Betti number $\geq 24$ is known, namely O'Grady's desingularization of a $10$ dimensional moduli space of sheaves on a $K3$ surface 
\cite{ogrady10,rapagnetta}.
The latter is related to the above strategy in the papers \cite{LSV,voisin-tenfold}.

%
\section{Effective LLV lines}
\label{sec-Effective-LLV-lines}
Let $X$ be a projective irreducible holomorphic symplectic manifold. 
We say that a line $\ell$ in $\tilde{H}(X,\QQ)$ is {\em effective}, if $\ell=\ell(F)$ for an object $F$ in $D^b(X)$, 
such that 
$F$  has a rank $1$ cohomological obstruction map
(Definition \ref{def-deforms-in-co-dimension-one}).
We say that a line $\ell$ in $\tilde{H}(X,\QQ)$ is {\em potentially effective}, 
if there exists a projective irreducible holomorphic symplectic manifold $Y$ and a derived parallel transport operator 
$\tilde{\phi}:\tilde{H}(X,\QQ)\rightarrow \tilde{H}(Y,\QQ)$, such that $\tilde{\phi}(\ell)$ is effective.
Clearly, if $\ell$ is potentially effective, then so is every line in the orbit of $\ell$ under the derived monodromy group. Theorem \ref{thm-modularity-of-a-stable-sheaf-with-a-rank-1-obstruction-map} makes the following plausible.

\begin{conj}
\label{conj-effective-LLV-lines}
A line $\ell$  in $\tilde{H}(X,\QQ)$ is effective, if it is potentially effective and is contained in $\tilde{H}^{0,0}(X,\QQ)$.
\end{conj}

Assume that $X$ is of $K3^{[2]}$-type. 
Choose a class $\delta\in H^2(X,\Integers)$ of divisibility $2$ satisfying $(\delta,\delta)=-2$. 
Set $B_\lambda:=\exp(e_\lambda)$, for $\lambda\in H^2(X,\QQ)$.
Set
\begin{equation}
\label{eq-Mukai-lattice-K3-2-case}
\Lambda_X:=B_{-\delta/2}(\ZZ\alpha+H^2(X,\ZZ)+\Integers\beta)\subset \tilde{H}(X,\QQ).
\end{equation}
The subgroup $\Lambda_X$ of $\tilde{H}(X,\QQ)$ is independent of the choice of $\delta$. We refer to $\Lambda_X$ as the {\em integral LLV lattice}, since it is invariant under the derived monodromy group of $X$, by \cite[Theorem 9.8]{taelman}.

\begin{rem}
\label{rem-integral-LLV-lattice}
Assume that $X$ is of $K3^{[n]}$-type. 
Choose a primitive class $\delta\in H^2(X,\Integers)$ of divisibility $2n-2$ satisfying $(\delta,\delta)=2-2n$. 
Define $\Lambda_X$ as in Equation (\ref{eq-Mukai-lattice-K3-2-case}). 
\begin{enumerate}
\item
\label{rem-integral-LLV-is-monodromy-invariant}
Let us show that the lattice $\Lambda_X$ depends only on the 
$\Mon(X)$-orbit of $\delta$. Notice that $B_{-\delta/2}(H^2(X,\ZZ)+\Integers\beta)=H^2(X,\ZZ)+\Integers\beta$ and
$B_{-\delta/2}(\alpha)=\alpha-\delta/2+\frac{1-n}{4}\beta$. Given $g\in O(H^2(X,\Integers))$ we get that
$B_{-g(\delta)/2}(\alpha)-B_{-\delta/2}(\alpha)=(\delta-g(\delta))/2$. It suffices to show that the latter difference belongs to $H^2(X,\ZZ)$.
Indeed, the discriminant group $H^2(X,\ZZ)^*/H^2(X,\ZZ)$ is cyclic generated by $\delta/(2n-2)$ and the coset $\delta/2+H^2(X,\ZZ)$ is 
invariant under $\Mon(X)$, as the monodromy action on $H^2(X,\ZZ)$ factors through a (surjective) homomorphism  onto
the subgroup of 
$O^+(H^2(X,\ZZ))$ of elements acting on the discriminant group by $\pm1$, by \cite[Theorem 1.2 and Lemma 4.2]{markman-constraints}.
\item
The lattice $\Lambda_X$ is invariant under $\Mon(X)$, by part (\ref{rem-integral-LLV-is-monodromy-invariant}).
Indeed, we have the equalities $g\circ B_{-\delta/2}=(g\circ B_{-\delta/2}\circ g^{-1})\circ g=B_{-g(\delta)/2}\circ g$, which takes $\Integers\alpha\oplus H^2(X,\Integers)\oplus \Integers\beta$ to $\Lambda_X$
if $g\in \Mon(X)$.
When $X=S^{[n]}$, for some $K3$-surface $S$, then 
$\Lambda_X$ is $DMon(S)$-invariant, by Theorem \ref{thm-action-of-DMon-S-on-LLV-lattice}, as well as with respect to 
the auto-equivalence $\Phi_\chi^{[n]}$ which is the $BKR$-conjugate of tensorization with the sign character of $\fS_n$ (Lemma
\ref{lemma-isometry-of-BKR-conjugate-of-tensorization-by-sign-character}). Hence, $\Lambda_X$ is invariant by the subgroup of $DMon(S^{[n]})$ generated by these three types of elements.
\hide{
\item
Following is an intrinsic definition of a $DMon(X)$-invariant lattice in $\widetilde{H}(X,\QQ)$.
Let $v:K_{top}(X)\rightarrow H^*(X,\QQ)$ send a class $x$ to $ch(x)\sqrt{td_X}$. Let $\hat{v}:K_{top}(X)\rightarrow SH^*(X,\QQ)$
be the composition of $v$ with the projection to $SH^*(X,\QQ)$. We get the lattice $\Lambda_{X,n}:=\Psi(\hat{v}(K_{top}(X)))$
in the kernel of $\Delta:\Sym^n(\widetilde{H}(X,\QQ))\rightarrow \Sym^{n-2}(\widetilde{H}(X,\QQ))$.
A class $\gamma$ in $\widetilde{H}(X,\QQ)$ is an {\em integral  class}, if  
the projection of $\gamma^n$ to $\ker(\Delta)$ belongs to $\Lambda_{X,n}$.
Let $\Lambda'_X\subset \widetilde{H}(X,\QQ)$
be the abelian subgroup generated by integral classes.
$\Lambda'_X$ is $DMon(X)$-invariant, since equivalences of derived categories map objects, which remain of Hodge type in co-dimension $1$, to such objects.
Set $t_X:=\min\{t\in \QQ \ : \ t>0 \ \mbox{and} \ t^nc_X\in \ZZ\}$, where $c_X$ is the Fujiki constant. 
The class $t_X\beta$ belongs to $\Lambda_X'$, by Equation (\ref{eq-Psi-pt}). 
If $X$ is of $K3^{[n]}$-type, then $t_X=c_X=1$ and so
$\Lambda'_X$ contains $\beta$ and $\alpha+\frac{n+3}{4}\beta$,
by Lemma \ref{lemma-LLV-line-of-structure-sheaf-k3-type}. If $X$ is of generalized Kummer deformation type, then 
$t_X=1$, as $c_X=n+1$ is not divisible by an $n$-th power of an integer larger than $1$. Thus
the classes $\beta$ and $\alpha+\frac{n+1}{4}\beta$ belongs to $\Lambda'_X$, by Lemma \ref{lemma-LLV-line-of-structure-sheaf-k3-type}.
}
\end{enumerate}
\end{rem}

Assume that $X$ is of $K3^{[2]}$-type. 
A rank $1$ sublattice of $\Lambda_X$ determines a unique orbit of primitive elements in 
$\Lambda_X$ for the action of the 
derived monodromy group $DMon(X)$.

\begin{question}
Which primitive elements of $\Lambda_X$ span potentially effective lines in $\tilde{H}(X,\QQ)$?
\end{question}

As is well known, the positive cone $\C^+:=\{x\in \tilde{H}(X,\RR) \ : \ (x,x)>0\}$ is homotopic to the $3$-sphere (see for example \cite[Sec. 4]{markman-survey}).
$DMon(X)$ contains the index $2$ subgroup $O^+(\Lambda_X)$ of $O(\Lambda_X)$ acting trivially on $H^3(\C^+,\ZZ)$, by \cite[Theorem 9.8]{taelman}.
It follows that the self intersection $(x,x)$ and the divisibility of a class $x\in\Lambda_X$ determine the $DMon(X)$-orbit of $x$,
by \cite[Cor. 3.7]{GHS}.

We have seen that the primitive isotropic element $\beta$ of divisibility $1$ is effective, by Example \ref{example-Mukai-line-of-sky-scraper-sheaf}. 
The element $0\alpha+\lambda+0\beta$ is potentially effective, for every primitive isotropic class $\lambda\in H^2(X,\ZZ)$, by Example \ref{examples-of-maximally-deformable-Lagrangian-surfaces}(\ref{example-item-lagrangian-torus}). 
The latter two examples represent the same $DMon(X)$-orbit, since every primitive isotropic class in $\Lambda_X$ has divisibility $1$ (see for Example \cite[Lemma 2.5]{markman-isotropic}).
The primitive element 
$2\alpha+\frac{5}{2}\beta$ of self intersection $-10$ and divisibility $2$ is effective, by Example  \ref{example-Mukai-line-of-structure-sheaf} of the trivial line bundle $\StructureSheaf{X}$. 
The class $\lambda+3\beta$ is potentially effective, for $\lambda\in H^2(X,\ZZ)$ primitive of divisibility $2$ with $(\lambda,\lambda)=-10$, by Example \ref{example-lagrangian-surfaces-which-deform-in-co-dimension-1} of a lagrangian $\PP^2$.
The latter two examples represent the same $DMon(X)$-orbit.
The class $\lambda-3\beta$, for $\lambda\in H^2(X,\ZZ)$  primitive of divisibility $2$ with $(\lambda,\lambda)=6$ is potentially effective, 
by Example \ref{example-lagrangian-surfaces-which-deform-in-co-dimension-1} of the lagrangian Fano variety of lines on a cubic threefold and Lemma \ref{lemma-chern-character-of-Lagrangian-structure-sheaf-deforms-in-co-dimension-1}. So is
the class $\lambda-3\beta$, for $\lambda\in H^2(X,\ZZ)$  primitive of divisibility $1$ with $(\lambda,\lambda)=2$, 
by Example \ref{example-lagrangian-surfaces-which-deform-in-co-dimension-1} of the lagrangian fixed locus of an anti-symplectic involution and Lemma \ref{lemma-chern-character-of-Lagrangian-structure-sheaf-deforms-in-co-dimension-1}. Reference \cite{CMP} may lead to more effective classes associated to singular lagrangian surfaces (see Remark \ref{rem-singular-Lagrangian-surfaces}).

%
\section{The Chern character of a very modular vector bundle isomorphic to $\Phi(\StructureSheaf{K3^{[2]}})$}
We calculate the Chern character of a very modular vector bundle obtained from the structure sheaf of $X$ of $K3^{[2]}$-type via
equivalences of derived categories and deformations. The higher Chern classes of $F$ are determined by its rank and first Chern class and the values of the latter two are constrained.

Let $X$ be of $K3^{[2]}$-type. 
Let $\Lambda_X$ be its integral LLV lattice and let $B_{-\delta/2}$ be the isometry of $\tilde{H}(X,\QQ)$ both given in (\ref{eq-Mukai-lattice-K3-2-case}).
We have
\begin{eqnarray}
\label{eq-B-minus-delta-over-2-of-alpha}
B_{-\delta/2}(\alpha)&=&\alpha-\frac{1}{2}\delta-\frac{1}{4}\beta, 
\\
\nonumber
B_{-\delta/2}(\lambda)&=& \lambda-\frac{(\lambda,\delta)}{2}\beta, \ \ \ \forall \lambda\in H^2(X,\QQ), \ \ \ \mbox{and} 
\\
\nonumber
B_{-\delta/2}(\beta)&=&\beta.
\end{eqnarray}
Note the equality 
\begin{equation}
\label{eq-B-takes-H-2-and-beta-to-itself}
B_{-\delta/2}(H^2(X,\ZZ)+\Integers\beta)=H^2(X,\ZZ)+\Integers\beta.
\end{equation}

\begin{lem}
\label{lemma-chern-character-of-derived-and-parallel-transforms-of-structure-sheaf}
 Let $F$ be a vector bundle over $X$ which is obtained via a sequence of equivalences of derived categories and deformations from the structure sheaf $\StructureSheaf{Y}$, where $Y$ is of $K3^{[2]}$-type as well. 
Then $\rank(F)=r_0^2$, for some integer $r_0$. Set $h:=c_1(F)$. Then $h=\frac{r_0}{\gcd(r_0,2)}\eta$ for an integral class $\eta\in H^{1,1}(X,\Integers)$. The class $\eta$ satisfies
\begin{eqnarray}
\label{r-0-odd}
r_0\mid(5+2(\eta,\eta)) & \mbox{if} & r_0 \ \mbox{is} \ \mbox{odd},
\\
\label{r-0-congruent-to-2-mod-4}
r_0\mid\left(5+\frac{(\eta,\eta)}{2}\right) & \mbox{if} & r_0 
 \ \mbox{is} \ \mbox{even}.
\end{eqnarray}
Furthermore, $\Psi(v(F))=\frac{1}{8}(\gamma^2+10\tilde{q})$, where 
\begin{equation}
\label{eq-gamma}
\gamma=
2r_0\alpha+2h/r_0+\left(\frac{5r_0^2+2(h,h)}{2r_0^3}\right)\beta
\end{equation}
is a primitive class of divisibility $2$ in the lattice $\Lambda_X$
satisfying $(\gamma,\gamma)=-10$. Consequently,
\begin{eqnarray}
\nonumber
ch_2(F) & = & \frac{h^2}{2r_0^2} +\frac{1-r_0^2}{24}c_2(X),
\\
\nonumber
ch_3(F) & = &\frac{h^3}{6r_0^4}+\frac{(1-r_0^2)}{24r_0^2}hc_2(X),
\\
\nonumber
ch_4(F) & = &\frac{4(h,h)^2+20r_0^2(1-r_0^2)(h,h)+25r_0^4-46r_0^6+21r_0^8}{32r_0^6}[pt],
\\
\label{eq-euler-characteristic-of-F}
\chi(F)&=& \frac{4(h,h)^2+20(h,h)r_0^2(r_0^2+1)+25r_0^4(r_0^4+1)+46r_0^6}{32r_0^6}
\\
\nonumber
&=&\frac{1}{8r_0^6}\left((h,h)+\frac{5r_0^2(r_0^2+1)}{2}-r_0^3\right)\left((h,h)+\frac{5r_0^2(r_0^2+1)}{2}+r_0^3\right).
\end{eqnarray}
\end{lem}

\begin{proof}
Let $\Lambda_Y$ be the integral LLV lattice of $Y$, given in (\ref{eq-Mukai-lattice-K3-2-case}). The line $\ell(\StructureSheaf{Y})$ is spanned by the class $4\alpha+5\beta$, by Example \ref{example-Mukai-line-of-structure-sheaf}.
We have $(4\alpha+5\beta)=B_{-\delta/2}(4\alpha+2\delta+4\beta)$ and
$2\alpha+\delta+2\beta$ is a primitive class in $\ZZ\alpha+H^2(X,\ZZ)+\Integers\beta$ of divisibility $2$.
Then $\gamma_0:=2\alpha+\frac{5}{2}\beta$ is a primitive class of divisibility $2$ in $\Lambda_Y$
satisfying $(\gamma_0,\gamma_0)=-10$.
Let $\widetilde{H}(\phi):\Lambda_Y\rightarrow \Lambda_X$ be the 
integral isometry associated to a
derived parallel transport operator $\phi:H^*(Y,\QQ)\rightarrow H^*(X,\QQ)$, which is induced by 
a sequence of equivalences of derived categories and deformations used to obtain $F$ from $\StructureSheaf{Y}$
(see Theorem \ref{thm-def-of-H-tilde}).
Set $\gamma:=(\widetilde{H}(\phi))(\gamma_0)$. Then $\gamma$ is a primitive class of divisibility $2$ in $\Lambda_X$
satisfying $(\gamma,\gamma)=-10$, since $\phi$ is an isometry.
Note that 
$H^2(X,\ZZ)$ is contained in $\Lambda_X$
and the component in $H^2(X,\QQ)$ of a class in $\Lambda_X$ belongs to $H^2(X,\ZZ)+\Integers\left(\frac{\delta}{2}\right)$,
by Equations (\ref{eq-B-minus-delta-over-2-of-alpha}) and (\ref{eq-B-takes-H-2-and-beta-to-itself}).
Write $\gamma=r\alpha+\lambda+s\beta$, with $r,s\in\QQ$ and $\lambda\in H^2(X,\ZZ)+\Integers\left(\frac{\delta}{2}\right)$. 
Then $r=-(\gamma,\beta)$ is an even integer by the divisibility of $\gamma$. 
Set $r_0:=r/2$.
Furthermore, $s$ belongs to $\frac{1}{2}\ZZ$, since
$(\gamma,B_{-\delta/2}(\alpha))=(\gamma,\alpha-\delta/2-\beta/4)=-s+\frac{r_0}{2}-(\lambda,\delta/2)$ is an even integer and 
$(\lambda,\delta/2)$ belongs to $\frac{1}{2}\ZZ$.
We have $v(F)=\phi(v(\StructureSheaf{Y}))$ and the equality
\begin{equation}
\label{eq-Psi-v-F}
\epsilon(\widetilde{H}(\phi))\Psi(v(F))=\frac{1}{8}(\gamma^2+10\tilde{q})=
\frac{1}{8}\left(
r^2\alpha^2+\lambda^2+s^2\beta^2+2r\lambda\alpha+2rs\alpha\beta+2s\lambda\beta +10\tilde{q}
\right), 
\end{equation}
by Equation (\ref{eq-Psi-of-sqrt-td-class}) and Theorem \ref{thm-def-of-H-tilde}.
The coefficient of $\alpha^2$ in $\Psi(v(F))$ is $\rank(F)/2$ and the coefficient of $\alpha^2$ in the right hand side is 
$r^2/8$, so $\epsilon(\widetilde{H}(\phi))=1$ (by Remark \ref{rem-functor-chi}) and $\rank(F)=r_0^2$. 

Let $h=c_1(F)$. Note that $\Psi(h)=\Psi(e_h(1))=e_h(\alpha^2/2)=h\alpha$.
The coefficient in $H^2(X,\QQ)$ of $\alpha$ in the left hand side of (\ref{eq-Psi-v-F}) is $h$ and the coefficient on the right hand side is $(r_0/2)\lambda$. Hence, $h=(r_0/2)\lambda$.
We have
$
-10=(\gamma,\gamma)=-4r_0s+\frac{4(h,h)}{r_0^2}.
$
So $s=\frac{5}{2r_0}+\frac{(h,h)}{r_0^3}=\frac{5r_0^2+2(h,h)}{2r_0^3},$ and
\[
\gamma=2r_0\alpha+\frac{2}{r_0}h+\frac{5r_0^2+2(h,h)}{2r_0^3}\beta,
\]
verifying Equation (\ref{eq-gamma}).
The above displayed formula yields
\[
\gamma=
B_{-\delta/2}\left(2r_0\alpha+\frac{2}{r_0}h+r_0\delta+
\left[\frac{5r_0^2+2(h,h)-2r_0^4}{2r_0^3}+\frac{(h,\delta)}{r_0}\right]\beta
\right).
\]
We see that $\frac{2}{r_0}h$ is an integral class. Hence, $\frac{(h,\delta)}{r_0}$ is an integer, since $\delta$ has divisibility $2$,  and 
so $2r_0^3$ divides $5r_0^2+2(h,h)-2r_0^4$.

If $r_0$ is odd, then $h=r_0\eta$, for $\eta\in H^2(X,\ZZ)$, and $r_0^3$ divides $r_0^2(5+2(\eta,\eta))$, so 
Equation (\ref{r-0-odd}) holds. If 
$r_0$ is even, 
then $h=(r_0/2)\eta,$ for  $\eta\in H^2(X,\ZZ)$, and $r_0^3$ divides $r_0^2(5+(\eta,\eta)/2)$, so Equation (\ref{r-0-congruent-to-2-mod-4}) holds.

Let us first calculate $\kappa(F)=\exp(-h/r_0^2)(ch(F))$. Set $\kappa(\gamma):=\exp(-h/r_0^2)(\gamma)=2r_0\alpha+\frac{5}{2r_0}\beta$. Equation (\ref{eq-Psi-v-F}) yields the first equality below
\[
8\Psi(\kappa(F)\sqrt{td_X})=\kappa(\gamma)^2+10\tilde{q}\stackrel{(\ref{eq-tilde-q-in-terms-of-tilde-b-X})}{=}
4r_0^2\alpha^2+\frac{25}{4r_0^2}\beta^2+\frac{46}{5}(\tilde{b}_X+\alpha\beta).
\]
Hence, $8\kappa(F)\sqrt{td_X}=8r_0^2+\frac{46}{5}b_X+\frac{25}{4r_0^2}[pt]$, by Equation (\ref{eq-values-of-Psi}),  and using $\frac{46}{5}b_X=\frac{1}{3}c_2(X)$, established in Equation (\ref{eq-b-X-in-terms-of-c-2}), we get
\[
\kappa(F)=(\sqrt{td_X})^{-1}\left(r_0^2+\frac{1}{24}c_2(X)+\frac{25}{32r_0^2}[pt]\right).
\]
Now, $(\sqrt{td_X})^{-1}=1-\frac{1}{24}c_2(X)+\frac{21}{32}[pt]$ and we get
\begin{equation}
\label{eq-kappa-F-depends-only-on-the-rank}
\kappa(F)=r_0^2+\frac{1-r_0^2}{24}c_2(X)+\frac{1}{32}\left(21r_0^2+\frac{25}{r_0^2}-46\right)[pt].
\end{equation}
Multiplying both sides by $1+\frac{h}{r_0^2}+\frac{h^2}{2r_0^4}+\frac{h^3}{6r_0^6}+\frac{h^4}{24r_0^8}$ and using $h^4=3(h,h)^2$ and $\int_Xc_2(X)h^2=30(h,h)$ we get
\begin{eqnarray*}
ch(F)&=&
r_0^2+h+
\left[\frac{h^2}{2r_0^2}+\frac{1-r_0^2}{24}c_2(X)\right]+
\left[\frac{h^3}{6r_0^4}+\frac{(1-r_0^2)}{24r_0^2}hc_2(X)\right]+
\\
&&\frac{4(h,h)^2+20r_0^2(1-r_0^2)(h,h)+25r_0^4-46r_0^6+21r_0^8}{32r_0^6}[pt].
\end{eqnarray*}
Formula (\ref{eq-euler-characteristic-of-F}) now follows, by Grothendieck-Riemann-Roch. 
\end{proof}

\begin{rem}
The determination of the Chern character of $F$ in 
Lemma \ref{lemma-chern-character-of-derived-and-parallel-transforms-of-structure-sheaf} 
in terms of $\rank(F)$ and $c_1(F)$ agrees with the computation of $ch_2(F)$ in O'Grady's \cite[Prop. 5.8]{ogrady-modular}. 
The formula (\ref{eq-euler-characteristic-of-F}) for $\chi(F)$ agrees with \cite[Example 7]{huybrechts-survey}, when $r_0=1$.
When $r_0=2$ and $(h,h)=6$ it yields $\chi(F)=6$, as in the case of the Fano variety of lines on a cubic fourfold \cite[Sec. 2.1]{ogrady-modular}. 
When $r_0=2$ and $(h,h)=22$ it yields $\chi(F)=10$, as in the case of Debarre-Voisin varieties \cite[Remark 8.2]{ogrady-modular}.
More generally, whenever $F$ is generated by global sections and has vanishing higher cohomology, we get a morphism $\varphi_F:X\rightarrow Gr(r_0^2,\chi(F))$. When $(X,F)=(S^{[2]},G^{[2]})$, as in Theorem \ref{thm-BKR-of-tensor-product-of-spherical-object}, and the vector bundle $G$ on the $K3$ surface $S$ has vanishing higher cohomology, then 
the vector bundle $G^{[2]}$ has vanishing higher cohomology, by the fact that $BKR$ is an equivalence and Equation (\ref{eq-BKR-of-structure-sheaf}), and formula (\ref{eq-euler-characteristic-of-F}) agrees with the fact that 
$\chi(G^{[2]})=\dim H^0(S^{[2]},G^{[2]})=\dim\Sym^2(H^0(S,G)).$
Formula  (\ref{eq-euler-characteristic-of-F}) and the integrality of $\chi(F)$ imply that
$(h,h)+\frac{5r_0^2(r_0^2+1)}{2}$ is an odd multiple of $r_0^3$. In particular, 
\[
(h,h)\equiv r_0^3-\frac{5r_0^2(r_0^2+1)}{2} \  (\mbox{mod} \ 2r_0^3). 
\]
The latter congruence relation and the equality $h=\frac{r_0}{\gcd(r_0,2)}\eta$ of Lemma \ref{lemma-chern-character-of-derived-and-parallel-transforms-of-structure-sheaf} imply the congruence relations involving $r_0$ and $(\eta,\eta)$, which are conjectured in \cite[Remark 1.6]{ogrady-modular}.
\end{rem}

\begin{rem}
Let us relate Lemma \ref{lemma-chern-character-of-derived-and-parallel-transforms-of-structure-sheaf}
to the lagrangian surfaces in Example \ref{examples-of-maximally-deformable-Lagrangian-surfaces}.
The total Segre class of an object in $D^b(X)$ is the inverse of its total Chern class.
The second Segre class of the vector bundle $F$ is Lemma \ref{lemma-chern-character-of-derived-and-parallel-transforms-of-structure-sheaf} is 
\[
s_2(F)=\frac{r_0^2+1}{2r_0^2}h^2-\frac{r_0^2-1}{24}c_2(X)
\]
and it is proportional to the lagrangian class associated in Lemma \ref{lemma-integral-lagrangian-classes} to the class $\eta=\frac{\gcd(2,r_0)}{r_0}h$, if and only if 
$(h,h)=\frac{5r_0^2(r_0^2-1)}{2(r_0^2+1)}$. The only integral solutions of the latter with positive $r_0$ are
$(r_0,(\eta,\eta))\in \{(1,0),(2,6),(3,2)\},$ and $\chi(F)=3, 6, 10$, respectively, by (\ref{eq-euler-characteristic-of-F}). 
In all three cases $s_2(F)$ is the primitive lagrangian class in Lemma \ref{lemma-integral-lagrangian-classes}. 
In the first two cases the generic such $F$ is generated by global sections, and so $s_2(F)=c_2(E)$, where $E$ is the kernel of the evaluation map $ev$ in the short exact sequence
\begin{equation}
\label{eq-E-with-chern-class-equal-segre-class-of-F}
0\rightarrow E\rightarrow H^0(X,F)\otimes_\ComplexNumbers\StructureSheaf{X}\RightArrowOf{ev} F\rightarrow 0,
\end{equation}
$E$ has rank $\chi(F)-r_0^2=2$, and the zero loci of regular sections of $E^*$ are the lagrangian surfaces in Example \ref{examples-of-maximally-deformable-Lagrangian-surfaces}(\ref{example-item-lagrangian-torus}),(\ref{example-item-Fano-variety-of-lines-on-a-cubic-threefold}). In the third case of 
$(r_0,(\eta,\eta))=(3,2)$ 
the lagrangian class associated to $\eta$ in Lemma \ref{lemma-integral-lagrangian-classes} corresponds to the surface $Z$ of fixed points of the anti-symplectic involution in Example \ref{examples-of-maximally-deformable-Lagrangian-surfaces}(\ref{example-item-fixed-locus-of-anti-symplectic-involution}), but in this case 
the kernel of $ev$ has rank $1$,
and so $F$ is not generated by global sections, since $s_2(F)$ does not vanish. 
Finally, let $L$ be the line bundle with $c_1(L)=\eta$ of BBF degree $2$, and let $E_L$ be the kernel of the evaluation homomorphism
$H^0(X,L)\otimes\StructureSheaf{X}\rightarrow L$. The fixed locus $Z$ of the anti-symplectic involution is the set theoretic degeneracy locus
of the differential $d\varphi_L:TX\rightarrow \SheafHom(E_L,L)$ of the morphism from $X$ to $\PP H^0(X,L)^*\cong\PP^5$ exhibiting $X$ as a double cover of an EPW sextic. The class of the scheme theoretic degeneracy locus is
$c_2(TX\rightarrow \SheafHom(E_L,L))$, by the Thom-Porteous formula, and is equal to $3[Z]$.
\end{rem}

\hide{
\begin{rem}
Lemma  \ref{lemma-chern-character-of-derived-and-parallel-transforms-of-structure-sheaf} may also help construct examples of lagrangian surfaces with the two classes in Remark \ref{rem-two-lagrangian-classes-satisfying-arithmetic-constraints} satisfying the arithmetic constraints of Lemma \ref{lemma-arithmetic-constraints}. Over the generic IHSM $X$ of $K3^{[2]}$ type with a primitive polarization of class $\eta$ with $(\eta,\eta)=54$ there exists a rank $4$ vector bundle $F$ with $c_1(F)=\eta$ as in 
Lemma \ref{lemma-chern-character-of-derived-and-parallel-transforms-of-structure-sheaf}, by \cite[Theorem 1.4]{ogrady-modular}.
Similarly, 
when $(\eta,\eta)=8$ there exists over $X$
a rank $9$ vector bundle $F$ with $c_1(F)=3\eta$ as in 
Lemma \ref{lemma-chern-character-of-derived-and-parallel-transforms-of-structure-sheaf}, by \cite[Theorem 1.4]{ogrady-modular}.
In both cases the second Segre class $s_2(F)$ minus $c_2(TX)$ is the primitive lagrangian class associated to $\eta$ in Lemma \ref{lemma-integral-lagrangian-classes}. 
\end{rem}
}

\begin{example} 
\label{example-LLV-line-of-BKR-image-of-G-boxtimes G}
Let $S$ be a $K3$ surface and 
$X=S^{[n]}$. Let $F$ be the vector bundle 
corresponding via the BKR equivalence to 
$G\boxtimes G \boxtimes \cdots \boxtimes G\otimes_\ComplexNumbers \chi^i$, for a rigid slope-stable vector bundle $G$ on $S$.
If $n=2$, Lemma \ref{lemma-chern-character-of-derived-and-parallel-transforms-of-structure-sheaf} applies to the vector bundle $F$, 
by Remark \ref{rem--is deformation-of-FM-image-of-structure-sheaf}. We use Lemma \ref{lemma-chern-character-of-derived-and-parallel-transforms-of-structure-sheaf} to show that $\ell(F)$ is spanned by
\[
\left\{
\begin{array}{ccl}
r_0^2\alpha+\left(r_0c_1(G)-\frac{r_0(r_0-1)}{2}\delta\right)+\left[
\frac{(c_1(G),c_1(G))+2}{2}+\frac{r_0^2}{4}-\frac{r_0(r_0-1)}{2}
\right]\beta &\mbox{if} & i=0,
\\
r_0^2\alpha+\left(r_0c_1(G)-\frac{r_0(r_0+1)}{2}\delta\right)+\left[
\frac{(c_1(G),c_1(G))+2}{2}+\frac{r_0^2}{4}-\frac{r_0(r_0+1)}{2}
\right]\beta &\mbox{if} & i=1.
\end{array}
\right.
\] 
We start with comments for all $n$, which will be used in the following Remark \ref{rem-eliminate-use-of-lemma-for-n-equal-2} to calculate $\ell(F)$ for all $n$.
The line $\ell(F)$ is invariant under the action of the stabilizer of $\ell(G)$ in the isometry group of the sublattice 
$\tilde{H}(S,\RationalNumbers)$ of $\tilde{H}(S^{[n]},\RationalNumbers)$, where the isometry group of the former is embedded as a subgroup of the isometry group of the latter via the homomorphism $h$  making the diagram below commutative.
\[
\xymatrix{
DMon(S) \ar[d]  \ar[r] & DMon(S^{[n]}) \ar[d]
\\
O(\tilde{H}(S,\QQ)) \ar[r]^h & O(\tilde{H}(S^{[n]},\QQ)),
}
\]
The top horizontal arrow above is $^{[n]}:\Hom_{K3^{[1]}}(S,S)\rightarrow \Hom_{K3^{[n]}}(S^{[n]},S^{[n]})$ 
induced by the functor (\ref{eq-functor-[n]}).
The homomorphism $h$ is given by 
$h(g)=\det(g)\cdot (B_{-\delta/2}\circ \iota(g)\circ B_{\delta/2})$, where $\iota(g)$ is the extension acting trivially on $\span\{\delta\}$ (see Example \ref{example-Mukai-lattice-of-Hilbert-schemes}), and $B_\lambda$ is  $\exp(e_\lambda)$, by
\cite[Theorem 9.4]{taelman} when $n=2$ and by Theorem \ref{thm-introduction-action-of-DMon-S-on-LLV-lattice} for general $n$.
It follows that $B_{\delta/2}(\ell(F))$ is a line in the plane spanned by $\ell(G)$ and $\delta$.
\begin{equation}
\label{eq-B-of-ell-F-is-in-the-plane-spanned-by-delta-and-ell-G}
B_{\delta/2}(\ell(F))\subset \ell(G)+\QQ\delta.
\end{equation}
Now, if the rank of $G$ is $r_0$, then the rank $r$ of $F$ is $r_0^n$ and $c_1(F)=r_0^{n-1}c_1(G)+t\delta$, under the inclusion $H^2(S,\Integers)\subset H^2(S^{[n]},\Integers)$ given in (\ref{eq-theta}), where $t$ is an integer
depending on the exponent $i\in \{0, 1\}$ of $\chi$ and $\delta$ is half the class of the divisor of non-reduced subschemes. 
The line $\ell(F)$ is spanned by a vector of the form
$r_0^n\alpha+c_1(F)+s(F)\beta$, for some $s(F)\in\QQ$,
by Theorem \ref{thm-Mukai-vector}(\ref{prop-item-spanning-Mukai-vector}).

Assume $n=2$. The coefficient $s(F)-t-\frac{r_0^2}{4}$ of $\beta$ in $B_{\delta/2}(r_0^2\alpha+c_1(F)+s(F)\beta)$ must be
equal $r_0 \cdot s(G)$, by (\ref{eq-B-of-ell-F-is-in-the-plane-spanned-by-delta-and-ell-G}), yielding $s(F)=r_0\cdot s(G)+\frac{r_0^2}{4}+t$.  Set $\lambda:=c_1(G)$. Then $s(G)=\frac{(\lambda,\lambda)+2}{2r_0}$ and 
\[
s(F)=\frac{(\lambda,\lambda)+2}{2}+\frac{r_0^2}{4}+t.
\]
On the other hand, multiplying the coefficient of $\beta$ in (\ref{eq-gamma}) by $r_0/2$ we get that 
\[
s(F)=\frac{5r_0^2+2(c_1(F),c_1(F))}{4r_0^2}=\frac{5r_0^2+2r_0^2(\lambda,\lambda)-4t^2}{4r_0^2}.
\]
Comparing the above two expressions for $s(F)$ we get that $t^2+r_0^2t+\frac{r_0^4-r_0^2}{4}=0,$
so that 
\[
t\in \left\{
-\frac{r_0(r_0-1)}{2}, \ -\frac{r_0(r_0+1)}{2}
\right\}.
\]
O'Grady calculated in \cite[Prop. 5.8]{ogrady-modular} that $t=-r_0(r_0-1)/2$ if $i=0$ and $t=-r_0(r_0+1)/2$ if $i=1$ as follows. 
Using the notation of (\ref{eq-b-q}) we have the short exact sequence over $\Gamma$
\[
0\rightarrow q^*F\rightarrow b^*(G\boxtimes G)\rightarrow \left[(G\boxtimes G\restricted{)}{\hat{D}}\right]^{\chi^{i+1}}\rightarrow 0,
\]
where $\hat{D}\subset \Gamma$ is the exceptional divisor of the blow-up $b:\Gamma\rightarrow S\times S$ centered along the diagonal. 
Now $q^*(\delta)=[E]$, $c_1(b^*(G\boxtimes G))=r_0q^*(\lambda)$, and so  
\begin{equation}
\label{eq-minus-t}
-t=\rank \left[(G\boxtimes G\restricted{)}{\hat{D}}\right]^{\chi^{i+1}}=
\left\{
\begin{array}{ccl}
r_0(r_0-1)/2 & \mbox{if} & i=0,
\\
r_0(r_0+1)/2 & \mbox{if} & i=1.
\end{array}
\right.
\end{equation}
O'Grady denotes $F$ by $G[2]^+$, if $i=0$, and by $G[2]^-$, if $i=1$.
Note that $G[2]^-$ is isomorphic to $(G^*[2]^+)^*\otimes \StructureSheaf{S^{[2]}}(-\delta)$, 
by Lemma \ref{lemma-BKR-conjugate-of-tensorization-by-sign-character}.
Indeed, the above description of $\ell(G[2]^-)$ agrees with that of $\ell((G^*[2]^+)^*\otimes \StructureSheaf{S^{[2]}}(-\delta))$
via Lemma \ref{lemma-Mukai-line-of-F-dual}.
\end{example}

\begin{rem}
\label{rem-eliminate-use-of-lemma-for-n-equal-2}
The argument of Example \ref{example-LLV-line-of-BKR-image-of-G-boxtimes G} generalizes to calculate $\ell(F)$ for
the vector bundle $F$
corresponding via the BKR equivalence to 
$G\boxtimes G \boxtimes \cdots \boxtimes G\otimes_\ComplexNumbers \chi^i$, for a rigid slope-stable vector bundle $G$ on a $K3$ surface $S$, for all $n$. Again $c_1(F)=r_0^{n-1}c_1(G)+t\delta$ and the coefficient 
$\frac{1-n}{4}r_0^n+(1-n)t+s(F)$
of $\beta$ in the image
$B_{\delta/2}(r_0^n\alpha+c_1(F)+s(F)\beta)$ must be equal to $r_0^{n-1}s(G)=r_0^{n-2}\left(\frac{(c_1(G),c_1(G))+2}{2}\right)$, by (\ref{eq-B-of-ell-F-is-in-the-plane-spanned-by-delta-and-ell-G}).
We get that
\[
s(F)=r_0^{n-2}\left(\frac{(c_1(G),c_1(G))+2}{2}\right)+r_0^n\left(\frac{n-1}{4}\right)+(n-1)t.
\]
It remains to calculate $t$.
The quotient $\Gamma/{\mathfrak A}_n$ of $\Gamma$ by the alternating group is a double cover 
of $S^{[n]}$ branched over the divisor of non-reduced subschemes
\[
\Gamma\LongRightArrowOf{q_1} \Gamma/{\mathfrak A}_n \LongRightArrowOf{q_2} S^{[n]}.
\]
We have the short exact sequence over $\Gamma/{\mathfrak A}_n$
\[
0\rightarrow q_2^*F\rightarrow \left(q_{1,*}(b^*(G\boxtimes \ \cdots\ \boxtimes G))\right)^{{\mathfrak A}_n}\rightarrow Q\rightarrow 0.
\]
The quotient sheaf $Q$ is again supported over the inverse image of the divisor of non-reduced subschemes in $S^{[n]}$. 
Away from the locus of co-dimension $2$, where more then two points come together, $Q$ is locally free of rank
$r_0^{n-2}$ times the right hand side of (\ref{eq-minus-t}). We conclude that
\[
-t=
\left\{
\begin{array}{ccl}
r_0^{n-1}(r_0-1)/2 & \mbox{if} & i=0,
\\
r_0^{n-1}(r_0+1)/2 & \mbox{if} & i=1.
\end{array}
\right.
\]

\end{rem}

\hide{
\begin{question}
Use Taelman's calculation of the derived monodromy group for $X$ of $K3^{[2]}$-type and the integral lattice $\Lambda\subset \tilde{H}(X,\RationalNumbers)$ of \cite[Theorem 9.8]{taelman} in order to characterize the LLV vectors in $\Lambda$ admitting a modular vector bundle $F$, which is a lift of the Azumaya algebra deforming $\End(G[2]^{\pm})$ (in O'Grady's notation), for a spherical bundle $G$ on a $K3$ surface. Recover the constraints on $\rank(F)$ and $c_1(F)$ in \cite[Theorem 1.4]{ogrady-modular}. Relate also those lifts, which correspond to the same LLV line. The rational translate of the LLV line  $\ell(F)$ to the line $\ell'(F)=exp(e_{-c_1(F)/r}\ell(F)$ in $U_\QQ$ should depend only on the rank $r$ of $F$ and whether $F$ is a deformation of $G[2]^+$ or $G[2]^-$ for some spherical stable vector bundle  $G$ on a $K3$ surface.
\end{question}
}

%
\section{Very modular vector bundles isomorphic to $\Phi(\StructureSheaf{K3^{[n]}})$}
\label{sec-stability-of-images-of-structure-sheaf}

We prove Theorem \ref{thm-BKR-of-tensor-product-of-spherical-object} in this section.

\begin{prop}
\label{prop-tensor-outer-product-is-stable}
Let $X_i$, $1\leq i \leq n$, be compact complex manifolds with trivial Picard groups and let $V_i$ be a $\theta_i$-twisted vector bundle over $X_i$ or rank $r_i$. Assume that $V_i$ does not have any non-zero subsheaf of rank less than $r_i$. 
Then the vector bundle $V_1\boxtimes V_2\boxtimes \cdots \boxtimes V_n$ does not have any non-zero subsheaf of rank less than $\prod_{i=1}^n r_i$.
\end{prop}

The proof of the proposition will use the following  lemma.

\hide{
\begin{lem}
\label{lemma-U-otimes-V-does-not-have-subsheaves-of-rank-leq-r}
Let $X$ be a connected compact complex manifold with a trivial Picard group. Let $V$ be a possibly twisted rank $r$ vector bundle over $X$. Assume that $V$ does not have any non-zero subsheaf of rank less than $r$. Let $U$ be a vector space and $E\subset U\otimes_\CC V$ a saturated non-zero subsheaf of rank $\leq r$. Then there exists a line $\ell\subset U$, such that $E=\ell\otimes_\CC V$. 
\end{lem}

Note that the statement of the theorem is false if we remove the assumption that $\Pic(X)$ is trivial. Indeed,
take $V=\StructureSheaf{\PP(U)}$ and $E=\StructureSheaf{\PP(U)}(-1)\subset U\otimes_\CC V.$ 

\begin{proof}
The proof is by induction on $n:=\dim(U)$. The statement is clear if $n=1$. Assume that $n>1$ and the statement holds for $n-1$.
Choose a non-zero element $\xi\in U^*$. Let $h_\xi:E\rightarrow V$
be the restriction of $\xi\otimes id:U\otimes_\CC V\rightarrow V$ to $E$.
If $h_\xi$ vanishes, then $E$ is a saturated subsheaf of $\ker(\xi)\otimes_\CC V$ and we are done, by the induction hypothesis. 
It remains to prove that $h_\xi$ vanishes, for some non-zero $\xi\in U^*$.

The proof is by contradiction.
Assume that $h_\xi$ does not vanish for every non-zero $\xi\in U^*$. Then the image of $h_\xi$ must have rank $r$. Hence, $\wedge^r(h_\xi)$ is a non-zero section of
$\det(E)^*\otimes\det(V)$. The latter is a trivial line-bundle, since $\Pic(X)$ is trivial. Hence, $h_\xi$ is an isomorphism, as $E$ is reflexive, being a saturated subsheaf of a locally free sheaf. The homomorphism $h_\xi$ maps the fiber $E_p$ surjectively onto $V_p$, for all $p\in X$. Hence, $E_p$ maps to an $r$-dimensional subspace of $U\otimes_\CC V_p$, for all $p\in X$.

Over $\PP(U^*)\times X$ we have the composite homomorphism $h:\pi_2^*E\rightarrow \pi_1^*\StructureSheaf{\PP(U^*)}(1)\otimes \pi_2^V$, given by
\[
\pi_2^*E
\rightarrow U\otimes_\CC \pi_2^*V 
\RightArrowOf{\pi_1^*(\xi)\otimes id}
\pi_1^*\StructureSheaf{\PP(U^*)}(1)\otimes \pi_2^*V,
\]
where $\xi:U\otimes_\CC\StructureSheaf{\PP(U^*)}\rightarrow \StructureSheaf{\PP(U^*)}(1)$
is the quotient homomorphism. Then $\wedge^r(h)$ is a nowhere-vanishing section of 
$\pi_1^*\StructureSheaf{\PP(U^*)}(r)\otimes \pi_2^*(\det(E^*)\otimes\det(V))$, by the above discussion.
Such a section does not exist, as the line bundle $\det(E^*)\otimes\det(V)$ is trivial and any non-zero section would vanish over a degree $r$ hypersurface in $\PP(U^*)$. A contradiction. 
\end{proof}

}

\begin{lem}
\label{lemma-U-otimes-V-does-not-have-subsheaves}
Let $X$ be a connected compact K\"{a}hler manifold with a trivial Picard group. Let $V$ be a possibly twisted rank $r$ vector bundle over $X$. Assume that $V$ does not have any non-zero subsheaf of rank less than $r$. Let $U$ be a vector space and $E\subset U\otimes_\CC V$ a saturated non-zero subsheaf. Then there exists a subspace $W\subset U$, such that $E=W\otimes_\CC V$. 
\end{lem}

\begin{proof}
The vector bundle $V$ is slope-stable with respect to every K\"{a}hler class on $X$. 
The determinant of $\Hom(E,U\otimes_\CC V)$ is trivial, since $\Pic(X)$ is trivial. The statement is thus a special case of the
 the more general fact that saturated subsheaves of polystable sheaves of the same slope are direct summands.
\end{proof}
\hide{
\begin{proof}
The proof is by induction on $e:=\rank(E)$. 
The statement holds if $e\leq r$, by Lemma \ref{lemma-U-otimes-V-does-not-have-subsheaves-of-rank-leq-r}. 
Assume that $e>r$ and the statement holds for all saturated subsheaves $E'$ with $\rank(E')<e$.
Choose $\xi\in U^*$, such that the restriction $h:E\rightarrow V$ of $\xi\otimes id:U\otimes_\CC V\rightarrow V$ to $E$ does not vanish. Then $h(E)$ is a rank $r$ subsheaf of $V$ and so $E_0:=\ker(h)$ is a saturated subsheaf of $U\otimes_\CC V$ of rank $e-r$. Hence, $E_0=W_0\otimes_\CC V$, for some subspace $W_0\subset U$, by the induction hypothesis. 
The quotient $E_1:=E/E_0$ is a saturated subsheaf of $(U/W_0)\otimes_\CC V$  of rank $r$.
Thus, $E_1=\ell\otimes_\CC V$, for some line $\ell\subset U$, by Lemma \ref{lemma-U-otimes-V-does-not-have-subsheaves-of-rank-leq-r}. It follows that $E=W\otimes_\CC V$, where $W$ is the unique subspace of $U$ containing $W_0$ and projecting onto $\ell$. 
\end{proof}

}

\begin{proof}
(of Proposition \ref{prop-tensor-outer-product-is-stable}). The proof is by induction on $n$. The case $n=1$ is clear. Assume that
$n>1$ and the statement holds for $n-1$. Let $E$ be a non-zero saturated subsheaf of $V:=V_1\boxtimes \cdots \boxtimes V_n$.
Denote the rank of $E$ by $e$. 
There exists an open subset $U$ of $X_n$, such that the restriction of $E$ to $(\prod_{i=1}^{n-1}X_i)\times\{p\}$ maps to a non-zero subsheaf $E_p$ of rank $e$ of the restriction $V_p$ of $V$, for all $p\in U$. Then there exists a  subspace $W_p$ of the fiber $V_{n,p}$ of $V_n$ at $p$, such that the saturation of $E_p$ in $V_p$ is 
$V_1\boxtimes \cdots \boxtimes V_{n-1}\otimes_\CC W_p$ and $e=\dim(W_p)\prod_{i=1}^{n-1}r_i$, by the induction hypothesis and Lemma \ref{lemma-U-otimes-V-does-not-have-subsheaves}. We conclude that there is a 
rank $e/\prod_{i=1}^{n-1}r_i$ subsheaf $W$ of $V_p$, such that 
$E$ is a subsheaf of $V_1\boxtimes \cdots \boxtimes V_{n-1}\boxtimes W$. But the saturation of $W$ must be equal to $V_n$, as $V_n$ does not have non-zero saturated subsheaves. Hence, $E$ must be equal to $V$.
\end{proof}

\begin{proof} (of Theorem \ref{thm-BKR-of-tensor-product-of-spherical-object}).
Being $H$-slope-stable, $G$ is projectively $H$-hyperholomorphic \cite{kaledin-verbitsky-book} (see also \cite{huybrechts-schroer}). Hence, 
the triple $(S,H,G)$ determines a twistor family $\pi:\S\rightarrow \PP^1$ and a possibly twisted vector bundle $\G$ over $\S$, such that each fiber $S_t$ of $\pi$ over $t\in \PP^1$ is endowed with a K\"{a}hler class $\omega_t$ and the restriction $G_t$ of $\G$ to $S_t$ is
$\omega_t$-slope-stable. We may assume, possibly after replacing $H$ by another ample class with respect to which $G$ is slope-stable, that $\Pic(S_t)$ is trivial, for a generic $t\in\PP^1$. Then $G_t$ does not have any non-zero saturated proper subsheaves. Thus, the vector bundle $G_t\boxtimes \cdots \boxtimes G_t$ over $(S_t)^n$ does not have any non-zero saturated proper subsheaves, by Proposition \ref{prop-tensor-outer-product-is-stable}. 

The isospectral Hilbert scheme $\Gamma\subset S^n\times S^{[n]}$ is the reduced fiber product of $S^n$ and $S^{[n]}$ over the symmetric product $S^{(n)}$, by \cite[Def. 3.2.4]{haiman}. Hence, we have a relative isospectral subscheme over the base of the twistor deformation with fiber $\Gamma_t\subset S_t^n\times S_t^{[n]}$.
Let $b_t:\Gamma_t\rightarrow S^n$ and $q_t:\Gamma_t\rightarrow S_t^{[n]}$ be the restriction of the two projections.
The Brauer class $\alpha_t\in H^2_{an}(S_t,\StructureSheaf{S})$ of $G_t$ is the image of a class $\tilde{\alpha}_t\in H^2(S_t,\mu_r)$,
where $r$ is the rank of $G_t$, as pointed out in Equation (\ref{eq-HS-lemma2.5}).
The Brauer class $\alpha_t^{\boxtimes n}$ of $G_t\boxtimes \cdots \boxtimes G_t$ is thus the image of an $\fS_n$-invariant class in
$H^2(S_t^n,\mu_r)$, which is the pullback of a class in $H^2(S_t^{(n)},\mu_r)$. 
Hence, the Brauer class $b_t^*(\alpha_t^{\boxtimes n})$ of $b_t^*(G_t\boxtimes \cdots \boxtimes G_t)$
is the pullback via $q_t$ of a class in $H^2_{an}(S_t^{[n]},\StructureSheaf{S_t^{[n]}})$. The morphism $q_t$ is flat, by the proof of
\cite[Prop. 3.7.4 ]{haiman}.
The pushforward $q_{t,*}b_t^*(G_t\boxtimes \cdots \boxtimes G_t)$ is thus a well defined locally free twisted sheaf.
Set 
\[
G_t^{[n]}:=(q_{t,*}b_t^*(G_t\boxtimes \cdots \boxtimes G_t))^{\fS_n}.
\]
A non-zero proper saturated subsheaf $E$ of $G_t^{[n]}$ over the Douady space $S_t^{[n]}$ determines an $\fS_n$-invariant non-zero subsheaf $\tilde{E}$ of $G_t\boxtimes \cdots \boxtimes G_t$ 
of the same rank as $E$, since $q_t$ restricts to a Galois covering over the complement of the divisor in $S_t^{[n]}$ of non-reduced subschemes. We have seen that such a subsheaf $\tilde{E}$ does not exist. Hence,  the vector bundle $G_t^{[n]}$ does not have any non-zero saturated proper subsheaves. Thus, $G_t^{[n]}$ is slope-stable with respect to every K\"{a}hler class on $S_t^{[n]}$. 

The paragraph before the statement of Theorem \ref{thm-BKR-of-tensor-product-of-spherical-object} establishes that $\obs_{G^{[n]}}$ has rank $1$. Hence, the class $\kappa(G^{[n]})$ remains of Hodge type under every K\"{a}hler deformation of $S^{[n]}$, by Proposition \ref{prop-kappa-class-remains-of-Hodge-type} (applied with $\Phi$ the identity endofunctor). The same thus holds for the class $\kappa(G_t^{[n]})$.
We conclude that $G_t^{[n]}$ deforms with $S_t^{[n]}$ to a pair $(X,F)$ of a vector bundle $F$ over every K\"{a}hler deformation $X$ of $S_t^{[n]}$, by \cite[Prop. 6.17]{markman-BBF-class-as-characteristic-class} 
(applied via an argument identical to the one used in the proof of Theorem
\ref{thm-modularity-of-a-stable-sheaf-with-a-rank-1-obstruction-map}).

The vector space $\Ext^j(G_t,G_t)$ is one dimensional, for $j=0$ and $j=2$, and it vanishes otherwise. Hence, 
the vector space $\Ext^{j}_{\fS_n}((G_t^{[n]},\rho_{\boxtimes}),(G_t^{[n]},\rho_{\boxtimes}))$ is one dimensional, for $j=2i$, $0\leq i\leq n$, and it vanishes otherwise. We conclude that $\Ext^j(G_t^{[n]},G_t^{[n]})$ is one dimensional, for $j=2i$, $0\leq i\leq n$, and it vanishes otherwise, since $BKR$ is an equivalence of categories. Now 
$\Ext^j(G_t^{[n]},G_t^{[n]})$ is isomorphic to $H^j(S_t^{[n]},\SheafEnd(G_t^{[n]}))$ and $\SheafEnd(G_t^{[n]})$
is a poly-stable hyperholomorphic vector bundle.
The dimensions of sheaf cohomologies of poly-stable hyperholomorphic vector bundles are constant along twistor deformations, by 
\cite[Cor. 8.1]{verbitsky-hyperholomorphic}. Hence, the deformations of $G_t^{[n]}$ remain $\PP^n$-objects in the sense that $\Ext^{i}(F,F)\cong H^i(\PP^n,\CC)$. In particular, $\Ext^1(F,F)$ vanishes and $F$ is infinitesimally rigid.
\end{proof}

%
\section{Very modular vector bundles with isotropic Mukai vector over Hilbert schemes}
\label{sec-modular-vector-bundles-with-isotropic-LLV-line}
In Section \ref{sec-chern-character-of-FM-image-of-sky-scraper-sheaf} we compute the Chern character of a vector bundle on an irreducible holomorphic symplectic manifold of $K3^{[n]}$-type, which is the Fourier-Mukai image of a sky-scraper sheaf.
The Chern character is determined by the rank and first Chern class of $F$.  
Let $S$ be a $K3$ surface. 
In Section \ref{sec-proof-of-modularity-of-images-of-sky-scraper-sheaves} we prove Theorem \ref{thm-Fourier-Mukai-images-of-sky-scraper-sheaves} about the modularity of the BKR images on $S^{[n]}$ of equivariant vector bundles on $S^n$ constructed from $n$ vector bundles $\{G_i\}_{i=1}^n$ on $S$, where all $G_i$ belong to  the same two dimensional  component of the moduli space of stable sheaves on $S$.
%
\subsection{The Chern character of the Fourier-Mukai image of a sky-scraper sheaf}
\label{sec-chern-character-of-FM-image-of-sky-scraper-sheaf}
%
\subsubsection{An example of a Fourier-Mukai image of a sky-scraper sheaf}
Let $\Phi: D^b(M)\rightarrow D^b(S)$ be an equivalence of derived categories of $K3$ surfaces. Assume that for every point $p\in M$, the image $\Phi(\StructureSheaf{p})$ of the sky-scraper sheaf of $p$ is isomorphic to a 
vector bundle of Mukai vector $v$
on $S$. Let $\{p_i\}_{i=1}^n$ be a set of $n$ distinct points on $M$
and denote by $G_i$ the  vector bundle isomorphic to $\Phi(\StructureSheaf{p_i})$.
Let $Z\subset M$ be the length $n$ zero dimensional subscheme of $M$ with support $\{p_i\}_{i=1}^n$. 
Denote by $z$ the corresponding point of $M^{[n]}$. 
Note that the object $\StructureSheaf{p_{\sigma(1)}}\boxtimes \cdots \boxtimes \StructureSheaf{p_{\sigma(n)}}$ is isomorphic in $D^b(M^n)$ to
the sky scraper sheaf $\StructureSheaf{(p_{\sigma(1)}, \dots, p_{\sigma(n)})}$.
The ``usual'' BKR equivalence is given by the functor
\begin{equation}
\label{eq-usual-bkr}
bkr:=\Phi_{\StructureSheaf{\Gamma_M},\rho}:D^b(M^{[n]})\rightarrow D^b_{\fS_n}(M^n)
\end{equation}
using the same object as in (\ref{eq-BKR}) as Fourier-Mukai kernel, but going in the reverse direction. 
We see that the sky-scraper sheaf $\StructureSheaf{z}$ corresponds via $bkr$
to the $\fS_n$-equivariant object
\[
\oplus_{\sigma\in\fS_n}\left[\StructureSheaf{p_{\sigma(1)}}\boxtimes \cdots \boxtimes \StructureSheaf{p_{\sigma(n)}}\right]
\]
over $M^n$. The equivalence $\Phi^{\boxtimes n}:D^b_{\fS_n}(M^n)\rightarrow D^b_{\fS_n}(S^n)$
takes the above displayed object to 
\begin{equation}
\label{eq-G-Z}
G_z:=\oplus_{\sigma\in\fS_n}\left[G_{\sigma(1)}\boxtimes \cdots \boxtimes G_{\sigma(n)}\right]
\end{equation}
with its natural $\fS_n$-linearization $\rho$. 
Let $\tilde{\Phi}^{[n]}:D^b(M^{[n]})\rightarrow D^b(S^{[n]})$ be the composition
\[
\tilde{\Phi}^{[n]}:= BKR\circ \Phi^{\boxtimes n}\circ bkr,
\]
where $BKR$ is given in (\ref{eq-BKR}).
Let $F_z$ be the sheaf over $S^{[n]}$ isomorphic to $BKR(G_z,\rho)$. Then
$F_z$ is isomorphic to $\tilde{\Phi}^{[n]}(\StructureSheaf{z})$ and hence 
$\obs_{F_z}$ has rank $1$.
Write $v=(r_0,\lambda,(\lambda,\lambda)/2r_0)$, so that $r_0:=\rank(G_i)$ and $\lambda:=c_1(G_i)$, for $1\leq i\leq n$. Let $S^n\LeftArrowOf{b} \Gamma \RightArrowOf{q} S^{[n]}$ be the morphisms from the  isospectral Hilbert scheme given in (\ref{eq-b-q}). Let $\theta:H^2(S,\Integers)\rightarrow H^2(S^{[n]},\Integers)$ be the homomorphism given in (\ref{eq-theta}) and let $\delta$ be half the class of the divisor of non-reduced subschemes. 

\begin{lem}
\label{lemma-vector-bundle-over-S[n]-associated-to-a-point-in-M[n]}
$F_z$ is a locally free sheaf  of rank $n!r_0^n$, which is isomorphic to $q_*b^*(G_1\boxtimes \cdots \boxtimes G_n),$
and $c_1(F_z)=n!r_0^{n-1}\left(\theta(\lambda)-r_0\delta/2\right).$
\end{lem}

\begin{proof}
Set $G_\sigma:=G_{\sigma(1)}\boxtimes \cdots \boxtimes G_{\sigma(n)}$. Denote the automorphisms of $S^n$ and of $\Gamma$ associated to a permutation $\sigma$ by $\sigma$ as well.
The equality $q=q\circ\sigma$ and the isomorphism $\rho_{\sigma^{-1}}:G_\sigma\rightarrow\sigma_*(G_{id})$ induced by  the linearization give rise to the natural isomorphism
\[
q_*b^*(\rho_{\sigma^{-1}}^{-1}): q_*b^*G_{id}\rightarrow q_*b^*G_\sigma.
\]
The invariant subsheaf $F_z$ of $\oplus_{\sigma\in\fS_n}q_*b^*G_\sigma$ is hence the diagonal embedding of $q_*b^*G_{id}$
into the direct sum. 
$F_z$ is locally free, by the flatness of the morphism $q$ (see the proof of \cite[Prop. 3.7.4]{haiman}). 
The computation of $c_1(F_z)$ is carried out in Corollary \ref{cor-vector-bundle-over-S[n]-associated-to-a-point-in-M[n]}.
\end{proof}

\hide{
Let $Y\subset S$ be a length $n$ subscheme supported on $n$ distinct points $\{y_1, \dots, y_n\}$. Let $y\in S^{[n]}$ be the point representing $Y$. Let $G_{i,y_{j}}$ be the fiber of $G_i$ over $y_j$. 
The fiber $F_{z,y}$ of $F_z$ over $y$ is 
 \[
F_{z,y}= \left(\bigoplus_{\tau\in\fS_n}\bigoplus_{\sigma\in\fS_n}\left(G_{\sigma(1),y_{\tau(1)}}\otimes \ \cdots \ \otimes G_{\sigma(n),y_{\tau(n)}}\right)\right)^{\fS_n}.
 \]
The  $\fS_n$-linearization of $G_z$ identifies the fiber 
$G_{\sigma(1),y_1}\otimes \cdots \otimes G_{\sigma(n),y_n}$ over $(y_1, \dots, y_n)$ of the direct summand of $G_z$ corresponding to $\sigma$ with the fiber $G_{\sigma(\tau(1)),y_{\tau(1)}}\otimes \cdots \otimes G_{\sigma(\tau(n)),y_{\tau(n)}}$ 
 over $(y_{\tau(1)}, \dots, y_{\tau(n)})$ of the direct summand corresponding to $\sigma\circ \tau$ by permutation of the factors, for all $\sigma,\tau\in\fS_n$. Hence, 
 \begin{eqnarray*}
F_{z,y}&\cong & \left(\bigoplus_{\tau\in\fS_n}\bigoplus_{\sigma\in\fS_n}\left(G_{\sigma(\tau^{-1}(1)),y_1}\otimes \ \cdots \ \otimes G_{\sigma(\tau^{-1}(n)),y_n}\right)\right)^{\fS_n}
\\
& \stackrel{\gamma=\sigma\circ\tau^{-1}}{\cong} &
\left(\bigoplus_{\tau\in\fS_n}\bigoplus_{\gamma\in\fS_n}\left(G_{\gamma(1),y_1}\otimes \ \cdots \ \otimes G_{\gamma(n),y_n}\right)\right)^{\fS_n}.
 \end{eqnarray*}
 and
the fiber $F_{z,y}$ is naturally isomorphic to
$
\oplus_{\gamma\in\fS_n}\left(G_{\gamma(1),y_1}\otimes \ \cdots \ \otimes G_{\gamma(n),y_n}\right).
$
}

%
\subsubsection{A relative version of the construction}
Assume next that $M$ is a two dimensional projective moduli space of sheaves with a primitive Mukai vector $v$ which are 
$H$-stable with respect to a $v$-generic polarization $H$ on $S$. Assume that a universal sheaf $\G$ exists over $M\times S$.
Then $\G$ is the Fourier-Mukai kernel of an equivalence $\Phi_\G:D^b(M)\rightarrow D^b(S)$.
Let $M^n\LeftArrowOf{b} \Gamma_M \RightArrowOf{q}M^{[n]}$ be the projections from the isospectral Hilbert scheme of the $K3$ surface $M$. The exterior product $\G^{\boxtimes n}$ over $M^n\times S^n$ admits a natural linearization $\rho_\boxtimes$
with respect to the diagonal action of $\fS_n$.
A relative version of the above construction produces the object
\[
\tilde{\F}:=(\StructureSheaf{\Gamma_M},\rho)\circ (\G^{\boxtimes n},\rho_\boxtimes)\circ (\StructureSheaf{\Gamma},\rho)
\]
over $M^{[n]}\times S^{[n]}$, which is the Fourier-Mukai kernel of the composition 
$\tilde{\Phi}_\G^{[n]}:=BKR\circ \Phi^{\boxtimes n}_\G\circ bkr$. Using \cite[Lemma 5(3)]{ploog} we get the isomorphism 
\begin{equation}
\label{eq-universal-sheaf-over-product-of-Hilbert-schemes}
\tilde{\F} \cong R(q\times q)_*^{\fS_{n,\Delta}}\left((b\times b)^*(\G^{\boxtimes n},\rho_\boxtimes\right),
\end{equation}
which is locally free by the flatness of $q\times q$. 

The isospectral Hilbert scheme $\Gamma$ is Gorenstein, by \cite[Theorem 3.1]{haiman},  and its dualizing sheaf $\omega_\Gamma$ is the line bundle $q^*\StructureSheaf{S^{[n]}}(\delta)$, by 
\cite[Theorem 3.1, Prop. 3.4.3, comment after Lemma 3.4.2]{haiman}. 
Hence, the right adjoint $q^!$ of $Rq_*$ is  $\omega_\Gamma\otimes Lq^*$ and so
\[
BKR^{-1}(\bullet) \cong (Rq_*^{\fS_n}\circ L_b^*)^{-1}(\bullet)\cong (Rb_*\circ q^!)(\bullet)\cong
Rb_*((\omega_\Gamma,\rho_\omega)\otimes Lq^*)(\bullet),
\]
where $\rho_{\omega}$ is the natural $\fS_n$-linearization of $\omega_\Gamma$. We have
\begin{equation}
\label{eq-linearized-dualizing-sheaf-of-isospectral-Hilbert-scheme}
(\omega_\Gamma,\rho_\omega)\cong \chi\otimes Lq^*\StructureSheaf{S^{[n]}}(\delta),
\end{equation} 
since $Lq^*\StructureSheaf{S^{[n]}}(\delta)$, with its induced linearization, does not have an invariant global section, while $(\omega_\Gamma,\rho_\omega)$ does.
Hence,
\begin{equation}
\label{eq-relation-betweek-BKR-inverse-and-bkr}
BKR^{-1}(\bullet) \cong \chi\otimes (Rb_*\otimes Lq^*)(\StructureSheaf{S^{[n]}}(\delta)\otimes (\bullet)) = 
\chi\otimes bkr(\StructureSheaf{S^{[n]}}(\delta)\otimes (\bullet)).
\end{equation}
We conclude that the kernel of $\Phi^{[n]}:=BKR\circ \Phi^{\boxtimes n}_\G\circ BKR^{-1}:D^b(M^{[n]})\rightarrow D^b(S^{[n]})$ is the locally free
\begin{equation}
\label{eq-second-universal-sheaf-over-product-of-Hilbert-schemes}
\F \cong \pi_{M^{[n]}}^*\left(\StructureSheaf{M^{[n]}}(\delta)\right)\otimes R(q\times q)_*^{\fS_{n,\Delta}}\left((b\times b)^*(\G^{\boxtimes n},\chi\otimes\rho_\boxtimes)\right).
\end{equation}
Equation (\ref{eq-relation-betweek-BKR-inverse-and-bkr}) and Remark \ref{rem-chi-invariant-objects} yield the isomorphism 
$bkr(\StructureSheaf{z})\cong BKR^{-1}(\StructureSheaf{z})$,  
for sky-scraper sheaves of points $z\in M^{[n]}\setminus D_{M^{[n]}}$, where $D_{M^{[n]}}$ is the divisor of non-reduced subschemes.  Consequently, 
\begin{equation}
\label{eq-F-z-is-image-of-sky-scpaper-via-Phi[n]}
F_z\cong\tilde{\Phi}^{[n]}(\StructureSheaf{z})\cong \Phi^{[n]}(\StructureSheaf{z}),
\end{equation}
for $z\in M^{[n]}\setminus D_{M^{[n]}}$. 

\begin{rem}
Let $X$ and $Y$ be two projective irreducible holomorphic symplectic manifolds.
Is every rational Hodge isometry $f:H^2(X,\RationalNumbers)\rightarrow H^2(Y,\RationalNumbers)$
induced by an algebraic correspondence? Shafarevich asked this question in case $X$ and $Y$ are $K3$ surfaces. Buskin, building on results of Mukai, proved that every rational Hodge isometry is algebraic \cite{mukai-hodge,buskin-thesis} (see \cite{huybrechts-rational-hodge-isometries} for another proof). Buskin's approach uses the fact the the universal bundle $\G$ as above is hyperholomorphic and it deforms over
the graph of the automorphism of the period domain of K\"{a}hler $K3$-surfaces induced by an isometry $g:H^2(M,\RationalNumbers)\rightarrow H^2(S,\RationalNumbers)$ associated to $\G$. It is natural to expect that both the vector bundle $\tilde{\F}$ given in (\ref{eq-universal-sheaf-over-product-of-Hilbert-schemes}) and 
$\F$ given in (\ref{eq-second-universal-sheaf-over-product-of-Hilbert-schemes}) 
have an analogous property and that the techniques  of \cite{buskin-thesis} would establish the algebraicity of many rational Hodge isometries between the second cohomologies of irreducible holomorphic symplectic manifolds of $K3^{[n]}$-type. We plan to return to this question in the future.
\end{rem}

%
\subsubsection{The Chern character in terms of the rank and first Chern class}
\begin{lem}
\label{lemma-isotropic-Mukai-line-all-n}
Let $X$ be of $K3^{[n]}$-type. Let $F$ be a vector bundle over $X$ which is obtained via a sequence of equivalences of derived categories and deformations from a sky scraper sheaf $\StructureSheaf{z}$, $z\in Y$, where $Y$ is of $K3^{[n]}$-type as well.
Then $v(F)$ belongs to $SH^*(X,\QQ)$. Set $h:=c_1(F)$. The rank of $F$ is equal $n!r_0^n$, for some rational number $r_0$, and
$
\Psi(v(F))=\gamma^n,
$
where
\begin{equation}
\label{eq-isotropic-gamma-for-all-n}
\gamma=r_0\alpha+\frac{h}{n!r_0^{n-1}}+\frac{(h,h)}{2(n!)^2r_0^{2n-1}}\beta
\end{equation}
is an isotropic class. In particular, $\ell(F)=\span_\QQ\{\gamma\}.$
\end{lem}

\begin{proof}
Let $\phi:H^*(Y,\QQ)\rightarrow H^*(X,\QQ)$ be the derived parallel transport operator associated 
to the sequence of derived equivalences and deformations used to obtain $F$ from the sky scraper sheaf.
Let $\widetilde{H}(\phi):\tilde{H}(Y,\QQ)\rightarrow \tilde{H}(X,\QQ)$ be the 
image of $\phi$ via the functor (\ref{eq-functor-tilde-H}).
Then $v(F)=\phi(v(\StructureSheaf{z}))=\phi([pt])$, the class $[pt]$ belongs to $SH^*(Y,\QQ)$, and
$\phi$ restricts to an isomorphism from $SH^*(Y,\QQ)$ onto $SH^*(Y,\QQ)$, by \cite[Theorem B]{taelman}. Hence, $v(F)$
belongs to $SH^*(Y,\QQ)$.

Set $\gamma:=(\widetilde{H}(\phi))(\beta)$. Then $\Psi(v(F))=\epsilon(\widetilde{H}(\phi))\gamma^n$, by the equality $\Psi(v(\StructureSheaf{z}))=\beta^n$ of Example \ref{example-Mukai-line-of-sky-scraper-sheaf} and the equivariance of $\Psi$ with respect to derived parallel transport operators (Theorem \ref{thm-def-of-H-tilde}). 
When $n$ is odd $\epsilon(\widetilde{H}(\phi))=1$, by definition, and 
the fact that the rank of $F$ is positive implies that $\epsilon(\widetilde{H}(\phi))=1$ also when $n$ is even, by Remark \ref{rem-functor-chi}.
Write $\gamma=r_0\alpha+\lambda+s\beta$. Then
\[
\Psi(v(F))=\gamma^n=\sum_{k=0}^n\Choose{n}{k}r_0^{n-k}\lambda^k\alpha^{n-k}+\sum_{k=1}^n\Choose{n}{k}s^k(r_0\alpha+\lambda)^{n-k}\beta^k.
\]
The coefficient of $\alpha^n$ on the left hand side is $\rank(F)/n!$ and on the right hand side it is $r_0^n$, so
$\rank(F)=n!r_0^n$. 
The coefficient in $H^2(X,\QQ)$ of $\alpha^{n-1}$ on the left hand side is $h/(n-1)!$, since $\Psi(h)=e_h(\alpha^n/n!)=h\alpha^{n-1}/(n-1)!$, and on the right hand side it is $nr_0^{n-1}\lambda$, so $\lambda=h/(n!r_0^{n-1})$. 
We have
\[
0=(\gamma,\gamma)=\lambda^2-2r_0s=\frac{(h,h)}{(n!)^2r_0^{2n-2}}-2r_0s.
\]
We get that $s=\frac{(h,h)}{2(n!)^2r_0^{2n-1}}$ and Equation (\ref{eq-isotropic-gamma-for-all-n}) is verified.
\end{proof}

\begin{lem}
\label{lemma-isotropic-LLV-line-and-chern-character}
Keep the notation and hypotheses of Lemma \ref{lemma-isotropic-Mukai-line-all-n} and assume that $n=2$. Then 
$\gamma=r_0\alpha+\frac{h}{2r_0}+\frac{(h,h)}{8r_0^3}\beta$ is a primitive isotropic class of divisibility $1$ in the LLV lattice $\Lambda_X$ given in (\ref{eq-Mukai-lattice-K3-2-case}). In particular, $r_0$ is an integer,
$h=r_0\gcd(2,r_0)\psi$, for a class $\psi\in H^2(X,\ZZ)$ such that $2$ divides $\div(\psi)$ if $r_0$ is odd, $\frac{2r_0}{\gcd(2,r_0)^2}$ divides $(\psi,\psi)$, and
$\frac{(\psi,\psi)\gcd(2,r_0)^2}{2r_0}\equiv r_0$ (mod $4$).
Furthermore,
\begin{eqnarray*}
ch_2(F) & = & \frac{1}{4r_0^2}h^2-\frac{r_0^2}{12}c_2(X)
\\
ch_3(F) & = & \frac{1}{24 r_0^4}h^3-\frac{1}{24}c_2(X)h
\\
ch_4(F) & = & \left[
\frac{(h,h)^2}{2^6r_0^6} 
-\frac{5(h,h)}{16 r_0^2}
+\frac{21r_0^2}{16}
\right][pt]
\\
\chi(F) & = & \left(
\frac{(h,h)+10r_0^4}{8r_0^3}
\right)^2.
\end{eqnarray*}
\end{lem}

\begin{proof}
We have $\Psi(v(F))=\gamma^n,$ where 
$\gamma=(\widetilde{H}(\phi))(\beta)=r_0\alpha+\frac{h}{2r_0}+\frac{(h,h)}{8r_0^3}\beta$, by
Lemma \ref{lemma-isotropic-Mukai-line-all-n}. We get that
\[
\gamma=B_{-\delta/2}\left(
r_0\alpha+\frac{1}{2r_0}h+\frac{r_0}{2}\delta+\left[\frac{(h,h)}{8r_0^3}+\frac{(h,\delta)}{4r_0}-\frac{r_0}{4}\right]\beta
\right).
\]
Hence, $r_0$ is an integer, $\eta=\frac{1}{2r_0}h+\frac{r_0}{2}\delta$ is a class in $H^2(X,\ZZ)$, and 
$h=r_0(2\eta-r_0\delta)$. Set $\psi:=[2\eta-r_0\delta]/\gcd(2,r_0)$. Then $2$ divides $\div(\psi)$, if $r_0$ is odd. 
Furthermore, the coefficient 
of $\beta$ in the above displayed formula is an integer equal to 
\[
\left[\frac{(h,h)}{8r_0^3}+\frac{(h,\delta)}{4r_0}-\frac{r_0}{4}\right]=
\frac{r_0^2[4(\eta,\eta)-4r_0(\eta,\delta)-2r_0^2]}{8r_0^3}+
\frac{2r_0[(\eta,\delta)+r_0]}{4r_0}-\frac{r_0}{4}=\frac{(\eta,\eta)}{2r_0}.
\]
Hence, 
$2r_0$ divides $(\eta,\eta)$. Now, $(\psi,\psi)=\frac{4(\eta,\eta)-2r_0^2}{\gcd(2,r_0)^2}$ and so $2r_0/\gcd(2,r_0)$ divides $(\psi,\psi)$. We have
\[
\frac{\gcd(2,r_0)^2(\psi,\psi)}{2r_0}=4\frac{(\eta,\eta)}{2r_0}-r_0,
\]
hence, $\frac{(\psi,\psi)}{2r_0}\equiv r_0$ (mod $4$).

We have $\Psi(h)=h\alpha$, $\Psi(h^2)=h^2+(h,h)\alpha\beta$, and $\Psi(h^3)=3(h,h)h\beta$. Hence,
\begin{eqnarray*}
\gamma^2 & = & r_0^2\alpha^2+h\alpha+\frac{1}{4r_0^2}h^2+\frac{(h,h)}{4r_0^2}\alpha\beta
+\frac{(h,h)}{8r_0^4}h\beta+\left(\frac{(h,h)}{8r_0^3}\right)^2\beta^2
\\
&=& \Psi\left(
2r_0^2+h+\frac{1}{4r_0^2}h^2+\frac{1}{24r_0^4}h^3+\frac{(h,h)^2}{2^6r_0^6}[pt]
\right)=\Psi(v(F)).
\end{eqnarray*}
Using $c_2(X)h^2=30(h,h)[pt]$ we get
\begin{eqnarray*}
ch(F)&=&v(F)(\sqrt{td_X})^{-1}=
\left(2r_0^2+h+\frac{1}{4r_0^2}h^2+\frac{1}{24r_0^4}h^3+\frac{(h,h)^2}{2^6r_0^6}[pt]\right)
(1-\frac{1}{24}c_2(X)+\frac{21}{32}[pt])
\\
&=&
2r_0^2+h+\left[
\frac{1}{4r_0^2}h^2-\frac{r_0^2}{12}c_2(X)\right]+
\left[\frac{1}{24r_0^4}h^3-\frac{1}{24}c_2(X)h\right]+
\left[
\frac{(h,h)^2}{2^6r_0^6}-\frac{5(h,h)}{16r_0^2}+\frac{21r_0^2}{16}
\right][pt]
\end{eqnarray*}
The formula for $\chi(F)$ follows, by Grothendieck-Riemann-Roch.
\end{proof}

%
\subsection{Proof of Theorem \ref{thm-Fourier-Mukai-images-of-sky-scraper-sheaves}}
\label{sec-proof-of-modularity-of-images-of-sky-scraper-sheaves}
The following will be needed for the proof of Theorem \ref{thm-Fourier-Mukai-images-of-sky-scraper-sheaves}.

\begin{lem}
\label{lemma-inflation-of--tensor-outer-product-is-stable}
Let $X$ be a connected compact K\"{a}hler manifold with a trivial Picard group.
Let $V_i$, $1\leq i\leq n$, be vector bundles, possibly twisted, but if so then with respect to the same Brauer class.
Assume that $V_i$ does not have any non-zero saturated proper  subsheaf, $1\leq i \leq n$.
Assume that $V_i$ and $V_j$ are not isomorphic, if $i\neq j$.
Let $\fS_n$ act on $X^n$ by permuting the factors. Choose the linearization $\rho$ of
$W:=\oplus_{\sigma\in \fS_n}V_{\sigma(1)}\boxtimes \cdots \boxtimes V_{\sigma(n)}$ permuting the factors of each direct summand.
Then $W$ does not have any $\fS_n$-invariant non-zero saturated proper subsheaf.
\end{lem}

\begin{proof}
The direct summands of $W$ are stable, by Proposition \ref{prop-tensor-outer-product-is-stable}, are pairwise non-isomorphic, and $W$ is polystable with respect to every K\"{a}hler class on $X^n$.
Let $E$ be a saturated subsheaf of $W$. Then $\det(\SheafHom(E,W))$ is the trivial line bundle, since $\Pic(X^n)$ is trivial.
Hence, $E$ is a direct sum of a subset of the direct summands of $W$. If $E$ is $\fS_n$-invariant, then this subset must be the whole of $\fS_n$ and so $E=W$.
\end{proof}

\begin{proof} (of Theorem \ref{thm-Fourier-Mukai-images-of-sky-scraper-sheaves}). 
Being $H$-slope-stable, each $G_i$ is projectively $H$-hyperholomorphic \cite{kaledin-verbitsky-book} (see also \cite{huybrechts-schroer}). 
The proof is the same as that of Theorem \ref{thm-BKR-of-tensor-product-of-spherical-object}, except that one replaces the use of
Proposition \ref{prop-tensor-outer-product-is-stable} by that of Lemma \ref{lemma-inflation-of--tensor-outer-product-is-stable}
and the reference to the paragraph before Theorem \ref{thm-BKR-of-tensor-product-of-spherical-object} by that before Theorem \ref{thm-Fourier-Mukai-images-of-sky-scraper-sheaves}.
\end{proof}
%
\section{The action of the derived monodromy group of $S$ on the LLV lattice of $S^{[n]}$}
We prove in this section Theorem \ref{thm-introduction-action-of-DMon-S-on-LLV-lattice}
describing the composite homomorphism
$(\widetilde{H}\circ ^{[n]}):DMon(S)\rightarrow O(\widetilde{H}(S^{[n]},\QQ))$ from the derived monodromy group of a $K3$ surface to the isometry group of the rational LLV lattice of its $n$-th Hilbert scheme.

Let $\Phi^{[n]}_\chi:D^b(S^{[n]})\rightarrow D^b(S^{[n]})$ be the auto-equivalence (\ref{eq-Phi-chi-[n]}), which is
BKR-conjugate to tensorization by the sign character $\chi$. 
Let
$\phi^{[n]}_\chi:H^*(S^{[n]},\QQ)\rightarrow H^*(S^{[n]},\QQ)$ be the derived monodromy operator corresponding to  $\Phi^{[n]}_\chi$.
Let $\widetilde{H}(\phi^{[n]}_\chi):\widetilde{H}(S^{[n]},\QQ)\rightarrow \widetilde{H}(S^{[n]},\QQ)$ be the induced isometry via 
(\ref{eq-action-of-DMon-on-Mukai-lattice}). Clearly, $(\widetilde{H}(\phi^{[n]}_\chi))^2=id$.
Let $\delta\in H^{1,1}(S^{[n]},\Integers)$ be half the class of the divisor of non-reduced subschemes.
Denote by $(r,\lambda,s)$ the element $r\alpha+\lambda+s\beta$ of $\widetilde{H}(S^{[n]},\QQ)$
\begin{equation}
\label{eq-notation-for-element-of-LLV-lattice}
(r,\lambda,s):=r\alpha+\lambda+s\beta.
\end{equation}

\begin{lem}
\label{lemma-isometry-of-BKR-conjugate-of-tensorization-by-sign-character}
The isometry $\widetilde{H}(\phi^{[n]}_\chi)$ is $(-1)^{n+1}$ times  the reflection with respect to the hyperplane orthogonal to $u_0:=(0,\delta,n-1)=B_{-\delta/2}(0,\delta,0)$,
\[
\widetilde{H}(\phi^{[n]}_\chi)(x) = (-1)^{n+1}\left[x+\frac{(x,u_0)}{n-1}u_0\right].
\]
\end{lem}

\begin{proof}
Let $e:\Gamma\rightarrow S^n\times S^{[n]}$ be the inclusion. 
The kernel of the functor $BKR^{-1}$ is isomorphic to $(\iota_*\StructureSheaf{\Gamma})^\vee[2n]$, the derived dual $(\iota_*\StructureSheaf{\Gamma})^\vee$ is $\iota_*\omega_\Gamma[-2n]$, and $\omega_\Gamma$ is a line bundle, since $\Gamma$ is a normal Cohen-Macaulay and Gorenstein variety, by \cite[Theorem 3.1]{haiman}. 
So $BKR^{-1}$ has Fourier-Mukai kernel $\iota_*\omega_\Gamma$, while the functor $bkr$ given in (\ref{eq-usual-bkr}),
has kernel $\iota_*\StructureSheaf{\Gamma}$, each with its natural linearization.
 Let $p\in S^{[n]}$ correspond to a reduced subscheme supported on the subset $\{x_1, \ \dots, \ x_n\}$ of $S$, $x_i\neq x_j$ for $i\neq j$. The restriction of $\omega_\Gamma$ to the $\fS_n$-orbit of $(x_1, \ \dots, \ x_n)$ is trivial.
 Hence, $BKR^{-1}(\CC_p)\cong bkr(\CC_p)\cong 
 \left(\oplus_{\sigma\in\fS_n}\CC_{(x_{\sigma(1)}, \ \dots, \ x_{\sigma(n)})},\rho\right)$, for a linearization $\rho$, which is isomorphic to
 $\left(\CC_{(x_{\sigma(1)}, \ \dots, \ x_{\sigma(2)})},\chi\otimes \rho\right)$ by Remark \ref{rem-chi-invariant-objects}, so $\Phi^{[n]}_\chi(\CC_p)\cong\CC_p$.
Hence, $\widetilde{H}(\phi^{[n]}_\chi)$ maps the line $\ell(\CC_p)=\span_\QQ\{(0,0,1)\}$ to itself, $\Sym^n(\widetilde{H}(\phi^{[n]}_\chi))$ maps
$\Psi(v(\CC_p))$ to itself, and $\phi^{[n]}_\chi(v(\CC_p))=v(\CC_p)$. If $n$ is even, it follows that $\det(\widetilde{H}(\phi^{[n]}_\chi))=1$, by the equality
$\Sym^n(\widetilde{H}(\phi^{[n]}_\chi))\circ\Psi=\Psi\circ [\det(\widetilde{H}(\phi^{[n]}_\chi))\phi^{[n]}_\chi]$ in Theorem \ref{thm-def-of-H-tilde}.
If $n$ is odd, it follows that $\ell(\CC_p)$ is an eigenline of $\widetilde{H}(\phi^{[n]}_\chi)$ with eigenvalue $1$, as the eigenvalue is $\pm 1$ as well as an $n$-th root of unity  
by the equality $\Sym^n(\widetilde{H}(\phi^{[n]}_\chi))\circ\Psi=\Psi\circ \phi^{[n]}_\chi$
in Theorem \ref{thm-def-of-H-tilde}.

Note that the action of $\widetilde{H}(\phi^{[n]}_\chi)$ is independent of the complex structure of the $K3$ surface $S$.
Furthermore, $\widetilde{H}(\phi^{[n]}_\chi)$ commutes with $\widetilde{H}(g^{[n]}):\widetilde{H}(S^{[n]},\QQ)\rightarrow \widetilde{H}(S^{[n]},\QQ)$, for every $g$ in $DMon(S)$, by Equation (\ref{eq-Phi-chi-[n]-commutes-with-surface-FM}). 
Let $V$ be the subspace of $\widetilde{H}(S^{[n]},\QQ)$ spanned by the $DMon(S)$ orbit of the line 
$\ell(\CC_p)=\span_\QQ\{(0,0,1)\}$. The subspace 
$V$ is $\widetilde{H}(\phi^{[n]}_\chi)$-invariant and $\widetilde{H}(\phi^{[n]}_\chi)$  acts on $V$ as the identity, if $n$ is odd, and $V$ is contained in either the $1$ or $-1$ eigenspace of $\widetilde{H}(\phi^{[n]}_\chi)$, if $n$ is even.

Assume that $S$ admits an elliptic fibration
$\pi:S\rightarrow\PP^1$  and let $E_1$, \dots, $E_n$ be $n$ distinct smooth fibers. The abelian variety 
$A:=E_1\times \cdots \times E_n$ embeds as a lagrangian subvariety of $S^{[n]}$. Let $\iota:A\rightarrow S^{[n]}$ be the embedding. 
Set $\tilde{A}:=\cup_{\sigma\in\fS_n}E_{\sigma(1)}\times \cdots \times E_{\sigma(n)}$. 
We get the equivariant embedding $\tilde{\iota}:\tilde{A}\rightarrow \Gamma$ into the universal subscheme $\Gamma\subset S^{[n]}\times S^n$.
Let $b:\Gamma\rightarrow S^n$ and $q:\Gamma\rightarrow S^{[n]}$ be the restriction to $\Gamma$ of the two projections.
Then $\tilde{A}$ is contained in the open subset where $b$ is an isomorphism and $q$ is \'{e}tale.
Given a line bundle $\LB$ on $S^{[n]}$, then 
$Lq^*(\iota_*\iota^*\LB)\cong \tilde{\iota}_*\tilde{\iota}^*(Lq^*\LB)$ in $D^b_{\fS_n}(\Gamma)$. Denote the linearization of $Lq^*\LB$ by $\lambda$. Then $\tilde{\iota}_*(\tilde{\iota}^*Lq^*\LB,\tilde{\iota}^*\lambda)$ is isomorphic to $\tilde{\iota}_*(\tilde{\iota}^*Lq^*\LB,\chi\otimes \tilde{\iota}^*\lambda)$, by Remark \ref{rem-chi-invariant-objects}, as the irreducible components $E_{\sigma(1)}\times \cdots \times E_{\sigma(n)}$  are disjoint for different permutations $\sigma\in \fS_n$. Furthermore,
$Rb_*^{\fS_{n,\Delta}}\tilde{\iota}_*(\tilde{\iota}^*Lq^*\LB,\chi\otimes \tilde{\iota}^*\lambda)$ is isomorphic to 
$\chi\otimes Rb_*^{\fS_{n,\Delta}}\tilde{\iota}_*(\tilde{\iota}^*Lq^*\LB,\tilde{\iota}^*\lambda)$, since $b$ is an equivariant 
morphism with respect to the isomorphism $\fS_n\rightarrow \fS_n$.
It follows that
$BKR^{-1}(\iota_*\iota^*\LB)\cong \chi\otimes BKR^{-1}(\iota_*\iota^*\LB)$. Thus, $\Phi^{[n]}_\chi(\iota_*\iota^*\LB)\cong \iota_*\iota^*\LB$.
Now, $\ell(\StructureSheaf{A})=\span\{(0,f,0)\}$, where $f$ is the class of a fiber of $\pi$ in $H^2(S,\Integers)$ considered as the subspace of $H^2(S^{[n]},\ZZ)$ orthogonal to $\delta$, by Lemma \ref{lemma-Mukai-line-of-structure-sheaf-of-subcanonical-lagrangian}. Thus, 
$\ell(\iota_*\iota^*\LB)=\exp(e_{c_1(\LB)})(0,f,0)=(0,f,(f,c_1(\LB)))$. 
We see that $\widetilde{H}(\phi^{[n]}_\chi)$ leaves invariant the lines $\span\{(0,f,s)\}$, whenever $f$ is an isotropic primitive class in $H^2(S,\ZZ)$ and there is a complex structure on $S$ such that $f$ is the class of a fiber of an elliptic fibration and there exists a line bundle $\LB$  on $S$, such that $s=(f,c_1(\LB))$. 
We conclude that $\widetilde{H}(\phi^{[n]}_\chi)$ acts by multiplication by $1$ or $-1$
on $W:=\{(0,0,1),(0,\delta,0)\}^\perp$. 

Assume that the elliptic fibration $\pi$ has a section and all its fibers are integral. Then $S$ is isomorphic to the moduli space $M_H(0,f,0)$
of $H$-stable  sheaves $F$  of rank $0$, $c_1(F)=f$, and $\chi(F)=0$, for a suitable choice of a polarization $H$ on $S$
\cite{briedgeland-elliptic-surfaces}. Furthermore, there exists over $M_H(0,f,0)\times S$ a universal sheaf $\P$ and
$\Phi_\P:D^b(M_H(0,f,0))\rightarrow D^b(S)$ is an equivalence \cite{briedgeland-elliptic-surfaces}.
Choose the length $n$ subscheme $p$ of $S$ to consists of the $n$ distinct points $p_i$ of intersection of the section with the fibers $E_i$, $1\leq i\leq n$. The  Fourier Mukai transform
$\Phi_\P:D^b(S)\cong D^b(M_H(0,f,0))\rightarrow D^b(S)$ 
satisfies $\Phi_\P^{[n]}(\CC_{p_i})\cong \StructureSheaf{E_i}$. Hence,
$\Phi_\P^{[n]}(\CC_p)\cong \iota_*\StructureSheaf{A}$. We conclude that $W$ is contained in $V$.

Note that $V$ contains the LLV line of the vector bundle (\ref{eq-G-Z}). The latter line is not contained in $W$, by Lemma \ref{lemma-isotropic-LLV-line-and-chern-character}. Hence, the co-dimension of $V$ in $\tilde{H}(S^{[n]},\QQ)$ is at most $1$. 
Recall that $\ell(\StructureSheaf{S^{[n]}})=\span\{(4,0,n+3)\}$, by Lemma \ref{lemma-LLV-line-of-structure-sheaf-k3-type}. 
Equation (\ref{eq-Phi-chi-sends-structure-sheaf-to-LB}) implies that $\widetilde{H}(\phi^{[n]}_\chi)$  interchanges the two lines 
$\ell(\StructureSheaf{S^{[n]}})$ and $\ell(\StructureSheaf{S^{[n]}}(-\delta))=\span\{(4,-4\delta,7-3n)\}$.
Hence, $V$ is a co-dimension $1$ subspace of $\tilde{H}(S^{[n]},\QQ)$. Furthermore, $\widetilde{H}(\phi^{[n]}_\chi)$  is neither $id$, nor $-id$.
Hence, $\widetilde{H}(\phi^{[n]}_\chi)$ is either the reflection $R_V$ in $V$ or $-R_V$.
The reflection $R_V$ must map $(4,0,n+3)$ to $\pm(4,-4\delta,7-3n)$, as the two have the same  self-intersection. 
Furthermore, $(4,0,n+3)+R_V((4,0,n+3))$ belongs to $V$. 
If $R_V((4,0,n+3))=-(4,-4\delta,7-3n)$, then $(4,0,n+3)+R_V((4,0,n+3))=4(0,\delta,n-1)$ and so $V=(0,0,1)^\perp$.
This contradicts the fact that $V$ contains LLV lines of vector bundles of positive rank as in Lemma \ref{lemma-isotropic-LLV-line-and-chern-character}.
Hence, $R_V((4,0,n+3))=(4,-4\delta,7-3n)$, and $V=(0,\delta,n-1)^\perp$, since
$(0,\delta,n-1)=\frac{1}{4}[(4,0,n+3)-R_V((4,0,n+3))]$.
When $n$ is even we observed above that  $\det(\widetilde{H}(\phi^{[n]}_\chi))=1$ and so $\widetilde{H}(\phi^{[n]}_\chi)=-R_V$. When $n$ is odd 
we observed above that $\widetilde{H}(\phi^{[n]}_\chi)$  maps $(0,0,1)$ to itself, and so $\widetilde{H}(\phi^{[n]}_\chi)=R_V$.
\end{proof}

Let $v=(1,0,1-n)$ be the Mukai vector of the ideal sheaf of a length $n$ subscheme of a $K3$ surface $S$.
Then the co-rank $1$ sublattice $v^\perp$ of the Mukai lattice $\widetilde{H}(S,\ZZ)$ contains $H^2(S,\ZZ)$.
Let $\theta:H^2(S,\ZZ)\rightarrow H^2(S^{[n]},\ZZ)$ be the homomorphism (\ref{eq-theta}).
Denote by 
$\theta:H^2(S,\QQ)\rightarrow H^2(S^{[n]},\QQ)$ the induced homomorphism as well.
Let $\tilde{\theta}:\widetilde{H}(S,\QQ)\rightarrow \widetilde{H}(S^{[n]},\QQ)$ be the extension mapping $\alpha$ to $\alpha$ and
$\beta$ to $\beta$. 
Let $\eta:O(\widetilde{H}(S,\QQ))\rightarrow O(\widetilde{H}(S^{[n]},\QQ))$ be the homomorphism extending an isometry by acting 
as the identity on the orthogonal complement $\span_\QQ\{(0,\delta,0)\}$ of the image of $\tilde{\theta}$. Explicitly, 
$\eta_g(\tilde{\theta}(x))=\tilde{\theta}(g(x))$, and $\eta_g(0,\delta,0)= (0,\delta,0)$, for all $g\in O(\widetilde{H}(S,\QQ))$.

The following was proven for $n=2$ in \cite[Theorem 9.4]{taelman}.
Given an element $\phi$ of $DMon(S)$ denote by $\phi^{[n]}$ its image in $DMon(S^{[n]})$ via the functor 
(\ref{eq-functor-[n]}). Denote the image of $\phi^{[n]}$ in $O(\tilde{H}(S^{[n]},\QQ))$ via (\ref{eq-functor-tilde-H}) by $\widetilde{H}(\phi^{[n]})$. Define $\widetilde{H}(\phi)\in O(\tilde{H}(S,\QQ))$ similarly. Set $\eta_\phi:=\eta_{\widetilde{H}(\phi)}$, for $\phi\in DMon(S)$.

\begin{thm}
\label{thm-action-of-DMon-S-on-LLV-lattice}
$\widetilde{H}(\phi^{[n]})=\det(\widetilde{H}(\phi))^{n+1}\left(
B_{-\delta/2}\circ \eta_\phi\circ B_{\delta/2}
\right).$
\end{thm}

\begin{proof} Recall that $DMon(S)=O^+(\widetilde{H}(S,\ZZ))$, by \cite[Prop. 9.2]{taelman}.
Let us first verify the theorem  for all $\phi\in O^+(\widetilde{H}(S,\QQ))$ satisfying $\phi(\alpha)=\alpha$ and $\phi(\beta)=\beta$. 
Note that in this case $B_{-\delta/2}\circ \eta_\phi\circ B_{\delta/2}=\eta_\phi$.
Such $\phi$ belongs to the monodromy group $O^+(H^2(S,\ZZ))$ of $S$. Then the image $\phi^{[n]}$ of $\phi$ in the monodromy group of $S^{[n]}$ acts on $H^2(S^{[n]},\ZZ)$ via $\phi^{[n]}(\theta(x))=\theta(\phi(x))$ and $\phi^{[n]}(\delta)=\delta$. This determines the action $SH(\phi^{[n]})$ 
of $\phi^{[n]}$ on the subring $SH(S^{[n]},\QQ))$ generated by $H^2(S^{[n]},\QQ)$. 
The action of $\widetilde{H}(\phi^{[n]})$
on the image of $\Psi$ in $\Sym^n(\widetilde{H}(S^{[n]},\QQ))$ is determined as follows.
If $n$ is odd, then $\Psi$ is equivariant, by Theorem \ref{thm-def-of-H-tilde}.
If $n$ is even, then 
\[
\Psi\circ \det(\widetilde{H}(\phi^{[n]}))SH(\phi^{[n]})=\Sym^n(\widetilde{H}(\phi^{[n]}))\circ\Psi,
\]
by Theorem \ref{thm-def-of-H-tilde}.
Now $\Psi$ is monodromy equivariant with respect to the action of
$\eta_\phi$ on $\widetilde{H}(S^{[n]},\QQ)$.
If $n$ is odd, the theorem follows for $\phi$. 
If $n$ is even we get that $\det(\widetilde{H}(\phi^{[n]}))=1$ and so 
$\widetilde{H}(\phi^{[n]})=\left\{\begin{array}{ccc}
\eta_\phi & \mbox{if} & \det(\widetilde{H}(\phi))=1,
\\
-\eta_\phi & \mbox{if} & \det(\widetilde{H}(\phi))=-1.
\end{array}\right.$

It remains to verify the theorem for the index $2$ subgroup $SO^+(\widetilde{H}(S,\ZZ))$ of $DMon(S)$. 
The statement of the theorem for this subroup is equivalent to the statement that the composition 
\[
B_{-\frac{\delta}{2}}\circ \tilde{\theta}:\widetilde{H}(S,\QQ)\rightarrow \widetilde{H}(S^{[n]},\QQ)
\]
is $SO^+(\widetilde{H}(S,\ZZ))$-equivariant, since $SO^+(\widetilde{H}(S,\ZZ))$ does not have any non-trivial characters, and so must act trivially on the one-dimensional orthogonal complement $\span\{u_0\}$, $u_0:=(0,\delta,n-1)=B_{-\frac{\delta}{2}}(0,\delta,0)$, of its image.
Note first the equality 
\begin{equation}
\label{eq-image-of-tensorization-by-lb-via-[n]}
(B_\lambda)^{[n]}=B_{\theta(\lambda)},
\end{equation} 
for all $\lambda\in H^2(S,\QQ)$.
It suffices to prove (\ref{eq-image-of-tensorization-by-lb-via-[n]}) for integral $\lambda$, and by varying the complex structure, we may assume that $\lambda$ is of type $(1,1)$.
Given a line bundle $L$ on $S$, the equivariant line bundle
$(L\boxtimes \ \cdots \ \boxtimes L,\rho_\boxtimes)$ over $S^n$ is a pullback of a line bundle $\LB$  on the symmetric product
$S^{(n)}$.  We have the commutative diagram
\[
\xymatrix{
& \Gamma \ar[dl]_{b} \ar[dr]^{q}
\\
S^n \ar[dr]_{\bar{q}} & & S^{[n]} \ar[dl]^{\bar{b}}
\\
& S^{(n)}.
}
\]
Hence, $BKR((L\boxtimes \ \cdots \ \boxtimes L,\rho_\boxtimes))\cong
Rq_*^{\fS_n}(Lb^*L\bar{q}^*\LB)\cong
Rq_*^{\fS_n}(Lq^*L\bar{b}^*\LB)\cong \bar{b}^*\LB$.
We get that 
\begin{eqnarray*}
BKR((F\boxtimes \ \cdots \ \boxtimes F,\lambda)\otimes (L\boxtimes \ \cdots \ \boxtimes L,\rho_\boxtimes))  \cong 
Rq_*^{\fS_n}(Lb^*(F\boxtimes \ \cdots \ \boxtimes F,\lambda)\otimes Lq^*L\bar{b}^*\LB) \cong
\\
Rq_*^{\fS_n}(Lb^*(F\boxtimes \ \cdots \ \boxtimes F,\lambda))\otimes L\bar{b}^*\LB
\cong  BKR(F\boxtimes \ \cdots \ \boxtimes F,\lambda)\otimes BKR((L\boxtimes \ \cdots \ \boxtimes L,\rho_\boxtimes)),
\end{eqnarray*}
where the second isomorphism is via the projection formula.
The equality (\ref{eq-image-of-tensorization-by-lb-via-[n]}) thus follows for $\lambda=c_1(L)$ from the well known equality
$\theta(c_1(L))=c_1(\bar{b}^*\LB)$.

Passing to Lie algebras we get the homomorphism 
\begin{equation}
\label{eq-Lie-algebra-homomorphism[n]}
{}^{[n]}:\LieAlg{so}(\widetilde{H}(S,\QQ))\rightarrow \LieAlg{so}(\widetilde{H}(S^{[n]},\QQ))
\end{equation} 
and the 
equality $e_\lambda^{[n]}=e_{\theta(\lambda)}$.
As $B_\lambda^{[n]}$ commutes with $\widetilde{H}(\phi^{[n]}_\chi)$, by Equation (\ref{eq-Phi-chi-[n]-commutes-with-surface-FM}),
so does $e_\lambda^{[n]}$ and hence so does $e_{\theta(\lambda)}$. This agrees with the fact  that the eigenspace $V:=(0.\delta,n-1)^\perp$ of $\widetilde{H}(\phi^{[n]}_\chi)$
is $e_{\theta(\lambda)}$-invariant, for all $\lambda\in H^2(S,\QQ)$.

Every non-trivial $24$-dimensional representation of $\LieAlg{so}(\widetilde{H}(S,\QQ))$ is isomorphic to
$\widetilde{H}(S,\QQ)$. Hence there exists an isomorphism
\[
\gamma : \widetilde{H}(S,\QQ)\rightarrow V
\]
of $\LieAlg{so}(\widetilde{H}(S,\QQ))$-modules, which must also be an isomorphism of $SO^+(\widetilde{H}(S,\QQ))$-representations. 
The element $\beta$ of $\widetilde{H}(S,\QQ)$ spans the common kernel of 
$e_\lambda$,  $\lambda\in H^2(S,\QQ)$. Hence,
the element $\gamma(\beta)$ must belong to the common kernel $\span\{\beta\}$
of $e_{\theta(\lambda)}$, $\lambda\in H^2(S,\QQ)$. We may normalize $\gamma$ to satisfy $\gamma(\beta)=\beta$.

\hide{
The Lie algebras appearing in (\ref{eq-Lie-algebra-homomorphism[n]}) are subalgebras of $\LieAlg{gl}(H^*(S,\QQ))$ and
of $\LieAlg{gl}(H^*(S^{[n]},\QQ))$ and Hodge sub-modules. We have the commutative diagram
\[
\xymatrix{
\LieAlg{gl}(H^*(S,\QQ)) \ar[r] &\LieAlg{gl}(H^*(S^n,\QQ)) \ar[r]& \LieAlg{gl}(H^*(S^{[n]},\QQ))
\\
\LieAlg{so}(\widetilde{H}(S,\QQ)) \ar[rr]_{[n]} \ar[u]^{\cup} & & \LieAlg{so}(\widetilde{H}(S^{[n]},\QQ)), \ar[u]_{\cup}
}
\]
where the top left horizontal arrow is the map from the endomorphism algebra of a vector space to that of its $n$-th tensor power. 
The top right horizontal arrow sends  a correspondence in $H^*(S^n\times S^n,\QQ)$ to its convolution with the correspondence
$\sqrt{td_{S^{[n]}}} ch(\StructureSheaf{\Gamma})\sqrt{td_{S^n}}$ in $H^*(S^{[n]}\times S^n,\QQ)$
and its adjoint in $H^*(S^n\times S^{[n]},\QQ)$ associated to the dual object $\StructureSheaf{\Gamma}^\vee$ (??? need to account for the difference between $Rq_*$ and $Rq_*^{\fS_n}$ ???).
Assign $H^*(S^{[n]},\QQ)$ a coarser Hodge structure of weight $0$,
where $H^{p,q}(S^{[n]})$ has weight $(p-\frac{p+q}{2},q-\frac{p+q}{2})$. Endow $\LieAlg{gl}(H^*(S^{[n]},\QQ))$ with the coarser Hodge structure as well. All Hodge classes has weight $(0,0)$ with respect to the coarser Hodge structure.
The homomorphism ${}^{[n]}$ is compatible with the coarser Hodge structures, as it is induced by rational Hodge classes.
Now, the $\LieAlg{so}(\widetilde{H}(S,\QQ))$-action determines the 
coarser 
Hodge structure of weight $0$ introduced in Section \ref{sec-Mukai-lattice}. Similarly,
the $\LieAlg{so}(\widetilde{H}(S^{[n]},\QQ))$-action determines the coarse Hodge structure of 
$\widetilde{H}(S^{[n]},\QQ)$. We conclude that $\gamma$ is an isomorphism of the coarse Hodge structures.
}

Let $\tilde{\gamma}:\widetilde{H}(S,\QQ)\rightarrow \widetilde{H}(S,\QQ)$ be the composition
$\tilde{\gamma}=\tilde{\theta}^{-1}\circ B_{\delta/2}\circ \gamma$, where $\tilde{\theta}^{-1}:(0,\delta,0)^\perp\rightarrow \widetilde{H}(S,\QQ)$ is the left inverse of $\tilde{\theta}$. Then $\tilde{\gamma}$ is $SO^+(H^2(S,\ZZ))$-equivariant (the two actions, via ${}^{[n]}$ and via $\eta$ coincide for this subgroup, as demonstrated in the first paragraph of the proof). 
Indeed, 
both $\tilde{\theta}^{-1}$ and $B_{\delta/2}$ are $SO^+(H^2(S,\ZZ))$-equivariant and  $\gamma$ is  $SO^+(H^2(S,\ZZ))$-equivariant as it is $SO^+(\widetilde{H}(S,\QQ))$-equivariant by assumption. 
Hence, $\tilde{\gamma}$ must map $H^2(S,\QQ)$ to itself and act on it by $\pm 1$. 
Being an isometry,  $\tilde{\gamma}$  maps the orthogonal complement
$U:=\span\{\alpha,\beta\}$ of  $H^2(S,\QQ)$ to itself.
Furthermore,  
$\tilde{\gamma}(\beta)=\beta$, and $\tilde{\gamma}$ commutes with $e_\lambda$, for all $\lambda\in H^2(S,\QQ)$.
Thus $\tilde{\gamma}(\alpha)=\alpha$, since $\alpha$ is the unique isotropic class in $U$ which pairs to $-1$ with $\beta$. 
Finally, $\tilde{\gamma}(\lambda)=\tilde{\gamma}(e_\lambda(\alpha))=e_\lambda(\tilde{\gamma}(\alpha))=e_\lambda(\alpha)=\lambda$. Hence $\tilde{\gamma}$ is the identity. 
Hence $B_{-\frac{\delta}{2}}\circ \tilde{\theta}$ is $SO^+(\widetilde{H}(S,\QQ))$-equivariant.
The theorem is thus verified also for 
the subgroup $SO^+(\widetilde{H}(S,\ZZ))$ of $DMon(S)$. 
\end{proof}

\begin{cor}
\label{cor-vector-bundle-over-S[n]-associated-to-a-point-in-M[n]}
The vector bundle $F_z$  in Lemma \ref{lemma-vector-bundle-over-S[n]-associated-to-a-point-in-M[n]} satisfies
$c_1(F_z)=n!r_0^{n-1}(\theta(\lambda)-r_0\delta/2)$.
\end{cor}

\begin{proof}
Note that $n!r_0^{n-1}(\theta(\lambda)-r_0\delta/2)$ is $n!r_0^{n-1}$ times the direct summand of $B_{-\frac{\delta}{2}}(\tilde{\theta}(v))$ in $H^2(S^{[n]},\QQ)$, where $v=(r_0,\lambda,(\lambda,\lambda)/2)$ 
is the Mukai vector of $G_i$, $1\leq i\leq n$.
Theorem \ref{thm-Mukai-vector}(\ref{prop-item-spanning-Mukai-vector}) reduces the proof of the Corollary to that of the equality 
$\ell(F_z)=\span\{B_{-\frac{\delta}{2}}(\tilde{\theta}(v))\}$.
Let $\phi$ be the element of $DMon(S)$ corresponding to the equivalence $\Phi$.
Denote by $\beta_1=(0,0,1)\in\widetilde{H}(S,\QQ)$ the Mukai vector of a point. Then $v=\widetilde{H}(\phi)(\beta_1)$,
since $G_i$ is isomorphic to $\Phi(\StructureSheaf{p_i})$.

The equality $\ell(F_z)=\widetilde{H}(\phi^{[n]})(\span\{\beta\})$ follows from  
Equation (\ref{eq-F-z-is-image-of-sky-scpaper-via-Phi[n]}). 
Theorem \ref{thm-action-of-DMon-S-on-LLV-lattice} implies that 
the embedding $B_{-\frac{\delta}{2}}\circ\tilde{\theta}:\widetilde{H}(S,\QQ)\rightarrow \widetilde{H}(S^{[n]},\QQ)$ is  equivariant up to sign with respect to (it spans a character of) $DMon(S):=\Aut_{K3^{[1]}}(S)$ acting via the functor $\widetilde{H}$ on the domain and via the functor $\widetilde{H}\circ ^{[n]}$ on the co-domain.  This equivariance explains the third equality below.
\begin{eqnarray*}
\ell(F_z)&=&
\widetilde{H}(\phi^{[n]})(\span\{\beta\})=
\widetilde{H}(\phi^{[n]})(\span\{B_{-\frac{\delta}{2}}\circ\tilde{\theta}(\beta_1)\})=
\span\{B_{-\frac{\delta}{2}}(\tilde{\theta}(\phi(\beta_1))\}
\\
&=& \span\{B_{-\frac{\delta}{2}}(\tilde{\theta}(v)\}.
\end{eqnarray*}
\end{proof}

%
\section{A reflexive torsion free Fourier-Mukai image of a lagrangian line bundle with a positive LLV line}
\label{sec-locally-free-FM-image-of-lagrangian-lb}

Let $S$ be a $K3$ surface of degree $6$ in $\PP^4$. The Fano variety of lines on a smooth cubic $3$-folds $C$ containing $S$
embeds as a smooth Lagrangian surface $Z$ in the Hilbert scheme $S^{[2]}$.
In Section \ref{sec-main-result-on-a-reflexive-FM-image} we state 
the main result of this section that the images $E_k$ of the natural line bundles $\StructureSheaf{Z}(k)$ on $Z$ via a particular Fourier-Mukai auto-equivalence $\Phi^{[2]}_{\P,\chi}$
of $D^b(S^{[2]})$ are reflexive torsion free sheaves, for $k$ sufficiently large,  and we compute the LLV line of $E_k$
(Proposition \ref{prop-a-locally-free-image-of-a-Lagrangian-line-bundle}). 
The auto-equivalence $\Phi^{[2]}_{\P,\chi}$ is the BKR-conjugate of the auto-equivalence of $D^b(S\times S)$ induced by tensoring with the sign character $\chi$ of $\fS_2$ the cartesian-square of the spherical twist $\Phi_\P$ of $D^b(S)$ with respect to $\StructureSheaf{S}$.
In Section \ref{sec-isospectral-Hilbert-scheme-for-Hilb-2}
we describe in more detail the equivalences $BKR:D^b_{\fS_2}(S\times S)\rightarrow D^b(S^{[2]})$ and 
$\Phi_{\P\boxtimes\P}:D^b(S\times S)\rightarrow D^b(S\times S)$.
In section \ref{sec-local-coordinates} we provide a local coordinate description of 
strict transforms of certain diagonals in the cartesian square $\widehat{S\times S}\times \widehat{S\times S}$
of the blow-up $\widehat{S\times S}$ of the diagonal in $S\times S$.
In Section \ref{sec-FM-kernel-over-S-[2]-squared} we provide an explicit complex over $S^{[2]}\times S^{[2]}$ representing  the Fourier-Mukai kernel of $\Phi^{[2]}_{\P,\chi}$ (Lemma \ref{lemma-kernel-of-Phi_P^[2]}).
In Section \ref{sec-reflexivity-of-FM-image-of-lb-over-lagrangian-surface} we provide a general criterion 
for the Fourier-Mukai image via $\Phi^{[2]}_{\P,\chi}$ of a line-bundle over a lagrangian surface in $S^{[2]}$ to be a reflexive torsion free sheaf (Lemma \ref{lemma-object-satisfying-3-conditions-is-isomorphic-to-a-locally-free-sheaf}).
We then verify that the hypotheses of this criterion are satisfied for sufficiently large powers $\StructureSheaf{Z}(k)$ of the natural ample line bundle over the Fano variety of lines $Z$ over a cubic $3$-fold. In Section \ref{sec-proof-of-the-reflexivity-proposition}
we complete the proof of Proposition \ref{prop-a-locally-free-image-of-a-Lagrangian-line-bundle}.


%
\subsection{Fourier-Mukai images of line bundles on the Fano variety of lines on a cubic 3-fold}
\label{sec-main-result-on-a-reflexive-FM-image}
Let $S$ be a $K3$ surface. Let $\Delta:S\rightarrow S\times S$ the diagonal embedding, and let $\Phi_\P:D^b(S)\rightarrow D^b(S)$ be the auto-equivalence with Fourier-Mukai kernel the complex $\P$ given by 
$\StructureSheaf{S\times S}\rightarrow \Delta_*\StructureSheaf{S}$ supported in degrees $0$ and $1$. $\P$ is quasi-isomorphic to the ideal sheaf of the diagonal. We will denote $\Delta_*\StructureSheaf{S}$ by $\StructureSheaf{\Delta_S}$ and use a similar abuse of notation for subschemes of $S^4$ and $S^{[2]}\times S^{[2]}$.
$\Phi_\P$ is the spherical twist with respect to the spherical object $\StructureSheaf{S}$. We get the 
auto-equivalence $\Phi_\P\boxtimes \Phi_\P:D^b_{\fS_2}(S\times S)\rightarrow D^b_{\fS_2}(S\times S)$ of $\fS_2$-equivariant 
categories.
Denote by $\Phi_\P^{[2]}:D^b(S^{[2]})\rightarrow D^b(S^{[2]})$ its conjugate (\ref{eq-BKR-conjugate-of-outer-product-of-FM}) via the BKR equivalence.
Let $\Phi_\chi$ be the auto-equivalence of $D^b_{\fS_2}(S\times S)$ corresponding to tensorization with the sign character $\chi$ of $\fS_2$ and let $\Phi^{[2]}_\chi:D^b(S^{[2]})\rightarrow D^b(S^{[2]})$ be its conjugate (\ref{eq-Phi-chi-[n]}) via the BKR equivalence. 
Set 
\[
\Phi^{[2]}_{\P,\chi}:=\Phi^{[2]}_\chi\circ \Phi^{[2]}_\P.
\]

Let $Q$ be a quadric in $\PP^4$ and let $C$ be a cubic in $\PP^4$ satisfying the following.

\begin{assumption}
\label{assumption-on-Q-and-C}
\begin{enumerate}
\item
$Q$ and $C$ are smooth and their intersection $S:=Q\cap C$ is smooth.   
\item
$S$ does not contain any line. 
\item
$C$ does not have any Eckhardt points, i.e., through every point of $C$ pass only finitely many lines on $C$. 
\end{enumerate}
\end{assumption}

Note that $S$ is a $K3$ surface.
The Fano variety $F(C)$ of lines on $C$ is a smooth surface, by \cite{clemens-griffiths}.
Let 
\begin{equation}
\label{eq-g}
g:F(C)\rightarrow S^{[2]}
\end{equation} 
be the morphism sending a line $\ell$ on $C$ to the length $2$ subscheme $\ell\cap Q$ of $S$. 
Let $\StructureSheaf{F(C)}(1)$ be the pull back of the line bundle $\StructureSheaf{Gr(2,5)}(1)$ via the inclusion $F(C)\subset Gr(2,5)$. Set $h:=c_1(\StructureSheaf{S}(1))$ and let $\tilde{h}:=\theta(h)\in H^2(S^{[2]},\Integers)$ be its image via (\ref{eq-theta}).
Set $Z:=g(F(C))$ and $\StructureSheaf{Z}(1):=g_*\StructureSheaf{F(C)}(1)$.
Keep the notation for elements of the rational LLV lattice as triples introduced in Equation (\ref{eq-notation-for-element-of-LLV-lattice}).
Following is the main result of this section.  
 
\begin{prop}
\label{prop-a-locally-free-image-of-a-Lagrangian-line-bundle}
The image $g(F(C))$ is a smooth lagrangian surface in $S^{[2]}$ and 
$g_*\StructureSheaf{F(C)}$ has a rank $1$ cohomological obstruction map
(Definition \ref{def-deforms-in-co-dimension-one}(\ref{def-item-remains-of-Hodge-type-in-co-dimension-1})). For all integers $k$ sufficiently large the object $\Phi_{\P,\chi}^{[2]}(g_*\StructureSheaf{F(C)}(k))$ is isomorphic to a reflexive torsion free sheaf $E_k$ over $S^{[2]}$ of rank $45k^2$ and  $c_1(E_k)=-15k\tilde{h}-\frac{45k(k+1)}{2}\delta$. The 
sheaf $E_k$  has a rank $1$ cohomological obstruction map
and $\ell(E_k)$ is spanned by $\left(45k^2,-15k\tilde{h}-\frac{45k(k-1)}{2}\delta,-\frac{45k(k-2)}{4}\right)$. The sheaf $E_k$ fits in a short exact sequence
\begin{equation}
\label{eq-short-exact-sequence-of-E-k}
0\rightarrow E_k \rightarrow
\left[
H^0(Z,\StructureSheaf{Z}(k+1))\otimes_\CC \StructureSheaf{S^{[2]}}
\right]
\oplus
\left[
H^0(Z,\StructureSheaf{Z}(k))\otimes_\CC \StructureSheaf{S^{[2]}}(-\delta)
\right]
\RightArrowOf{a_k} F_k\rightarrow 0,
\end{equation}
where $F_k$ is a torsion free sheaf of rank $12$ and $c_1(F_k)=15k\tilde{h}-6\delta$. The restriction of $a_k$ to the first direct summand is surjective onto $F_k$, while the image of its restriction to the second summand is isomorphic to $F_{k-1}\otimes \StructureSheaf{S^{[2]}}(-\delta)$.
\end{prop}

The proposition is proved at the end of this section. Note that $E_k$ is the image of $g_*\StructureSheaf{F(C)}$ via an auto-equivalence, hence $\dim\Ext^i(E_k,E_k)=\dim\Ext^i(g_*\StructureSheaf{F(C)},g_*\StructureSheaf{F(C)})=\dim H^i(F(C),\ComplexNumbers)$, where the latter equality follows from \cite[Theorem 0.1.3]{mladenov}. 
In particular,
$E_k$ is simple, i.e., $\Hom(E_k,E_k)\cong\CC$,  $\dim\Ext^1(E_k,E_k)=10$, and $\dim\Ext^2(E_k,E_k)=45$, by Example
\ref{examples-of-maximally-deformable-Lagrangian-surfaces}(\ref{example-item-Fano-variety-of-lines-on-a-cubic-threefold}).
We do not know\footnote{We check in Remark \ref{rem-two-subsheaves-do-not-destabilize} that 
the intersections of $E_k$ with each of the two direct summands in (\ref{eq-short-exact-sequence-of-E-k}) do not slope-destabilize $E_k$ for polarizations in an open subcone of the ample cone.
} 
if $E_k$ is slope-stable for line bundles in some open subcone of the ample cone. If it is, then Theorem \ref{thm-modularity-of-a-stable-sheaf-with-a-rank-1-obstruction-map} would apply and the Azumaya algebra $\SheafEnd(E_k)$ would deform with $S^{[2]}$ to every irreducible holomorphic symplectic manifold of $K3^{[2]}$-type.

\begin{rem}
Assumption \ref{assumption-on-Q-and-C} does not rule out the case where $S$ contains a smooth conic and hence admits an elliptic fibration onto the pencil of hyperplane sections containing the conic. In this case $S^{[2]}$ admits a lagrangian fibration the generic fiber of which is the product of two elliptic curves. Techniques for proving stability of vector bundles over IHSMs admitting lagrangian fibrations were developed in \cite[Prop. 3.6]{ogrady-modular}.
\end{rem}
%
\subsection{Explicit description of the isospectral Hilbert scheme for $n=2$}
\label{sec-isospectral-Hilbert-scheme-for-Hilb-2}
Let $S$ be a $K3$ surface. Let $b:\widehat{S\times S}\rightarrow S\times S$ be the blow-up of the diagonal,  and $q:\widehat{S\times S}\rightarrow S^{[2]}$ the quotient map by the natural $\fS_2$-action.
The BKR equivalence, given in (\ref{eq-BKR}), 
\[
BKR
:D^b_{\fS_2}(S\times S)\rightarrow D^b(S^{[2]})
\] 
is the composition $Rq_*^{\fS_2}\circ Lb^*$,
where 
$Rq_*^{\fS_2}:D^b_{\fS_2}(\widehat{S\times S})\rightarrow D^b(S^{[2]})$ is the equivariant pushforward associated to the homomorphism $\fS_2\rightarrow \{1\}$, and
$Lb^*:D^b_{\fS_2}(S\times S)\rightarrow D^b_{\fS_2}(\widehat{S\times S})$  is the equivariant pullback.
The functor $Rq_*^{\fS_2}$ takes an $\fS_2$-equivariant object $F$ to the $\fS_2$-invariant direct summand $(Rq_*(F))^{\fS_2}$ of the object
$Rq_*(F).$ 

The morphism $(b,q):\widehat{S\times S}\rightarrow (S\times S)\times S^{[2]}$ embeds $\widehat{S\times S}$ isomorphically onto
the isospectral Hilbert scheme $\Gamma$.
The sheaf $\StructureSheaf{\Gamma}$ admits a natural $\fS_2\times 1$ linearization $\rho$ and $(\StructureSheaf{\Gamma},\rho)$
is the Fourier-Mukai kernel of the functor $BKR$ given in (\ref{eq-BKR}).
Its inverse has Fourier-Mukai kernel $(\StructureSheaf{\Gamma}^\vee[2],\rho^\vee)\cong(\omega_{\Gamma},\rho_\omega)$ in $D^b_{1\times \fS_2}(S^{[2]}\times (S\times S))$, where $\rho_\omega$ is the natural linearization (see \cite[Cor. 3.40 and Prop. 5.9]{huybrechts-FM}). 
The functor 
$\Phi_{(\omega_{\Gamma},\rho_\omega)}:D^b(S^{[2]})\rightarrow D^b_{\fS_2}(S\times S)$ is the composition of tensorization by $\StructureSheaf{S^{[2]}}(\delta)$, followed by
$R^{\fS_2}b_*\circ  \chi\otimes Lq^*$, where $\chi$ is the non-trivial character of $\fS_2$, by Equation (\ref{eq-linearized-dualizing-sheaf-of-isospectral-Hilbert-scheme}).
Indeed, 
\begin{eqnarray}
\label{eq-tensorization-by-chi-interchanges-to-lines}
BKR(\StructureSheaf{S\times S},1)&\cong &\StructureSheaf{S^{[2]}}, \  \mbox{and}
\\ 
\nonumber
BKR(\StructureSheaf{S\times S},\chi)&\cong &\StructureSheaf{S^{[2]}}(-\delta),
\end{eqnarray}
as observed already in (\ref{eq-BKR-of-structure-sheaf}),
and so 
$BKR^{-1}(\StructureSheaf{S^{[2]}})\cong (\StructureSheaf{S\times S},1)$ and 
$BKR^{-1}(\StructureSheaf{S^{[2]}}(-\delta))\cong (\StructureSheaf{S\times S},\chi)$, which agrees with the above description of $BKR^{-1}$. 

The following lemma describes the involutive auto-equivalence $\Phi^{[2]}_\chi:D^b(S^{[2]})\rightarrow D^b(S^{[2]})$, given in (\ref{eq-Phi-chi-[n]}), which is the BKR conjugate of tensorization with the sign character $\chi$.

\begin{lem}
\label{lemma-BKR-conjugate-of-tensorization-by-sign-character}
The following isomorphisms holds for every object $E$ in $D^b(S^{[2]})$ and every object $(F,\rho)$ in $D^b_{\fS_2}(S\times S)$. 
\begin{eqnarray*}
\Phi^{[2]}_\chi(E):=BKR(BKR^{-1}(E)\otimes\chi)&\cong& (BKR\left[\left(BKR^{-1}(E))^\vee\right]\right)^\vee\otimes\StructureSheaf{S^{[2]}}(-\delta),
\\
BKR(F,\rho\otimes\chi)&\cong& (BKR(F^\vee,\rho^\vee))^\vee\otimes\StructureSheaf{S^{[2]}}(-\delta).
\end{eqnarray*}
\end{lem}

\begin{proof}
The existence of the second isomorphism is equivalent to that of  the first, by setting $E=BKR(F,\rho)$. We prove the second. 
\begin{eqnarray*}
[BKR(F^\vee,\rho^\vee)]^\vee &\cong &\left[Rq_*^{\fS_2}(Lb^*(F^\vee,\rho^\vee))\right]^\vee
\cong
Rq_*^{\fS_2}((\omega_{\widehat{S\times S}},\rho_\omega)\otimes Lb^*(F,\rho))
\\
&\cong&
Rq_*^{\fS_2}\left[\chi\otimes Lq^*\StructureSheaf{S^{[2]}}(\delta)\otimes Lb^*(F,\rho)\right]
\cong
\StructureSheaf{S^{[2]}}(\delta)\otimes Rq_*^{\fS_2}(Lb^*(F,\rho\otimes\chi))
\\
&\cong&
\StructureSheaf{S^{[2]}}(\delta)\otimes BKR(F,\rho\otimes\chi),
\end{eqnarray*}
where the first isomorphism is a definition, the second is by equivariant Grothendieck-Verdier duality, the third by the isomorphism $(\omega_{\widehat{S\times S}},\rho_\omega\otimes\chi)\cong Lq^*\StructureSheaf{S^{[2]}}(\delta)$, the fourth by the projection formula,
and the fifth by definition of $BKR$.
\end{proof}

\hide{
Let $\phi^{[2]}_\chi:H^*(S^{[2]]},\QQ)\rightarrow H^*(S^{[2]]},\QQ)$ be the automorphism induced by $\Phi^{[2]}_\chi$.
Let $\widetilde{H}(\phi^{[2]}_\chi):\widetilde{H}(S^{[2]},\QQ)\rightarrow \widetilde{H}(S^{[2]},\QQ)$ be the induced isometry via 
(\ref{eq-action-of-DMon-on-Mukai-lattice}), which appeared in Lemma \ref{lemma-isometry-of-BKR-conjugate-of-tensorization-by-sign-character}.

\begin{lem}
\label{lemma-isometry-of-BKR-conjugate-of-tensorization-by-sign-character}
The isometry $\widetilde{H}(\phi^{[2]}_\chi)$ is minus the reflection with respect to the hyperplane orthogonal to $u_0:=(0,\delta,1)=B_{-\delta/2}(0,\delta,0)$,
\[
\widetilde{H}(\phi^{[2]}_\chi)(x) = -[x+(x,u_0)u_0].
\]
\end{lem}

\begin{proof}
Let $e:\Gamma\rightarrow S^n\times S^{[n]}$ be the inclusion. 
The kernel of the functor $BKR^{-1}$ is isomorphic to $(\iota_*\StructureSheaf{\Gamma})^\vee[2n]$, the derived dual $(\iota_*\StructureSheaf{\Gamma})^\vee$ is $\iota_*\omega_\Gamma[-2n]$, and $\omega_\Gamma$ is a line bundle, since $\Gamma$ is a normal Cohen-Macaulay and Gorenstein, by \cite[Theorem 3.1]{haiman}. 
 Let $p\in S^{[2]}$ correspond to a reduced subscheme supported on the subset $\{x_1,x_2\}$ of $S$, $x_1\neq x_2$.
 We have $BKR^{-1}(\CC_p)\cong bkr(\CC_p)\cong \left(\CC_{(x_1,x_2)}\oplus \CC_{(x_2,x_1)},\rho\right)$, for a linearization $\rho$, which is isomorphic to
 $\left(\CC_{(x_1,x_2)}\oplus \CC_{(x_2,x_1)},\chi\otimes \rho\right)$ by Remark \ref{rem-chi-invariant-objects}, so $\Phi^{[2]}_\chi(\CC_p)\cong\CC_p$.
Hence, $\widetilde{H}(\phi^{[2]}_\chi)$ maps the line $\ell(\CC_p)=\span\{(0,0,1)\}$ to itself, $\Sym^2(\widetilde{H}(\phi^{[2]}_\chi))$ restricts as the identity to
$\Psi(v(\CC_p))$, and $\phi(v(\CC_p))=v(\CC_p)$, where $\phi:H^*(S^{[2]},\QQ)\rightarrow H^*(S^{[2]},\QQ)$ is the derived monodromy operator corresponding to  $\Phi^{[2]}_\chi$. It follows that $\det(\widetilde{H}(\phi^{[2]}_\chi))=1$, since we must have
$\Sym^2(\widetilde{H}(\phi^{[2]}_\chi))\circ\Psi=\Psi\circ [\det(\widetilde{H}(\phi^{[2]}_\chi))\phi]$, by \cite[Theorem 4.9]{taelman}.

Note that the action of $\widetilde{H}(\phi^{[2]}_\chi)$ is independent of the complex structure on the $K3$ surface $S$.
If $\pi:S\rightarrow\PP^1$  is an elliptic fibration and $E_1$ and $E_2$ are two distinct smooth fibers, then
$E_1\times E_2$ embeds as a lagrangian surface in $S^{[2]}$. Let $\iota:E_1\times E_2\rightarrow S^{[2]}$ be the embedding. 
Set $Z:=E_1\times E_2\cup E_2\times E_1$. We get the equivariant embedding $\hat{e}:Z\rightarrow \widehat{S\times S}$.
Given a line bundle $\LB$ on $S^{[2]}$, then 
$Lq^*(\iota_*\iota^*\LB)\cong \hat{e}_*\hat{e}^*(Lq^*\LB)$ in $D^b_{\fS_2}(\widehat{S\times S})$. Denote the linearization of $Lq^*\LB$ by $\lambda$. Then $\hat{e}_*(\hat{e}^*Lq^*\LB,\hat{e}^*\lambda)$ is isomorphic to $\hat{e}_*(\hat{e}^*Lq^*\LB,\chi\otimes \hat{e}^*\lambda)$, by Remark \ref{rem-chi-invariant-objects}, as $E_1\times E_2$ and $E_2\times E_1$ are disjoint. Furthermore,
$Rb_*^{\fS_{2,\Delta}}\hat{e}_*(\hat{e}^*Lq^*\LB,\chi\otimes \hat{e}^*\lambda)$ is isomorphic to 
$\chi\otimes Rb_*^{\fS_{2,\Delta}}\hat{e}_*(\hat{e}^*Lq^*\LB,\hat{e}^*\lambda)$, since $b$ is an equivariant 
morphism with respect to the isomorphism $\fS_2\rightarrow \fS_2$.
It follows that
$BKR^{-1}(\iota_*\iota^*\LB)\cong \chi\otimes BKR^{-1}(\iota_*\iota^*\LB)$. Thus, $\Phi^{[2]}_\chi(\iota_*\iota^*\LB)\cong \iota_*\iota^*\LB$.
Now, $\ell(\StructureSheaf{E_1\times E_2})=\span\{(0,f,0)\}$, where $f$ is the class of a fiber of $\pi$ in $H^2(S,\Integers)$ considered as the subspace of $H^2(S^{[2]},\ZZ)$ orthogonal to $\delta$, by Lemma \ref{lemma-Mukai-line-of-structure-sheaf-of-subcanonical-lagrangian}. Thus, 
$\ell(\iota_*\iota^*\LB)=\exp(e_{c_1(\LB)})(0,f,0)=(0,f,(f,c_1(\LB))$. 
We see that $\widetilde{H}(\phi^{[2]}_\chi)$ leaves invariant the lines $\span\{(0,f,s)\}$, whenever $f$ is an isotropic primitive class in $H^2(S,\ZZ)$ and there is a complex structure on $S$ such that $f$ is the class of a fiber of an elliptic fibration and there exists a line bundle $\LB$  on $S$, such that $s=(f,c_1(\LB))$. 
We conclude that $\widetilde{H}(\phi^{[2]}_\chi)$ acts by multiplication by $1$ or $-1$
on $M:=\{(0,0,1),(0,\delta,0)\}^\perp$. 

Recall that $\ell(\StructureSheaf{S^{[2]}})=\span\{(4,0,5)\}$, by Example \ref{example-Mukai-line-of-structure-sheaf}. 
Equation (\ref{eq-tensorization-by-chi-interchanges-to-lines}) implies that $\widetilde{H}(\phi^{[2]}_\chi)$  interchanges the two lines 
$\ell(\StructureSheaf{S^{[2]}})\}$ and $\ell(\StructureSheaf{S^{[2]}}(-\delta))=\span\{(4,-4\delta,1)\}$.
Set $N:=\span_\QQ\{(4,0,5),(0,\delta,1)\}$.
The LLV lattice $\Lambda=B_{-\delta/2}(\widetilde{H}(S^{[2]},\ZZ))$  of $S^{[2]}$ given in (\ref{eq-Mukai-lattice-K3-2-case}) is invariant under the derived monodromy group, hence it is $\widetilde{H}(\phi^{[2]}_\chi)$-invariant.
Thus, so is $N_\ZZ:=N\cap\Lambda$. 
Now $B_{\delta/2}(0,\delta,1)=(0,\delta,0)$, $B_{\delta/2}(4,0,5)=2(2,\delta,2)$, and 
\[
B_{\delta/2}(N_\ZZ)=\span_\QQ\{(0,\delta,0),(2,\delta,2)\}\cap \widetilde{H}(S^{[2]},\ZZ)=\span_\ZZ\{(0,\delta,0),(1,0,1)\}
\]
is $B_{\delta/2}\circ\widetilde{H}(\phi^{[2]}_\chi)\circ B_{-\delta/2}$-invariant.
There are precisely $4$ elements of self intersection $-2$ in $B_{\delta/2}(N_\ZZ)$, namely $\{\pm(0,\delta,0),\pm(1,0,1)\}$.
An isometry of $B_{\delta/2}(N_\ZZ)$, which extends to $\widetilde{H}(S^{[2]},\ZZ)$, must take $(0,\delta,0)$ to $\pm(0,\delta,0)$
and $(1,0,1)$ to $\pm(1,0,1)$, as the divisibility of $(1,0,1)$ in $\widetilde{H}(S^{[2]},\ZZ)$ is $1$ and that of $(0,\delta,0)$ is $2$.
Hence, such an isometry must be one of 
\[
\{\pm 1, \pm R_{(0,\delta,0)},\pm R_{(1,0,1)},\pm R_{(0,\delta,0)}R_{(1,0,1)}\}, 
\]
where $R_u$ is the reflection with respect to $u^\perp$. Of those, only 
$\pm R_{(0,\delta,0)}$ interchange the lines spanned by 
$B_{\delta/2}(4,0,5)=(4,2\delta,4)$ and $B_{\delta/2}(4,-4\delta,1)=(4,-2\delta,4)$. 
Note that $R_{(0,\delta,0)}=R_{B_{\delta/2}(u_0)}=B_{\delta/2}\circ R_{u_0}\circ B_{-\delta/2}$.
We conclude that $\widetilde{H}(\phi^{[2]}_\chi)$
restricts to $N$ as $R_{u_0}$ or $-R_{u_0}$.

Note that $M$ and $N$ are complementary $\widetilde{H}(\phi^{[2]}_\chi)$-invariant subspaces and the latter restricts to $N$ with determinant $-1$ and $\det(\widetilde{H}(\phi^{[2]}_\chi))=1$. Hence, $\widetilde{H}(\phi^{[2]}_\chi)$ restricts to the $23$-dimensional $M$ with determinant $-1$ and so acts on $M$ by $-1$. Thus, $-R_{u_0}\circ \widetilde{H}(\phi^{[2]}_\chi)$ acts on $M$ by $1$ and on $N$ by $\pm 1$. Now, $M$ and $N$ are not orthogonal, and so $-R_{u_0}\circ \widetilde{H}(\phi^{[2]}_\chi)$ must act on $N$ by $1$ as well. Hence, 
$\widetilde{H}(\phi^{[2]}_\chi)=-R_{u_0}$.
\end{proof}
}

Denote by $\fS_{2,\Delta}$ the diagonal subgroup of $\fS_2\times\fS_2$. 
The non-unit of $\fS_{2,\Delta}$ acts by $(x_1,x_2,x_3,x_4)\mapsto (x_2,x_1,x_4,x_3).$
 We get the $\fS_{2,\Delta}$-equivariant object 
 \[
 \P\boxtimes\P:=L\pi_{13}^*\P\otimes L\pi_{24}^*\P
 \] 
 over $(S\times S)\times (S\times S)$, endowed with the linearization $\rho_{\boxtimes}$ acting by permuting the factors. The object $\P\boxtimes\P$  is represented by the complex
\begin{equation}
\label{eq-right-exact-sequence-of-beta}
\StructureSheaf{S^4}\RightArrowOf{\alpha}\StructureSheaf{\Delta_{13}}\oplus \StructureSheaf{\Delta_{24}}
\RightArrowOf{\beta} \StructureSheaf{\Delta_{13}\cap\Delta_{24}},
\end{equation}
where $\Delta_{ij}\subset S^4$ is the subvariety of quadruples $(x_1,x_2,x_3,x_4)$ satisfying $x_i=x_j$, 
$\beta$ is the difference of the
natural homomorphisms from each direct summand, and $\alpha$ is the direct sum of the two natural homomorphisms. 
Note that $\ker(\beta)$ is $\StructureSheaf{\Delta_{13}\cup\Delta_{24}}$ and so the above complex is quasi-isomorphic to the ideal sheaf $\Ideal{\Delta_{13}\cup\Delta_{24}}$.
The integral transform $\Phi_{(\P\boxtimes\P,\rho_\boxtimes)}:D^b_{\fS_2}(S\times S)\rightarrow D^b_{\fS_2}(S\times S)$ is an auto-equivalence, which is conjugated via the BKR equivalence to an auto-equivalence
\begin{equation}
\label{eq-P-[2]-as-a-composition}
\Phi_\P^{[2]}:= \Phi_{{(\StructureSheaf{\Gamma},\rho)}}\circ \Phi_{(\P\boxtimes\P,\rho_{\boxtimes})}\circ 
\Phi_{(\omega_{\Gamma},\rho_\omega)}: D^b(S^{[2]})\rightarrow D^b(S^{[2]}),
\end{equation}
by \cite[Lemma 5(3)]{ploog}.
Composing $\Phi_{(\P\boxtimes\P,\rho_\boxtimes)}$ with the auto-equivalence $\Phi_\chi$ of tensoring with the sign character $\chi$ of $\fS_2$ we get the auto-equivalence $\Phi_{(\P\boxtimes\P,\chi\otimes\rho_\boxtimes)}$ and we denote its BKR-conjugate by 
$\Phi_{\P,\chi}^{[2]}:D^b(S^{[2]})\rightarrow D^b(S^{[2]})$. 
The projections $\pi_i:(S\times S)\times (S\times S)\rightarrow (S\times S)$, $i=1,2$, are equivariant with respect to the natural isomorphism $\fS_{2,\Delta}\rightarrow\fS_2$, and we get  $\Phi_\chi\circ \Phi_{(\P\boxtimes\P,\rho_{\boxtimes})}=\Phi_{(\P\boxtimes\P,\rho_{\boxtimes})}\circ \Phi_\chi$.
We have
\[
\Phi_{\P,\chi}^{[2]}=
\Phi_\chi^{[2]}\circ\Phi_\P^{[2]}=\Phi_\P^{[2]}\circ\Phi_\chi^{[2]},
\]
by Equation (\ref{eq-Phi-chi-[n]-commutes-with-surface-FM}).

\begin{lem}
\label{lemma-shifts-of-structure-sheaf-and-delta-line-bundle-interchanged}
$\Phi^{[2]}_{\P.\chi}(\StructureSheaf{S^{[2]}})  \cong \StructureSheaf{S^{[2]}}(-\delta)[-4]$ and
$\Phi^{[2]}_{\P.\chi}(\StructureSheaf{S^{[2]}}(-\delta))  \cong \StructureSheaf{S^{[2]}}[-4].$
\end{lem}

\begin{proof}
The isomorphism $\Phi_\P(\StructureSheaf{S})\cong \StructureSheaf{S}[-2]$ explains the third isomorphism below.
\begin{eqnarray*}
\Phi^{[2]}_{\P.\chi}(\StructureSheaf{S^{[2]}}) & \cong &
BKR\left(
\chi\otimes \Phi_{(\P\boxtimes\P,\rho_\boxtimes)}(
BKR^{-1}(\StructureSheaf{S^{[2]}})
)
\right)
\\
&\stackrel{(\ref{eq-BKR-of-structure-sheaf})}{\cong}& 
BKR\left(
\chi\otimes \Phi_{(\P\boxtimes\P,\rho_\boxtimes)}(\StructureSheaf{S\times S},1)
\right)
\cong
BKR\left(
\chi\otimes (\StructureSheaf{S\times S}[-4],1)
\right)
\\
&
\stackrel{\mbox{Lemma \ref{lemma-BKR-conjugate-of-tensorization-by-sign-character}}}{\cong}
&
\left(BKR(\StructureSheaf{S\times S}[4],1)\right)^\vee\otimes \StructureSheaf{S^{[2]}}(-\delta)
\stackrel{(\ref{eq-BKR-of-structure-sheaf})}{\cong} \StructureSheaf{S^{[2]}}(-\delta)[-4].
\end{eqnarray*}
The proof of the second isomorphism is similar.
\begin{eqnarray*}
\Phi^{[2]}_{\P.\chi}(\StructureSheaf{S^{[2]}}(-\delta))&\cong&
BKR\left(\chi\otimes \Phi_{(\P\boxtimes\P,\rho_\boxtimes)}(BKR^{-1}(\StructureSheaf{S^{[2]}}(-\delta)))
\right)
\\
&\cong& BKR\left(\chi\otimes \Phi_{(\P\boxtimes\P,\rho_\boxtimes)}(\StructureSheaf{S\times S},\chi))
\right)
\cong
BKR((\StructureSheaf{S\times S},1)[-4])\cong\StructureSheaf{S^{[2]}}[-4].
\end{eqnarray*}
\end{proof}

%
\subsection{A local coordinate description of the strict transforms of diagonals}
\label{sec-local-coordinates}
We will need a local coordinates description of $\widehat{S\times S}\times \widehat{S\times S}$ and its exceptional divisors.
We replace $S$ by $\CC^2$ and let $(x_i,y_i)$ be coordinates on the $i$-th factor of $S$. Let $(s:t)$ be homogeneous coordinates on $\PP^1$.
The blow-up of $\CC^2$ at the origin is the hypersurface $V(xt-ys)$ in $\CC^2\times \PP^1$. Let $z=s/t$ denote the affine coordinate on $\PP^1\setminus\{(1:0)\}$.  Away from 
$\CC^2\times\{(1:0)\}$
the blow-up 
has coordinates $y,z$ and $x=yz$. The exceptional divisor $\hat{D}$ is $V(y)$ in these coordinates.
Consider the map
\begin{eqnarray*}
\CC^2\times \CC^2 &\rightarrow & \CC^2\times \CC^2,
\\
((x_1,y_1),(x_2,y_2)) & \mapsto & ((x_1,y_1),(x_1-x_2,y_1-y_2)).
\end{eqnarray*}
The blow-up of the left copy of $\CC^2\times \CC^2$ along the diagonal corresponds 
to the blow-up along $\CC^2\times \{0\}$ via the above change of coordinates. 
Hence, $\widehat{S\times S}$ embeds as the subvariety 
\[
V((x_1-x_2)t-(y_1-y_2)s)
\]
of $\CC^2\times \CC^2\times \PP^1$ and has coordinates $x_1,y_1,y_2,z$, and $x_1-x_2=(y_1-y_2)z$, away from the divisor
$\CC^2\times\CC^2\times \{(1:0)\}$.
The exceptional divisor is $V(y_1-y_2)$ in these coordinates.
Over $\widehat{S\times S}\times \widehat{S\times S}$ we get the coordinates
\[
(x_1,y_1,y_2,z_1;x_3,y_3,y_4,z_2)
\]
satisfying $x_1-x_2=(y_1-y_2)z_1$ and $x_3-x_4=(y_3-y_4)z_2$.
The strict transform $\hat{\Delta}_{13}$ of $\Delta_{13}$ in $\widehat{S\times S}\times \widehat{S\times S}$ corresponds to $V(x_1-x_3,y_1-y_3)$.
The strict transform $\hat{\Delta}_{24}$  of $\Delta_{24}$ corresponds to $V(y_2-y_4,x_2-x_4)$, where
$x_2=x_1-z_1(y_1-y_2)$ and $x_4=x_3-z_2(y_3-y_4)$.
The intersection $\hat{\Delta}_{13}\cap \hat{\Delta}_{24}$ is given by
\begin{eqnarray*}
\hat{\Delta}_{13}\cap \hat{\Delta}_{24}&=&V(x_1-x_3,y_1-y_3,y_2-y_4,(y_1-y_2)(z_1-z_2)), \mbox{where}
\\
\hat{D}\times_S\hat{D}&=& V(x_1-x_3,y_1-y_3,y_2-y_4,y_1-y_2)
\\
\Delta_{\widehat{S\times S}} &=& V(x_1-x_3,y_1-y_3,y_2-y_4,z_1-z_2).
\end{eqnarray*}
Note that $x_1-x_3,y_1-y_3,y_2-y_4,(y_1-y_2)(z_1-z_2)$ is a regular sequence, and so $\hat{\Delta}_{13}\cap \hat{\Delta}_{24}$
is a complete intersection. Consider the sections  $s_1:=(x_1-x_3,y_1-y_3)$ and $s_2:=(y_2-y_4,x_2-x_4)$  of
$\StructureSheaf{\widehat{S\times S}\times \widehat{S\times S}}^{\oplus 2}$. The Koszul complex associated to $s_1$
is a locally free resolution of $\StructureSheaf{\hat{\Delta}_{13}}$ and the one associated to $s_2$ is a locally free resolution of $\StructureSheaf{\hat{\Delta}_{24}}$. The tensor product of these two Koszul complexes is the
Koszul complex associated to the section $(s_1,s_2)$ of 
$\StructureSheaf{\widehat{S\times S}\times \widehat{S\times S}}^{\oplus 4}$, which is exact, as the section  $(s_1,s_2)$ is regular.
Hence, the torsion sheaves $\Tor_i(\StructureSheaf{\hat{\Delta}_{13}},\StructureSheaf{\hat{\Delta}_{24}})$ vanish, for $i>0$.

%
\subsection{The kernel of the BKR-conjugate of the square of a spherical twist by $\StructureSheaf{S}$}
\label{sec-FM-kernel-over-S-[2]-squared}
Let $I\subset S^{[2]}\times S^{[2]}$ be the incidence variety consisting of pairs of ideals of non-disjoint subschemes.
Let $\nu:\tilde{I}\rightarrow I$ be its normalization.
Let $\hat{\Delta}_{ij}$ be the strict transform of $\Delta_{ij}$ in $\widehat{S\times S}\times \widehat{S\times S}$.
Let $\hat{\delta}:\widehat{S\times S}\rightarrow \widehat{S\times S}\times \widehat{S\times S}$ be the diagonal embedding
and denote its image by $\Delta_{\widehat{S\times S}}$.
Let $\hat{D}\subset \widehat{S\times S}$ be the exceptional divisor of $b$. 
Then $\hat{\Delta}_{13}\cap \hat{\Delta}_{24}=\Delta_{\widehat{S\times S}}\cup \hat{D}\times_S\hat{D}$   and each of
$\hat{\Delta}_{13}$ and $\hat{\Delta}_{24}$ is naturally isomorphic to $\tilde{I}$. The diagonal involution interchanges $\hat{\Delta}_{13}$ and $\hat{\Delta}_{24}$ leaving $\hat{D}\times_S\hat{D}$ pointwise invariant. Thus 
$\hat{D}\times_S\hat{D}$ embeds naturally in $\tilde{I}$ and we denote its image by $\hat{D}\times_S\hat{D}$ as well. 
Denote by $\Ideal{\tilde{I},\hat{D}\times_S\hat{D}}$ the ideal sheaf of $\hat{D}\times_S\hat{D}$ in $\tilde{I}$. 
The homomorphism $\StructureSheaf{\tilde{I}}\rightarrow\StructureSheaf{\Delta_{\widehat{S\times S}}}$
maps a local function vanishing to order $1$ along $\hat{D}\times_S\hat{D}$ to an anti-invariant function with respect to the diagonal involution. 
We have the short exact sequence
\[
0\rightarrow \Ideal{I,D\times_SD\cup\Delta_{S^{[2]}}}\rightarrow \nu_*\Ideal{\tilde{I},\hat{D}\times_S\hat{D}}
\rightarrow \StructureSheaf{\Delta_{S^{[2]}}}\otimes\pi_1^*\StructureSheaf{S^{[2]}}(-\delta)\rightarrow 0.
\]

\begin{lem}
\label{lemma-kernel-of-Phi_P^[2]}
\begin{enumerate}
\item
\label{lemma-item-P[2]}
The auto-equivalence 
$\Phi_\P^{[2]}:D^b(S^{[2]})\rightarrow D^b(S^{[2]})$ is the Fourier-Mukai transformation with kernel 
\[
\P^{[2]} :=
\pi_1^*\StructureSheaf{S^{[2]}}(\delta)\otimes \left[\left\{
\pi_1^*\StructureSheaf{S^{[2]}}(-\delta)\oplus \pi_2^*\StructureSheaf{S^{[2]}}(-\delta)
\right\}
\RightArrowOf{a} \nu_*\Ideal{\tilde{I},\hat{D}\times_S\hat{D}}
\right],
\]
where the homomorphism $a$ is surjective and is induced from the restriction homomorphism
$\StructureSheaf{\widehat{S\times S}\times \widehat{S\times S}}\rightarrow \StructureSheaf{\hat{\Delta}_{13}\cup\hat{\Delta}_{24}}$ via push-forward via $q\times q$ and taking the $\fS_{2,\Delta}$-anti-invariant subsheaves.
\item
\label{lemma-item-P[2]-chi}
The auto-equivalence 
$\Phi_{\P,\chi}^{[2]}:D^b(S^{[2]})\rightarrow D^b(S^{[2]})$ is the Fourier-Mukai transformation with kernel
\[
\P^{[2]}_\chi := \pi_1^*\StructureSheaf{S^{[2]}}(\delta)\otimes \left[
\left\{\StructureSheaf{S^{[2]}\times S^{[2]}}\oplus 
\left[\pi_1^*\StructureSheaf{S^{[2]}}(-\delta)\otimes \pi_2^*\StructureSheaf{S^{[2]}}(-\delta)\right]\right\}
\RightArrowOf{a_\chi} \StructureSheaf{I}
\right],
\]
where the homomorphism $a_\chi$ is surjective and is induced from the restriction homomorphism
$\StructureSheaf{\widehat{S\times S}\times \widehat{S\times S}}\rightarrow \StructureSheaf{\hat{\Delta}_{13}\cup\hat{\Delta}_{24}}$ via push-forward via $q\times q$ and taking the $\fS_{2,\Delta}$-invariant subsheaves.
\end{enumerate}
\end{lem}

\begin{proof}
(\ref{lemma-item-P[2]}). \underline{Step 1:}
The kernel of $\Phi_\P^{[2]}$, as defined in (\ref{eq-P-[2]-as-a-composition}) using \cite[Lemma 5(3)]{ploog}, is 
isomorphic to
$\pi_1^*\StructureSheaf{S^{[2]}}(\delta)\otimes R(q\times q)_*^{\fS_{2,\Delta}}\left( L(b\times b)^*(\P\boxtimes\P)\right),$
by the argument below using the following diagram.
\[
\xymatrix{
& \widehat{S\times S}\times \widehat{S\times S} \ar[ld]_{q\times q} \ar[d]_{(q,b)\times(b,q)} \ar[rd]^{b\times b}
\\
S^{[2]}\times S^{[2]} & 
[S^{[2]}\times (S\times S)]\times [(S\times S)\times S^{[2]}] \ar[l]^-{\pi_{14}} \ar[r]_-{\pi_{23}} &
(S\times S)\times (S\times S).
}
\]
Above $\pi_{ij}$ is the projection onto the product of the $i$-th and $j$-th factors, where each $(S\times S)$ is regarded as a single factor of the middle bottom variety in the diagram.
Let $\hat{\pi}_i$ be the projection from $\widehat{S\times S}\times \widehat{S\times S}$ to the $i$-th factor, $i=1,2$.
The object $L\pi_{34}^*R(b,q)_*\StructureSheaf{\widehat{S\times S}}$ is isomorphic to the structure sheaf of 
$[(S\times S)\times S^{[2]}]\times \Gamma$. Similarly, $L\pi_{12}^*R(q,b)_*\StructureSheaf{\widehat{S\times S}}$
is isomorphic to the structure sheaf of 
$\Gamma\times [S^{[2]}\times (S\times S)]$. The transversality of these two subschemes yields the isomorphism
\begin{equation}
\label{eq-tensor-product-is-push-forward}
[L\pi_{34}^*R(b,q)_*\StructureSheaf{\widehat{S\times S}}]\otimes L\pi_{12}^*R(q,b)_*\StructureSheaf{\widehat{S\times S}}
\cong
R((q,b)\times (b,q))_*\StructureSheaf{\widehat{S\times S}\times \widehat{S\times S}}.
\end{equation}
The kernel of $\Phi_\P^{[2]}$ is isomorphic to 
\begin{eqnarray*}
R\pi_{14,*}^{\fS_{2,\Delta}}\left\{
[L\pi_{34}^*R(b,q)_*^{\fS_2}(\StructureSheaf{\widehat{S\times S}},1)]
\otimes
[L\pi_{12}^*R(q,b)_*^{\fS_2}(\omega_{\widehat{S\times S}},\rho_\omega)]
\otimes
L\pi_{23}^*(\P\boxtimes\P,\rho_\boxtimes)
\right\} &\cong&
\\
R\pi_{14,*}^{\fS_{2,\Delta}}\left\{
R((q,b)\times (b,q))_*^{\fS_{2,\Delta}}(L\hat{\pi}_1^*(\omega_{\widehat{S\times S}},\rho_\omega))\otimes L\pi_{23}^*(\P\boxtimes\P,\rho_\boxtimes)
\right\}&\cong&
\\
R\pi_{14,*}^{\fS_{2,\Delta}}\left\{
R((q,b)\times (b,q))_*^{\fS_{2,\Delta}}\left(L\hat{\pi}_1^*(\omega_{\widehat{S\times S}},\rho_\omega)\otimes L(b\times b)^*(\P\boxtimes\P,\rho_\boxtimes)\right)
\right\}&\cong&
\\
R(q\times q)_*^{\fS_{2,\Delta}}\left\{
L\hat{\pi}_1^*(\omega_{\widehat{S\times S}},\rho_\omega)\otimes L(b\times b)^*(\P\boxtimes\P,\rho_\boxtimes)
\right\}&\cong&
\\
\pi_1^*\StructureSheaf{S^{[2]}}(\delta)\otimes R(q\times q)_*^{\fS_{2,\Delta}}\left( L(b\times b)^*(\P\boxtimes\P,\chi\otimes\rho_\boxtimes)\right),
\end{eqnarray*}
where the first isomorphism is due to (\ref{eq-tensor-product-is-push-forward}) and cohomology and base change, 
the second isomorphism is due to the projection formula and the equality
$b\times b=\pi_{23}\circ ((q,b)\times (b,q))$, the third is due to the equality 
$q\times q=\pi_{14}\circ ((q,b)\times (b,q))$, and the fourth is due to the projection formula, the equality
$\pi_1\circ(q\times q)=q\circ \hat{\pi}_1$, and the isomorphism $(\omega_{\widehat{S\times S}},\chi\otimes\rho_\omega)\cong Lq^*\StructureSheaf{S^{[2]}}(\delta)$. It remains to exhibit the isomorphism
\begin{equation}
\label{eq-isomorphism-of-invariant-subcomplex-of-push-forward-of-ideal-sheaf}
R(q\times q)_*^{\fS_{2,\Delta}}\left( L(b\times b)^*(\P\boxtimes\P,\chi\otimes \rho_\boxtimes)\right)
\cong
\left[
\pi_1^*\StructureSheaf{S^{[2]}}(-\delta)\oplus \pi_2^*\StructureSheaf{S^{[2]}}(-\delta)
\RightArrowOf{a} \nu_*\Ideal{\tilde{I},\hat{D}\times_S\hat{D}}
\right].
\end{equation}

\underline{Step 2:} (Calculation of $L(b\times b)^*(\P\boxtimes\P,\chi\otimes \rho_\boxtimes)$).
Recall that $\hat{\Delta}_{ij}$ is the strict transform of $\Delta_{ij}$ in $\widehat{S\times S}\times \widehat{S\times S}$.
Let $\tilde{\Delta}_{13}$ and $\tilde{\Delta}_{23}$ be the strict transforms of $\Delta_{13}$ and $\Delta_{24}$  in
$\widehat{S\times S}\times (S\times S)$.
The subscheme $\Delta_{13}$ is 
transversal to $\Delta_{12}$ and $\tilde{\Delta}_{13}$ is 
transversal to  $\tilde{\Delta}_{34}$, by a computation in local coordinates,
and so $L(b\times b)^*(\StructureSheaf{\Delta_{13}})$ is represented by $\StructureSheaf{\hat{\Delta}_{13}}$.
Then $L(b\times b)^*(\P\boxtimes \P)$ is isomorphic to $\Ideal{\hat{\Delta}_{13}}\stackrel{L}{\otimes}\Ideal{\hat{\Delta}_{24}}$, which is isomorphic to
\[
[\StructureSheaf{\widehat{S\times S}\times \widehat{S\times S}}\rightarrow 
\StructureSheaf{\hat{\Delta}_{13}}] \stackrel{L}{\otimes} 
[\StructureSheaf{\widehat{S\times S}\times \widehat{S\times S}}\rightarrow 
\StructureSheaf{\hat{\Delta}_{24}}],
\]
and hence to the complex 
\begin{equation}
\label{eq-hat-alpha}
\StructureSheaf{\widehat{S\times S}\times \widehat{S\times S}}\RightArrowOf{\hat{\alpha}}
\StructureSheaf{\hat{\Delta}_{13}}\oplus \StructureSheaf{\hat{\Delta}_{24}} \RightArrowOf{\hat{\beta}}
\left(\StructureSheaf{\hat{\Delta}_{13}}\stackrel{L}{\otimes} \StructureSheaf{\hat{\Delta}_{24}}\right).
\end{equation}
Note that $\hat{\Delta}_{13}\cap\hat{\Delta}_{24}$ is the union of 
the diagonal embedding of $\widehat{S\times S}$ and the fiber product $\hat{D}\times_S\hat{D}$, where $\hat{D}\subset \widehat{S\times S}$ is the exceptional divisor of $b$. Hence, the intersection of $\hat{\Delta}_{13}$ and $\hat{\Delta}_{24}$ is not transversal along the diagonal embedding of $\hat{D}$. 
Nevertheless, $\hat{\Delta}_{13}\cap\hat{\Delta}_{24}$ is a complete intersection and so 
the derived tensor product $\StructureSheaf{\hat{\Delta}_{13}}\stackrel{L}{\otimes} \StructureSheaf{\hat{\Delta}_{24}}$
is represented by the usual tensor product $\StructureSheaf{\hat{\Delta}_{13}}\otimes \StructureSheaf{\hat{\Delta}_{24}}$, which is isomorphic to $\StructureSheaf{\hat{\Delta}_{13}\cap \hat{\Delta}_{24}}$, and the sheaves 
$\Tor_i(\StructureSheaf{\hat{\Delta}_{13}},\StructureSheaf{\hat{\Delta}_{24}})$ vanish for $i>0$ (see Section \ref{sec-local-coordinates}). 
The complex (\ref{eq-hat-alpha})  thus becomes
\begin{equation}
\label{eq-right-exact-sequence-of-hat-beta}
\StructureSheaf{\widehat{S\times S}\times \widehat{S\times S}}\RightArrowOf{\hat{\alpha}}
\StructureSheaf{\hat{\Delta}_{13}}\oplus \StructureSheaf{\hat{\Delta}_{24}} \RightArrowOf{\hat{\beta}}
\StructureSheaf{\hat{\Delta}_{13}\cap \hat{\Delta}_{24}}
\end{equation}
in degrees $0$, $1$, and $2$ and with sheaf cohomology only in degree $0$, so it is a right exact sequence. The
homomorphisms $\hat{\alpha}$ and $\hat{\beta}$ are analogous to $\alpha$ and $\beta$ in (\ref{eq-right-exact-sequence-of-beta}). Furthermore, the short exact sequence 
\[
0\rightarrow \Ideal{\hat{\Delta}_{13}}\rightarrow \StructureSheaf{\widehat{S\times S}\times \widehat{S\times S}}\rightarrow 
\StructureSheaf{\hat{\Delta}_{13}}\rightarrow 0
\]
yields the isomorphism $\Tor_i(\Ideal{\hat{\Delta}_{13}},F)\cong\Tor_{i+1}(\StructureSheaf{\hat{\Delta}_{13}},F)$, for every coherent sheaf $F$ and for $i>0$, and similarly for $\Ideal{\hat{\Delta}_{24}}$. The symmetry of the functor $\Tor$ yields
$\Tor_i(\Ideal{\hat{\Delta}_{13}},\Ideal{\hat{\Delta}_{24}})\cong \Tor_{i+2}(\StructureSheaf{\hat{\Delta}_{13}},\StructureSheaf{\hat{\Delta}_{24}})=0,$ for all $i>0$. Hence, the complex (\ref{eq-right-exact-sequence-of-hat-beta})
represents the sheaf $\Ideal{\hat{\Delta}_{13}}\otimes\Ideal{\hat{\Delta}_{24}}$ (as well as the object $\Ideal{\hat{\Delta}_{13}}\stackrel{L}{\otimes}\Ideal{\hat{\Delta}_{24}}$). Finally, $\Ideal{\hat{\Delta}_{13}}\otimes\Ideal{\hat{\Delta}_{24}}$ is isomorphic to
the kernel $\Ideal{\hat{\Delta}_{13}\cup \hat{\Delta}_{24}}$ of $\hat{\alpha},$ by the right exactness of (\ref{eq-right-exact-sequence-of-hat-beta}).

\underline{Step 3:} (First step in the construction of the homomorphism $a$ in (\ref{eq-isomorphism-of-invariant-subcomplex-of-push-forward-of-ideal-sheaf})). 
The quotient map $q\times q:\widehat{S\times S}\times \widehat{S\times S}\rightarrow S^{[2]}\times S^{[2]}$ restricts to 
the complement of $\hat{\Delta}_{13}\cap\hat{\Delta}_{24}$ in 
each of $\hat{\Delta}_{13}$ and $\hat{\Delta}_{24}$ as an isomorphism onto $I\setminus[\Delta_{S^{[2]}}\cup D\times_S D]$,
but $I$ is singular along $\Delta_{S^{[2]}}$, as $q\times q:\Delta_{\widehat{S\times S}}\rightarrow \Delta_{S^{[2]}}$ has degree $2$, while it is smooth\footnote{
The kernel of $d(q\times q)$ at a point $p$ of $\widehat{S\times S}\times \widehat{S\times S}$ is the sum of non-trivial characters 
of $\fS_2\times\fS_2$. Let $\sigma_{ij}\in \fS_2\times\fS_2$ be the transposition interchanging $x_i$ and $x_j$ in $(x_1,x_2,x_3,x_4)\in S^4$. At a point $p$ of 
$\hat{\Delta}_{13}\cap\hat{\Delta}_{24}$ the trivial eigenspace in $T_p[\widehat{S\times S}\times \widehat{S\times S}]$ of each of $\sigma_{12}$ and $\sigma_{34}$ has codimension $1$, and the kernel of $d(q\times q)$ is the $2$-dimensional $-1$-eigenspace of the diagonal involution $\sigma_{12}\sigma_{34}$. The latter interchanges $T_p\hat{\Delta}_{13}$ and
$T_p\hat{\Delta}_{24}$, and so  if $\xi\in T_p\hat{\Delta}_{13}$ is a $-1$-eigenvector of $\sigma_{12}\sigma_{34}$,
then $\xi$ belongs to $T_p\hat{\Delta}_{13}\cap T_p\hat{\Delta}_{24}$. The computation in local coordinates in 
Section \ref{sec-local-coordinates} shows that 
$\hat{\Delta}_{13}$ and $\hat{\Delta}_{24}$ intersect transversally along $\hat{D}\times_S \hat{D}\setminus [\Delta_{\widehat{S\times S}}]$. Hence,
at a point $p$ of the latter, $T_p\hat{\Delta}_{13}\cap T_p\hat{\Delta}_{24}=T_p[\hat{D}\times_S\hat{D}]$ is contained in the $1$-eigenspace of $\sigma_{12}\sigma_{34}$ and so $d(q\times q)$ restricts to
$T_p[\hat{D}\times_S\hat{D}]$ as an injective homomorphism.
} 
along $[D\times_S D]\setminus \Delta_{S^{[2]}}$.
The diagonal involution interchanges $\hat{\Delta}_{13}$ and $\hat{\Delta}_{24}$. 
We conclude that both $\hat{\Delta}_{13}$ and $\hat{\Delta}_{24}$ are isomorphic to the normalization $\tilde{I}$ of $I$ and  $q\times q$ restricts to each as the normalization morphism $\nu:\tilde{I}\rightarrow I$.
The linearization $\bar{\rho}_\boxtimes$  induced by $\rho_\boxtimes$ on 
$\StructureSheaf{\hat{\Delta}_{13}}\oplus \StructureSheaf{\hat{\Delta}_{24}}$ is the usual one, 
\[
\bar{\rho}_{\boxtimes,g}:\StructureSheaf{\hat{\Delta}_{13}}\oplus \StructureSheaf{\hat{\Delta}_{24}}\rightarrow g^*\left[\StructureSheaf{\hat{\Delta}_{13}}\oplus \StructureSheaf{\hat{\Delta}_{24}}\right]=\StructureSheaf{\hat{\Delta}_{24}}\oplus \StructureSheaf{\hat{\Delta}_{13}},
\] 
$\bar{\rho}_{\boxtimes,g}(f_1,f_2)=(f_2,f_1)$,
for $g$ the involution in $\fS_{2,\Delta}$. 
On the other hand, the linearization 
on $\StructureSheaf{\hat{\Delta}_{13}\cap \hat{\Delta}_{24}}$ induced by $\rho_{\boxtimes}$ is the sign character $\chi$,  so that  the difference homomorphism $\hat{\beta}$ in (\ref{eq-right-exact-sequence-of-hat-beta}) is $\fS_{2,\Delta}$-equivariant.  
Now, 
\[
(q\times q)_*(\StructureSheaf{\hat{\Delta}_{13}}\oplus \StructureSheaf{\hat{\Delta}_{24}})\cong 
\nu_*\StructureSheaf{\tilde{I}}\oplus (\nu_*\StructureSheaf{\tilde{I}}\otimes_\CC\chi) 
\]
 and
we get the isomorphisms 
\begin{eqnarray}
\nonumber
\nu_*\StructureSheaf{\tilde{I}} &\cong &
R(q\times q)_*^{\fS_{2,\Delta}}(\StructureSheaf{\hat{\Delta}_{13}}\oplus \StructureSheaf{\hat{\Delta}_{24}},\bar{\rho}_\boxtimes),
\\
\label{eq-pushforward-of-structure-sheaf-of-normalization-isomorphic-toequivariant-pushforward-with-chi-turned-on}
\nu_*\StructureSheaf{\tilde{I}} &\cong &
R(q\times q)_*^{\fS_{2,\Delta}}(\StructureSheaf{\hat{\Delta}_{13}}\oplus \StructureSheaf{\hat{\Delta}_{24}},\chi\otimes \bar{\rho}_\boxtimes). 
\end{eqnarray}

Set $\widehat{DI}_1:=\hat{\Delta}_{13}\cap [\hat{D}\times \widehat{S\times S}]$ and  $\widehat{DI}_2:=\hat{\Delta}_{13}\cap  [\widehat{S\times S}\times \hat{D}]$. We have $\widehat{DI}_1\cap \widehat{DI}_2=\hat{D}\times_S\hat{D}$.
Now,
$(q\circ\hat{\pi}_i)^*\StructureSheaf{S^{[2]}}(-\delta)$ restricts to $\hat{\Delta}_{13}$ as 
$\StructureSheaf{\hat{\Delta}_{13}}(-\widehat{DI}_i)$. Hence, the sum of the restriction homomorphisms yields the surjective homomorphisms
\begin{eqnarray*}
(q\circ\hat{\pi}_1)^*\StructureSheaf{S^{[2]}}(-\delta)\oplus (q\circ\hat{\pi}_2)^*\StructureSheaf{S^{[2]}}(-\delta)
&\rightarrow &\Ideal{\hat{\Delta}_{13},\hat{D}\times_S\hat{D}},
\\
(q\circ\hat{\pi}_1)^*\StructureSheaf{S^{[2]}}(-\delta)\oplus (q\circ\hat{\pi}_2)^*\StructureSheaf{S^{[2]}}(-\delta)
&\rightarrow &\Ideal{\hat{\Delta}_{24},\hat{D}\times_S\hat{D}}.
\end{eqnarray*}
Note the isomorphism 
\[
(q\circ\hat{\pi}_1)^*\StructureSheaf{S^{[2]}}(-\delta)\oplus (q\circ\hat{\pi}_2)^*\StructureSheaf{S^{[2]}}(-\delta)\cong L(q\times q)^*(\pi_1^*\StructureSheaf{S^{[2]}}(-\delta)\oplus \pi_2^*\StructureSheaf{S^{[2]}}(-\delta)).
\]
The right hand side comes with a natural $\fS_{2,\Delta}$-linearization.
Evaluating the functor $R(q\times q)_*^{\fS_{2,\Delta}}$ on the homomorphism $\hat{\alpha}$ in (\ref{eq-right-exact-sequence-of-hat-beta})$\otimes \chi$
we get 
\begin{equation}
\label{eq-morphism-that-restricts-to-a}
R_{q\times q, *}^{\fS_{2,\Delta}}\left(\StructureSheaf{\widehat{S\times S}\times \widehat{S\times S}},\chi\right)
\rightarrow 
\nu_*\StructureSheaf{\tilde{I}}.
\end{equation}

\underline{Step 4:}
We show next that the above displayed homomorphism factors through a surjective homomorphism
\begin{equation}
\label{eq-bar-alpha-into-ideal}
\pi_1^*\StructureSheaf{S^{[2]}}(-\delta)\oplus \pi_2^*\StructureSheaf{S^{[2]}}(-\delta)
\RightArrowOf{a} \nu_*\Ideal{\tilde{I},\hat{D}\times_S\hat{D}}.
\end{equation}

The complex  $R_{q\times q, *}\StructureSheaf{\widehat{S\times S}\times \widehat{S\times S}}$ is isomorphic to the
pushforward $(q\times q)_ *\StructureSheaf{\widehat{S\times S}\times \widehat{S\times S}}$, which decomposes as the direct sum of four character line subbundles for the $\fS_2\times \fS_2$-action
\[
R_{q\times q, *}\StructureSheaf{\widehat{S\times S}\times \widehat{S\times S}}\cong
\StructureSheaf{S^{[2]}\times S^{[2]}}\oplus\pi_1^*\StructureSheaf{S^{[2]}}(-\delta) \oplus\pi_2^*\StructureSheaf{S^{[2]}}(-\delta)
\oplus[\pi_1^*\StructureSheaf{S^{[2]}}(-\delta)\otimes\pi_2^*\StructureSheaf{S^{[2]}}(-\delta)].
\] 
Thus $R_{q\times q, *}^{\fS_{2,\Delta}}\left(\StructureSheaf{\widehat{S\times S}\times \widehat{S\times S}},\chi\right)$ is the rank two vector bundle
$\pi_1^*\StructureSheaf{S^{[2]}}(-\delta)\oplus\pi_2^*\StructureSheaf{S^{[2]}}(-\delta).$ 
The homomorphism $\hat{\alpha}$ in (\ref{eq-right-exact-sequence-of-hat-beta}) restricts to  
\[
R(q\times q)_*^{\fS_{2,\Delta}}(\hat{\alpha}):\pi_1^*\StructureSheaf{S^{[2]}}(-\delta)\oplus\pi_2^*\StructureSheaf{S^{[2]}}(-\delta)\rightarrow R(q\times q)_*^{\fS_{2,\Delta}}(\StructureSheaf{\hat{\Delta}_{13}}\oplus \StructureSheaf{\hat{\Delta}_{24}},\chi\otimes \bar{\rho}_\boxtimes)\cong \nu_*\StructureSheaf{I},
\] 
whose image is claimed to be $\nu_*\Ideal{\tilde{I},\hat{D}\times_S\hat{D}}$.
The exactness of $R(q\times q)_*^{\fS_{2,\Delta}}$ implies that to verify the latter claim 
it remains to prove that $R(q\times q)_*^{\fS_{2,\Delta}}(\ker(\hat{\beta})\otimes\chi)$ is isomorphic to $\nu_*\Ideal{\tilde{I},\hat{D}\times_S\hat{D}}$, where $\hat{\beta}$ is the homomorphism in (\ref{eq-right-exact-sequence-of-hat-beta})$\otimes\chi$.

Note that 
$R(q\times q)_*^{\fS_{2,\Delta}}(\StructureSheaf{\hat{D}\times_S\hat{D}},\chi)=0$ and 
$R(q\times q)_*^{\fS_{2,\Delta}}(\StructureSheaf{\hat{D}\times_S\hat{D}},1)=\StructureSheaf{D\times_S D}$.
Similarly, $R(q\times q)_*^{\fS_{2,\Delta}}(\StructureSheaf{\Delta_{\widehat{S\times S}}\cap\hat{D}\times_S\hat{D}},\chi)=0$.
The short exact sequence
\[
0\rightarrow \StructureSheaf{\Delta_{\widehat{S\times S}}\cup\hat{D}\times_S\hat{D}}\rightarrow
\StructureSheaf{\Delta_{\widehat{S\times S}}}\oplus \StructureSheaf{\hat{D}\times_S\hat{D}} \rightarrow 
\StructureSheaf{\Delta_{\widehat{S\times S}}\cap\hat{D}\times_S\hat{D}}\rightarrow 0
\]
yields the isomorphisms
\begin{eqnarray}
\label{eq-pushforward-of-structure-sheaf-of-intersection-linearized-by-chi}
\mbox{\hspace{2ex}}
R(q\times q)_*^{\fS_{2,\Delta}}(\StructureSheaf{\hat{\Delta}_{13}\cap\hat{\Delta}_{24}},\chi)&\cong&
R(q\times q)_*^{\fS_{2,\Delta}}(\StructureSheaf{\Delta_{\widehat{S\times S}}},\chi)\cong 
\StructureSheaf{\Delta_{S^{[2]}}}\otimes\pi_1^*\StructureSheaf{S^{[2]}}(-\delta),
\\
\label{eq-pushforward-of-structure-sheaf-of-intersection-linearized-by-1}
\mbox{\hspace{2ex}} R(q\times q)_*^{\fS_{2,\Delta}}(\StructureSheaf{\hat{\Delta}_{13}\cap\hat{\Delta}_{24}},1) &\cong& \StructureSheaf{[\Delta_{S^{[2]}}\cup D\times_SD]}.
\end{eqnarray}
The kernel of 
\[
\nu_*\StructureSheaf{\tilde{I}}
\stackrel{(\ref{eq-pushforward-of-structure-sheaf-of-normalization-isomorphic-toequivariant-pushforward-with-chi-turned-on})}{\cong} 
R(q\times q)_*^{\fS_{2,\Delta}}(\StructureSheaf{\hat{\Delta}_{13}}\oplus \StructureSheaf{\hat{\Delta}_{24}},\chi\otimes\bar{\rho}_\boxtimes)
\RightArrowOf{R(q\times q)_*^{\fS_{2,\Delta}}(\hat{\beta})}
R(q\times q)_*^{\fS_{2,\Delta}}(\StructureSheaf{\hat{\Delta}_{13}\cap\hat{\Delta}_{24}},1)\cong
\StructureSheaf{[\Delta_{S^{[2]}}\cup D\times_SD]}
\]
is equal to the image $\nu_*\Ideal{\tilde{I},\hat{D}\times_S\hat{D}}$ of $R(q\times q)_*^{\fS_{2,\Delta}}(\hat{\alpha})$. Indeed, if $f$ is a local function on $\hat{\Delta}_{13}$ then the local anti-invariant section
$(f,-\sigma_{12}^*\sigma_{34}^*(f))$ of $\StructureSheaf{\hat{\Delta}_{13}}\oplus \StructureSheaf{\hat{\Delta}_{24}}$
is mapped via the difference homomorphism $\hat{\beta}$ to the restriction of $f+\sigma_{12}^*\sigma_{34}^*(f)$ to $\hat{D}\times_S\hat{D}\cup \Delta_{\widehat{S\times S}}$. If, furthermore, $f$ is in the kernel of 
$R(q\times q)_*(\hat{\beta})$ then the restriction of $f$ to $\Delta_{\widehat{S\times S}}$ is anti-invariant and its restriction to
$\hat{D}\times_S\hat{D}$ vanishes. This completes the proof that the morphism (\ref{eq-morphism-that-restricts-to-a}) factors through the sheaf homomorphism $a$ in (\ref{eq-bar-alpha-into-ideal}).
The construction of the isomorphism (\ref{eq-isomorphism-of-invariant-subcomplex-of-push-forward-of-ideal-sheaf})
is thus complete.

Proof of part (\ref{lemma-item-P[2]-chi}):
Evaluating the functor $R(q\times q)_*^{\fS_{2,\Delta}}$ on the short exact sequence (\ref{eq-right-exact-sequence-of-hat-beta})
we get 
\[
0\rightarrow
\StructureSheaf{S^{[2]}\times S^{[2]}}\oplus
[\pi_1^*\StructureSheaf{S^{[2]}}(-\delta)\otimes\pi_2^*\StructureSheaf{S^{[2]}}(-\delta)]
\LongRightArrowOf{(q\times q)_ *^{\fS_{2,\Delta}}(\hat{\alpha})} 
\nu_*\StructureSheaf{\tilde{I}}
\LongRightArrowOf{(q\times q)_ *^{\fS_{2,\Delta}}(\hat{\beta})} 
\StructureSheaf{\Delta_{S^{[2]}}}\otimes\pi_1^*\StructureSheaf{S^{[2]}}(-\delta)
\rightarrow 0,
\]
where we used the isomorphism
(\ref{eq-pushforward-of-structure-sheaf-of-intersection-linearized-by-chi}).
Now, $(q\times q)_ *^{\fS_{2,\Delta}}(\hat{\alpha})$ factors through the surjective homomorphism $a_\chi$ onto
$\StructureSheaf{I}$.
\end{proof}

\begin{rem}
\begin{enumerate}
\item 
The kernel $\P^{[2]}_\chi$ in $D^b(S^{[2]}\times S^{[2]})$ of $\Phi_{\P,\chi}^{[2]}$ is the tensor product of $\pi_1^*\StructureSheaf{S^{[2]}}(\delta)$ 
with the pushforward to $S^{[2]}\times S^{[2]}$ of the ideal sheaf in 
the quotient  $\widehat{S\times S}\times \widehat{S\times S}/\fS_{2,\Delta}$ of the quotient  
$(\hat{\Delta}_{13}\cup\hat{\Delta}_{24})/\fS_{2,\Delta}$.
\item
The sheaves $E_k$ and $F_k$ in Proposition
\ref{prop-a-locally-free-image-of-a-Lagrangian-line-bundle} have the following geometric interpretation.
Set $X:=\widehat{S\times S}\times \widehat{S\times S}/\fS_{2,\Delta}$ and $I':=(\hat{\Delta}_{13}\cup\hat{\Delta}_{24})/\fS_{2,\Delta}$. 
Let $\tilde{\pi}_i$ be the two projections from $X$ to $S^{[2]}$, $i=1,2$. Set $Z:=g(F(C))$. 
Then $(\tilde{\pi}_1,\tilde{\pi}_2)_*\StructureSheaf{X}$ is $\StructureSheaf{S^{[2]}\times S^{[2]}}\oplus 
[\pi_1^*\StructureSheaf{S^{[2]}}(-\delta)\otimes\pi_2^*\StructureSheaf{S^{[2]}}(-\delta)]$.
We have 
\begin{eqnarray*}
\tilde{\pi}_{2,*}(\tilde{\pi}_1^*g_*\StructureSheaf{F(C)}(k))&\cong& 
\pi_{2,*}\left[(\tilde{\pi}_1,\tilde{\pi}_2)_*\tilde{\pi}_1^*g_*\StructureSheaf{F(C)}(k)
\right]\cong 
\pi_{2,*}\left[\pi_1^*g_*\StructureSheaf{F(C)}(k)\otimes 
\left\{
(\tilde{\pi}_1,\tilde{\pi}_2)_*\StructureSheaf{X}
\right\}\right]
\\
&\cong&
\pi_{2,*}\left[\pi_1^*g_*\StructureSheaf{F(C)}(k)\otimes 
\left\{
\StructureSheaf{S^{[2]}\times S^{[2]}}\oplus 
[\pi_1^*\StructureSheaf{S^{[2]}}(-\delta)\otimes\pi_2^*\StructureSheaf{S^{[2]}}(-\delta)]
\right\}\right].
\end{eqnarray*}
Thus, the graded sheaf of algebras
\[
\A:=\bigoplus_{k\geq 0}\left\{
[H^0(Z,\StructureSheaf{Z}(k))\otimes_\CC\StructureSheaf{S^{[2]}}]
\oplus
[H^0(Z,\StructureSheaf{Z}(k)(-\delta))\otimes_\CC\StructureSheaf{S^{[2]}}(-\delta)
]
\right\}
\]
over $S^{[2]}$ 
has relative Proj the subscheme $\tilde{\pi}_1^{-1}(Z)$ of $X$. 
The $k+1$ direct summand of $\A$ appears in (\ref{eq-short-exact-sequence-of-E-k}) due to the isomorphism $\StructureSheaf{Z}(1)\cong\StructureSheaf{Z}(\delta)$ established in Lemma \ref{lem-modular-interpretation-of-the-contraction-of-A} below.
The restriction of $\tilde{\pi}_1\times\tilde{\pi}_2$ to $I'$ is an isomorphism of the latter onto $I$. Hence, $(Z\times S^{[2]})\cap I$ is isomorphic to $\tilde{\pi}^{-1}(Z)\cap I'$. 
Let $p_i:(Z\times S^{[2]})\cap I\rightarrow S^{[2]}$ be the restriction of $\pi_i$.  The line bundle $p_1^*\StructureSheaf{Z}(\delta)$ corresponds to the graded $\A$-module
$\oplus_{k>0}[p_{2,*}p_1^*\StructureSheaf{Z}(k)(\delta)]$, which is isomorphic to 
$\oplus_{k\geq 0}F_{k+1}$, and $\oplus_{k\geq 0}E_{k+1}$ is a graded $\A$-module corresponding to the ideal of $(Z\times S^{[2]})\cap I\cong \tilde{\pi}^{-1}(Z)\cap I'$ in $\tilde{\pi}_1^{-1}(Z)$.
\hide{
\item
$\Phi_\P$ induces an involution of the Mukai lattice $\tilde{H}(S,\Integers)$. Hence, both $\Phi_\P^{[2]}$ 
and $\Phi_{\P,\chi}^{[2]}$ induce involutions of
$\tilde{H}(S^{[2]},\QQ)$. This is equivalent to the equality $ch(\P^{[2]})=\tau^*(ch((\P^{[2]})^\vee))$ and 
$ch(\P^{[2]}_\chi)=\tau^*(ch((\P^{[2]}_\chi)^\vee))$. In particular,
\[
ch(\P^{[2]}\oplus \P^{[2]}_\chi)=\tau^*(ch((\P^{[2]}\oplus \P^{[2]}_\chi)^\vee)).
\]
Let us verify the above equality.
The direct sum $\P^{[2]}\oplus \P^{[2]}_\chi$ is isomorphic to $\pi_1^*\StructureSheaf{S^{[2]}}(\delta)\otimes R(q\times q)_*(\Ideal{\hat{\Delta}_{13}\cup\hat{\Delta}_{24}})$, and so the equality above is equivalent to
\[
ch\left\{
\pi_1^*\StructureSheaf{S^{[2]}}(\delta)\otimes R(q\times q)_*(\Ideal{\hat{\Delta}_{13}\cup\hat{\Delta}_{24}})
\right\}
\cong 
\tau^*ch\left\{
\pi_1^*\StructureSheaf{S^{[2]}}(-\delta)\otimes \left[R(q\times q)_*(\Ideal{\hat{\Delta}_{13}\cup\hat{\Delta}_{24}})\right]^\vee
\right\}.
\]
The $\tau$-invariance of $R(q\times q)_*(\Ideal{\hat{\Delta}_{13}\cup\hat{\Delta}_{24}})$ implies that latter is equivalent to 
\begin{equation}
\label{eq-involution-boils-down-to}
ch\left\{
R(q\times q)_*(\Ideal{\hat{\Delta}_{13}\cup\hat{\Delta}_{24}})\otimes \left\{\StructureSheaf{S^{[2]}}(\delta)\boxtimes \StructureSheaf{S^{[2]}}(\delta)\right\}
\right\}
\cong 
ch\left\{
\left[R(q\times q)_*(\Ideal{\hat{\Delta}_{13}\cup\hat{\Delta}_{24}})\right]^\vee
\right\}.
\end{equation}
We have seen above that $\StructureSheaf{\hat{\Delta}_{13}}\cong L(b\times b)^*\StructureSheaf{\Delta_{13}}$ and
$\StructureSheaf{\hat{\Delta}_{13}\cap\hat{\Delta}_{24}}\cong \StructureSheaf{\hat{\Delta}_{13}}\stackrel{L}{\otimes} \StructureSheaf{\hat{\Delta}_{24}}$. Hence, the derived dual of 
each of $\StructureSheaf{\widehat{S\times S}\times \widehat{S\times S}}$, $\StructureSheaf{\hat{\Delta}_{13}}$, $\StructureSheaf{\hat{\Delta}_{24}}$, and $\StructureSheaf{\hat{\Delta}_{13}\cap\hat{\Delta}_{24}}$ is isomorphic to an even shift of itself. Now, 
$ch(\Ideal{\hat{\Delta}_{13}\cup\hat{\Delta}_{24}})=ch(\StructureSheaf{\widehat{S\times S}\times \widehat{S\times S}})-
ch(\StructureSheaf{\hat{\Delta}_{13}})-ch(\StructureSheaf{\hat{\Delta}_{24}})+ch(\StructureSheaf{\hat{\Delta}_{13}\cap\hat{\Delta}_{24}})$.  Thus, $ch(\Ideal{\hat{\Delta}_{13}\cup\hat{\Delta}_{24}}^\vee)=ch(\Ideal{\hat{\Delta}_{13}\cup\hat{\Delta}_{24}})$.

We have $\omega_{q\times q}\cong (q\times q)^*\left(\StructureSheaf{S^{[2]}}(\delta)\boxtimes \StructureSheaf{S^{[2]}}(\delta)\right)$. Hence,
\[
\left[R(q\times q)_*(\Ideal{\hat{\Delta}_{13}\cup\hat{\Delta}_{24}})\right]^\vee\cong
R(q\times q)_*(\Ideal{\hat{\Delta}_{13}\cup\hat{\Delta}_{24}}^\vee\otimes \omega_{q\times q})\cong
R(q\times q)_*(\Ideal{\hat{\Delta}_{13}\cup\hat{\Delta}_{24}}^\vee)\otimes \{\StructureSheaf{S^{[2]}}(\delta)\boxtimes \StructureSheaf{S^{[2]}}(\delta)\},
\]
where the first isomorphism is due to Grothedieck-Verdier duality and the second due to the projection formula. The self duality of $ch(\Ideal{\hat{\Delta}_{13}\cup\hat{\Delta}_{24}})$   yields the desired Equality (\ref{eq-involution-boils-down-to}).
}
\end{enumerate}
\end{rem}

%
\subsection{Reflexivity of the Fourier-Mukai image of $\StructureSheaf{Z}(k)$}
\label{sec-reflexivity-of-FM-image-of-lb-over-lagrangian-surface}
Let $\LB$ be a line bundle over $S^{[2]}$. 
Let $\pi_i$ be the projection from $S^{[2]}\times S^{[2]}$ onto the $i$-th factor, $i=1,2$.
Let $Z$ be a two-dimensional non-singular subvariety of $S^{[2]}$.  Denote by $e:Z\rightarrow S^{[2]}$ the inclusion.

\begin{lem}
\label{lemma-object-satisfying-3-conditions-is-isomorphic-to-a-locally-free-sheaf}
Assume that the following three conditions hold.
\begin{enumerate}
\item
\label{assumption-finiteness}
The projection $\pi_2$ restricts to the intersection $(Z\times S^{[2]})\cap I$ as a finite and surjective morphism onto $S^{[2]}$
\begin{equation}
\label{eq-restriction-of-pi-2-to-4-dimensional-intersection}
\pi_2:(Z\times S^{[2]})\cap I\rightarrow S^{[2]}.
\end{equation}
\item
\label{assumption-vanishing}
$H^i(Z,e^*\LB)=0$ and $H^i(Z,e^*(\LB(\delta)))=0$, for $i>0$. 
\item
\label{assumption-surjectivity}
Let $d$ be the maximal length of fibers of the finite morphism (\ref{eq-restriction-of-pi-2-to-4-dimensional-intersection}).
The homomorphism $H^0(Z,e^*(\LB(\delta)))\rightarrow H^0(A,\restricted{\LB(\delta)}{A})$ is surjective, for every zero-dimensional subscheme $A$ of $Z$ of length $\leq d$. 
\end{enumerate}
Then $\Phi^{[2]}_{\P,\chi}(e_*e^*\LB)$  is represented by a torsion free reflexive sheaf.
\end{lem}

\begin{proof}
Note that 
$\Phi^{[2]}_{\P,\chi}(e_*e^*\LB)$ 
is represented by the complex 
\begin{equation}
\label{eq-complex-of-locally-free-sheaves-with-locally-free-kernel}
H^0(Z,e^*(\LB(\delta)))\otimes_\CC\StructureSheaf{S^{[2]}}
\oplus
H^0(Z,e^*\LB)\otimes_\CC\StructureSheaf{S^{[2]}}(-\delta)
\RightArrowOf{a_\chi} \pi_{2,*}\left(\pi_1^*(\LB(\delta))\otimes\StructureSheaf{(Z\times S^{[2]})\cap I}\right),
\end{equation}
by Lemma \ref{lemma-kernel-of-Phi_P^[2]} and
assumptions (\ref{assumption-finiteness}) and (\ref{assumption-vanishing}). Furthermore, the homomorphism $a_\chi$ is surjective, by assumption (\ref{assumption-surjectivity}), and so $\Phi^{[2]}_{\P,\chi}(e_*e^*\LB)$ is represented by the sheaf $\ker(a_\chi)$. Each of the irreducible components of $(Z\times S^{[2]})\cap I$ is $4$-dimensional, since it has dimension $\geq 4$ and is finite over $S^{[2]}$.  Hence, the sheaf $\pi_{2,*}\left(\pi_1^*(\LB(\delta))\otimes\StructureSheaf{(Z\times S^{[2]})\cap I}\right)$ is torsion free and so $\ker(a_\chi)$ is reflexive, being a saturated subsheaf of a locally free sheaf. 
\end{proof}

\begin{rem}
\begin{enumerate}
\item
The sheaf $\ker(a_\chi)$ is locally free away from a closed subscheme of $S^{[2]}$ of co-dimension 
$\geq 3$ as it is reflexive. We claim that this closed subscheme of singularities of $\ker(a_\chi)$ is contained in $Z$. It suffices to show that the morphism 
(\ref{eq-restriction-of-pi-2-to-4-dimensional-intersection}) is flat over $S^{[2]}\setminus Z$, since then the sheaf 
$F:=\pi_{2,*}\left(\pi_1^*(\LB(\delta))\otimes\StructureSheaf{(Z\times S^{[2]})\cap I}\right)$ is locally free over $S^{[2]}\setminus Z$ and the kernel of a surjective homomorphism of locally free sheaves is locally free as well.
The length of the fiber of (\ref{eq-restriction-of-pi-2-to-4-dimensional-intersection}) over $y\in S^{[2]}$ is equal to 
 the length of the intersection subscheme $(Z\times S^{[2]})\cap I\cap (S^{[2]}\times\{y\})$ of $S^{[2]}\times S^{[2]}$, which is independent of $y$ as long as it is of the expected dimension $0$ and 
 the object $\StructureSheaf{Z\times S^{[2]}}\stackrel{L}{\otimes}\StructureSheaf{I}\stackrel{L}{\otimes}\StructureSheaf{S^{[2]}\times\{y\}}$ is represented by the tensor product sheaf, i.e., all the higher torsion sheaves vanish. 
This is the case if the intersection  $(Z\times S^{[2]})\cap I\cap (S^{[2]}\times\{y\})$  is contained in the non-singular locus of $I$.
Indeed, in that case each of $Z\times S^{[2]}$, $I$, and $S^{[2]}\times\{y\}$ is non singular at every point $p$ of their common intersection and
their local equations at $p$ yield a regular sequence for the subscheme $(Z\times S^{[2]})\cap I\cap (S^{[2]}\times\{y\})$ locally around $p$. The Koszul complex associated to that regular sequence is exact on the one hand, and is isomorphic to the tensor product of the three Koszul complexes associated to the local equations of each of the three subvarieties. Hence, the higher torsion sheaves vanish, and the intersection multiplicity at each point $p$ is equal to the length of the intersection subscheme at $p$.
We have seen in Step 3 of the proof of Lemma \ref{lemma-kernel-of-Phi_P^[2]} that the singular locus of $I$ is the diagonal $\Delta_{S^{[2]}}$. Hence, 
if $y$ does not belong to $Z$ the fiber over $y$ is contained in the non-singular locus of $I$.
The morphism (\ref{eq-restriction-of-pi-2-to-4-dimensional-intersection}) is thus flat over $S^{[2]}\setminus Z$, by \cite[Theorem III.9.9]{hartshorne}. 
\item
The sheaf $F$ is not locally free over $Z$, hence not reflexive. Indeed, if $y\in Z$ and $y$ is reduced, then the intersection $I\cap (S^{[2]}\times\{y\})$ consists of two smooth surfaces in $S^{[2]}\times\{y\}$ meeting at $(y,y)$, and so the intersection of $I\cap (S^{[2]}\times\{y\})$ with 
$Z\times \{y\}$ is a subscheme with stalk at $(y,y)$ of length at least $3$ (while the intersection multiplicity is $2$ if $Z$ is transversal at $y$ to each of the two surfaces), by 
\cite[Appendix A, Example 1.1.1]{hartshorne}. Nevertheless, $\ker(a_\chi)$ is locally free over the generic point of $Z$, being reflexive.
\end{enumerate}
\end{rem}

\begin{lem}
\label{lemma-lagrangian-surface-meeting-I-along-a-12-sheeted-finite-cover}
The morphism $g$, given in (\ref{eq-g}), is a closed immersion. 
$Z:=g(F(C))$ is a lagrangian surface in $S^{[2]}$ and  the projection $\pi_2$ restricts to $(Z\times S^{[2]})\cap I$ as a finite and surjective morphism of degree $12$ onto $S^{[2]}$. 
\end{lem}

\begin{proof}
Let $q(x_0, \dots, x_4)$ and $c(x_0, \dots,x_4)$ be homogeneous quadratic and cubic forms with zero divisors $Q$ and $C$ respectively satisfying Assumption \ref{assumption-on-Q-and-C}.
Set 
\[
d(x_0, \dots, x_4,x_5):=x_5q(x_0, \dots, x_4)+c(x_0, \dots,x_4)
\] 
and let $X$ be the corresponding cubic fourfold in $\PP^5$. 
Then $X$ has a single ordinary double point at $p:=(0\!:\!0\!:\!0\!:\!0\!:\!0\!:\!1)$. 
The Fano variety $F(X)$ of lines on $X$ is birational to $S^{[2]}$, as observed in 
\cite[Example 4.3]{mukai-survey}. 
Let $F_0(X)\subset F(X)$ be the open subset of lines not passing through $p$.
Let $A\subset S^{[2]}$ the the divisor of subschemes $y$, such that the line $\bar{\ell}_y$ in $\PP^4$ is contained in $Q$.
Note that $A$ admits a morphism of degree $3$ onto the Fano variety $F(Q)$ of lines on $Q$, and that the latter is isomorphic to $\PP^3$. Note also that $A$ admits a morphism 
\begin{equation}
\label{eq-phi-from-A-to-S}
\phi:A\rightarrow S, 
\end{equation}
by sending a subscheme $y\in A$ to the point of $\bar{\ell}_y\cap C$ complementary to $y$. In fact, $\phi:A\rightarrow S$ is a conic bundle, which is the restriction to $S$ of the following conic bundle $\tilde{\phi}:\tilde{A}\rightarrow Q$. Given a point $a\in Q$, let $a^\perp$ be the hyperplane in $\CC^5$ orthogonal to $a$ with respect to $q$ and denote by $q_a$ the non-degenerate quadratic form on $a^\perp/a$ induced by $q$. 
Then the fiber of $\tilde{\phi}$ over $a\in Q$ is the conic in the plane $\PP(a^\perp/a)$ defined by $q_a$, as the latter conic parametrizes lines on $Q$ through $a$.
Let $S^{[2]}_0$ be the complement of $A$. 
Let $\iota:S\rightarrow F(X)$ be the embedding of $S$ as the singular locus of $F(X)$, sending $q\in S$ to the line $\overline{pq}$ through $p$ and $q$. 
The following is a slightly more detailed version of \cite[Example 4.3]{mukai-survey}.

\begin{claim}
\label{claim-fano-of-cubic-with-ODP-is-resolved-by-Hilbert-scheme}
There exists a canonical morphism $f:S^{[2]}\rightarrow F(X)$, which restricts to an 
isomorphism $f:S^{[2]}_0\rightarrow F_0(X)$.
Furthermore, $f$ restricts to an isomorphism from $Z$ onto $F(C)$, which is the inverse of $g$, and 
the following diagram is commutative
\[
\xymatrix{
S^{[2]} \ar[r]^{f} & F(X)
\\
A \ar[u]_{\cup} \ar[r]_{\phi} & S. \ar[u]_{\iota} 
}
\]
\end{claim}

\begin{proof}
Let $\hat{Q}$ and $\hat{C}$ be the cones over $Q$ and $C$ given by $q$ and $c$ considered as forms in $x_0$, \dots, $x_5$. A line on $X$ passes through $p$, if and only if it is contained in both $\hat{Q}$ and $\hat{C}$.
Indeed, if $\ell$ passes through $p$ then a general point of $\ell$ is of the form 
$(sa_0\!:\!sa_1\!:\!sa_2\!:\!sa_3\!:\!sa_4\!:\!t)$, $(s\!:\!t)\in\PP^1$, 
so that $d(sa_0,sa_1,sa_2,sa_3,sa_4,t)=ts^2q(a_0, \dots, a_4)+s^3c(a_0, \dots, a_4)$ vanishes for all $(s\!:\!t)$, so that 
$a:=(a_0\!:\!\cdots \!:\!a_4)$ is a point of $S$. Conversely, if $\ell$ is contained in $\hat{Q}\cap \hat{C}$, then it must pass through $p$, since otherwise it would project to a line in $S$ contrary to our assumption that $S$ does not contain any lines. 

Let $H$ be the hyperplane $x_5=0$. 
Given a point $y$ of $S^{[2]}$, we get the line $\bar{\ell}_y$ in $H$ intersecting $S$ along a subscheme containing $y$,\footnote{The scheme $\bar{\ell}_y\cap S$ is equal to $y$, if and only if $y$ does not belong to $A$. If $y$ belongs to $A$, then $y$ is a proper subscheme of $\bar{\ell}_y\cap S$, as the latter is equal to $\bar{\ell}_y\cap C$ and so has length $3$.} hence we get the plane $P_y$ in $\PP^5$ spanned by $p$ and $\bar{\ell}_y$. $P_y$ is not contained in $X$, since otherwise it would be contained in $\hat{Q}\cap\hat{C}$, by the above discussion, and so project to a line in $S$.
If $y$ consists of two distinct points $a$ and $b$ of $S$, then $P_y$ intersects $X$ along three lines, two of them are $\overline{pa}$ and $\overline{pb}$, and we set $f(y)$ to be the third line.
If $y$ is non-reduced supported set theoretically on a point $a$ of $S$, 
then $P_y$ intersects $X$ along $2\overline{pa}+f(y)$ for a line $f(y)$ in $P_y$.  
We claim that $f(y)$ passes through $p$, if and only if $y$ belongs to $A$. Indeed, let
$a=(a_0\!:\!a_1\!:\!a_2\!:\!a_3\!:\!a_4\!:\!0)$, $b=(b_0\!:\!b_1\!:\!b_2\!:\!b_3\!:\!b_4\!:\!0)$, and choose coordinates $(s\!:\!t\!:\!u)$ for $P_y$ via
\[
\eta(s\!:\!t\!:\!u)= s(a_0,a_1,a_2,a_3,a_4,0)+t(b_0,b_1,b_2,b_3,b_4,0)+(0,0,0,0,0,u).
\]
Then the cubic form $d$ pulls back to
\[
2q(a,b)ust+3s^2tc(a,a,b)+3st^2c(a,b,b)=st[2q(a,b)u+3c(a,a,b)s+3c(a,b,b)t],
\] 
where the symmetric bilinear form $q(\bullet,\bullet)$ is the polarization of $q$ and the symmetric trilinear form $c(\bullet,\bullet,\bullet)$ is that of $c$. 
The morphism $\eta$ maps
the line $s=0$ to $\overline{bp}$, the line $t=0$  to $\overline{ap}$, and 
the line cut out by $2q(a,b)u+3c(a,a,b)s+3c(a,b,b)t$ to $f(y)$. Note that $q(a,b)= 0$ if and only if the line $\ell_y$ is contained in $Q$. Now, $\eta(0\!:\!0\!:\!1)=p$, so $f(y)$ passes through $p$ if and only if $q(a,b)= 0$. 
A similar argument shows that $f(y)$ passes through $p$ if and only if $y$ belongs to $A$ also in the case where $y$ is non-reduced. The commutativity of the diagram in the statement of the claim follows.

If a line $\ell$ on $X$ is contained in $\hat{Q}$, then it is contained in $\hat{C}$ and so passes through $p$. Hence, if a line $\ell$ on $X$ does not pass through $p$, then it intersects $\hat{Q}$ along a length $2$ subscheme. 
Thus, the projection $\bar{\ell}$ of $\ell$ from $p$ intersects $Q$ along a length $2$ subscheme $\tilde{g}(\ell)\in S_0^{[2]}$. 
We get a morphism $\tilde{g}:F_0(X)\rightarrow S_0^{[2]}$, which is clearly the inverse of the restriction of $f$ to $S_0^{[2]}$.
We conclude that $f:S_0^{[2]}\rightarrow F_0(X)$ is an isomorphism. This completes the proof of Claim \ref{claim-fano-of-cubic-with-ODP-is-resolved-by-Hilbert-scheme}.
\end{proof}

The Fano variety of lines $F(C)$  is identified with $F(X\cap H)$.
The point $p$ does not belong to $X\cap H$, and so $F(X\cap H)$ is contained in $F_0(X)$. 
The  subvariety $F(X\cap H)$ of $F(X)$ is lagrangian, by \cite{voisin-lagrangian}. Hence, $Z:=\tilde{g}(F(X\cap H))$
is a lagrangian surface in $S^{[2]}$. 
Through a generic point of $C$ pass $6$ lines lying on $C$, by \cite{clemens-griffiths}, and the set of lines on $C$ through every point of $C$ is finite, by Assumption \ref{assumption-on-Q-and-C}. 
The fiber of the restriction of $\pi_2$ to $Z\times S^{[2]}\cap I$ over a point $y\in S^{[2]}$ 
consists of length $2$ subschemes of $S$ lying on a line passing through one of the two points of $y$ and is hence finite. If $y$ consists of two distinct points $a$, $b$ through each of which pass exactly $6$ lines on $C$, and such that the line $\overline{ab}$ does not lie on $C$,  
then
the fiber of 
\begin{equation}
\label{eq-finite-flat-projective-morphism}
\pi_2:(Z\times S^{[2]})\cap I\rightarrow S^{[2]}
\end{equation}
over $y$ consists of precisely $12$ points, as 
 none of the $12$  lines on $C$ through $a$ or $b$  lies on $Q$, since $S$ does not contain a line.
 The morphism (\ref{eq-finite-flat-projective-morphism}) is  projective and has finite fibers, hence is a finite morphism (see \cite[Ex. III.11.2]{hartshorne}). The morphism (\ref{eq-finite-flat-projective-morphism}) is  clearly surjective.
This completes the proof of Lemma \ref{lemma-lagrangian-surface-meeting-I-along-a-12-sheeted-finite-cover}.
\end{proof}

%
\subsection{Proof of Proposition \ref{prop-a-locally-free-image-of-a-Lagrangian-line-bundle}}
\label{sec-proof-of-the-reflexivity-proposition}
Let $h=c_1(\StructureSheaf{S}(1))$ be the class of a hyperplane section of $S$ in $\PP^4$.
Let $\tilde{h}$ be the nef class of $\Pic(S^{[2]})$ corresponding to $h$ via (\ref{eq-theta}).
Let $\delta\in \Pic(S^{[2]})$ be half the class of the divisor of non-reduced subschemes.


\begin{lem}
\label{lem-modular-interpretation-of-the-contraction-of-A}
\begin{enumerate}
\item
\label{lemma-item-class-of-A}
The class of the divisor $A$ is $[A]=\tilde{h}-2\delta$.
\item
\label{lemma-item-pullback-via-f-of-hyperplane-section}
The pullback via $f$ of
$\StructureSheaf{F(X)}(1)$ is $\lambda:=2\tilde{h}-3\delta$.
\item
\label{lemma-item-nef-cone}
If $\Pic(S)=\Integers h$, then
the nef cone of $S^{[2]}$ is $\langle\lambda,\tilde{h}\rangle.$ 
\item
\label{lemma-item-restriction-of-lambda-and-delta-are-equal}
The restrictions of $\lambda$ and $\delta$ to $F(C)$ are equal.
\end{enumerate}
\end{lem}

\begin{proof} (\ref{lemma-item-class-of-A}) 
We may and do assume that $\Pic(S)$ is cyclic generated by $h$, as the parts of the lemma not assuming it follow from this case by continuity.
The  contraction $f:S^{[2]}\rightarrow F(X)$ of the divisor $A$ in Claim \ref{claim-fano-of-cubic-with-ODP-is-resolved-by-Hilbert-scheme} admits the following modular interpretation. The divisor $A$ consists precisely of those  subschemes $y$, such that $I_y(1)$ is not generated by global sections. Indeed, $y$ belongs to $A$, if and only if $\bar{\ell}_y$ intersects $S$ in a length $3$ subscheme. Every hyperplane section of $S$ containing $y$ contains $S\cap\bar{\ell}_y$, so 
$H^0(S,I_y(1))$ generates the subsheaf $I_{\bar{\ell}_y\cap S}(1)$ of $I_y(1)$. 
Let $M(2,h,1)$ be the moduli space of $h$-stable sheaves on $S$ with Mukai vector $(2,h,1)$.
There is a natural isomophism $\eta:M(1,h,2)\rightarrow M(2,h,1)$, by 
\cite[Theorem 3.21]{markman-brill-noether} and the fact that 
the moduli space $M(2,h,3)$ is empty. $M(1,h,2)$ parametrizes sheaves of the form $I_y(1)$, where $y$ is a length $2$ subscheme of $S$, and is thus naturally isomorphic to $S^{[2]}$. 
The divisor $\eta(A)$ is precisely the divisor of non-locally-free sheaves in 
$M(2,h,1)$, by construction of $\eta$. Explicitly, given $y\in A$, let $ev:H^0(I_y(1))\otimes_\CC\StructureSheaf{S}\rightarrow I_y(1)$ be the evaluation homomorphism. Then $\ker(ev)$ belongs to $M(2,-h,2)$, which is zero dimensional, so $\ker(ev)$ is isomorphic to the unique rigid and stable vector bundle $E$ with Mukai vector $(2,-h,2)$. The torsion free sheaf 
$\eta(I_y(1))$ is a subsheaf of $E^*$ and the quotient $E^*/\eta(I_y(1))$ is supported on the point $\phi(y)$ of $S$ complementary to the subscheme $y$ in $\bar{\ell}_y\cap S$. We get the commutative diagram
\[
\xymatrix{
A\ar[r]^{\restricted{\eta}{A}} \ar[d]_\phi &\PP(E^*)\ar[d]
\\
S\ar[r]_{=}&S,
}
\]
where $\phi$ is given in (\ref{eq-phi-from-A-to-S}) and the right vertical arrow is the natural projection.
The contraction $f:S^{[2]}\rightarrow F(X)$ is identified with the 
Gieseker-Li morphism from $M(2,h,1)$ to the normalization of the 
Uhlenbeck-Yau compactification \cite{jun-li}. 
Set $v:=(1,h,2)$ and $w=(2,h,1)$. 
The divisor $A$ corresponds to the image of the Mukai vector $(2,h,2)$
in $H^2(M(w),\Integers)$ via the Mukai Hodge isometry $\theta_w:w^\perp\rightarrow H^2(M(w),\Integers)$
between the latter and the sublattice $w^\perp$ of the Mukai lattice of $S$ orthogonal to $w$, by \cite[Lemma 10.14]{markman-exceptional}. 
The isometry $\eta_*:H^2(M(v),\Integers)\rightarrow H^2(M(w),\Integers)$  
is the image $\monrep(\gamma)$ via the parallel-transport representation $\monrep$ in \cite[Theorem 1.6]{markman-monodromy}
of the isometry $\gamma:=-D\circ R_{(1,0,1)}$ of the Mukai lattice, where $D(r,c,s)=(r,-c,s)$ and $R_{(1,0,1)}$ is the reflection with respect to the sublattice orthogonal to $(1,0,1)$ \cite[Theorem 3.21]{markman-brill-noether}.
Now, $\monrep(\gamma)$ 
corresponds, via the Mukai isometries $\theta_v$ and $\theta_w$, to the isometry $-\gamma_|:v^\perp\rightarrow w^\perp$, which is minus
the restriction of 
the isometry $\gamma$.
The isometry $\gamma$ sends $v$ to $w$ and sends the class $(2,h,2)$ to itself. Hence, the class of $A$ in
$H^2(M(v),\Integers)$ is $[A]=-\theta_v((2,h,2))$. 

The class of the divisor  of non-reduced subschemes in $M(v)$ is  $2\delta$, where $\delta:=\theta_v((1,h,4))$. 
The nef class $\tilde{h}$ orthogonal to $\delta$, corresponding to the class $h\in H^2(S,\Integers)$,  is 
$\tilde{h}=\theta_v((0,h,6))$. The equality $[A]=\tilde{h}-2\delta$ follows.

(\ref{lemma-item-pullback-via-f-of-hyperplane-section}), (\ref{lemma-item-nef-cone}) 
The extremal rays associated to the two divisorial contractions correspond via the isomorphism $H^2(M(v),\QQ)\rightarrow H_2(M(v),\QQ)$, $\alpha\mapsto (\alpha,\bullet)$, to the rays spanned by the classes $[A]$ and $\delta$ (see \cite[Cor. 3.6]{markman-exceptional}). 
The nef cone of $M(v)$ is thus dual to the cone generated by the classes  $[A]$ and $\delta$. $F(X)$ is a subvariety of $Gr(2,6)$
and $\StructureSheaf{Gr(2,6)}(1)$ pulls back via $f$ to a nef class $\lambda$ satisfying $(\lambda,\lambda)=6$ and $\lambda$ is orthogonal to $[A]$.
Hence $\lambda$ corresponds to $-\theta_v((3,h,0))$. The equality $\lambda=2\tilde{h}-3\delta$ follows.
The nef cone is thus $\langle 
\lambda,\tilde{h}
\rangle$.

(\ref{lemma-item-restriction-of-lambda-and-delta-are-equal}) 
The restriction of $\lambda$ and of $\delta$ to $F(C)$
are equal, as $\lambda=2[A]+\delta$ and $A$ is disjoint from $F(C)$.
\end{proof}

We are ready to prove Proposition \ref{prop-a-locally-free-image-of-a-Lagrangian-line-bundle}.

\begin{proof} 
The surface $g(F(C))$ is smooth and lagrangian, by Lemma \ref{lemma-lagrangian-surface-meeting-I-along-a-12-sheeted-finite-cover}. Let $L$ be the line bundle over $S^{[2]}$ with class $\lambda$, so that $L=f^*\StructureSheaf{F(X)}(1)$, by Lemma \ref{lem-modular-interpretation-of-the-contraction-of-A}. The ampleness of the restriction of $L$ to $g(F(C))$ and 
Lemma \ref{lemma-lagrangian-surface-meeting-I-along-a-12-sheeted-finite-cover} imply that 
the object  $\Phi^{[2]}_\P(L^k\otimes g_*\StructureSheaf{F(C)})$ 
satisfies the three conditions of Lemma \ref{lemma-object-satisfying-3-conditions-is-isomorphic-to-a-locally-free-sheaf}, 
with $\LB=L^k$, for all $k$ sufficiently large,  and
is hence isomorphic to a torsion free reflexive sheaf.

The pair  $(S^{[2]},g_*\StructureSheaf{F(C)})$ is the limit of a flat family of pairs $(F(X_t),\StructureSheaf{F(X_t\cap H_t)})$, where $X_t$ is a smooth cubic fourfold and $X_t\cap H_t$ is a smooth hyperplane section, as $S^{[2]}$ is the minimal resolution of an $A_1$ singularity along the embedding of $S$ as the singular locus of $F(X)$, by Claim \ref{claim-fano-of-cubic-with-ODP-is-resolved-by-Hilbert-scheme}. Such a flat family is obtained from a flat deformation of $F(X)$ after passage to a double cover of the base of that family, see \cite{markman-modular-galois-covers}. (This was the main point of Mukai's \cite[Example 4.3]{mukai-survey}, as remarked in \cite[Footnote 38]{mukai-survey}).
The sheaf $g_*\StructureSheaf{F(C)}$ thus  has a rank $1$ cohomological obstruction map,
by Lemma \ref{lemma-chern-character-of-Lagrangian-structure-sheaf-deforms-in-co-dimension-1} and Example \ref{example-lagrangian-surfaces-which-deform-in-co-dimension-1}.
$E_k$  has a rank $1$ cohomological obstruction map
as well, as the property is invariant under derived equivalences. 

The exact sequence (\ref{eq-short-exact-sequence-of-E-k}) is obtained from 
(\ref{eq-complex-of-locally-free-sheaves-with-locally-free-kernel}), with $\LB=L^k$, 
by adding the kernel of the homomorphism $a_\chi$ and using the equality of the restrictions  to $Z=g(F(C))$ of $\delta$ and of $L$
established in Lemma \ref{lem-modular-interpretation-of-the-contraction-of-A}.
The rank of $F_k$ is 12, by Lemma \ref{lemma-lagrangian-surface-meeting-I-along-a-12-sheeted-finite-cover}.
The rank of $E_k$, for $k$ sufficiently large,  is thus 
\[
-12+\dim H^0(Z,\StructureSheaf{Z}(k+1))+\dim H^0(Z,\StructureSheaf{Z}(k)).
\] 
The multiplication table (\ref{eq-multiplication-table}) and 
Lemma \ref{lemma-ch-of-Lagrangian-surface-in-K3-2}, with $(\lambda,\lambda)=6$, $c=5/8$, and $t=1$ as in Example \ref{example-lagrangian-surfaces-which-deform-in-co-dimension-1}, yields the third equality below.
\begin{eqnarray*}
\chi(\StructureSheaf{Z}(k))&=&
\int_{S^{[2]}}ch(\StructureSheaf{Z})\exp(k\lambda)td_X
=\chi(\StructureSheaf{Z})+ \frac{k^2}{2}\int_{S^{[2]}} [Z]\lambda^2+k\int_{S^{[2]}} ch_3(\StructureSheaf{Z}))\lambda
\\ 
&=& 6 +\frac{45k^2}{2}-\frac{45k}{2}=6+45\frac{k(k-1)}{2}.
\end{eqnarray*}
Hence, the rank of $E_k$, for $k$ sufficiently large, is $45k^2$.

It remains to calculate $\ell(E_k)$ and $c_1(E_k)$. The latter is determined by $\rank(E_k)$ and the former, by Theorem \ref{thm-Mukai-vector}(\ref{prop-item-spanning-Mukai-vector}).
Now 
$\ell(\StructureSheaf{Z})=(0,\lambda,-3)$, 
by Lemma \ref{lemma-chern-character-of-Lagrangian-structure-sheaf-deforms-in-co-dimension-1}, 
$\ell(\StructureSheaf{Z}(k))=B_{k\lambda}(0,\lambda,-3)=(0,\lambda,6k-3)$, and 
$\ell(E_k)$ is the image via $\Phi^{[2]}_{\P,\chi}=\Phi^{[2]}_\chi\circ \Phi^{[2]}_\P$ of $\ell(\StructureSheaf{Z}(k))$. 
$\Phi_\P$ acts on the Mukai lattice of $S$ via $\phi_\P(r,\alpha,s)= (s,-\alpha,r)$.
According to \cite[Prop. 9.3]{taelman} $\Phi_\P^{[2]}$ acts on the LLV lattice of $S^{[2]}$ by
\[
\widetilde{H}(\phi_\P^{[2]})=\det(\phi_\P)(B_{-\delta/2}\circ \iota(\phi_\P)\circ B_{\delta/2}),
\]
where $\iota(\phi_\P)$ leaves the class $\delta\in H^2(S^{[2]})$ invariant and acts of $\delta^\perp$ by $\phi_\P$ via the natural identification 
of $\delta^\perp$ with the LLV lattice of $S$. Note that $\det(\phi_\P)=-1$, $\lambda=2\tilde{h}-3\delta$, $(\delta/2,\lambda)=3$, 
$B_{\delta/2}(0,\lambda,6k-3)=(0,\lambda,6k)$, and 
$\iota(\phi_\P)(0,\lambda,6k)=(6k,-2\tilde{h}-3\delta,0)$. Hence,
\[
\widetilde{H}(\phi_\P^{[2]})(0,\lambda,6k-3)=-B_{-\delta/2}(6k,-2\tilde{h}-3\delta,0)=-(6k,-2\tilde{h}-(3+3k)\delta,-3-(3/2)k).
\]

Let $\phi^{[2]}_\chi:H^*(S^{[2]},\QQ)\rightarrow H^*(S^{[2]},\QQ)$ be the automorphism induced by $\Phi^{[2]}_\chi$.
Let $\widetilde{H}(\phi^{[2]}_\chi):\widetilde{H}(S^{[2]},\QQ)\rightarrow \widetilde{H}(S^{[2]},\QQ)$ be the induced isometry via 
(\ref{eq-action-of-DMon-on-Mukai-lattice}), which appeared in Lemma \ref{lemma-isometry-of-BKR-conjugate-of-tensorization-by-sign-character}.
Lemma \ref{lemma-isometry-of-BKR-conjugate-of-tensorization-by-sign-character} yields 
the equality $\widetilde{H}(\phi_\chi^{[2]}\circ\phi_\P^{[2]})=-R_{(0,\delta,1)}\circ \widetilde{H}(\phi_\P^{[2]})$,  and so
\[
\widetilde{H}(\phi_\chi^{[2]}\circ\phi_\P^{[2]})(0,\lambda,6k-1)))=\left(6k,-2\tilde{h}-(3k-3)\delta,3-\frac{3}{2}k\right).
\]
Thus, $c_1(E_k)=-15k\tilde{h}-\frac{45k(k-1)}{2}\delta$, by Theorem \ref{thm-Mukai-vector}(\ref{prop-item-spanning-Mukai-vector}).
\end{proof}

\begin{rem}
\label{rem-two-subsheaves-do-not-destabilize}
Let $a'_k$ be the restriction of $a_k$ to the first direct summand in (\ref{eq-short-exact-sequence-of-E-k}) and let $a_k''$ be its restriction to the second direct summand.
Then $G'_k:=\ker(a'_k)$ and $G''_k:=\ker(a''_k)$ are saturated subsheaves of $E_k$ and $G''_k\cong G'_{k-1}(-\delta)$. Let us determine the open subcone of the ample cone of $S^{[2]}$, where $G'_k$ and $G''_k$ do not slope-destabilize $E_k$, when
$\Pic(S)$ is cyclic. In that case the nef cone of $S^{[2]}$ is $\langle\tilde{h},2\tilde{h}-3\delta\rangle$, by Lemma \ref{lem-modular-interpretation-of-the-contraction-of-A}. Hence, the class $\kappa:=j\tilde{h}-\delta$, $j\in\QQ$, is ample precisely when $j>2/3$ and every rational ample class is a multiple of such a $\kappa$. 
The {\em normalized $\kappa$-slope} of a sheaf $F$ on $S^{[2]}$ of positive rank $r$ is
$\mu_\kappa(F):=(\kappa,c_1(F))/r$. The equality
\[
\int_{S^{[2]}}\kappa^3c_1(F)=3(\kappa,\kappa)(\kappa,c_1(F))
\]
implies that the  slope of $F$ is $3(\kappa,\kappa)\mu_\kappa(F)$, and is hence proportional to $\mu_\kappa(F)$ by a positive constant independent of $F$.
The normalized $\kappa$-slopes with respect to $\kappa$ of the three sheaves are as follows.
\[
\begin{array}{cccc}
\underline{\mbox{sheaf}} & \underline{\mbox{rank}} & \underline{\hspace{10ex}c_1\hspace{10ex}} & \underline{\hspace{8ex}\mu_\kappa\hspace{8ex}}
\\
E_k & 45k^2 & -15 k\tilde{h}-\frac{45k(k-1)}{2}\delta & -\left(\frac{2j+k-1}{k}\right)
\\
G'_k & -6+\frac{45k(k+1)}{2} & -15k\tilde{h}+6\delta & \frac{-60kj+8}{15k(k+1)-4}
\\
G''_k & -6+\frac{45k(k-1)}{2} & -15(k-1)\tilde{h}+\left(12-\frac{45k(k-1)}{2}\right)\delta & 
\frac{16-(k-1)[30k+60j]}{15k(k-1)-4}
\end{array}
\]
The inequality $\mu_\kappa(G'_k)>\mu_\kappa(G''_k)$ holds, for $k\geq 2$.
Hence, for 
$j>\frac{k+1}{2}$
and $k\geq 2$ the normalized $\kappa$-slopes $\mu_\kappa(G'_k)$ and $\mu_\kappa(G''_k)$ are smaller than $\mu_\kappa(E_k)$.
\end{rem}

\begin{lem}
\label{lem-homomorphism-from-E-k-to-almost-trivial-vb}
We have the short exact sequence
\begin{equation}
\label{eq-exact-triangle}
0\rightarrow G'_k\rightarrow E_k\RightArrowOf{\eta'} H^0(Z,\StructureSheaf{Z}(k))\otimes_\CC\StructureSheaf{S^{[2]}}(-\delta)\rightarrow 0,
\end{equation}
where the homomorphism $\eta'$ 
is the natural one arising from the isomorphism
$H^0(Z,\StructureSheaf{Z}(k))\cong \Hom(E_k,\StructureSheaf{S^{[2]}}(-\delta))^*$.
The sheaf $G'_k$ is isomorphic to the image via $\Phi^{[2]}_{\P.\chi}$ of the object 
\begin{equation}
\label{eq-preimage-of-G'-k}
\StructureSheaf{Z}(k)\rightarrow \StructureSheaf{S^{[2]}}[4]\otimes_\CC 
\Hom\left(\StructureSheaf{Z}(k),\StructureSheaf{S^{[2]}}[4]\right)^*
\end{equation}
in degrees $0$ and $1$.
\end{lem}

\begin{proof}
The exact sequence (\ref{eq-exact-triangle})
follows from the Snake Lemma and the commutative diagram with short exact rows the second of which is (\ref{eq-short-exact-sequence-of-E-k})
\[
\xymatrix{
0 \ar[r] &  G'_k \ar[r] \ar[d] & H^0(Z,\StructureSheaf{Z}(k+1))\otimes_\CC\StructureSheaf{S^{[2]}} \ar[r]^{a_k'} \ar[d] & F_k \ar[r] \ar[d]_{id}
&0
\\
0\ar[r] & E_k \ar[r] &
\left[
H^0(Z,\StructureSheaf{Z}(k+1))\otimes_\CC \StructureSheaf{S^{[2]}}
\right]
\oplus
\left[
H^0(Z,\StructureSheaf{Z}(k))\otimes_\CC \StructureSheaf{S^{[2]}}(-\delta)
\right]
\ar[r]_-{a_k} & F_k \ar[r] &0.
}
\]
The exactness of both rows follows from Proposition \ref{prop-a-locally-free-image-of-a-Lagrangian-line-bundle}.
$E_k$ is isomorphic to 
$\Phi^{[2]}_{\P.\chi}(\StructureSheaf{Z}(k))$ and $\StructureSheaf{S^{[2]}}(-\delta)$ is isomorphic to
$\Phi^{[2]}_{\P.\chi}(\StructureSheaf{S^{[2]}}[4])$, 
by Lemma \ref{lemma-shifts-of-structure-sheaf-and-delta-line-bundle-interchanged}. 
The morphism $\eta'$ is the image via $\Phi^{[2]}_{\P.\chi}$ of the natural morphism in
(\ref{eq-preimage-of-G'-k}),
where we identify $H^0(Z,\StructureSheaf{Z}(k))$ and $\Hom\left(\StructureSheaf{Z}(k),\StructureSheaf{S^{[2]}}[4]\right)^*$
via Serre's duality and the triviality of the canonical line bundle of $S^{[2]}$.
\end{proof}

\begin{lem}
We have the exact sequence
\[
0\rightarrow G''_k\rightarrow E_k \RightArrowOf{\eta''} H^0(Z,\StructureSheaf{Z}(k+1))\otimes_\CC\StructureSheaf{S^{[2]}}
\rightarrow F_k/F_{k-1}\otimes\StructureSheaf{S^{[2]}}(-\delta)\rightarrow 0.
\]
The homomorphism $\eta''$ is the natural one under the identification of $H^0(Z,\StructureSheaf{Z}(k+1))$
with $\Hom(E_k,\StructureSheaf{S^{[2]}})^*$.
\end{lem}

\begin{proof}
The exact sequence is obtained via the Snake Lemma after replacing the first row in the diagram in Lemma \ref{lem-homomorphism-from-E-k-to-almost-trivial-vb} by 
\[
0\rightarrow G''_k\rightarrow H^0(Z,\StructureSheaf{Z}(k))\otimes_\CC\StructureSheaf{S^{[2]}}(-\delta)
\rightarrow F_{k-1}\otimes\StructureSheaf{S^{[2]}}(-\delta)\rightarrow 0.
\]
The trivial line bundle is isomorphic to
$\Phi^{[2]}_{\P.\chi}(\StructureSheaf{S^{[2]}}(-\delta)[4])$, 
by Lemma \ref{lemma-shifts-of-structure-sheaf-and-delta-line-bundle-interchanged}. 
The morphism $\eta''$ is the image via $\Phi^{[2]}_{\P.\chi}$ of the natural homomorphism
\[
\StructureSheaf{Z}(k)\rightarrow \StructureSheaf{S^{[2]}}(-\delta)[4]\otimes_\CC 
\Hom(\StructureSheaf{Z}(k),\StructureSheaf{S^{[2]}}(-\delta)[4])^*,
\]
where we identify $H^0(Z,\StructureSheaf{Z}(k+1))$ with $\Hom(\StructureSheaf{Z}(k),\StructureSheaf{S^{[2]}}(-\delta)[4])^*$
via Serre's duality, the triviality of the canonical line bundle of $S^{[2]}$ and Lemma \ref{lem-modular-interpretation-of-the-contraction-of-A}(\ref{lemma-item-restriction-of-lambda-and-delta-are-equal}).
\end{proof}


\hide{
%
\section{Appendix: Non rigidity of tensor products of rigid very modular vector bundles}
Let $X=S^{[2]}$, for a $K3$ surface $S$. Let $G_i$ be a spherical vector bundle on $S$, $i=1,2$. 
 Set $F_i=G_i[2]^\pm$, $i=1,2$ and let $r_i$ be the rank of $F_i$. 
 \begin{lem}
 $\Ext^1(F_1\otimes F_2,F_1\otimes F_2)$ does not vanish, if $r_1>0$ and $r_2>0$.
 \end{lem}
 \begin{proof}
 Let us calculate the Mukai pairing 
$(v(F_1\otimes F_2),v(F_1\otimes F_2))$. 
We have
\[
\int_Xc_2(X)^2=828, \ \ \ td_X=1+\frac{1}{12}c_2(X)+3[pt], 
\ \ \ \sqrt{td_X}=1+\frac{1}{24}c_2(X)+\frac{25}{32}[pt].
\] 
If $\kappa(F_1)=r_1+a_1c_2(X)+b_1[pt]$ and $\kappa(F_2)=r_2+a_2c_2(X)+b_2[pt]$, then
\begin{eqnarray*}
w&:=&\kappa(F_1\otimes F_2)\sqrt{td_X}
\\
&=& (r_1r_2+[r_1a_2+r_2a_1]c_2(X)+\left[r_1b_2+r_2b_1+828a_1a_2\right][pt])\sqrt{td_X}
\\
&=& r_1r_2+\left(r_1a_2+r_2a_1+\frac{r_1r_2}{24}\right)c_2(X)+\left([r_1a_2+r_2a_1]\frac{69}{2}+[r_1b_2+r_2b_1+828a_1a_2]+\frac{25}{32}r_1r_2\right)[pt],
\end{eqnarray*}
and, since $\kappa(F)=\kappa(F^*)$ if $\kappa(F)$ is $\bar{\LieAlg{g}}$-invariant (so that $\kappa_3(F)$ vanishes since $H^6(X,\RationalNumbers)$ is an irreducible $\bar{\LieAlg{g}}$-module),  
\begin{eqnarray*}
(w,w)&=&-\int_X\kappa(F_1\otimes F_2)^2td_X
\\
&=&-2r_1r_2\left([r_1a_2+r_2a_1]\frac{69}{2}+[r_1b_2+r_2b_1+828a_1a_2]+\frac{25}{32}r_1r_2\right)-828\left(r_1a_2+r_2a_1+\frac{r_1r_2}{24}\right)^2.
\end{eqnarray*}
We have
\[
3=\int_X\kappa(F_i)^2td_X=(2r_ib_i+828a_i^2)+138 r_ia_i+3r_i^2.
\]
Hence, $b_i=\frac{3}{2r_i}\left(1-276a_i^2-46a_ir_i-r_i^2\right)$.
Set $e_i=a_i/r_i$. Then 
\[
\frac{b_i}{r_i}=\frac{3}{2}\left(\frac{1}{r_i^2}-276e_i^2-46e_i-1\right).
\]
We get
\begin{equation}
\label{eq-non-rigidity-n=2}
(v(F_1\otimes F_2),v(F_1\otimes F_2))=(w,w)=3r_1^2r_2^2\left(1-\frac{1}{r_1^2}-\frac{1}{r_2^2}\right),
\end{equation}
which is positive if  $r_1>1$ and $r_2>1$.
Serre's duality yields 
\[
(w,w)=-2\dim\Hom(F_1\otimes F_2,F_1\otimes F_2)+2\dim\Ext^1(F_1\otimes F_2,F_1\otimes F_2)-\dim\Ext^2(F_1\otimes F_2,F_1\otimes F_2).
\]
Hence, $\Ext^1(F_1\otimes F_2,F_1\otimes F_2)$ does not vanish if $r_1>1$ and $r_2>1$.
\end{proof}

\begin{rem}
Assume that $F_1\otimes F_2$ is simple, i.e., $\dim \End(F_1\otimes F_2)=1$.
A choice of a hyperk\"{a}hler structure on $X$ determines an action of the algebra $\HH$ of quaternions on $\Ext^1(F_1\otimes F_2,F_1\otimes F_2)$, by \cite[Prop. 3.2 and Cor. 4.3]{verbitsky-eprint-version-of-JAG96}. In particular, it is even dimensional. Hence,
$(w,w)+2+\dim \Ext^2(F_1\otimes F_2,F_1\otimes F_2)$ is divisible by $4$. Now $r_i=\rho_i^2$, so 
$(w,w)\equiv\left\{\begin{array}{ccl}
0 & \mbox{if} & r_1 \ \mbox{and}\ r_2 \ \mbox{are even}
\\
1 & \mbox{if} & r_1 \ \mbox{or}\ r_2 \ \mbox{is odd}
\end{array}\right.$ (mod $4$), by Equation (\ref{eq-non-rigidity-n=2}).
Hence,
\[
\dim \Ext^2(F_1\otimes F_2,F_1\otimes F_2)
\equiv\left\{\begin{array}{ccl}
2 & \mbox{if} & r_1 \ \mbox{and}\ r_2 \ \mbox{are even}
\\
1 & \mbox{if} & r_1 \ \mbox{or}\ r_2 \ \mbox{is odd}
\end{array}\right. \ \ (\mbox{mod} \ 4).
\]
Assume there exists an elliptic fibration $\pi:S\rightarrow \PP^1$ and let  $p:S^{[2]}\rightarrow \PP^2$ the associated lagrangian fibration. Denote by $C_x$ the fiber of $\pi$ over $x\in\PP^1$. Let $x$ and $y$ be two distinct points in $\PP^1$ with $C_y$ and $C_y$ smooth and denote by $x+y$ the  point of $\PP^2$ corresponding to the length two subscheme of $\PP^1$. The restriction of $F_i$ to the fiber $C_x\times C_y$ pf $p$ is $G_{\restricted{i}{C_x}}\boxtimes G_{\restricted{i}{C_y}}$.
The restriction of $F_1\otimes F_2$ is thus $(G_1\otimes G_2\restricted{)}{C_x}\boxtimes (G_1\otimes G_2\restricted{)}{C_y}$. Now, if $\gcd(\rho_1,\rho_2)$ divides the rank and degree of $(G_1\otimes G_1\restricted{)}{C_x}$. Hence,
if $\rho_1$ and $\rho_2$ are both even, then $(G_1\otimes G_1\restricted{)}{C_x}$ has even rank and even degree, and is thus polystable but unstable.
\end{rem}

\begin{rem}
The Lie algebra $\bar{\LieAlg{g}}$ acts on $H^*(X,\QQ)$ preserving the grading and leaving the cup product infinitesimally invariant, i.e.,
$\xi(\alpha\beta)=\xi(\alpha)\beta+\alpha\xi(\beta)$, for all $\alpha,\beta\in H^*(X,\QQ)$ and all $\xi\in\bar{\LieAlg{g}}$. If $F$ and $G$ are objects in $D^b(X)$ of non-zero rank with $\bar{\LieAlg{g}}$-invariant $\kappa(F)$ and $\kappa(G)$, then so is $\kappa(F\otimes G)=\kappa(F)\kappa(G)$. Hence, if $F$ and $G$ are locally free very modular sheaves, then so is $F\otimes G$.
The latter need not be rigid, even if $F$ and $G$ are. If, for example, $F$ and $G$ are spherical locally free sheaves on a $K3$ surface $X$ with $v(F)=r_1\alpha+\lambda_1+s_1\beta$ and $v(G)=r_1\alpha+\lambda_1+s_1\beta$, where 
$(v(F),v(F))=(v(G),v(G))=-2$, so that $(\lambda_i,\lambda_i)=2r_is_i-2$,  then
$v(F\otimes G)=r_1r_2\alpha+(r_1\lambda_2+r_2\lambda_1)+(r_1s_2+r_2s_1+(\lambda_1,\lambda_2)-r_1r_2)\beta$ and
\begin{equation}
\label{eq-non-rigidity-n=1}
(v(F\otimes G),v(F\otimes G))=2[(r_1^2-1)(r_2^2-1)-1]=2r_1^2r_2^2\left(1-\frac{1}{r_1^2}-\frac{1}{r_2^2}\right),
\end{equation}
which is larger than  $-2$, whenever $r_1>1$ and $r_2>1$. Hence, $\dim \Ext^1(F\otimes G,F\otimes G)>0$, if $r_1>1$ and $r_2>1$. 

Consider, for example, a $K3$ surface $X$ with a cyclic Picard group generated by an ample class $\lambda$ with $(\lambda,\lambda)=2$. Write $(r,t\lambda,s)$ for $r\alpha+t\lambda+s\beta$.
Let $F$ and $G$ be $\lambda$-stable sheaves with $v(F)=(2,\lambda,1)$ and $v(G)=(5,2\lambda,1)$. Then both are spherical, hence locally free, and the Mukai vector of the $\lambda$-polystable vector bundle $F\otimes G$ is $v(F\otimes G)=(10,9\lambda,1)$, which is primitive as well, and so $F\otimes G$ is $\lambda$-stable with
$(v(F\otimes G),v(F\otimes G))=142$. Hence, $F\otimes G$ belongs to a $144$-dimensional moduli space.
\end{rem}

\begin{rem}
Equations (\ref{eq-non-rigidity-n=2}) and (\ref{eq-non-rigidity-n=1}) suggest the following question. Do 
the vector bundles $F_i=G_i[n]^\pm$ of rank $r_i$
over the Hilbert scheme $S^{[n]}$ of a $K3$ surface $S$,  associated to spherical vector bundles $G_i$ over $S$, $i=1,2$, satisfy the following equation?
\[
-\chi\left([F_1\otimes F_2]^\vee\otimes F_1\otimes F_2\right)=
(n+1)r_1^2r_2^2\left(1-\frac{1}{r_1^2}-\frac{1}{r_2^2}\right).
\]
\end{rem}
}


\end{document}